%% file: Lesnick-Thesis-For-Arxiv.tex
\documentclass[12pt]{report}

\usepackage[online]{suthesis-2e-Arxiv}

 \figurespagefalse
 \tablespagefalse

\input{./Preamble_2.tex}
\begin{document}
   \blankthesis
    \title{Multidimensional Interleavings \\
            and Applications to Topological Inference}
    \author{Michael Phillip Lesnick}
    \dept{Institute for Computational and Mathematical Engineering}
    \principaladviser{Gunnar Carlsson}
    \firstreader{Leonidas Guibas}
    \secondreader{Dmitriy Morozov}
 
    \beforepreface
    \prefacesection{Abstract}

       \input{./Thesis_Abstract.tex}

\newpage

    \prefacesection{Acknowledgements}
          \input{./acknowledgements_arxiv}

    \afterpreface

 \chapter{Introduction}

 \input{./Thesis_Introduction_6}

 \chapter[Interleavings on Multi-D Persistence Modules]{Interleavings on Multidimensional Persistence Modules}

. 
  \input{./Part_I/T_Preliminaries_7}

  \input{./Part_I/T_Structure_Theorem_For_Well_Behaved_B_1-modules_3}

  \input{./Part_I/T_Interleaving_Equals_Bottleneck_3}

  \input{./Part_I/T_Extrinsic_Characterization}

  \input{./Part_I/T_Geometric_Preliminaries_2}

  \input{./Part_I/T_Stability}

  \input{./Part_I/T_Optimality_Generalities_2}

  \input{./Part_I/T_Optimality_Specifics_9}

  \input{./Part_I/T_Computation}

 \input{./Part_I/T_Discussion_I}
 
 \chapter[Interleavings on Multi-D Filtrations]{Strong and Weak Interleavings on Multidimensional Filtrations}

 \input{./Part_II/T_Filtrations_Revisited}

 \input{./Part_II/T_Interleavings_Of_Filtrations}

 \chapter[Approximation and Inference Results]{Approximation and Inference Results for Multidimensional Filtrations}
 \input{./Part_II/T_Preliminaries_3}

 \input{./Part_II/T_Deterministic_Approximation}

 \input{./Part_II/T_Main_Inference_Results}

 \input{./Part_II/T_Inference_Discussion}

 \chapter{Conclusion}
 \input{./Part_II/T_Discussion_Final}

\appendix
\chapter[Minimal Presentations]{Minimal Presentations of Multidimensional Persistence Modules}
\input{./Part_I/T_Coherence_2}
\input{./Part_I/T_Minimal_Presentations_Algebra_3}


    \bibliographystyle{plain}
    \bibliography{Thesis_Refs}


\end{document}

%% file: Preamble_2.tex
\usepackage{amsmath,amscd} 
\usepackage{amsthm}
\usepackage{graphicx}
\usepackage{amsfonts}
\usepackage{float}
\usepackage{mdwlist}
\usepackage{appendix}
\newtheorem{thm}{Theorem}[section]
\newtheorem*{ThmNoNum}{Theorem}
\newtheorem*{definition}{Definition}

\newtheorem{prop}[thm]{Proposition}
\newtheorem{lem}[thm]{Lemma} 
\newtheorem{cor}[thm]{Corollary}
\newtheorem*{CorNoNum}{Corollary}
\newtheorem{conjecture}[thm]{Conjecture}

\theoremstyle{definition}

\theoremstyle{definition}
\newtheorem{example}[thm]{Example}

\theoremstyle{definition}
\newtheorem{remark}[thm]{Remark}

\newcommand{\supp}[0]{\text{\rm supp}}
\newcommand{\obj}[0]{\textup{obj}}
\newcommand{\rank}[0]{\text{\rm rank}}
\newcommand{\im}[0]{\text{\rm im}}
\newcommand{\coker}[0]{\text{\rm coker}}
\newcommand{\Z}[0]{\mathbb Z}
\newcommand{\R}[0]{\mathbb R}
\newcommand{\NN}[0]{\mathbb N}
\newcommand{\Q}[0]{\mathbb Q}
\newcommand{\A}[0]{\mathcal A}
\newcommand{\B}[0]{\mathcal B}

\newcommand{\D}[0]{\mathcal D}

\newcommand{\F}[0]{\mathcal F} 
\newcommand{\G}[0]{\mathcal G}
\renewcommand{\H}[0]{\mathcal H}
\renewcommand{\L}[0]{\mathcal L}
\renewcommand{\P}[0]{\mathcal P}
\newcommand{\M}[0]{\mathbf M}
\newcommand{\N}[0]{\mathbf N}
\renewcommand{\Im}[0]{\mathrm{Im}}

\newcommand{\T}[0]{\mathcal T}
\newcommand{\W}[0]{\mathcal W}
\newcommand{\Y}[0]{\mathcal Y}
\newcommand{\J}[0]{\mathcal J}
\newcommand{\V}[0]{\mathcal V}
\newcommand{\NNN}[0]{\mathcal N}
\newcommand{\Rips}[0]{\textup{Rips}}
\newcommand{\e}[0]{\mathbf e}
\newcommand{\bC}[0]{\mathbf C}
\newcommand{\newspan}[0]{{\rm span}}
\newcommand{\Cech}[0]{\textup{\v{C}ech }}
\newcommand{\SCe}[0]{S\check{C}e}
\newcommand{\Ce}[0]{\check{C}e}
\newcommand{\Top}[0]{\mathbf{Top}}
\newcommand{\filt}[0]{\textup{filt}}
\newcommand{\RR}[0]{\mathcal{R}}
\renewcommand{\mod}[0]{\textup{mod}}

\newcommand{\Sub}[0]{{\rm Sub}}
\newcommand{\colim}[0]{{\rm colim}}
\newcommand{\id}[0]{\rm id}
\newcommand{\glazcitationhack}[0]{\cite[Theorem 2.3.3]{glaz1989commutative}}

%% file: Thesis_Abstract.tex
This thesis concerns the theoretical foundations of persistence-based topological data analysis.  The primary focus of the work is on the development of theory of topological inference in the multidimensional persistence setting, where the set of available theoretical and algorithmic tools has remained comparatively underdeveloped, relative to the 1-D persistence setting.  The thesis establishes a number of theoretical results centered around this theme, some of which are new and interesting even for 1-D persistent homology.  In addition, this work presents theory of topological inference formulated directly on the (topological) level of filtrations rather than on the (algebraic) level of persistent homology modules.

The main mathematical objects of study in this work are {\bf interleavings}.  These are tools for quantifying the similarity between multidimensional filtrations and persistence modules.  They were introduced for 1-D filtrations and persistence modules by Chazal et al. \cite{chazal2009proximity}, where they were used to prove a strong and very useful generalization of the stability of persistence result of \cite{cohen2007stability}; we generalize the definition of interleavings appearing in \cite{chazal2009proximity} in several directions and use these generalizations to define pseudometrics on multidimensional filtrations and multidimensional persistence modules called {\bf interleaving distances}.

The first part of this thesis, adapted from the preprint \cite{lesnick2011optimality}, studies in detail the theory of interleavings and interleaving distances on multidimensional persistence modules.  We present six main results about the interleaving distance. 

First, we show that in the case of 1-D persistence, the interleaving distance is equal to the bottleneck distance on tame persistence modules.  

Second, we prove a theorem which implies that the restriction of the interleaving distance to finitely presented multidimensional persistence modules is a metric.  The same theorem, together with our first result, also yields a converse to the algebraic stability theorem of \cite{chazal2009proximity}; this answers a question posed in that paper.  

Third, we present an ``extrinsic" characterization of interleaved pairs of multidimensional persistence modules which makes transparent the sense in which interleaved modules are algebraically similar.  This characterization turns out to hold for a definition of interleavings of multidimensional persistence modules rather more general than that which we need to define the interleaving distance; the more general form of our result is an important ingredient in inferential theory we develop in the second part of this thesis.

Fourth, we observe that the interleaving distance is stable in four senses analogous to those in which the bottleneck distance is known to be stable. 

Fifth, we introduce several notions of optimality of metrics on persistence modules and show that when the underlying field is $\Q$ or a field of prime order, the interleaving distance is optimal with respect to one of these notions. This optimality result, which is new even for 1-D persistence, is the central result of the first part of this thesis. We also prove that a version of this result holds for ordinary persistence modules over any field, provided we restrict attention to a class of well behaved ordinary persistence modules containing the finitely presented ones.  

Sixth, we show that the computation of the interleaving distance between two finitely presented multidimensional persistence modules M and N reduces to deciding the solvability of $O(\log m)$ systems of multivariate quadratic equations, each with $O(m^2)$ variables and $O(m^2)$ equations, where m is the total number of generators and relations in a minimal presentation for $M$ and a minimal presentation for $N$.

In the second part of the thesis, we define interleavings and interleaving distances on multidimensional filtrations, and present theoretical results for these.  We then use interleavings and interleaving distances on multidimensional filtrations to formulate and prove several analogues of a topological inference theorem of \cite{chazal2009persistence} in the multidimensional setting, and directly on the level of filtrations.  In particular, we employ {\it localization of categories}, a standard construction in homotopy theory, to define and study homotopy theoretic versions of interleavings and the interleaving distance on multidimensional filtrations, which we call {\bf weak interleavings} and the {\bf weak interleaving distance}.  We formulate our main inference results using weak interleavings and the weak interleaving distance.
    
To describe these results in detail, let $\gamma: \R^m\to \R$ be a probability density function on $\R^m$, and for $z\in \Z_{\geq 0}$ let $T_z$ be an i.i.d. sample of size $z$ of a probability distribution with density $\gamma$.  Let $F^{\SCe}_z$ be the random \Cech bifiltration with vertices $T_z$, filtered by the superlevelsets of $E(T_z)$, where $E$ is a density estimator, and by the usual scale parameter for \Cech complexes.

Our first main inference result is that under mild conditions on $\gamma$ and $E$, $F^{\SCe}_z$ converges in probability (with respect to the weak interleaving distance, and as $z\to \infty$) to a bifiltration constructed directly from $\gamma$, which we call the {\bf superlevel-offset bifiltration}. 

Our second main inference result is an analogue of the first result for Vietoris-Rips bifiltrations, filtered by the superlevelsets of $E(T_z)$ and by the usual scale parameter for Vietoris-Rips complexes.  

We also present analogues for each of these results for probability density functions defined on Riemannian manifolds.  

These inference results on the level of filtrations yield as corollaries analogous results on the level of persistent homology, formulated in terms of the interleaving distance and interleavings on persistence modules.  Our extrinsic characterization of interleavings from the first part of our thesis yields concrete interpretations of these corollaries as statements about the similarity between presentations of persistent homology modules.

%% file: acknowledgements_arxiv.tex
My discussions with my adviser Gunnar Carlsson have catalyzed the research presented here in several key ways.  In our early conversations about multidimensional persistent homology, Gunnar impressed on me the value of bringing the machinery of commutative algebra to bear on the study of multidimensional persistence.  This idea has very much informed my work, and is manifest in the content of the first part of this thesis.  Gunnar's input has also shaped the content of this thesis in more concrete---and important---ways: There were several instances in the course of my doing this research where Gunnar's feedback was instrumental in my finding ways to formulate my ideas that were clean and consonant with the broader mathematical culture of topology and commutative algebra.  Specifically, I would like to acknowledge Gunnar's input on three key places in this work.
\begin{enumerate*}
\item Gunnar suggested the use of the theory of coherent rings in this thesis.  They are used in the proof of Theorem~\ref{MinimalPresentationTheorem}.
\item Gunnar offered valuable input on how to cleanly explain the functorial relationship between $B_n$-persistence modules and ${\mathbf A}_n$-persistence modules (as defined in Section~\ref{Sec:MultidimensionalPersistenceModules}); his feedback was helpful in my finding the right way to exposit the material in Section~\ref{WellBehavedStructThmSection}.
\item Gunnar suggested that the definition of the homotopy category of $u$-filtrations appearing in this thesis be formulated using localization of categories.  This suggestion shaped the mathematics of Chapters 3 and 4 of this thesis in a significant way.    
\end{enumerate*}

Patrizio Frosini, Steve Oudot, Fred Chazal, and Henry Adams also offered helpful feedback on parts of this work.

This research was supported by an NDSEG fellowship and a research assistantship funded though ONR grant number N00014-09-1-0783.

%% file: Thesis_Introduction_6.tex
This thesis concerns the theoretical foundations of persistence-based topological inference.  It focuses in particular on topological inference in the multidimensional setting and at the level of filtrations.  

In this introduction, we offer some context and motivation for the theory developed in this thesis and present an overview of our results.  The chapter is divided into six sections: Section~\ref{Sec:StatisticalFoundationsOfTDA} introduces the broader context for the mathematics of this thesis and discusses the need for further development of the statistical foundations of topological data analysis; Section~\ref{Sec:OverviewInferentialTheory} introduces the problem of developing inferential theory for multidimensional filtrations, and presents background and motivation for the problem; Sections~\ref{Sec:Chapter2Overview}-\ref{Sec:Chapter4Overview} present detailed overviews of the results of chapters 2-4 of this thesis; and Section~\ref{Sec:InferenceAtLevelOfFiltrations} closes this introduction with discussion of our motivation for considering in this thesis persistence based inference directly at the level of filtrations.

This introduction is not intended as a tutorial; we will assume that the reader has some familiarity with some of the basic terminology and ideas of applied topology and topological data analysis.  See the reviews \cite{edelsbrunner2008persistent,ghrist2008barcodes,carlsson2009topology} and the textbook \cite{edelsbrunner2010computational} for treatment of the basics, and Sections~\ref{SectionAlgebraicPreliminaries},~\ref{Sec:FiltrationsRevisited}, and~\ref{Sec:InferencePreliminaries} for foundational definitions.

\section{On The Need for Firm Statistical Foundations of Topological Data Analysis}\label{Sec:StatisticalFoundationsOfTDA}
To explain the mathematical context of this thesis, it will be useful for us to begin by formulating definitions of topological inference and topological data analysis.  

Recall first that in statistics, we distinguish between {\it descriptive statistics} and {\it statistical inference}.  Descriptive statistics, as the name suggests, is that part of statistics concerned with defining and studying descriptors of data.  It involves no probability theory and aims simply to offer tools for describing, summarizing, and visualizing data.  Statistical inference, on the other hand, concerns the more sophisticated enterprise of estimating descriptors of an unknown probability distribution from random samples of the distribution.  The theory and methods of statistical inference are built on the tools of descriptive statistics: The estimators considered in statistical inference are of course, when stripped of their inferential interpretation, merely descriptors of data.

We define {\bf descriptive topological data analysis (descriptive TDA)} to be the branch of descriptive statistics which uses topology to define and study qualitative descriptors of data sets.

We define {\bf topological inference} to be the branch of statistical inference which
\begin{enumerate*}
\item uses topology to define qualitative descriptors of probability distributions.
\item develops and studies estimators for inferring such descriptors from finite samples of the distributions.
\end{enumerate*}
We define {\bf topological data analysis (TDA)} to refer collectively to descriptive TDA, topological inference, and the applications of these to science and engineering.

In the last 10 years the TDA community has introduced a number of novel tools for descriptive TDA \cite{zomorodian2005computing,carlsson2009theory,carlsson2009zigzag,singh2007topological,de2009persistent,adams2011morse}, and has begun developing applications of these in science and engineering.  There is great interest in applying these tools to the study of random data, and a good deal of work has already been done in this direction \cite{carlsson2008local,singh2008topological,nicolau2011topology,chazal2009persistence}.  

In the last few years, there also has been some important progress in topological inference \cite{niyogi2009finding,chazal2009persistence,chazal2011geometric}.  However, the development of tools for descriptive TDA has, as a rule, outpaced the development of the theory of topological inference supporting the use of these tools in the study of random data.  Indeed, statistical foundations for many of the most discussed tools in TDA---for example, Mapper \cite{singh2007topological}, circle valued coordinatization \cite{de2009persistent}, and persistent homology of Vietoris-Rips (multi-)filtraitons \cite{carlsson2009theory}---have either not been laid out at all, or have only partially (and recently) been laid out.  In many cases, there is no theoretical framework in place for interpreting these descriptors of data as appropriately behaved estimators of descriptors of probability distributions.  (An important exception is the work of Chazal et al. \cite{chazal2009persistence} on estimating the superlevelset persistent homology of density functions using Rips complexes; we'll discuss this work in detail in Section~\ref{Sec:ResultOfChazal}.)   

Additionally there is, to date, little to no discussion in the TDA literature on how to compute theoretically sound measures of confidence for any estimator of a topological descriptor.  This last gap is an especially critical one in the statistical foundations of TDA: It is a basic principle of statistical inference that an estimate of some descriptor of a probability distribution is only meaningful insofar as we also have a good measure of confidence for that estimate.

In these senses, the statistical foundations of TDA remain quite underdeveloped; much further progress on the theory of topological inference is needed before the tools of TDA can sit comfortably amongst the more conventional tools in a statistician's toolbox.

Carrying out this work is central to the program of fully realizing TDA as a data analysis methodology of value in science and engineering.  Those of us working in TDA agree that, broadly speaking, TDA offers a very natural and powerful set of tools for qualitative statistical inference, even if the mathematical foundations of these tools are still in development.  We share a common goal of seeing the machinery of TDA mature to the point that it can make an impact on science and engineering commensurate with what we believe its potential to be.   For that to happen, a great deal of work needs to be done in a number of directions, but one of the most direct ways we as researchers in TDA can hasten the integration of TDA into the modern data analysis pipeline is by properly fleshing out the statistical foundations of TDA.  After all, if the statistical foundations of topological data analysis were firmer, the output of the tools of topological data analysis on random input would be more meaningful and more useful, and presumably statisticians, scientists, and engineers would be more receptive to the tools and more inclined to invest the time to develop new applications of them.  

This perspective motivates the work of this thesis.  Indeed, this work is an effort to contribute to the program of fleshing out the statistical foundations of TDA.  The program is a broad one, however, and the results of this thesis amount only to a narrow slice of what is needed to really put TDA on firm statistical footing.  Nevertheless, it is my hope that the results presented here can offer clarity on some basic issues in topological inference and TDA and open the door for further progress in the development of the theory.

\section{Overview: Inferential Theory for Multidimensional Filtrations}\label{Sec:OverviewInferentialTheory}

We focus here on the theoretical foundations of {\it persistence-based topological data analysis}, that central (though not all-encompassing) branch of topological data analysis that considers topological descriptors and random variables defined using filtrations and persistent homology.  Our aims in particular are, first, to develop an inferential theory for multidimensional persistence, and, second, to develop inferential theory at the level of filtrations, rather than merely at the level of persistent homology modules.

\subsection{Context and Motivation for Our Inference Results}

The history of the problem of developing an inferential theory for multidimensional persistent homology dates back several years, to the original paper on multidimensional persistent homology \cite{carlsson2009theory}.  Motivated by needs arising in their study of natural scene statistics \cite{carlsson2008local}, in \cite{carlsson2009theory} Carlsson and Zomorodian proposed the use of {\it Vietoris-Rips bifiltrations} to probe the qualitative structure of point cloud data of nonuniform density in exploratory data analysis applications.  This proposal was motivated by the idea that when the point cloud data set $T_z$ is obtained as an i.i.d. sample of size $z$ of some probability distribution with density function $\gamma:\R^m\to \R$, the random Vietoris-Rips  bifiltration with vertices $T_z$, filtered by superlevelsets of a density estimator $E$ and by the usual scale parameter for a Vietoris-Rips complex, should encode topological information about $\gamma$.  

A central aim of this thesis is to put this idea on firm mathematical footing.

Our approach builds in an essential way on recent work on persistence-based topological inference by Chazal, Guibas, Oudot, and Skraba \cite{chazal2009persistence}.  In that work, the authors build on theory developed in their earlier papers \cite{chazal2009proximity,chazal2009analysis} to prove a result \cite[Theorem 5.1]{chazal2009persistence} on the inference of the 1-D (superlevelset) persistent homology of a probability density function $\gamma$ from a pair of filtered Vietoris-Rips complexes built on i.i.d. samples of a probability distribution with density $\gamma$.  The primary aim of the paper \cite{chazal2009persistence} is to leverage persistent homology to introduce a clustering algorithm with good theoretical properties; the exposition there is such that the inference result \cite[Theorem 5.1]{chazal2009persistence} plays a supporting role, serving as the engine for the development of theoretical guarantees on the clustering algorithm presented in that paper\footnote{Presumably, the authors make this expository choice because the inference result \cite[Theorem 5.1]{chazal2009persistence} is, modulo some details, a reasonably easy corollary of a deterministic result that is the centerpiece of the earlier paper~\cite{chazal2009analysis}.}.  Nevertheless, the result is a very significant one in of itself: It proves (implicitly, at least), for the first time, the consistency of an estimator of the superlevelset persistent homology of a density function on a Euclidean space.  Moreover, because that estimator is defined using only a pair of Vietoris-Rips complexes, it is simple and quite computable.  One disadvantage of the estimator introduced in \cite{chazal2009persistence}, however, is that its construction depends on a choice of scale parameter; as noted in \cite{chazal2009persistence}, a correct value for this parameter can be difficult to choose in practice.

It is natural to ask if and how the topological inference result \cite[Theorem 5.1]{chazal2009persistence} might adapt to the multidimensional setting, and whether it might also adapt to yield inference results directly on the level of filtrations rather than on the level of persistent homology.  In Chapter 4 of this thesis, we will show that such adaptations are indeed possible.  In fact, we will see that the multidimensional setting has the advantage that in constructing estimators there analogous to that considered in \cite{chazal2009persistence}, we do not need to choose a scale parameter.\footnote{For both the estimator considered in \cite{chazal2009persistence} and the estimators we consider here, it is necessary to chose a bandwidth parameter for a density estimator.  In the Euclidean case, the need to select a bandwidth parameter is not an unmanageable difficulty, at least in low dimensions: For density functions on Euclidean spaces, the problem of optimally selecting a bandwidth parameter for a kernel density estimator has been well studied---see \cite{scott1992multivariate}.
}

\subsection[A Result of Chazal et al. on the Inference of Sublevelset Persistent Homology Using Filtered Rips Complexes]{The Result of Chazal et al. on the Inference of Sublevelset Persistent Homology Using Rips Complexes}\label{Sec:ResultOfChazal}

To understand the context of our inference results, it is important to understand the result \cite[Theorem 5.1]{chazal2009persistence}.  We now present an asymptotic corollary of \cite[Theorem 5.1]{chazal2009persistence} which will serve as our mathematical point of departure in our pursuit of our own inference results.

To prepare for the result, we review some basics about density estimation.

\subsubsection{Density Estimation Preliminaries}

Let ${\mathcal D}(\R^m)$ denote the set of probability density functions on $\R^m$.  Define a density estimator $E$ on a $\R^m$ to be a sequence of functions $\{E_z:(\R^m)^z\to {\mathcal D}(\R^m)\}_{z\in \NN}$  such that for each $z\in \NN$, the restriction of  $E_z$ to any point in $\R^m$ is a measurable function from $(\R^m)^z$ to $\R$.  In formulating our results, we will consider pairs $(\gamma,E)$, where $\gamma$ is a density function and $E$ is a density estimator, for which one of the following two assumptions holds:

\begin{enumerate}
\item[\bf{A1}:] $E$ converges uniformly in probability to $\gamma$.
\item[\bf{A2}:] $E$ converges uniformly in probability to the convolution of $\gamma$ with some kernel function $K$.
\end{enumerate}

A1 is known to hold for kernel density estimators, for a wide class of kernels and density functions $\gamma$, provided the kernel width tends to $0$ at an appropriate rate as $z$ tends to infinity.  A2 is known to hold for the kernel density estimator with kernel $K$ (with fixed width as the number of samples $z$ varies) for a wide class of kernels $K$ and density functions $\gamma$.  See Section~\ref{Sec:DensityEstimators} for more details and references.

\subsubsection{Our Asymptotic Corollary of the Inference Result \cite[Theorem 5.1]{chazal2009persistence}}

The form of the result we quote here is different in several ways than that of \cite[Theorem 5.1]{chazal2009persistence}.  First, as we noted above, rather than recall the original form of \cite[Theorem 5.1]{chazal2009persistence}, we will present a tidier asymptotic corollary of it whose form is closer to that of our main inference results.  The proof of the asymptotic corollary, given the original form of the result, follows from an $\epsilon$-$\delta$ argument very similar to that used in the proof of our Theorem~\ref{Thm:CechFixedScaleConsistency}.  Second, because of the emphasis on clustering in \cite{chazal2009persistence}, the result \cite[Theorem 5.1]{chazal2009persistence} is stated only for $0^{th}$ persistent homology, but as the authors note, the result adapts immediately to higher persistent homology via the results of \cite{chazal2009analysis}.  We will present here a version of the asymptotic corollary which holds for $i^{th}$ persistent homology, $i\in \Z_{\geq 0}$.  Third, to minimize the technicalities in this introduction, we present the result for density functions defined on Euclidean space; the result adapts readily to Riemannian manifolds with sectional curvature bounded above and below.    

Let $\gamma:\R^m\to \R$ be a $c$-Lipchitz density function such that for $i\in \Z_{\geq 0}$, the superlevelset persistent homology module $H_i(-\gamma)$ is tame\footnote{See Section~\ref{Sec:MultidimensionalPersistenceModules} for the definition of a tame persistence module.  The tameness condition ensures that the bottleneck distance in the statement of  \cite[Theorem 5.1]{chazal2009persistence} is well defined.}; let $T_z$ be an i.i.d. random sample of a probability distribution with density $\gamma$ of size $z$; let $d^p$ be the restriction of some $L^p$ metric to $T_z$, for $1\leq p\leq \infty$; and let $E$ be a density estimator.

For $\delta>0$, let $F^{SR}(T_z,d^p,-E(T_z),\delta)$ be a Rips filtration on $(T_z,d^p)$ with fixed scale parameter $\delta$, filtered by sublevelsets of $-E(T_z)$.  For $i\in \Z_{\geq 0}$, let $H_i$ denote the $i^{th}$ persistent homology functor.  The inclusion \[F^{SR}(T_z,d^p,-E(T_z),\delta)\hookrightarrow F^{SR}(T_z,d^p,-E(T_z),2\delta)\] induces a homomorphism \[j_{z,\delta}:H_i(F^{SR}(T_z,d^p,-E(T_z),\delta))\to H_i(F^{SR}(T_z,d^p,-E(T_z),2\delta))\] of persistence modules. $\Im(j_{z,\delta_z})$ is then itself a persistence module.

For a 1-D persistence module $M$, let $M^+$ denote the submodule of $M$ generated by homogeneous summands $M_a$ of $M$ with $a\geq 0$ and let $R_0(M)=M/M^+$.

Let $d_B$ denote the bottleneck distance on tame persistence modules \cite{chazal2009proximity}.  If $\{X_z\}_{z\in \NN}$ is a sequence of random variables, and $X$ is a random variable such that $X_z$ converges in probability to $X$, we write $X_z\xrightarrow{\P} X$.
 
\begin{thm}[Inference result of \cite{chazal2009persistence} (asymptotic form)]\label{Thm:ChazalInferenceResult}
\mbox{}
\begin{enumerate*}
\item[(i)] If $(\gamma,E)$ satisfies A1 then $\exists$ a sequence $\{\delta_z\}_{z\in \NN}$ such that \[d_B(R_0(\Im(j_{\delta_z,z})),R_0(H_i(-\gamma)))\xrightarrow{\mathcal P} 0\] (as $z\to \infty$).

\item[(ii)] If $(\gamma,E)$ satisfies A2 for some kernel $K$ then $\exists$ a sequence $\{\delta_z\}_{z\in \NN}$ such that 
\[d_B(R_0(\Im(j_{\delta_z,z})),R_0(H_i(-\gamma*K)))\xrightarrow{\mathcal P} 0.\]
\end{enumerate*}
\end{thm}

\subsection{Distances and Other Notions of Proximity on Filtrations and Persistence Modules}

One of the main goals of this work is to adapt the consistency result Theorem~\ref{Thm:ChazalInferenceResult} to the multidimensional setting and to the level of filtrations.  Since Theorem~\ref{Thm:ChazalInferenceResult} is formulated in terms of the bottleneck distance, to carry out this adaptation we require analogues of this distance on multidimensional persistence modules and filtrations.  The main obstacle to adapting Theorem~\ref{Thm:ChazalInferenceResult} to the multidimensional setting is that the bottleneck distance of ordinary persistent homology does not admit a naive extension to the multidimensional setting.  Similarly, the main obstacle to adapting the theorem to the level of filtrations is that the bottleneck distance does not admit a naive adaptation to a distance on filtrations, even in 1-D.

Most of the effort of this thesis is put not directly towards proving inference results, but rather towards introducing and studying generalizations of the bottleneck distance to multidimensional filtrations and persistence modules which we need to formulate our inference results.  Indeed, one of the main themes of this work is that developing the theory of topological inference largely boils down to the problem of selecting and understanding distances (or more general notions of proximity, which strictly speaking, are not distances in any reasonable mathematical sense) on appropriate objects: Once we develop the mathematical vocabulary needed to properly formulate our inference results in this thesis, the proofs of the results turn out to be reasonably straightforward, given existing ideas in the literature (and particularly those put forth \cite{chazal2009persistence}).  The main mathematical challenge then is to make sense of the distances and notions of proximity between filtrations and between persistence modules with which we formulate the inference results, and to show that these notions are, in suitable senses, the right ones to use.
  
We formulate our inference results using {\bf interleavings} and distances defined in terms of interleavings called {\bf interleaving distances}.   As noted in the abstract, interleavings are tools for quantifying the similarity between multidimensional filtrations and persistence modules.   Interleaving distances are generalizations of the bottleneck distance to pseudometrics on multidimensional filtrations and on multidimensional persistence modules.  The most basic type of interleavings, called {\bf $\epsilon$-interleavings}, were introduced for 1-D persistence filtrations and persistence modules by Chazal et al. \cite{chazal2009proximity}.  This thesis introduces generalizations of the definition of $\epsilon$-interleavings given in \cite{chazal2009proximity} to definitions of $\epsilon$-interleavings on multidimensional filtrations and multidimensional persistence modules. We define interleaving distances in terms of these generalized $\epsilon$-interleavings.

In fact, in this thesis we generalize the notion of $\epsilon$-interleavings yet further to arrive at the definition of {\bf $(J_1,J_1)$-interleavings}, which we also sometimes call {\bf generalized interleavings}.  We apply interleaving distances and $(J_1,J_2)$-interleavings to formulate results on multidimensional inference at the level of filtrations.  $(J_1,J_2)$-interleavings allow us to quantify anisotropic and asymmetric similarities between filtrations and between persistence modules that cannot be completely described using interleaving distances.  They turn out to be the right mathematical tools for interpreting random Vietoris-Rips bifiltrations as inferential objects.  

\section{Chapter 2: Interleavings and Interleaving Distances for Multidimensional Persistence Modules}\label{Sec:Chapter2Overview}
In Chapter 2 of this thesis, we introduce interleavings and the interleaving distance $d_I$ on multidimensional persistence modules, and treat the theory of these in detail.  Chapter 2 of this thesis is adapted from the preprint \cite{lesnick2011optimality}; the content of Chapter 2 is very close to that of \cite{lesnick2011optimality}.  The most important difference is that we define interleavings in greater generality and extend our extrinsic characterization of $\epsilon$-interleavings in \cite{lesnick2011optimality} to an extrinsic characterization of generalized interleavings.  There are some other minor differences between Chapter 2 of the thesis and \cite{lesnick2011optimality}, but none of any great significance. 

\subsection{Results}

In Chapter 2, we present six main results on interleavings and the interleaving distance.  The first result, Theorem~\ref{InterleavingEqualsBottleneck}, shows that in the case of ordinary persistence, the interleaving distance is in fact equal to the bottleneck distance on tame persistence modules.  Our proof relies on a generalization of the structure theorem \cite{zomorodian2005computing} for finitely generated ordinary persistence modules to (discrete) tame persistence modules.  This generalization is proven e.g. in \cite{webb1985decomposition}.

Our second main result is Theorem~\ref{InterleavingThm}, which tells us that if $M$ and $N$ are two finitely presented persistence modules and $d_I(M,N)=\epsilon$ then $M$ and $N$ are $\epsilon$-interleaved.  As an immediate consequence of this theorem, we have Corollary~\ref{MetricCorollary}, which says that the interleaving distance restricts to a metric on finitely presented persistence modules.  Theorems~\ref{InterleavingEqualsBottleneck} and~\ref{InterleavingThm} together also yield Corollary~\ref{Cor:Converse}, a converse to the algebraic stability theorem of \cite{chazal2009proximity}.  The converse says, firstly, that if two tame 1-D persistence modules $M$ and $N$ are are distance $\epsilon$ apart in the bottleneck distance, then they are $\epsilon+\delta$-interelaved for any $\delta>0$.  Secondly, the converse says that if in addition $M$ and $N$ are each finitely presented (which is stronger than tameness), then $M$ and $N$ are in fact $\epsilon$-interleaved.  This result answers a question posed in \cite{chazal2009proximity}.

Our third main result is Theorem~\ref{GeneralAlgebraicRealization}, an ``extrinsic" characterization of $(J_1,J_2)$-interleaved pairs of persistence modules; it expresses transparently the sense in which $(J_1,J_2)$-interleaved persistence modules are algebraically similar.  Since $\epsilon$-interleavings are a special type of $(J_1,J_2)$-interleaving, and since the interleaving distance is defined in terms of $\epsilon$-interleavings, the result also yields an extrinsic characterization of the interleaving distance.  The result is reminiscent of the extrinsic characterization of the Gromov-Hausdorff distance, which expresses the Gromov-Hausdorff distance between two compact metric spaces in terms of the Hausdorff distance between embeddings of two metric spaces into a third metric space.  Roughly speaking, Theorem~\ref{GeneralAlgebraicRealization} characterizes $\epsilon$-interleaved pairs of persistence modules in terms of a distance between embeddings of presentations of such modules into a free persistence module.

As noted above, Theorem~\ref{GeneralAlgebraicRealization} was presented in \cite{lesnick2011optimality} only for the special case of $\epsilon$-interleavings.  The more general form of our result will be an important ingredient of the inferential theory we develop in the second part of this thesis.

Our fourth result is the observation that the interleaving distance is stable in four senses analogous to those in which the bottleneck distance is known to be stable.  These stability results, while notable, require very little mathematical work; two of the stability results turn out to be trivial, the third follows from a minor modification of an argument given in \cite{chazal2009gromov}, and the fourth admits a straightforward proof.

Our fifth main result, Corollary~\ref{CorOptimality}, is an optimality result for the interleaving distance.  It tells us that when the underlying field is ${\mathbb Q}$ or a field of prime order, the interleaving distance is stable in a sense analogous to that which the bottleneck distance is shown to be stable in \cite{cohen2007stability,chazal2009proximity}, and further, that the interleaving distance is, in a uniform sense, the most sensitive of all stable pseudometrics.  This ``maximum sensitivity" property of the interleaving distance is equivalent to the property that, with respect to the interleaving distance, multidimensional persistent homology preserves the metric on source objects as faithfully as is possible for any choice of stable pseudometric on multidimensional persistence modules; see Remark~\ref{InterpretationRemark} for a precise statement.  Our optimality result is new even for 1-D persistence.  In that case, it offers some mathematical justification, complementary to that of \cite{cohen2007stability,chazal2009proximity}, for the use of the bottleneck distance.  

In fact, provided we restrict attention to a class of well behaved ordinary persistence modules containing the finitely presented ones, the assumption that the underlying field is ${\mathbb Q}$ or a field of prime order is unnecessary; our Theorem~\ref{WellBehavedOptimality} gives an analogue of Corollary~\ref{CorOptimality} for this class of modules, over arbitrary fields. 

The main ingredient in the proof of Corollary~\ref{CorOptimality} is our characterization result Theorem~\ref{GeneralAlgebraicRealization}.  Using Theorem~\ref{GeneralAlgebraicRealization}, we present a constructive argument which shows that when the underlying field is ${\mathbb Q}$ or a field of prime order, $\epsilon$-interleavings can, in a suitable sense, be lifted to a category of $\R^n$-valued functions.  This is Proposition~\ref{RealizationProp}.  From this proposition, our optimality result follows readily.

Given our first five main results, it is natural to ask if and how the interleaving distance can be computed.  Our sixth main result  speaks to this question.  The result, which follows from Theorem~\ref{InterleavingAndQuadratics} and Proposition~\ref{PossibilitiesForInterleavingDistance}, is that the computation of the interleaving distance between two finitely presented multidimensional persistence modules $M$ and $N$ reduces to deciding the solvability of $O(\log m)$ systems of multivariate quadratic equations, each with $O(m^2)$ variables and $O(m^2)$ equations, where $m$ is the total number of generators and relations in a minimal presentation for $M$ and a minimal presentation for $N$.  This result is just a first step towards understanding the problem of computing the interleaving distance; we hope to address the problem more fully in future work.  

\subsubsection{A Note On Prior Work}
After making the results of the first part of this thesis publicly available (in form of the preprint \cite{lesnick2011optimality}), it was brought to our attention that in \cite{d2010natural}, d'Amico et al.  proved an optimality result for the bottleneck distance similar to the optimality results given here, for the special case of 0-dimensional ordinary persistent homology.  Our Theorem~\ref{WellBehavedOptimality} generalizes a slight weakening of \cite[Theorem 32]{d2010natural}; see Remark~\ref{FrosiniRemark}.

\section{Chapter 3: Strong and Weak Interleavings and Interleaving Distances for Multidimensional Filtrations}\label{Sec:Chapter3Overview}

In Chapter 3 we introduce and study interleavings and interleaving distances on multidimensional filtrations.  The first goal of Chapter 3 is to present theory for interleavings on filtrations analogous to that which Chapter 2 presents for interleavings on multidimensional persistence modules; the second goal of the chapter is to establish the technical foundations needed to formulate and prove topological inference results directly at the level of filtrations.    

We in fact introduce two types of interleavings on multidimensional filtrations, {\bf strong interleavings} and {\bf weak interleavings}; each induces a distance on multidimensional filtrations, the {\bf strong interleaving distance} $d_{SI}$ and the {\bf weak interleaving distance} $d_{WI}$, respectively.  We define {strong interleavings} in a way closely analogous to the way in which we define interleavings on persistence modules, and they share some of the favorable theoretical properties of interleavings on persistence modules.

However, strong interleavings turn out to be too sensitive for the purpose of proving inference results analogous to Theorem~\ref{Thm:ChazalInferenceResult} at the level of filtrations.  The reason, put somewhat coarsely, is that strong interleavings are sensitive to the topology of the spaces in filtrations up to homeomorphism; in developing theory of topological inference, it turns out to be necessary to work with a notion of interleaving that is sensitive only to the homotopy type of the spaces in the filtrations.  We thus introduce a homotopy theoretic variant of strong interleavings called {\bf weak interleavings}, which we define using {\it localization of categories}, a standard construction in homotopy theory.  We use weak interleavings and the weak interleaving distance to formulate our main inference results on the level of filtrations.

Because our inference results at the level of filtrations are formulated using weak interleavings rather than strong interleavings, we are less interested in understanding strong interleavings and the strong interleaving distance than their weak counterparts.  Nevertheless, we believe the theory of strong interleavings is worth developing, if only as a bridge to developing the theory of weak interleavings.

\subsection{Strong Interleavings and the Strong Interleaving Distance}

Section~\ref{Sec:StrongInterleavings} presents the definition and basic theory of strong interleavings.  We present three main results on the theory of strong interleavings.

Our first main result, Theorem~\ref{Thm:CharacterizationOfStrongInterleavings}, is a characterization of strongly interleaved pairs of {\bf filtrations of nested type}---filtrations of nested type are simply filtrations, each of whose transition maps is an injection; all filtrations we have occasion to consider in the development of our inferential theory in Chapter 4 are of nested type.  

Our characterization of strongly interleaved pairs of {\bf filtrations of nested type} is loosely analogous to our characterization Theorem~\ref{GeneralAlgebraicRealization} of interleaved pairs of persistence modules in Chapter 2.  Whereas our characterization of interleaved pairs of persistence modules in Chapter 2 is given in terms of {\it free covers} of persistence modules (i.e. the $0^{th}$ modules in free resolutions of persistence modules) our characterization of strongly interleaved pairs of {filtrations of nested type} is given in terms of {\it colimits} of filtrations.  

Theorems~\ref{Thm:FiltrationOptimality1} and~\ref{Thm:FiltrationOptimality2}, our second and third main results on the theory of strong interleavings, are optimality results  for $d_{SI}$, each analogous to our optimality result Corollary~\ref{CorOptimality} for $d_I$.  We prove these using Proposition~\ref{RealizationProp} (which, as noted above, is the main step in our proof of Corollary~\ref{CorOptimality}, our optimality result for multidimensional persistence modules) and our characterization result Theorem~\ref{Thm:CharacterizationOfStrongInterleavings}.    

\subsection{Weak Interleavings and the Weak Interleaving Distance}\label{Sec:WeakInterleavingsIntro}
In Section~\ref{Sec:WeakInterleavings}, we define and study weak interleavings and the weak interleaving distance.  Our definition of weak $(J_1,J_2)$-interleavings, via which we formulate our second main inference result, Theorem~\ref{Thm:RipsAsymptotics}, is perhaps the most interesting object of study of this thesis.  

We'll now discuss at a high level the motivation for the definition of weak interleavings, the use of localization of categories in formulating the definition, and the theory of weak interleavings.

\subsubsection{Motivation For Considering Weak Interleavings}

As we have noted above, strong interleavings are too sensitive for our use in formulating our inference results.  We begin by explaining the sense in which this is so: To adapt Theorem~\ref{Thm:ChazalInferenceResult} to the level of filtrations, we want a notion of distance $d$ between interleavings between filtrations such that if $f:X\to Y$ is a morphism of filtrations\footnote{See Section~\ref{Sec:FiltrationsDef1} for the definition of a morphism of filtrations.} which is a {\it levelwise homotopy equivalence}, meaning that each component map $f_a:X_a\to Y_a$ is a homotopy equivalence, then $d(X,Y)=0$.  The reason we want a distance on filtrations with this property is simple: To adapt the proof of Theorem~\ref{Thm:ChazalInferenceResult} to proofs of inference results at the level of filtrations, we need an adaptation of the persistent nerve lemma of \cite{chazal2008towards} which holds on the level of filtrations rather than only on the level of persistent homology; such an adaptation can be formulated in terms of a distance $d$ on filtrations with the above property.  

When we give the definition of $d_{SI}(X,Y)$ in Section~\ref{Sec:StrongInterleavings}, it will be easy to see that it is not true that $d_{SI}(X,Y)=0$ whenever there is a levelwise homotopy equivalence between $X$ and $Y$; see Remark~\ref{Rem:InterleavingDefinition} for a counterexample.  On the other hand, our Proposition~\ref{prop:ZeroInterleavingOfSOandCech} shows that $d_{WI}$ does have this property.  In fact, we formulate the definitions of weak interleavings and $d_{WI}$ so as to explicitly enforce the property.  

\subsubsection{Localization of Categories and Weak Interleavings}

To define weak interleavings, we first modify the category of multidimensional filtrations by formally adjoining inverses of levelwise homotopy equivalences.  The mathematical tool for this is {\it localization of categories}.  Localization of categories is analogous to the localization of rings and modules in commutative algebra.  In analogy with those versions of localization, it is characterized by a simple universal property.  Localization of categories is intimately connected to homotopy theory and in particular to closed model categories and the axiomatic homotopy theory of Quillen \cite{quillen1967homotopical}.  We discuss this connection in Section~\ref{Sec:LocalizationAndClosedModelCategories}.

We observe that we can define interleavings between filtrations in a localized category of multidimensional filtrations in much the same way that we define strong interleavings in the ordinary category of multidimensional filtrations.  We define weak interleavings to be the interleavings in the localized category, and then define $d_{WI}$ in terms of weak interleavings.

\subsubsection{Properties of Weak Interleavings and the Weak Interleaving Distance}

Our definitions are such that $d_{WI} \leq d_{SI}$.  On the other hand, $d_{WI}$ is sensitive enough that persistent homology functors defined over arbitrary commutative coefficient rings are stable, with respect to the weak interleaving distance on $n$-filtrations and the interleaving distance on persistence modules.  This is the content of our Theorem~\ref{Thm:StabilityWRTInterleavings}.  As we will see, this theorem is quite useful in passing from inference results on the level of filtrations to inference results on the level of persistent homology.  

In this thesis, we do not prove an optimality result for $d_{WI}$ analogous to those which we prove for $d_I$ and $d_{SI}$.  Such a result, if we had it, would offer a fuller picture of the sensitivity properties of $d_{WI}$.  Also, we do not offer a geometrically transparent characterization of $d_{WI}$ analogous either to the algebraic characterization of $d_I$ given by Theorem~\ref{GeneralAlgebraicRealization} or to the geometric characterization of $d_{SI}$ given by Theorem~\ref{Thm:CharacterizationOfStrongInterleavings}.  Thus, in spite of our Theorem~\ref{Thm:StabilityWRTInterleavings} and the close analogy between $d_{WI}$ and both $d_{SI}$ and $d_I$, for which this thesis presents transparent characterizations, at the conclusion of this work $d_{WI}$ still remains somewhat of a mysterious object.  An optimality result for $d_{WI}$ and a geometrically transparent characterization of $d_{WI}$ would do much to remove the shroud of mystery around $d_{WI}$.  In Section~\ref{Sec:WeakInterleavingsQuestions}, we discuss these problems further and offer a conjectural optimality result for $d_{WI}$. 

\section{Chapter 4: The Inferential Interpretation of Random \Cech and Rips Multifiltrations}\label{Sec:Chapter4Overview}
In Chapter 4 we present our main inference results for multidimensional persistence, Theorems~\ref{Thm:CechConsistency} and~\ref{Thm:RipsAsymptotics}.  These adapt Theorem~\ref{Thm:ChazalInferenceResult} to the multidimensional setting and to the level of filtrations.  We present two main inference results: The first concerns inference using \Cech bifiltrations and the second concerns inference using Rips bifiltrations.

We also present two additional results, one deterministic and one probabilistic, which are similar in spirit to our main results.  These are our Theorems~\ref{Thm:RipsFixedScaleApproximation} and~\ref{Thm:CechFixedScaleConsistency}.

\subsection{Definitions of Filtrations}\label{Sec:FirstDefinitionsOfFiltrations}
To state our first main inference result, we need to define a few types of filtrations.  All definitions we present here also appear later in the thesis in fuller generality; see, in particular, Sections~\ref{GeometricPreliminariesSection} and~\ref{Sec:FiltrationsRevisited}.

For $u,v\in (-\infty,\infty]$, a $(u,v)$-filtration $X$ (of nested type) is a collection of topological spaces $\{X_{(a,b)}\}_{a<u,b<v}$ such that if $(a,b)\leq (a',b')$ (w.r.t. the usual partial ordering on $\R^2$) then $X_{(a,b)}\subset X_{(a',b')}$.  We can define the (multidimensional) persistent homology functor of a $(u,v)$-filtration for any $(u,v)\in (-\infty,\infty]^2$---see Section~\ref{Sec:GeneralizedPersistentHomologyDefinition}.

For an $(\infty,\infty)$-filtration $X$ and $u,v\in (-\infty,\infty]$, let $R_{(u,v)}(X)$ denote the $(u,v)$-filtration such that for $(a,b)<(u,v)$, $R_{(u,v)}(X)_{(a,b)}=X_{a,b}$.

\subsubsection{Superlevelset-Offset Filtrations}

We now introduce sublevelset-offset filtrations and superlevelset-offset filtrations.
Informally, whereas the topology of a superlevelset filtration encodes only the height of topographical features of the graph of an $\R$-valued function, the topology of a superlevelset-offset filtration simultaneously encodes both the height and the width of those topographical features.  

Superlevelset-offset filtrations are a natural common extension of superlevelset filtrations and {\it offset filtrations} \cite{chazal2008towards}, two types of 1-D filtrations which are standard objects of study in the computational topology literature.  

If $(Y,d)$ is a metric space, $X\subset Y$, and $\gamma:X \to \R$ is a function, $F^{SO}(X,Y,d,\gamma)$, the {\bf sublevelset-offset filtration of $\gamma$ (w.r.t the metric $d$)}, is the $(\infty,\infty)$-filtration for which \[F^{SO}(X,Y,d,\gamma)_{a,b}= \{y\in Y|d(y,\gamma^{-1}((\infty,a]))\leq b\}.\]  We write $F^{SO}(Y,Y,d,\gamma)$ simply as $F^{SO}(Y,d,\gamma)$.

Informally, $F^{SO}(X,Y,d,\gamma)_{a,b}$ is the thickening of the $a$-sublevelset of $\gamma$ by offset parameter $b$.  We call $F^{SO}(X,Y,d,-\gamma)$. the {\bf superlevelset-offset filtration} of $\gamma$.

\subsubsection{Superleveset-\Cech Bifiltrations}
Given $X,Y,$ and $d$ as above and $b \in \R$, let $\Cech(X,Y,d,b)$, the {\it (closed) \Cech complex} of $(X,d)$ with parameter $b$, be the abstract simplicial complex with vertex set $X$ such that for $l\geq 2$ and $x_1,x_2,..,x_l\in X$, $\Cech(X,Y,d,b)$ contains the $(l-1)$-simplex $[x_1,...,x_l]$ iff there is a point $y\in Y$ such that $d(y,x_i)\leq b$ for $1\leq i \leq l$.

For any $\gamma:X\to \R$, let $F^{\SCe}(X,Y,d,\gamma)$, the {\bf sublevelset-\Cech filtration} of $\gamma$ (w.r.t the metric $d$), be the $(\infty,\infty)$-filtration for which \[F^{\SCe}(X,Y,d,\gamma)_{a,b}=\Cech(\gamma^{-1}((\infty,a]),Y,d,b).\]
We call $F^{\SCe}(X,Y,d,-\gamma)$ the {\bf superlevelset-\Cech filtration} of $\gamma$.

\subsubsection{Superlevelset-Rips Bifiltrations}
Given a metric space $(X,d)$ and $b \in \R$, let $\Rips(X,d,b)$, the {\it Rips complex} of $(X,d)$ with parameter $b$, be the maximal abstract simplicial complex with vertex set $X$ such that for $x_1,x_2\in X$, the 1-skeleton of $R(X,d,b)$ contains the edge $[x_1,x_2]$ iff $d(x_1,x_2)\leq 2b$. 

For any $\gamma:X\to \R$, let $F^{SR}(X,d,\gamma)$, the {\bf sublevelset-Rips filtration} of $\gamma$ (w.r.t the metric $d$), be the $(\infty,\infty)$-filtration for which \[F^{SR}(X,d,\gamma)_{a,b}=\Rips(\gamma^{-1}((\infty,a]),d,b).\]
We call $F^{\SCe}(X,d,-\gamma)$ the {\bf superlevelset-Rips filtration} of $\gamma$.

\subsection{An Inferential Interpretation of Superlevelset-\Cech Bifiltrations}
We now describe our first main inference result.

For simplicity, we state the result here only for density functions on Euclidean space; it adapts readily to the case of density functions on Riemannian manifolds.

Let $\gamma:\R^m \to \R$ be a $c$-Lipchitz density function for some $c>0$; for $z\in \NN$, let $T_z$ be an i.i.d. sample of size $z$ of a random variable with density $\gamma$; fix $p\in [1,\infty]$, let $d^p$ denote both the $L^p$-metric on $\R^m$ and (by slight abuse of notation) the restriction of the $d^p$ to any subset of $\R^m$; let $E$ be a density estimator on $\R^m$.

Our first main result, Theorem~\ref{Thm:CechConsistency}, is the following.

\begin{ThmNoNum}\mbox{}
\begin{enumerate*}
\item[(i)]If $(\gamma,E)$ satisfies A1 then 
\[d_{WI}(R_{(0, \infty)}(F^{\SCe}(T_z,\R^m,d^p,-E(T_z))),R_{(0, \infty)}(F^{SO}(\R^m,d^p,-\gamma)))\xrightarrow{\mathcal P}0.\]
\item[(ii)]If $(\gamma,E)$ satisfies A2 for a kernel K then 
\[d_{WI}(R_{(0, \infty)}(F^{\SCe}(T_z,\R^m,d^p,-E(T_z))),R_{(0, \infty)}(F^{SO}(\R^m,d^p,-\gamma*K)))\xrightarrow{\mathcal P}0.\]  
\end{enumerate*}
\end{ThmNoNum}

Note that unlike Theorem~\ref{Thm:ChazalInferenceResult}, the statement of this result does not involve a sequence $\{\delta_z\}_{z\in \Z}$ of scale parameters.

Let $H_i$ denote the $i^{th}$ persistent homology functor (the definition is given in Section~\ref{Sec:GeneralizedPersistentHomologyDefinition}).  From our stability result for the weak interleaving distance, Theorem~\ref{Thm:StabilityWRTInterleavings}, we immediately obtain the following corollary, which is Corollary~\ref{Cor:CechHomologyConsistency}.

\begin{CorNoNum}\mbox{}
\begin{enumerate*}
\item[(i)]If $(\gamma,E)$ satisfies A1 then 
\[d_I(H_i(R_{(0, \infty)}(F^{\SCe}(T_z,\R^m,d^p,-E(T_z)))),H_i(R_{(0, \infty)}(F^{SO}(\R^m,d^p,-\gamma))))\xrightarrow{\mathcal P}0.\]
\item[(ii)]If $(\gamma,E)$ satisfies A2 for a kernel K then 
\[d_I(H_i(R_{(0, \infty)}(F^{\SCe}(T_z,\R^m,d^p,-E(T_z)))),H_i(R_{(0, \infty)}(F^{SO}(\R^m,d^p,-\gamma*K))))\xrightarrow{\mathcal P}0.\]
\end{enumerate*}
\end{CorNoNum}
Theorem~\ref{GeneralAlgebraicRealization}, our extrinsic characterization of $(J_1,J_2)$-interleaved pairs of modules, offers a concrete interpretation of Corollary~\ref{Cor:CechHomologyConsistency} as a statement about the similarity between presentations of persistent homology modules.

Theorem~\ref{Thm:CechConsistency} and Corollary~\ref{Cor:CechHomologyConsistency} also immediately give us corresponding results about Rips filtrations, in two special cases.  First, it is easy to check that for any $T\subset \R^m$, and any scale parameter $a\geq 0$, $\Rips(T,d^\infty,a)=\Cech(T,\R^m,d^\infty,a)$.  Thus \[F^{\SCe}(T_z,\R^m,d^\infty,-E(T_z))=F^{SR}(T_z,d^\infty,-E(T_z))\] and so when $p=\infty$, we may replace the superlevelset-\Cech filtrations $F^{\SCe}(T_z,\R^m,d^p,-E(T_z))$ with the superlevelset-Rips filtations $F^{SR}(T_z,d^p,-E(T_z))$ everywhere in the statements of the above theorems.  

Second, since the superlevelset-Rips filtrations and superlevelset-\Cech filtrations always have equal 1-skeletons, \[H_0(R_{(0, \infty)}(F^{\SCe}(T_z,d^p_z,-E(T_z))))=H_0(R_{(0, \infty)}(F^{SR}(T_z,d^p_z,-E(T_z))))\] in the the statement of Corollary~\ref{Cor:CechHomologyConsistency}, and so the corollary gives us in particular a consistency result for the the $0^{th}$ persistent homology of superlevelset-Rips filtrations.

\subsection{An Inferential Interpretation of Superlevelset-Rips Bifiltrations}\label{Sec:RipsInferenceIntro}
For $\gamma,T_z$, and $E$ as in the previous section, write $F^{SR}_z=F^{SR}(T_z,d^p,-E(T_z))$ and $F^{\SCe}_z=F^{\SCe}(T_z,\R^m,d^p,-E(T_z))$, and $F^{SO}=F^{SO}(\R^m,\R^m,d^p,-\gamma)))$.

Our second main inference result, Theorem~\ref{Thm:RipsAsymptotics}, is an analogue of our first main inference result for the random superlevelset-Rips bifiltrations $F^{SR}_z$.  The result is formulated directly in terms of weak interleavings, rather than in terms of $d_{WI}$.  Because of this, the precise statement of the result is technical and not readily presented before giving a careful treatment of weak $(J_1,J_2)$-interleavings.  Thus we will describe the result here but defer its statement to Section~\ref{Thm:RipsAsymptotics}.

Theorem~\ref{Thm:RipsAsymptotics} quantifies the sense in which $F^{SR}_z$ is topologically similar to $F^{SO}$ in the asymptotic limit as $z\to \infty$, under the same assumptions as in our first main inference result.  Roughly, the result says that in the limit as $z\to \infty$, $F^{SR}_z$ and $F^{SO}$ satisfy in the weak sense the same interleaving relationship that $F^{SR}_z$ and $F^{\SCe}_z$ satisfy in the strong sense for all $z$ (with probability 1).

The result is not, in any reasonable sense, a consistency result.  Indeed, $F^{\SCe}_z$ and $F^{SR}_z$ exhibit topological differences that do not become negligible with high probability as $z\to \infty$.  Thus, given that under the assumptions of our first main inference result, the bifiltrations $F^{\SCe}_z$ are (in a topological sense) consistent estimators of $F^{SO}$, we do not expect that under the same assumptions the bifiltrations $F^{SR}_z$ would also be consistent estimators of $F^{SO}$.  On the other hand, the topological differences between $F^{\SCe}_z$ and $F^{SR}_z$ are controlled by the well known inclusion relationships between \Cech complexes and Rips complexes: Recall that for any metric space $(Y,d)$, $X\subset Y$ and any scale parameter $r\geq 0$, we have that  
\begin{align}
\Cech(X,Y,d,r)\subset \Rips(X,d,r)\subset \Cech(X,Y,d,2r).
\label{Eq:Rips_And_Cech}
\end{align}
Using the language of weak  $(J_1,J_2)$-interleavings, we can encode the relationship between $F^{\SCe}_z$ and $F^{SR}_z$ induced by these inclusions.  Then, via a simple triangle inequality type lemma for interleavings, we can quantify how that relationship, taken together with our first main inference result, controls the topological differences between $F^{SR}_z$ and $F^{SO}$ in the asymtoptic limit.  Theorem~\ref{Thm:RipsAsymptotics} does exactly this.

As in the case of our first main result, via the stability Theorem~\ref{Thm:StabilityWRTInterleavings} we obtain an analogue of Theorem~\ref{Thm:RipsAsymptotics} for persistent homology modules, Corollary~\ref{Cor:RipsHomologyAsymptotics}.  Corollary~\ref{Cor:RipsHomologyAsymptotics} is formulated in terms of $(J_1,J_2)$-interleavings of persistence modules.  As is the case for Corollary~\ref{Cor:CechHomologyConsistency}, Theorem~\ref{GeneralAlgebraicRealization} offers a concrete interpretation of Corollary~\ref{Cor:RipsHomologyAsymptotics} as a statement about the similarity between presentations of persistent homology modules.

\subsection{Deterministic Approximation Results}

In proving our inference results, we follow a strategy analogous to that used by Chazal et. al to prove the inference result \cite[Theorem 5.1]{chazal2009persistence}--namely we first prove deterministic approximation results which assume that the domain of the functions we consider is well sampled, and then use these deterministic results to obtain probabilistic results.

In specific, we first prove Theorem~\ref{Thm:CechApproximation}, a bound on the weak interleaving distance between the sublevelset-offset filtration of a function $\gamma$ and the sublevelset-\Cech filtration of an approximation of $\gamma$ defined on a finite subset of the domain of $\gamma$.  This bound implies an analogous result, Theorem~\ref{Thm:RipsApproximation}, for sublevelset-Rips filtrations, formulated in terms of weak $(J_1,J_2)$-interleavings.  These results are the content of Sections~\ref{Sec:CechApproximation} and~\ref{Sec:RipsApproximation}; they are analogues of \cite[Theorem 3.1]{chazal2009analysis}, and its extension \cite[Theorem 4.5]{chazal2009persistence}, but in the multidimensional setting and for filtrations rather than for persistent homology modules.  In fact, our results treat the general situation that $\gamma$ is an $\R^n$-valued function, $n\geq 1$, and that only a sublevelset of the domain of $\gamma$ is well sampled.  

The result \cite[Theorem 4.5]{chazal2009persistence} for ordinary persistence is interesting and useful, independent of its application to statistical inference---see, for example, the application presented in \cite{skraba2010persistence}.  We imagine that the deterministic bounds we present here may similarly be of independent interest.

\subsection{Deterministic Approximation of Multidimensional Sublevelset Persistent Homology}
Using the interleaving distance on multidimensional persistence modules, it also is possible to adapt the deterministic result \cite[Theorem 4.5]{chazal2009persistence} to the multidimensional setting in a different---and more straightforward---way than that of Theorems~\ref{Thm:CechApproximation} and~\ref{Thm:RipsApproximation}: Our Theorem~\ref{Thm:RipsFixedScaleApproximation} generalizes the result \cite[Theorem 4.5]{chazal2009persistence} to $\R^n$-valued functions, using sublevelset multifiltrations (as defined e.g. in Section~\ref{GeoFunctorsSection}), and sublevelset-Rips multifiltrations with a fixed scale parameter.  The proof is essentially same as that of \cite[Theorem 4.5]{chazal2009persistence} in the case of $\R$-valued functions.

This result does not lift to the level of filtrations, but an analogue of it for sublevelset-\Cech filtrations rather than sublevelset-Rips filtrations does in fact lift to the level of filtrations.  This is our Theorem~\ref{Thm:CechFixedScaleApproximation}.

\subsection{An Inference Result for 1-D Filtrations}
In addition to our main inference results, described above, in Section~\ref{Sec:CechFixedScaleInference} we apply the weak interleaving distance and Theorem~\ref{Thm:CechFixedScaleApproximation} to obtain an inference result, Theorem~\ref{Thm:CechFixedScaleConsistency}, for the superlevelset filtration of a probability density function.  This is a variant of Theorem~\ref{Thm:ChazalInferenceResult}, holding at the level of filtrations. In general, the result holds only for filtered \Cech complexes, not for filtered Rips complexes---Theorem~\ref{Thm:ChazalInferenceResult}, as stated above for filtered Rips complexes, does not lift to the level of filtrations.  However, as for Theorem~\ref{Thm:CechConsistency}, Theorem~\ref{Thm:CechFixedScaleConsistency} can of course be interpreted as a result about Rips filtrations in the special case that we construct our filtrations using $L^{\infty}$ metrics on $\R^m$. 

As we have already noted, our proof of Theorem~\ref{Thm:CechFixedScaleConsistency} adapts to give a proof of Theoem~\ref{Thm:ChazalInferenceResult}, which we stated but did not prove.

\section{Why Formulate Theory of Topological Inference Directly at the Level of Filtrations?}\label{Sec:InferenceAtLevelOfFiltrations}
To close this introduction, we explain our choice to develop inferential theory directly at the level of filtrations rather than only at the level of persistent homology modules.

Our aim is to understand as deeply as possible the connection between the topological structure of discrete filtrations built on randomly sampled point cloud data and that of filtrations built directly from the probability distribution generating the random data; as we'll now argue, to do this, it is more natural to develop the inferential theory (and, in particular, the requisite notions of similarity of filtrations) directly at the level of filtrations than at the homology level.

We note first the trivial fact that singular homology induces an equivalence relation on topological spaces---namely, we can say that two topological spaces $X$ an $Y$ are {\it homology equivalent}\footnote{Note: This not standard terminology; we are introducing it here only in the service of the present discussion and only for use in this subsection.} if $H_i(X)$ and $H_i(Y)$ are isomorphic for all $i$.  However, such a notion of equivalence between topological spaces is not often adopted as a fundamental object of interest in algebraic topology.  The reason is that there are stronger, more easily defined, and more geometrically transparent notions of equivalence of topological spaces, such as those of homotopy equivalence and weak homotopy equivalence, which exhibit the kinds of invariance properties one wants a notion of equivalence on spaces to have in algebraic topology and homotopy theory.\footnote{When we say that homotopy equivalence is a geometrically transparent notion of equivalence of topological spaces we are alluding specifically to the theorem which says that two spaces are homotopy equivalent if and only if they are each deformation retracts of some third space \cite[Corollary 0.21]{hatcher2002algebraic}.}  Thus homotopy equivalence and weak homotopy equivalence are generally regarded as the fundamental notions of equivalence of topological spaces in algebraic topology, and homology serves primarily as a computational tool to understand spaces up to these notions of equivalence; homology is usually not taken as means of defining a notion of equivalence of topological spaces in of itself.

Now, for filtrations the analogue of homology is persistent homology, and we can define a notion of persistent homology equivalence of filtrations in a way analogous to the way in which we just defined homology equivalence of topological spaces.  In fact, for ordinary persistence the bottleneck distance $d_B$ on persistence modules affords us an {\it approximate} notion of persistent homology equivalence between two 1-D filtrations: we can interpret the statement that
\[\forall \, i\in \Z_{\geq 0},\, d_B(H_i(X),H_i(Y))\leq \epsilon\] 

for two filtrations $X$ and $Y$ as saying that $X$ and $Y$ are {\it approximately persistent homology equivalent}, up to an error of $\epsilon$.  More generally, the interleaving distance introduced in this thesis affords us in the same way an approximate notion of persistent homology equivalence between two multidimensional filtrations.  

Given that for topological spaces it turns out to be most natural to regard homotopy equivalence, rather than homology equivalence, as the fundamental notion of equivalence between topological spaces, by analogy it is reasonable for us to ask for a notion of {\it approximate homotopy equivalence} of filtrations which has good invariance properties and which enjoys the same advantages over approximate homology equivalence that homotopy equivalence enjoys over homology equivalence for topological spaces.  Further, if we were to have such a notion, then in light of its good properties and by analogy to classical algebraic topology, it would be natural for us to formulate persistence theory and in particular theory of topological inference directly at the level of filtrations in terms of that notion.

The weak interleaving distance affords us precisely such a well behaved notion of approximate homotopy equivalence of multidimensional filtrations, except that, as we noted in Section~\ref{Sec:WeakInterleavingsIntro}, it is not yet clear to what extent the {weak interleaving distance} is in fact geometrically transparent. 

Thus, the present unavailability of a geometric characterization of the weak interleaving distance aside, it is natural to formulate our main inference results directly at the level of filtrations, in terms of the weak interleaving distance and the closely related notion of weak interleavings.

%% file: Part_I/T_Preliminaries_7.tex
In this chapter we introduce and study interleavings and interleaving distance on multidimensional persistence modules.  See Section~\ref{Sec:Chapter2Overview} for an overview of the chapter.

\section{Algebraic Preliminaries}\label{SectionAlgebraicPreliminaries}
In this section we define persistence modules and review some (primarily) algebraic facts and definitions which we will need throughout the thesis.    

\subsection{First Definitions and Notation}\label{Sec:FirstDefs}
 
\subsubsection{Basic Notation}  Let $k$ be a field.  Let ${\mathbb N}$ denote the natural numbers and let ${\Z}_{\geq 0}$ denote the non-negative integers.  Let $\hat \R=(-\infty,\infty]$.  

We view $\R^n$ as a partially ordered set, with $(a_1,...,a_n)\leq (b_1,...,b_n)$ iff $a_i\leq b_i$ for all $i$.  Let $\e_i$ denote the $i^{th}$ standard basis vector in $\R^n$.  

For $A\subset \R$ any subset, let ${\bar A}=A\cup \{-\infty,\infty\}$.   

For $a=(a_1,...,a_n)\in \hat \R^n$ and $b\in \hat \R$, let $(a,b)=(a_1,...,a_n,b)\in \hat \R^{n+1}$.  

For $f:X \to {\R}^n$, $a\in \R^n$, let $f_a=\{x\in X|f(x)\leq a\}$.  We call $f_a$ the {\it a-sublevelset of $f$} and we call $(-f)_a$ the {\it a-superlevelset of $f$}.

For $\epsilon\in \hat \R$, let ${\vec \epsilon}_n\in \R^n$ denote the vector whose components are each $\epsilon$.  We'll also often write ${\vec \epsilon}_n$ simply as ${\vec \epsilon}$ when $n$ is understood.  As a notational convenience, for $u\in \R^n$ and $\epsilon \in \R$, let $u+\epsilon$ denote $u+{\vec \epsilon}$.  

\subsubsection{Notation Related to Categories}
 
For a category $C$, let $\obj(C)$ denote the objects of $C$ and let $\obj^*(C)$ denote the isomorphism classes of objects of $C$.  Let $\hom(C)$ denote the morphisms in $C$, and for $X,Y\in \obj(C)$ let $\hom_C(X,Y)$ denote the morphisms from $X$ to $Y$.  When $C$ is understood, we'll often write $\hom_C(X,Y)$ simply as $\hom(X,Y)$.

\subsubsection{Metrics, Pseudometrics, and Semi-pseudometrics}
Recall that a {\it pseudometric} on a set $X$ is a function $d:X\times X\to [0,\infty]$ with the following three properties:
\begin{enumerate*}
\item $d(x,x)=0$ for all $x\in X$.
\item $d(x,y)=d(y,x)$ for all $x,y\in X$.
\item $d(x,z)\leq d(x,y)+d(y,z)$ for all $x,y,z\in X$.
\end{enumerate*}
We'll often use the term {\it distance} in this thesis as a synonym for pseudometric.   

A {\it metric} is a pseudometric $d$ with the additional property that $d(x,y)\ne 0$ whenever $x\ne y$.  We define a {\it semi-pseudometric} to be a function $d:X\times X\to [0,\infty]$ satisfying properties 1 and 2 above.
\subsubsection{Metrics on Categories}

In this thesis we'll often have the occasion to define a pseudometric on $\obj^*(C)$, for $C$ some category.  For $d$ such a pseudometric, $M,N\in \obj(C)$, and $[M],[N]$ the isomorphism classes of $M$ and $N$, we'll always write $d(M,N)$ as shorthand for $d([M],[N])$.   


\subsection{Commutative Monoids and Commutative Monoid Rings}\label{MonoidRings}
Monoid rings are generalizations of polynomial rings.   

A {\it commutative monoid} is a pair $(G,+_G)$, where $G$ is a set and $+_G$ is an associative, commutative binary operation on $G$ with an identity element.  Abelian groups are by definition commutative monoids with the additional property that each element has an inverse.  We'll often denote the monoid $(G,+_G)$ simply as $G$.  A submonoid of a monoid is defined in the obvious way, as is an isomorphism between two monoids.  

Given a set $S$, let $k[S]$ denote the vector space of formal linear combinations of elements of $S$.  If ${\bar G}=(G,+_G)$ is a monoid, then the operation $+_G$ induces a ring structure on $k[G]$, where multiplication is characterized by the property $(k_1 g_1)(k_2 g_2)=k_1k_2(g_1 +_G g_2)$ for $k_1,k_2\in k$, $g_1,g_2\in G$.  We call the resulting ring the {\it monoid ring} generated by ${\bar G}$, and we denote it $k[{\bar G}]$.

Let ${\bf A_n}$ denote $k[x_1,...,x_n]$, the polynomial ring in $n$ variables with coefficients in $k$.  For $n>0$, $\Z_{\geq 0}^n$ is a monoid under the usual addition of vectors.  It's easy to see that $k[\Z_{\geq 0}^n]$ is isomorphic to ${\bf A_n}$.     

Similarly, $\R_{\geq 0}^n$ is a monoid under the usual addition of vectors.  Let $B_n$ denote the monoid ring $k[\R_{\geq 0}^n]$.  We may think of $B_n$ as an analogue of the usual polynomial ring in $n$-variables where exponents of the indeterminates are allowed to take on arbitrary non-negative real values rather than only non-negative integer values.  With this interpretation in mind, we'll often write $(r_1,...,r_n)$ as $x_1^{r_1}x_2^{r_2}\cdots x_n^{r_n}$, for $(r_1,...,r_n)\in \R_{\geq 0}^n$.  

\subsection{Multidimensional Persistence Modules}\label{Sec:MultidimensionalPersistenceModules}

We first review the definition of a multidimensional persistence module given in \cite{carlsson2009theory}.  We then define analogues of these over the ring $B_n$.  

In what follows, we'll often refer to multidimensional persistence modules simply as ``persistence modules."

\subsubsection{${\mathbf A}_n$-Persistence Modules}

Fix $n\in \NN$.  An {\bf ${\mathbf A}_n$-persistence module} is an ${\mathbf A}_n$-module $\M$ with a direct sum decomposition as a $k$-vector space $\M \cong \bigoplus_{a \in \Z^n} \M_a$ such that the action of ${\mathbf A}_n$ on $\M$ satisfies $x_i(\M_a) \subset \M_{a+\e_i}$ for all $a\in \Z^n.$  In other words, an ${\mathbf A}_n$-persistence module is simply an ${\mathbf A}_n$-module endowed with an $n$-graded structure.  

For $\M$ and $\N$ ${\mathbf A}_n$-persistence modules, we define $\hom(\M,\N)$ to be the set of module homomorphisms $f:\M\to \N$ such that $f(\M_a)\subset \N_a$ for all $a \in \Z^n$.  This defines a category whose objects are the ${\mathbf A}_n$-persistence modules.  Let ${\mathbf A}_n$-mod denote this category.

\subsubsection{$B_n$-persistence modules}
In close analogy with the definition of an $n$-graded ${\mathbf A}_n$-module, we define a {\bf $B_n$-persistence module} to be a $B_n$-module $M$ with a direct sum decomposition as a $k$-vector space $M \cong \bigoplus_{a \in {\R}^n} M_a$ such that the action of $B_n$ on $M$ satisfies $x_i^{\alpha}(M_a) \subset M_{a+\alpha\e_i}$ for all $a\in \R^n,$ $\alpha\geq 0$.

For $M$ and $N$ $B_n$-persistence modules, we define $\hom(M,N)$ to consist of module homomorphisms $f:M\to N$ such that $f(M_a)\subset N_a$ for all $a \in \R^n$.  This defines a category whose objects are the $B_n$-persistence modules.  Let $B_n$-mod denote this category.

Our notational convention will be to use boldface to denote ${\mathbf A}_n$-persistence modules and italics to denote $B_n$-persistence modules.  We'll often refer to ${\mathbf A}_1$-persistence modules and $B_1$-persistence modules as {\it ordinary} persistence modules.

\subsubsection{On the Relationship Between ${\mathbf A}_n$-persistence Modules and $B_n$-persistence Modules}\label{TensorUpRemarks} 

Since ${\mathbf A}_n$ is a subring of $B_n$, we can view $B_n$ as an ${\mathbf A}_n$-module.  If $\M$ is an ${\mathbf A}_n$-persistence module then $\M\otimes_{{\mathbf A}_n} B_n$ is a $B_n$-module.  Further, $\M\otimes_{{\mathbf A}_n} B_n$ inherits an $n$-grading from those on $B_n$ and $\M$ which gives $\M\otimes_{{\mathbf A}_n} B_n$ the structure of a $B_n$-persistence module.  

In fact, $(\cdot)\otimes_{{\mathbf A}_n} B_n$ defines a functor from ${\mathbf A}_n$-mod to $B_n$-mod.  It can be checked that this functor is fully faithful and descends to an injection on isomorphism classes of objects.  Thus the functor induces an identification of ${\mathbf A}_n$-mod with a subcategory of $B_n$-mod.

In light of this, we can think of $B_n$-persistence modules as generalizations of ${\mathbf A}_n$-persistence modules.  Finitely presented $B_n$-persistence modules arise naturally in applications, as discussed in Section~\ref{GeoFunctorsSection}.  In a sense that can be made precise using machinery mentioned in Remark~\ref{MultidimensionExtension}, it is possible to view them as ${\mathbf A}_n$-persistence modules endowed with some additional data.  However, this is awkward from the standpoint of constructing pseudometrics between $B_n$-persistence modules.  We thus regard $B_n$-persistence modules as the fundamental objects of interest here, and use ${\mathbf A_n}$-persistence modules in this thesis only in the case $n=1$ to translate results about ${\mathbf A_1}$-persistence modules into analogous results about $B_1$-persistence modules.

In the remainder of Section~\ref{Sec:MultidimensionalPersistenceModules}, we present some basic definitions related to $B_n$-persistence modules.  All of these definitions have obvious analogues for ${\mathbf A}_n$-persistence modules; we'll use these analogues where needed without further comment.

\subsubsection{Homogeneity}
Let $M$ be a $B_n$-persistence module.  For $u\in \R^n$, we say that $M_u$ is a {\it homogeneous summand} of $M$.  We refer to an element $v\in M_u$ as a {\it homogeneous element} of grade $u$, and write $gr(v)=u$.  A {\it homogeneous submodule} of a $B_n$-persistence module is a submodule generated by a set of homogeneous elements.  The quotient of a $B_n$-persistence module $M$ by a homogeneous submodule of $M$ is itself a $B_n$-persistence module; the $n$-graded structure on the quotient is induced by that of $M$.



\subsubsection{Tameness}
Following \cite{chazal2009proximity} we'll call 
a $B_n$-persistence module {\it tame} if each homogeneous summand of the module is finite dimensional.  Note that this is a more general notion of tameness than that which appears in the original paper on the stability of persistence \cite{cohen2007stability}. 

\subsubsection{Transition Maps}  
For $M$ a $B_n$-persistence module and any $u\leq v\in \R^n$, the restriction to $M_u$ of the action on $M$ of the monomial $x_1^{v_1-u_1}x_2^{v_2-u_2} \cdots x_n^{v_n-u_n}$  defines a linear map with codomain $M_v$, which we call a {\it transition map}.  Denote this map by $\varphi_M(u,v)$.  


\subsection{Shift Functors and Transition Morphisms}\label{ShiftsOfModules}

\subsubsection{Shift Functors and Related Notation}
For $v\in \R^n$, we define the {\bf shift functor} $(\cdot)(v)$ as follows: For $M$ a $B_n$-persistence module we let $M(v)$ be the $B_n$-persistence module such that for all $a\in \R^n$, $M(v)_a=M_{a+v}$.  For $a\leq b\in \R^n$, we take $\varphi_{M(v)}(a,b)=\varphi_M(a+v,b+v)$.
For $f\in \hom(B_n$-mod) we let $f(v)_a=f_{a+v}$.  

For $v\in \R^n$ and $f\in \hom(B_n$-mod), we'll sometimes abuse notation and write $f(v)$ simply as $f$.  

For $\epsilon\in \R$, let $M(\epsilon)$ denote $M({\vec \epsilon})$.  More generally, for any subset $Q\subset M$, let $Q(\epsilon)\subset M(\epsilon)$ denote the image of $Q$ under the bijection between $M$ and $M(\epsilon)$ induced by the identification of each summand $M(\epsilon)_u$ with $M_{u+\epsilon}$.  

\subsubsection{Transition Homomorphisms}

For a $B_n$-persistence module $M$ and $\epsilon\in \R_{\geq 0}$ let $S(M,\epsilon):M\to M(\epsilon)$, the {\bf (diagonal) $\epsilon$-transition homomorphism}, be the homomorphism whose restriction to $M_u$ is the linear map $\varphi_M(u,u+\epsilon)$ for all $u\in \R^n$.  

\subsection{$\epsilon$-interleavings and the Interleaving Distance}\label{InterleavingsOfModules} 

For $\epsilon\geq 0$, we say that two $B_n$-persistence modules $M$ and $N$ are {\bf $\epsilon$-interleaved} if there exist homomorphisms $f:M\to N(\epsilon)$ and $g:N\to M(\epsilon)$ such that 
\begin{align*}
g(\epsilon) \circ f&=S(M,2\epsilon)\textup{ and }\\
f(\epsilon) \circ g&=S(N,2\epsilon); 
\end{align*}
we refer to such $f$ and $g$ as {\bf $\epsilon$-interleaving} homomorphisms. 

The definition of $\epsilon$-interleaving homomorphisms was introduced for $B_1$-persistence modules in \cite{chazal2009proximity}.  

\begin{remark}\label{EasyInterleavingRemark}
It's easy to see that if $0\leq\epsilon_1 \leq \epsilon_2$ and $M$ and $N$ are $\epsilon_1$-interleaved, then $M$ and $N$ are $\epsilon_2$-interleaved. 
\end{remark}

In Section~\ref{Sec:J1J2Interleavings} we will observe that the definition of $\epsilon$-interleavings generalizes considerably, but for now we will not concern ourselves with the generalized form of the definition.

\subsubsection{The Interleaving Distance on $B_n$-persistence modules}

We define $d_I:\obj^*(B_n$-mod$)\times \obj^*(B_n$-mod$)\to [0,\infty]$, the {\bf interleaving distance}, by taking 
\[d_I(M,N)=\inf \{\epsilon\in \R_{\geq 0}|M\textup{ and  }N\textup{ are }\epsilon\textup{-interleaved}\}.\]

Note that $d_I$ is pseudometric.  However, the following example shows that $d_I$ is not a metric.  

\begin{example}\label{NonIsoSameDiagram}  Let $M$ be the $B_1$-persistence module with $M_0=k$ and $M_{a}=0$ if $a\ne 0$.  Let $N$ be the trivial $B_1$-persistence module.  Then $M$ and $N$ are not isomorphic, and so are not $0$-interleaved, but it is easy to check that $M$ and $N$ are $\epsilon$-interleaved for any $\epsilon>0$.  Thus $d_I(M,N)=0$. 
\end{example}

\subsection{Free $B_n$-persistence Modules, Presentations, and Related Algebraic Basics}

In our study of $B_n$-persistence modules, we'll make substantial use of free $B_n$-persistence modules and presentations of $B_n$-persistence modules; we define these objects here and present some basic results about them.  We begin with some foundational definitions.

\subsubsection{$n$-graded Sets}

Define an {\it $n$-graded set} to be a pair $G=({\bar G},\iota_G)$ where ${\bar G}$ is a set and $\iota_G:G\to\R^n$ is any function.  When $\iota_G$ is clear from context, as it will usually be, we'll write $\iota_G(y)$ as $gr(y)$ for $y\in {\bar G}$.  We'll sometimes abuse notation and write $G$ to mean the  set ${\bar G}$ when no confusion is likely.  The union of disjoint graded sets is defined in the obvious way.  For $\epsilon\geq 0$ and $G=({\bar G},\iota_G)$ an $n$-graded set, let $G(\epsilon)$ be the $n$-graded set $({\bar G},\iota'_G)$, where $\iota'_G(y)=\iota(y)-\epsilon$.  


For $G$ an $n$-graded set, define $gr(G):\R^n\to \Z_{\geq 0}\cup \{\infty\}$ by taking $gr(G)(u)$ to be the number of elements $y\in G$ such that $gr(y)=u$.  Note that for any $B_n$-persistence module $M$, a set $Y\subset M$ of homogeneous elements inherits the structure of an $n$-graded set from the graded structure on $M$, so that $gr(Y)$ is well defined.   

\subsubsection{Free $B_n$-persistence modules}
The usual notion of a free module extends to the setting of $B_n$-persistence modules as follows: For $G$ an $n$-graded set, let $\langle G\rangle=\oplus_{y\in {\bar G}}B_n(-gr(y))$.  A {\it free $B_n$-persistence module} $F$ is a $B_n$-persistence module such that for some $n$-graded set $G$, $F\cong \langle G \rangle$.  

Equivalently, we can define a free $B_n$-persistence module as a $B_n$-persistence module which satisfies a certain universal property.  Free ${\mathbf A}_n$-persistence modules are defined via a universal property e.g. in \cite[Section 4.2]{carlsson2009theory}.  The definition for $B_n$-persistence modules is analogous; we refer the reader to \cite{carlsson2009theory} for details.    

A {\it basis} for a free module $F$ is a minimal set of homogeneous generators for $F$.  For $G$ any graded set, identifying $y\in G$ with the copy of $1(-gr(y))$ in the summand $B_n(-gr(y))$ of $\langle G\rangle$ corresponding to $y$ gives an identification of $G$ with a basis for $\langle G \rangle$.  It can be checked that if $B$ and $B'$ are two bases for a free $B_n$-persistence module $F$ then $gr(B)=gr(B')$.  Clearly then, $gr(B)$ of an arbitrarily chosen basis $B$ for $F$ is an isomorphism invariant of $F$ and determines $F$ up to isomorphism.  



For $R$ a homogeneous subset of a free $B_n$-persistence module $F$, $\langle R \rangle$ will always denote the submodule of $F$ generated by $R$.  Since, as noted above, $R$ can be viewed as an $n$-graded set, we emphasize that for such $R$, $\langle R \rangle$ {\bf does not} denote $\oplus_{y\in R}B_n(-gr(y))$.  

\subsubsection{Free Covers and Lifts}
For $M$ a $B_n$-persistence module, define a {\it free cover} of $M$ be a pair $(F_M,\rho_M)$, where $F_M$ is a free $B_n$-persistence module and $\rho_M:F_M \to M$ a surjective morphism of $B_n$-persistence modules.  

For $M,N$ $B_n$-persistence modules, $(F_M,\rho_M)$ and $(F_N,\rho_N)$ free covers of $M$ and $N$, and $f:M \to N$ a morphism, define a {\it lift} of $f$ to be a morphism ${\tilde f}:F_M \to F_N$ such that the following diagram commutes.
 \begin{equation*}
 		\begin{CD}
 		F_M  @>{\tilde f}>> F_N  \\
 		 @VV{\rho_M}V       @VV{\rho_N}V \\    
 	    M    @>{f}>>        N 
 		\end{CD}
 	\end{equation*}

\begin{lem}[Existence and Uniqueness up to Homotopy of Lifts]\label{ExistenceAndHomotopyUniquenessOfLifts} For $B_n$-persistence modules $M$ and $N$, free covers $(F_M,\rho_M),(F_N,\rho_N)$  of $M,N$, and a morphism $f:M \to N$, there exists a lift ${\tilde f}:F_M\to F_N$ of $f$.  If ${\tilde f'}:F_M\to F_N$ is another lift of $f$, then $\im({\tilde f}-{\tilde f'})\subset \ker(\rho_N)$.  \end{lem}

\begin{proof} This is just a specialization of the standard result on the existence and homotopy uniqueness of free resolutions \cite[Eisenbud A3.13]{eisenbud1995commutative} to the $0^{th}$ modules in free resolutions for $M$ and $N$.  The proof is straightforward. \end{proof} 

\subsubsection{Presentations of $B_n$-persistence Modules}
A {\bf presentation} of a $B_n$-persistence module $M$ is a pair $(G,R)$ where $G$ is an $n$-graded set and $R \subset \langle G \rangle$ is a set of homogeneous elements such that $M\cong \langle G \rangle /\langle R \rangle$.  We denote the presentation $(G,R)$ as $\langle G|R \rangle$.  For n-graded sets $G_1,...,G_l$ and sets $R_1,...,R_m\subset \langle G_1 \cup ...\cup G_l\rangle$, we'll let $\langle G_1,...,G_l|R_1,...,R_m \rangle$ denote $\langle G_1 \cup ...\cup G_l|R_1 \cup ...\cup R_m \rangle$.   

If $M$ is a $B_n$-persistence module such that there exists a presentation $\langle G|R \rangle$ for $M$ with $G$ and $R$ finite, then we say $M$ is {\it finitely presented}.     

\subsubsection{Minimal Presentations of $B_n$-persistence Modules}
Let $M$ be a $B_n$-persistence module.  Define a presentation $\langle G|R \rangle$ of $M$ to be {\it minimal} if 
\begin{enumerate*}
\item the quotient $\langle G\rangle \to \langle G\rangle/\langle R\rangle$ maps $G$ to a minimal set of generators for $\langle G\rangle/\langle R\rangle$.
\item $R$ is a minimal set of generators for $\langle R \rangle$.
\end{enumerate*}

It's clear that a minimal presentation for $M$ exists.

\begin{thm}\label{MinimalPresentationTheorem}
If $M$ is a finitely presented $B_n$-persistence module and $\langle G|R \rangle$ is a minimal presentation of $M$, then for any other presentation $\langle G'|R' \rangle$ of $M$, $gr(G)\leq gr(G')$ and $gr(R)\leq gr(R')$.   
\end{thm}

Note that the theorem implies in particular that if $\langle G|R \rangle$ and $\langle G'|R' \rangle$ are two minimal presentations of $M$ then $gr(G)=gr(G')$ and $gr(R)=gr(R')$.

We defer the proof of the theorem to Appendix~\ref{MinimalPresentationAppendix}.  The proof is an adaptation to our setting of a standard result \cite[Theorem 20.2]{eisenbud1995commutative} about free resolutions of modules over a local ring.  The main effort required in carrying out the adaptation is to prove that the ring $B_n$ has a property known as {\it coherence}; we define coherence and prove that $B_n$ is coherent in Appendix~\ref{CoherenceSection}.

\section{Algebraic Preliminaries for 1-D Persistence}\label{1DPreliminaries}
In this section, we review algebraic preliminaries and establish notation specific to 1-D persistent homology.  This material will be used in Sections~\ref{WellBehavedStructThmSection} and~\ref{InterleavingMetricSection} to develop the machinery needed to prove Theorem~\ref{InterleavingEqualsBottleneck}.

\subsubsection{Basic Notation}
For $S$ any subset of ${\bar \R}^2$, let $S_+=\{(a,b)\in S|a<b\}$.  For $S$ a set and $f:S\to\R$ a function, let $\supp(f)=\{s\in S|f(s)\ne 0\}$.  

\subsection{Structure Theorems For Tame ${\mathbf A}_1$-Persistence Modules}
The structure theorem for finitely generated ${\mathbf A}_1$ persistence modules \cite{zomorodian2005computing} is well known in the applied topology community.  In fact, this theorem generalizes to tame ${\mathbf A}_1$-modules.  The existence portion of the generalized theorem is given e.g. in \cite{webb1985decomposition}; the uniqueness is not mentioned there but is very easy to show; we do so below.  To our knowledge, this generalization has not previously been discussed in the computational topology literature.  We will use the more general theorem to show that the bottleneck distance is equal to the interleaving distance for ordinary persistence.

Before stating the results, we establish some notation.  For $a<b\in \Z$, Let $\bC(a,b)$ denote the module $(k[x]/(x^{b-a}))(-a)$.  Let $\bC(a,\infty)=k[x](-a)$.  Note that for fixed $b$ (possibly infinite), the set of modules $\{\bC(a,b)\}_{a\in (-\infty,b)}$ has a natural directed system structure; let $\bC(-\infty,b)$ denote the colimit of this directed system.  

For $M$ a module and $m \in {\Z_{\geq 0}}$, let $M^m$ denote the direct sum of $m$ copies of $M$.

\begin{thm}[Structure Theorem for finitely generated ${\mathbf A}_1$-persistence modules \cite{zomorodian2005computing}]\label{FinitePersistenceStructureThm}  Let $\M$ be a finitely generated ${\mathbf A}_1$-module.  Then there is a unique function ${\mathcal D}_\M:(\Z\times {\bar \Z})_+\to \Z_{\geq 0}$ with finite support such that \[\M\cong \oplus_{(a,b)\in \supp({\mathcal D}_\M)} \bC(a,b)^{{\mathcal D}_\M(a,b)}.\]
\end{thm} 

\begin{thm}[Structure Theorem for tame ${\mathbf A}_1$-persistence modules \cite{webb1985decomposition}]\label{TamePersistenceStructureThm}  Let $\M$ be a tame ${\mathbf A}_1$-module.  Then there is a unique function ${\mathcal D}_\M:{\bar \Z}^2_+\to \Z_{\geq 0}$ such that \[\M\cong \oplus_{(a,b)\in \supp({\mathcal D}_\M)} \bC(a,b)^{{\mathcal D}_\M(a,b)}.\]
\end{thm}
The uniqueness part of Theorem~\ref{TamePersistenceStructureThm} is an immediate consequence of the following lemma, upon noting that the right hand sides of the equations in the statement of the lemma do not depend on ${\mathcal D}_\M$.  

\begin{lem}\label{MultiplicityFormula} Let $\M$ be a tame ${\mathbf A}_1$-module, and let ${\mathcal D}_\M:{\bar \Z}^2_+\to \Z_{\geq 0}$ be a function such that $\M\cong \oplus_{(a,b)\in \supp({\mathcal D}_\M)} \bC(a,b)^{{\mathcal D}_\M(a,b)}$.  Then   
\begin{enumerate*}
\item[(i)] For $(a,b)\in \Z^2_+$, \[{\mathcal D}_\M(a,b)=\rank(\varphi_\M(a,b-1))-\rank(\varphi_\M(a,b))-\rank(\varphi_\M(a-1,b-1))+\rank(\varphi_\M(a-1,b)).\]  
\item[(ii)] For $b\in \Z$, ${\mathcal D}_\M(-\infty,b)=\lim_{a\to -\infty} \rank(\varphi_\M(a,b-1))-\lim_{a\to -\infty} \rank(\varphi_\M(a,b)).$
\item[(iii)] For $a\in \Z$, ${\mathcal D}_\M(a,\infty)=\lim_{b\to \infty} \rank(\varphi_\M(a,b))-\lim_{b\to \infty} \rank(\varphi_\M(a-1,b)).$
\item[(iv)] ${\mathcal D}_\M(-\infty,\infty)=\lim_{a\to -\infty} \lim_{b\to \infty} \rank(\varphi_\M(a,b)).$
\end{enumerate*}
\end{lem}

\begin{proof} This is trivial. \end{proof}

We call ${\mathcal D}_\M$ the {\bf (discrete) persistence diagram} of $\M$.

In Section~\ref{WellBehavedStructThmSection}, we prove a structure theorem analogous to Theorem~\ref{TamePersistenceStructureThm} for a subset of the tame $B_1$-persistence modules which contains the finitely presented $B_1$-persistence modules.  We do not address the problem of generalizing this structure theorem to the full set of tame $B_1$-persistence modules, but to echo a sentiment expressed in \cite{chazal2009proximity}, it would be nice to have such a result.  


\subsection{Discrete Persistence Modules.}
In order to define persistence diagrams of $B_1$-persistence modules, we need a mild generalization of ${\mathbf A}_1$-persistence modules.  

Let $S\subset \R$ be a countably infinite set with no accumulation point.  The authors of \cite{chazal2009proximity} define a {\bf discrete persistence module} $M_S$ to be a collection of vector spaces $\{M_s\}_{s\in S}$ indexed by $S$ together with linear maps $\{\varphi_{M_S}(s_1,s_2)\}_{s_1\leq s_2\in S}$.  

Define a {\bf grid function} $t:\Z\to \R$ to be a strictly increasing function with no accumulation point.

\begin{remark}Discrete persistence modules are of course closely related to ${\mathbf A}_1$-persistence modules.  A countably infinite subset of $S\subset \R$ with no accumulation point can be indexed by a grid function $t$ with image $S$, and such a grid function is uniquely determined by the value of $t(0)$.  Thus, pairs $(M_S,s)$, where $M_S$ is a discrete persistence module and $s$ is an element of $S$, are equivalent to pairs $(\M',t)$, where $\M'$ is an ${\mathbf A}_1$-persistence module and $t$ is a grid function; there is an equivalence sending each pair $(M_S,s)$ to the pair $(\M',t)$, where $t$ is a grid function with $\im(t)=S$, $t(0)=s$, and $\M'$ is the ${\mathbf A}_1$-persistence module such that for $z\in \Z, \M'_z=M_{t(z)}$ and $\varphi_{\M'}(z_1,z_2)=\varphi_{M_S}(t(z_1),t(z_2))$.  \end{remark}

As a matter of expository convenience, from now on we'll define discrete modules to be pairs $(\M,t)$ where $\M$ is an ${\mathbf A}_1$-persistence module and $t$ is a grid function.  This in effect means we are carrying around the extra data of an element of $S$ in our discrete persistent modules relative to those defined in \cite{chazal2009proximity}, but this won't present a problem--in particular, the definition of the persistence diagram of a discrete persistence module that we present below is independent of the choice of this element, and is equivalent to that of \cite{chazal2009proximity}.

\subsection{Persistence Diagrams and the Bottleneck Distance}
The definition of a persistence diagram that we present here differs in some cosmetic respects from that in \cite{chazal2009proximity}.  Our choice in this regard is a matter of notational convenience; the reader may check that our definition of the bottleneck distance between tame $B_1$-persistence modules is equivalent to that of \cite{chazal2009proximity}.  

For a grid function $t$, define ${\bar t}:{\bar \Z}\to {\bar \R}$ as 
\begin{equation*}
{\bar t}(z)=
\begin{cases} t(z) &\text{if }z\in \Z,
\\
-\infty &\text{if } z=-\infty,\\
\infty &\text{if }z=\infty.
\end{cases}
\end{equation*}

Let ${\bar t}\times {\bar t}:{\bar \Z}^2_+\to {\bar \R}^2_+$ be defined by ${\bar t}\times {\bar t}(a,b)=({\bar t}(a),{\bar t}(b))$.   

We define a {\bf persistence diagram} to be a function ${\mathcal D}:{\bar \R^2}_+\to {\Z_{\geq 0}}$. 

For $(\M,t)$ a discrete persistence module, define ${\mathcal D}_{(\M,t)}$, the {\bf persistence diagram of $(\M,t)$}, to be the persistence diagram for which $\supp({\mathcal D}_{(\M,t)})={\bar t}\times {\bar t}(\supp({\mathcal D}_\M))$ and so that ${\mathcal D}_{(\M,t)}({\bar t}(a),{\bar t}(b))={\mathcal D}_\M(a,b)$ for all $(a,b)\in {\bar \Z}^2_+$.

\subsubsection{Bottleneck Metric}
For $x\in \R$, define $x+\infty=\infty$ and $x-\infty=-\infty$. 
Then the usual definition of $l^\infty$ norm on the plane extends to ${\bar \R}^2$; we denote it by $\|\cdot\|_{\infty}$.  
  
Now define a {\bf multibijection} between two persistence diagrams ${\mathcal D_1},{\mathcal D_2}$ to be a function $\gamma:\supp({\mathcal D_1})\times \supp({\mathcal D_2}) \to \Z_{\geq 0}$ such that 
\begin{enumerate*}
\item For each $x\in \supp({\mathcal D_1})$, the set $\{y\in \supp({\mathcal D_2})|(x,y)\in \supp(\gamma)\}$ is finite and   
 \[{\mathcal D_1}(x)=\sum_{y\in \supp({\mathcal D_2})} \gamma(x,y),\]
\item For each $y\in \supp({\mathcal D_2})$, the set $\{x\in \supp({\mathcal D_1})|(x,y)\in \supp(\gamma)\}$ is finite and 
 \[{\mathcal D_2}(y)=\sum_{x\in \supp({\mathcal D_1})} \gamma(x,y).\]
\end{enumerate*}

For persistence diagrams ${\mathcal D_1},{\mathcal D_2}$, let $\L({\mathcal D_1},{\mathcal D_2})$ denote the triples $({\mathcal D'_1},{\mathcal D'_2},\gamma)$, where ${\mathcal D'_1}$, and ${\mathcal D'_2}$ are persistence diagrams with ${\mathcal D'_1}\leq {\mathcal D_1}, {\mathcal D'_2}\leq{\mathcal D_2}$, and $\gamma$ is a multibijection between ${\mathcal D'_1}$ and ${\mathcal D'_2}$.

We define the {\bf bottleneck metric} $d_B$ between two persistence diagrams ${\mathcal D_1},{\mathcal D_2}$ as \[d_B({\mathcal D_1},{\mathcal D_2})=\inf_{\substack {({\mathcal D'_1},{\mathcal D'_2},\gamma)\\ \in \L({\mathcal D_1},{\mathcal D_2})}} \max \left( \sup_{\substack{(a,b)\in \supp({\mathcal D_1}-{\mathcal D'_1})\\ \cup \ \supp({\mathcal D_2}-{\mathcal D'_2})}} \frac{1}{2}(b-a),\sup_{(x,y)\in \supp(\gamma)} \|y-x\|_{\infty}\right ).\]
   
\subsubsection{Discretizations of $B_1$-modules}\label{DiscretizationsOfModules}  
Let $t$ be a grid function.  For $M$ a $B_1$-persistence module, we define the $t$-discretization of $M$ to be the discrete persistence module $({\mathbf P}_t(M),t)$ with ${\mathbf P}_t(M)$ defined as follows:
\begin{enumerate*}
\item For $z\in \Z$, ${\mathbf P}_t(M)_z=M_{t(z)}$; let ${\mathcal I}_{M,t,z}:{\mathbf P}_t(M)_z\to M_{t(z)}$ denote this identification.
\item For $y,z\in \Z$, $y\leq z$, $\varphi_{{\mathbf P}_t(M)}(y,z)={\mathcal I}_{M,t,z}^{-1} \circ \varphi_M(t(y),t(z)) \circ{\mathcal I}_{M,t,y}$. 
\end{enumerate*}  
  
\subsubsection{Persistence diagrams of $B_1$-persistence modules}\label{SectionDiagramsOfEuclideanGradedModules}
We'll say a grid function $t$ is an {\bf $\epsilon$-cover} if for any $a\in \R$, there exists $b\in \im(t)$ with $|a-b|\leq \epsilon$. 
Now fix $\alpha\in \R$ and let $\{t_i\}_{i=1}^\infty$ be a sequence of grid functions with $t_i$ a $1/2^i$-cover.     

It is asserted in \cite{chazal2009proximity} that for any tame $B_1$-persistence module $M$ the persistence diagrams $\D_{({\mathbf P}_{t_i}(M),t_i)}$ converge in the bottleneck metric to a limiting persistence diagram $\D_M$ and that $\D_M$ is independent of the choice of the sequence $\{t_i\}$.  We call ${\mathcal D}_M$ the persistence diagram of $M$.  For $M$ and $N$ tame $B_1$-persistence modules, we define $d_B(M,N)=d_B(\D_M,\D_N)$.

\begin{remark}
Two non-isomorphic tame $B_1$-persistence modules can have identical persistence diagrams.  For example, take $M$ and $N$ to be the $B_1$-persistence modules of Example~\ref{NonIsoSameDiagram}.  $M$ and $N$ are not isomorphic but it is easy to check that they have the same persistence diagram.  Thus $d_B$ defines a pseudometric (but not a metric) on isomorphism classes of tame $B_1$-persistence modules.
\end{remark}

%% file: Part_I/T_Structure_Theorem_For_Well_Behaved_B_1-modules_3.tex
\section{A Structure Theorem for Well Behaved $B_1$-persistence modules}\label{WellBehavedStructThmSection}
In this section we prove an analogue of Theorem~\ref{TamePersistenceStructureThm} for a certain subset of the tame $B_1$-persistence modules which we call the {\it well behaved} persistence modules.  The set of well behaved persistence modules contains the set of finitely presented $B_1$-persistence modules.  These modules are in a sense ``essentially discrete."  Indeed, they are exactly the $B_1$-persistence modules that are the images of tame ${\mathbf A}_1$-persistence modules under a certain family of functors from ${\mathbf A}_1$-mod to $B_1$-mod.

Our strategy for proving the structure theorem for well behaved persistence modules is to exploit Theorem~\ref{TamePersistenceStructureThm}, taking advantage of the functorial relationship between ${\mathbf A}_1$-persistence modules and well behaved $B_1$-persistence modules.

\subsection{Well Behaved Persistence Modules}

A {\it critical value} of a $B_1$-persistence module $M$ is a point $a\in \R$ such that for no $\epsilon\in \R_{\geq 0}$ is it true that for all $u\leq v\in [a-\epsilon,a+\epsilon]$, $\varphi_M(u,v)$ is an isomorphism.

We'll say a tame $B_1$-persistence module $M$ is {\bf well behaved} if 
\begin{enumerate*}
\item The critical values of $M$ are countable and have no accumulation point.
\item For each critical point $a$ of $M$, there exists $\epsilon>0$ such that $\varphi_M(a,y)$ is an isomorphism for all $y\in [a,a+\epsilon]$.    
\end{enumerate*}

\begin{prop}\label{FinitelyPresentedIsWellBehaved} A finitely presented $B_1$-persistence module is well behaved. \end{prop}
\begin{proof}  Let $M$ be a finitely presented $B_1$-persistence module and let $U\subset \R$ be the set of grades of the generators and relations in a minimal presentation for $M$.  (It follows from Theorem~\ref{MinimalPresentationTheorem} that $U$ is well defined).  Lemma~\ref{FirstIsomorphismLemma} below tells us that for any $a\leq b\in \R$ such that $(a,b]\cap U=\emptyset$, $\varphi_M(a,b)$ is an isomorphism.  Since $U$ is finite, the result follows immediately.  \end{proof}

Let $t$ be a grid function.  Define $t^{-1}:\R\to \Z$ by $t^{-1}(y)=\max\{z\in \Z|t(z)\leq y\}$.  

Define ${\bar t}^{-1}:{\bar \R}\to {\bar \Z}$ by 
\begin{equation*}
{\bar t}^{-1}(u)=
\begin{cases} t^{-1}(u) &\text{if }u\in \R,\\
-\infty &\text{if } u=-\infty,\\
\infty &\text{if }u=\infty.
\end{cases}
\end{equation*}

We'll now define a functor $E_t:{\mathbf A}_1$-mod$\to B_1$-mod as follows:  
\begin{enumerate*}
\item Action of $E_t$ on objects: For $\M$ an ${\mathbf A}_1$-persistence module and $u\in \R$, $E_t(\M)_u=\M_{t^{-1}(u)}$; let ${\mathcal J}_{\M,t,u}:E_t(\M)_u \to \M_{t^{-1}(u)}$ denote this identification.  For $u,v\in\R$, $u\leq v$, let $\varphi_{E_t(\M)}(u,v)={\mathcal J}_{\M,t,v}^{-1} \circ \varphi_\M(t^{-1}(u),t^{-1}(v)) \circ{\mathcal J}_{\M,t,u}$.  

\item Action of $E_t$ on morphisms: For $\M$ and $\N$ ${\mathbf A}_1$-persistence modules and $f\in \hom(\M,\N)$, define $E_t(f):E_t(\M)\to E_t(\N)$ by letting $E_t(f)_u={\mathcal J}_{\N,t,u}^{-1}\circ f_{t^{-1}(u)} \circ {\mathcal J}_{\M,t,u}$ for all $u\in \R$.
\end{enumerate*}
We leave to the reader the easy verification that $E_t$ is in fact a functor with target $B_1$-mod.  

It's clear that if $\M$ is a tame ${\mathbf A}_1$-persistence module, then for any grid function $t$, $E_t(\M)$ is tame.  Moreover, it's easy to check that for any grid function $t$ and any tame ${\mathbf A}_1$-persistence module $\M$, $E_t(\M)$ is well behaved.  

Conversely, we have the following:

\begin{prop}\label{WellBehavedModsAreImages} If $M$ is a well behaved $B_1$-persistence module, then there is some tame ${\mathbf A}_1$-persistence-module $\M$ and some grid function $t$ such that $M\cong E_t(\M)$. \end{prop}

\begin{proof} Let $t:\Z\to \R$ be a grid function whose image contains the critical points of $M$.  Let $(\M,t)$ denote the $t$-discretization of $M$, as defined in Section~\ref{DiscretizationsOfModules}.  $\M$ clearly is tame.  We'll show that $M\cong E_t(\M)$.     

For $u\in \R$, define $\sigma_u:E_t(\M)_u \to M_u$ by $\sigma_u=\varphi_{M}(t\circ t^{-1}(u),u) \circ {\mathcal I}_{M,t,t^{-1}(u)}\circ{\mathcal J}_{\M,t,u}$.  By definition, ${\mathcal J}_{\M,t,u}$ and ${\mathcal I}_{M,t,t^{-1}(u)}$ are  isomorphisms.  Moreover, a simple compactness argument shows that since $M$ is well behaved, $\varphi_{M}(t\circ t^{-1}(u),u)$ is an isomorphism.  Thus $\sigma_u$ is an isomorphism.

We claim that the collection of maps $\{\sigma_u\}_{u\in \R}$ defines an isomorphism of modules.  To see this, we need to show that for all $u,v\in \R, u\leq v$, $\sigma_v \circ \varphi_{E_t(\M)}(u,v)=\varphi_{M}(u,v) \circ \sigma_u$:
\begin{align*}
&\sigma_v \circ \varphi_{E_t(\M)}(u,v)\\
&=\sigma_v \circ {\mathcal J}_{\M,t,v}^{-1} \circ \varphi_{\M}(t^{-1}(u),t^{-1}(v)) \circ{\mathcal J}_{\M,t,u}\\
&=\varphi_{M}(t\circ t^{-1}(v),v) \circ {\mathcal I}_{M,t,t^{-1}(v)}\circ{\mathcal J}_{\M,t,v}\circ {\mathcal J}_{\M,t,v}^{-1}\circ \varphi_{\M}(t^{-1}(u),t^{-1}(v)) \circ{\mathcal J}_{\M,t,u}\\
&=\varphi_{M}(t\circ t^{-1}(v),v) \circ {\mathcal I}_{M,t,t^{-1}(v)}\circ \varphi_{\M}(t^{-1}(u),t^{-1}(v)) \circ{\mathcal J}_{\M,t,u}\\&=\varphi_{M}(t\circ t^{-1}(v),v) \circ {\mathcal I}_{M,t,t^{-1}(v)}\circ {\mathcal I}_{M,t,t^{-1}(v)}^{-1} \\ 
&\quad \circ \varphi_{M}(t\circ t^{-1}(u),t\circ t^{-1}(v)) \circ{\mathcal I}_{M,t,t^{-1}(u)} \circ{\mathcal J}_{\M,t,u}\\
&=\varphi_{M}(t\circ t^{-1}(v),v) \circ \varphi_{M}(t\circ t^{-1}(u),t\circ t^{-1}(v)) \circ{\mathcal I}_{M,t,t^{-1}(u)} \circ{\mathcal J}_{\M,t,u}\\
&=\varphi_{M}(t\circ t^{-1}(u),v) \circ{\mathcal I}_{M,t,t^{-1}(u)} \circ{\mathcal J}_{\M,t,u}\\
&=\varphi_{M}(u,v)\circ \varphi_{M}(t\circ t^{-1}(u),u) \circ {\mathcal I}_{M,t,t^{-1}(u)}\circ{\mathcal J}_{\M,t,u}\\
&=\varphi_{M}(u,v) \circ \sigma_u.  \qedhere
\end{align*} \end{proof}


\begin{remark}\label{MultidimensionExtension} The material above can be adapted with only minor changes to the setting of $B_n$-persistence modules, where it sheds some light on the relationship between ${\mathbf A}_n$-persistence modules and $B_n$-persistence modules.  Namely, the definitions of a well behaved persistence module, grid function, and the functors $E_t$ generalize to the multidimensional setting, and analogues of Propositions~\ref{FinitelyPresentedIsWellBehaved} and~\ref{WellBehavedModsAreImages} hold in that setting.  It can be shown that the functor $(\cdot)\otimes_{{\mathbf A}_n} B_n$ mentioned in Section~\ref{TensorUpRemarks} is naturally isomorphic to a generalized functor $E_t$.  The generalization of the above material also can be used to translate algebraic results about ${\mathbf A}_n$-persistence modules into analogous results about $B_n$-persistence modules.  For example, it can be used to show that any finitely presented $B_n$-persistence module has a free resolution of length at most $n$--that is, an analogue of the Hilbert syzygy theorem holds for $B_n$-persistence modules.  

However, as we have no immediate need for the generalization or its consequences in this thesis, we omit it. \end{remark}

\subsection{The Structure Theorem}
First, note that for any $a\in\R_{\geq 0}$, $k[[a,\infty)]$ (as defined in Section~\ref{MonoidRings}) is an ideal of $B_1$.     

For $a<b\in \R$, let $C(a,b)$ denote $(B_1/k[[b-a,\infty)])(-a)$; let $C(a,\infty)$ denote $B_1(-a)$.  In analogy to the discrete case, for fixed $b$ (possibly infinite), the set of modules $\{C(a,b)|a\in \R, a<b\}$ has a natural directed system structure; let $C(-\infty,b)$ denote the colimit of this directed system.




\begin{lem}\label{DirectSumDecompositionHasFixedForm} Let $M$ be a well-behaved persistence module and let ${\mathcal D}$ be a persistence diagram such that $M\cong \oplus_{(a,b)\in \supp({\mathcal D})} C(a,b)^{{\mathcal D}(a,b)}$.  Then ${\mathcal D}_M={\mathcal D}$.
\end{lem}

\begin{proof} Let \[A=\{a\in \R|(a,b)\in \supp({\mathcal D})\textup{ for some }b\in {\bar \R}\}\cup \{b\in \R|(a,b)\in \supp({\mathcal D})\textup{ for some }a\in {\bar \R}\}.\]
Let $t$ be a grid function such that $A\subset \im(t)$.  We claim that $\D_{({\mathbf P}_t(M),t)}={\mathcal D}$.  Since $\supp({\mathcal D})\in \im({\bar t}\times {\bar t})$, this is true if and only if ${\mathcal D}_{{\mathbf P}_t(M)}(y,z)={\mathcal D}({\bar t}(y),{\bar t}(z))$ for all $(y,z)\in {\bar \Z}^2_+$.

To show that ${\mathcal D}_{{\mathbf P}_t(M)}(y,z)={\mathcal D}({\bar t}(y),{\bar t}(z))$ for all $(y,z)\in {\bar \Z}^2_+$,
we'll need the following analogue of Lemma~\ref{MultiplicityFormula}.  

\begin{lem}\label{WellBehavedMultiplicityFormula} Let $M, {\mathcal D}$, and $t$ be as above.  
\begin{enumerate*}
\item[(i)]For $(y,z)\in \Z^2_+$,
\begin{align*}
{\mathcal D}(t(y),t(z))&=\rank(\varphi_M(t(y),t(z-1)))-\rank(\varphi_M(t(y),t(z)))\\
&\quad-\rank(\varphi_M(t(y-1),t(z-1)))+\rank(\varphi_M(t(y-1),t(z))).  
\end{align*}
\item[(ii)]For $z\in \Z$, \[{\mathcal D}(-\infty,t(z))=\lim_{y\to -\infty} \rank(\varphi_M(t(y),t(z-1)))-\lim_{y\to -\infty} \rank(\varphi_M(t(y),t(z))).\]
\item[(iii)]For $y\in \Z$, \[{\mathcal D}(t(y),\infty)=\lim_{z\to \infty} \rank(\varphi_M(t(y),t(z)))-\lim_{b\to \infty} \rank(\varphi_M(t(y-1),t(z))).\]
\item[(iv)]\[{\mathcal D}(-\infty,\infty)=\lim_{y\to -\infty} \lim_{z\to \infty} \rank(\varphi_M(t(y),t(z))).\]
\end{enumerate*}
\end{lem}

\begin{proof} The proof is straightforward; we omit it.\end{proof}

For $(y,z)\in {\Z}^2_+$ we have 
\begin{align*}
{\mathcal D}_{{\mathbf P}_t(M)}(y,z)
&=\rank(\varphi_{{\mathbf P}_t(M)}(y,z-1))-\rank(\varphi_{{\mathbf P}_t(M)}(y,z))\\
&\quad -\rank(\varphi_{{\mathbf P}_t(M)}(y-1,z-1))+\rank(\varphi_{{\mathbf P}_t(M)}(y-1,z))\\
&=\rank(\varphi_M(t(y),t(z-1)))-\rank(\varphi_M(t(y),t(z)))\\
&\quad-\rank(\varphi_M(t(y-1),t(z-1)))+\rank(\varphi_M(t(y-1),t(z)))\\
&={\mathcal D}(t(y),t(z)),
\end{align*}

where the first equality follows from Lemma~\ref{MultiplicityFormula}(i), and the last equality follows from Lemma~\ref{WellBehavedMultiplicityFormula}(i). 

Thus we have ${\mathcal D}_{{\mathbf P}_t(M)}(y,z)={\mathcal D}({\bar t}(y),{\bar t}(z))$ for all $(y,z)\in \Z^2_+$.  Similar arguments using Lemma~\ref{MultiplicityFormula}(ii)-(iv) and Lemma~\ref{WellBehavedMultiplicityFormula}(ii)-(iv) in the cases where $y=-\infty$ or $z=\infty$ show that in fact this holds for $(y,z) \in {\Z}^2_+$.  This proves the claim. 

It follows easily from the fact that $M$ is well behaved that $A$ is equal to the set of critical values of $M$.  There thus exists a sequence of grid functions $\{t_i\}_{i\in \NN}$ such that $t_i$ is a $1/2^i$ cover and $A\subset \im(t_i)$ for each $i$.
The lemma follows by writing ${\mathcal D}_M$ as the limit of the persistence diagrams ${\mathcal D}_{({\mathbf P}_{t_i}(M),t_i)}$.  \end{proof}

\begin{thm}\label{WellBehavedStructureThm} Let $M$ be a well behaved $B_1$-persistence module.  Let ${\mathcal D}_M$ be the persistence diagram of $M$.  Then 
\[M\cong \oplus_{(a,b)\in \supp({\mathcal D}_M)} C(a,b)^{{\mathcal D}_M(a,b)}.\]  This decomposition of $M$ is unique in the sense that if ${\mathcal D}$ is another persistence diagram such that  
$M\cong \oplus_{(a,b)\in \supp({\mathcal D})} C(a,b)^{{\mathcal D}(a,b)}$, then ${\mathcal D}={\mathcal D}_M$.\end{thm} 

\begin{proof} By Lemma~\ref{DirectSumDecompositionHasFixedForm}, it's enough show that there exists some persistence diagram $\D$ such that $M\cong \oplus_{(a,b)\in \supp(\D)} C(a,b)^{\D(a,b)}$.

By Proposition~\ref{WellBehavedModsAreImages}, there exists a grid function $t$ and a tame ${\mathbf A}_1$-persistence module $\M$ such that $E_t(\M)\cong M$.  The structure theorem for tame ${\mathbf A}_1$-persistence modules gives us that there's a persistence diagram ${\mathcal D}_{\M}$ supported in ${\bar \Z}^2_+$ such that we may take \[\M=\oplus_{(a,b)\in \supp({\mathcal D}_{\M})} \bC(a,b)^{{\mathcal D}_{\M}(a,b)}.\]  We'll show that \[M\cong \oplus_{(a,b)\in \supp({\mathcal D}_{\M})} E_t(\bC(a,b))^{{\mathcal D}_{\M}(a,b)}.\]  We have that $E_t(\bC(a,b))\cong C({\bar t}(a),{\bar t}(b))$ for any $(a,b)\in {\bar \Z}^2_+$, so this gives the result.

To show that $M\cong \oplus_{(a,b)\in \supp({\mathcal D}_{\M})} E_t(\bC(a,b))^{{\mathcal D}_{\M}(a,b)}$, we'll use the category theoretic characterization of direct sums of modules as coproducts \cite{mac1998categories}.  Let $A$ be a set.  Recall that in an arbitrary category, an object $X$ is a {\it coproduct} of objects $\{X^\alpha\}_{\alpha \in A}$ iff there exist morphisms $\{i^\alpha:X^\alpha\to X\}_{\alpha\in A}$, called canonical injections, with the following universal property: for any object $Y$ and morphisms $\{f^\alpha:X^\alpha\to Y\}_{\alpha \in A}$, there exists a unique morphism $f:X \to Y$ such that $f \circ i^\alpha=f^\alpha$ for each $\alpha\in A$.  In a category of modules over a ring $R$, The coproduct of modules $X^\alpha$ is $\oplus_\alpha X^\alpha$; the canonical injections are just the usual inclusions $X^\alpha \hookrightarrow \oplus_\alpha X^\alpha$.  The same is thus true for the module subcategories ${\mathbf A}_n$-mod and $B_n$-mod. 

Now let $\{\M^\alpha\}$ denote the indecomposable summands of $\M$ in the direct sum decomposition $\M=\oplus_{(a,b)\in \supp({\mathcal D}_{\M})} \bC(a,b)^{{\mathcal D}_{\M}(a,b)}$, so that each $\M^\alpha=\bC(a,b)$ for some $(a,b)\in {\bar \Z}^2_+$.
Let $\{i^\alpha:\M^\alpha\to \M\}$ denote the canonical injections.

We'll show that the maps $E_t(i^\alpha):E_t(\M^\alpha)\to E_t(\M)$ satisfy the universal property of a coproduct, so that $M\cong E_t(\M)\cong \oplus_\alpha E_t(\M^\alpha)$ as desired.
  
To show that the maps $E_t(i^\alpha):E_t(\M^\alpha)\to E_t(\M)$ satisfy the universal property of a coproduct, let $Y$ be an arbitrary $B_1$-persistence module and $\{f^\alpha:E_t(\M^\alpha)\to Y\}$ be homomorphisms.  

For any $z\in \Z$, $\M_z\cong \oplus_\alpha \M^\alpha_z$.  
It follows from the definition of $E_t$ that for any $r\in \R$, \[E_t(\M)_r\cong \oplus_\alpha E_t(\M^\alpha)_r\]
with the maps $E_t(i^\alpha)_r$ the canonical inclusions.  

For each $r\in \R$, define $f_r:E_t(\M)_r\to Y_r$ as $\oplus_\alpha f^\alpha_r$ (i.e. $f_r$ is the map guaranteed to exist by the universal property of direct sums for vector spaces.)  It's easy to check that the maps $f_r$ commute with the transition maps in $E_t(\M)$ and $Y$, 
so that they define a morphism $f:E_t(\M)\to Y$.  We also have that $f \circ E_t(i^\alpha)=f^\alpha$ for each $\alpha$.  By the universal property of direct sums of vector spaces, for each $r$ $f_r$ is the unique linear transformation from $E_t(\M)_r$ to $Y_r$ such that for each $\alpha$, $f_r\circ E_t(i^\alpha)_r=f^\alpha_r$.  Therefore $f$ must itself satisfy the desired uniqueness property.
This completes the proof.  \end{proof}

%% file: Part_I/T_Interleaving_Equals_Bottleneck_3.tex
\section{The Equality of the Interleaving and Bottleneck distances on Tame $B_1$-persistence Modules}\label{InterleavingMetricSection}
We show in this section that the restriction of the interleaving distance to tame $B_1$-persistence modules is equal to the bottleneck distance.  This shows that the interleaving distance is in fact a generalization of the bottleneck distance, as we want.  The result is also instrumental in proving Corollary~\ref{Cor:Converse}, our converse to the algebraic stability result of \cite{chazal2009proximity}.  

\subsubsection{The Algebraic Stability of Persistence}
The main result of \cite{chazal2009proximity}, generalizing considerably the earlier result of \cite{cohen2007stability}, is the following:

\begin{thm}[Algebraic Stability of Persistence]\label{AlgebraicStability} Let $\epsilon>0$, and let $M$ and $N$ be two tame $B_1$-persistence modules. If $M$ and $N$ are $\epsilon$-interleaved, then $d_B(M,N)\leq \epsilon$.\end{thm}

\subsubsection{A Converse to the Algebraic Stability of Persistence?}
In the conclusion of \cite{chazal2009proximity}, the authors ask whether it's true that if $M$ and $N$ are tame $B_1$-persistence modules with $d_B(M,N)=\epsilon$ then $M$ and $N$ are $\epsilon$-interleaved.  Example~\ref{NonIsoSameDiagram} shows that this is not true.  However, Corollary~\ref{Cor:Converse} below, which follows immediately from Theorems~\ref{InterleavingEqualsBottleneck} and~\ref{InterleavingThm}, asserts that the result is true provided $M$ and $N$ are finitely presented.  More generally, Corollary~\ref{Cor:Converse} tells us that if $M$ and $N$ are tame modules with $d_B(M,N)=\epsilon$ then $M$ and $N$ are $(\epsilon+\delta)$-interleaved for any $\delta>0$.  In other words, the converse of Theorem~\ref{AlgebraicStability} holds for tame modules to arbitrarily small error.

\begin{thm}\label{InterleavingEqualsBottleneck} $d_B(M,N)=d_I(M,N)$ for any tame $B_1$-persistence modules $M$ and $N$.\end{thm}

\begin{proof} Theorem~\ref{AlgebraicStability} tells us that $d_B(M,N)\leq d_I(M,N)$, so we just need to show that $d_B(M,N)\geq d_I(M,N)$.  It will follow from the structure theorem for well behaved persistence modules (Theorem~\ref{WellBehavedStructureThm}) that $d_B(M',N')\geq d_I(M',N')$ for well behaved persistence modules $M'$ and $N'$ (Lemma~\ref{InterleavingEqualsBottleneckForWellBehaved} below).  To extend this result to arbitrary tame modules, we will approximate the modules $M$ and $N$ up to arbitrarily small error in the interleaving distance by well behaved persistence modules (Lemma~\ref{WellBehavedApproximation} below).  The full result will follow readily from this this approximation.

\begin{lem}\label{CyclicConverse} If $(a,b),(a',b')\in {\bar \R}^2_+$ with $\|(a,b)-(a',b')\|_{\infty}\leq \epsilon$, then $C(a,b)$ and $C(a',b')$ are $\epsilon$-interleaved. \end{lem}

\begin{proof} This is easy to prove; we leave the details to the reader. \end{proof} 

\begin{lem}\label{InterleavingEqualsBottleneckForWellBehaved} If $M$ and $N$ are two well behaved persistence modules and $d_B(M,N)=\epsilon$ then $d_I(M,N)=\epsilon$. \end{lem}

\begin{proof} By stability we just need to show that $d_I(M,N)\leq\epsilon$.  By the structure theorem for well behaved persistence modules (Theorem~\ref{TamePersistenceStructureThm}), we have that 
\begin{align*}
M&\cong\oplus_{(a,b)\in \supp({\mathcal D}_M)} C(a,b)^{{\mathcal D}_{M}(a,b)},\\ 
N&\cong\oplus_{(a,b)\in \supp({\mathcal D}_N)} C(a,b)^{{\mathcal D}_N(a,b)}.
\end{align*}
Since $d_B(M',N')=\epsilon$, for any $\delta>0$ there exist persistence diagrams $D'_M$ and $D'_N$ with $D'_M\leq D_M$, $D'_N\leq D_N$, and a multibijection $\gamma$ between $D'_M$ and $D'_N$ such that
\begin{enumerate*}
\item For any $(a,b)\in \supp(D_M-D'_M) \cup \supp(D_N-D'_N)$, $(b-a)/2\leq \epsilon+\delta$, 
\item For any $(x,y)\in \supp(\gamma)$, $\|x-y\|_{\infty}\leq \epsilon+\delta$. 
\end{enumerate*}
Fix such $D'_M$, $D'_N$, and $\gamma$.  Now we can choose well behaved modules $M',M''\subset M$ and $N',N''\subset N$ such that $M=M'\oplus M''$, $N=N'\oplus N''$, $D_{M'}=D'_M$, $D_{N'}=D'_N$, $D_{M''}=D_M-D'_M$, and $D_{N''}=D_N-D'_N$.  

If follows from Lemma~\ref{CyclicConverse} that for each $(a,b),(a',b')\in \supp(\gamma)$, $C(a,b)$ and $C(a',b')$ are $(\epsilon+\delta)$-interleaved.   We may write 
\begin{align*}
M'&\cong \oplus_{(a,b),(a',b')\in \supp(\gamma)}C(a,b)^{\gamma((a,b),(a',b'))},\\ 
N'&\cong \oplus_{(a,b),(a',b')\in \supp(\gamma)}C(a',b')^{\gamma((a,b),(a',b'))}.
\end{align*}
It's clear from the form of these decompositions for $M'$ and $N'$ that a choice of a pair of $(\epsilon+\delta)$-interleaving homomorphisms between $C(a,b)$ and $C(a',b')$ for each pair $(a,b),(a',b')\in \supp(\gamma)$ induces a pair of $(\epsilon+\delta)$-interleaving homomorphisms ${\hat f}:M'\to N'(\epsilon+\delta)$ and ${\hat g}:N'\to M'(\epsilon+\delta)$.  

Now we extend this pair to a pair of homomorphisms $f:M\to N(\epsilon+\delta)$, $g:N\to M(\epsilon+\delta)$ by defining $f(y)={\hat f}(y)$ for $y\in M'$, $f(M'')=0$, $g(y)={\hat g}(y)$ for $y\in N'$, and $g(M'')=0$.  Obviously, $f$ and $g$ restrict to $(\epsilon+\delta)$-interleaving homomorphisms between $M'$ and $N'$.  Moreover, we have that $S(M'',2\epsilon+\delta)=0$ and $S(N'',2\epsilon+\delta)=0$, so $f$ and $g$ restrict to $(\epsilon+\delta)$-interleaving homomorphisms between $M''$ and $N''$.  Thus by linearity, $f$ and $g$ are $(\epsilon+\delta)$-interleaving homomorphisms between $M$ and $N$.  Since $\delta$ may be taken to be arbitrarily small, we must have $d_I(M,N)\leq\epsilon$, as we wanted to show. \end{proof}

\begin{lem}\label{WellBehavedApproximation} For any tame $B_1$-persistence module $M$ and $\delta>0$, there exists a well behaved persistence module $M'$ with $d_I(M,M')\leq \delta$.  \end{lem}

\begin{proof} Let $t$ be an $\delta/2$-cover of $\R$, as defined in Section~\ref{SectionDiagramsOfEuclideanGradedModules}.  For any $r\in \R$, there exists $r'\in im(t)$, with $0\leq r'-r\leq \delta$.  Define a function $\lambda:\R\to \im(t)$ such that $\lambda(r)=\min \{r'\geq r|r'\in im(t)\}$.  Then $0\leq \lambda(r)-r\leq \delta$ for all $r\in \R$.  

Let $\M={\mathbf P}_t(M)$ (as defined in Section~\ref{DiscretizationsOfModules}) and let $M'=E_t(\M)$.  Then $M'$ is well-behaved.  We now show that $M$ and $M'$ are $\delta$-interleaved, which implies that $d_I(M,M')\leq \delta$.  

Define $f:M\to M'(\delta)$ to be the morphism for which \[f_u:M_u\to M'_{u+\delta}=\varphi_{M'}(\lambda(u),u+\delta)\circ {\mathcal J}^{-1}_{\M,t,\lambda(u)}\circ {\mathcal I}^{-1}_{M,t,t^{-1}(\lambda(u))} \circ \varphi_M(u,\lambda(u)).\]

Define $g:M'\to M(\delta)$ to be the morphism for which \[g_u:M'_u\to M_{u+\delta}=\varphi_{M}(\lambda(u),u+\delta)\circ {\mathcal I}_{M,t,t^{-1}(\lambda(u))}\circ {\mathcal J}_{\M,t,\lambda(u)} \circ \varphi_{M'}(u,\lambda(u)).\] 

We need to check that $f$ and $g$ thus defined are in fact morphisms.  We verify this for $f$; the verification for $g$ is similar; we omit it.  

If $u\leq v\in \R$, we have 
\begin{align}
f_v \circ \varphi_{M}(u,v)=\varphi_{M'}(\lambda(v),v+\delta)\circ {\mathcal J}^{-1}_{\M,t,\lambda(v)}\circ {\mathcal I}^{-1}_{M,t,t^{-1}(\lambda(v))} \circ \varphi_M(v,\lambda(v))\circ \varphi_{M}(u,v).\label{firsteq}
\end{align} 
By definition, for any $u\leq v\in \R$, we have 
\begin{align}
\varphi_{M'}(u,v)&={\mathcal J}_{\M,t,v}^{-1} \circ \varphi_{\M}(t^{-1}(u),t^{-1}(v)) \circ{\mathcal J}_{\M,t,u}\notag\\&=
{\mathcal J}_{\M,t,v}^{-1} \circ {\mathcal I}_{M,t,t^{-1}(v)}^{-1} \circ \varphi_M(u,v) \circ{\mathcal I}_{M,t,t^{-1}(u)} \circ{\mathcal J}_{\M,t,u}.\label{funeq} 
\end{align}
Using (\ref{funeq}) to substitute for $\varphi_{M'}(\lambda(v),v+\delta)$ in (\ref{firsteq}) and simplifying gives us:
\begin{align*}
f_v \circ \varphi_{M}(u,v)&={\mathcal J}_{\M,t,v+\delta}^{-1} \circ {\mathcal I}_{M,t,t^{-1}(v+\delta)}^{-1} \circ \varphi_M(\lambda(v),v+\delta)\circ \varphi_M(v,\lambda(v))\circ \varphi_{M}(u,v)\\
&={\mathcal J}_{\M,t,v+\delta}^{-1} \circ {\mathcal I}_{M,t,t^{-1}(v+\delta)}^{-1} \circ \varphi_M(u,v+\delta).
\end{align*}
On the other hand we have, using (\ref{funeq}) again,
\begin{align*}
&\varphi_{M'}(u+\delta,v+\delta)\circ f_u\\
&=\varphi_{M'}(u+\delta,v+\delta)\circ\varphi_{M'}(\lambda(u),u+\delta)\circ {\mathcal J}^{-1}_{\M,t,\lambda(u)}\circ {\mathcal I}^{-1}_{M,t,t^{-1}(\lambda(u))} \circ \varphi_M(u,\lambda(u))\\
&=\varphi_{M'}(\lambda(u),v+\delta) \circ {\mathcal J}^{-1}_{\M,t,\lambda(u)}\circ {\mathcal I}^{-1}_{M,t,t^{-1}(\lambda(u))} \circ \varphi_M(u,\lambda(u))\\
&={\mathcal J}_{\M,t,v+\delta}^{-1} \circ {\mathcal I}_{M,t,t^{-1}(v+\delta)}^{-1} \circ \varphi_M(\lambda(u),v+\delta) \circ{\mathcal I}_{M,t,t^{-1}(\lambda(u))}\\ 
& \quad \circ {\mathcal J}_{\M,t,\lambda(u)}\circ {\mathcal J}^{-1}_{\M,t,\lambda(u)}\circ {\mathcal I}^{-1}_{M,t,t^{-1}(\lambda(u))} \circ \varphi_M(u,\lambda(u))\\
&={\mathcal J}_{\M,t,v+\delta}^{-1} \circ {\mathcal I}_{M,t,t^{-1}(v+\delta)}^{-1} \circ \varphi_M(\lambda(u),v+\delta)\circ  \varphi_M(u,\lambda(u))\\
&={\mathcal J}_{\M,t,v+\delta}^{-1} \circ {\mathcal I}_{M,t,t^{-1}(v+\delta)}^{-1} \circ \varphi_M(u,v+\delta).
\end{align*}
Thus $f_v \circ \varphi_{M}(u,v)=\varphi_{M'}(u+\delta,v+\delta)\circ f_u$, as we wanted to show.

Finally, we need to check that $g\circ f=S(M,2\delta)$ and $f\circ g=S(M',2 \delta)$.  We perform the first verification and omit the second, since the verifications are similar.  

For $u\in \R$,
\begin{align*}
&g_{u+\delta} \circ f_u\\
&=\varphi_{M}(\lambda(u+\delta),u+2\delta)\circ {\mathcal I}_{M,t,t^{-1}(\lambda(u+\delta))}\circ {\mathcal J}_{\M,t,\lambda(u+\delta)} \\ 
&\quad \circ \varphi_{M'}(u+\delta,\lambda(u+\delta))\circ \varphi_{M'}(\lambda(u),u+\delta)\circ {\mathcal J}^{-1}_{\M,t,\lambda(u)}\circ {\mathcal I}^{-1}_{M,t,t^{-1}(\lambda(u))} \circ \varphi_M(u,\lambda(u))\\
&=\varphi_{M}(\lambda(u+\delta),u+2\delta)\circ {\mathcal I}_{M,t,t^{-1}(\lambda(u+\delta))}\circ {\mathcal J}_{\M,t,\lambda(u+\delta)} \\ 
&\quad \circ\varphi_{M'}(\lambda(u),\lambda(u+\delta))\circ {\mathcal J}^{-1}_{\M,t,\lambda(u)}\circ {\mathcal I}^{-1}_{M,t,t^{-1}(\lambda(u))} \circ \varphi_M(u,\lambda(u)).
\end{align*}
Using (\ref{funeq}) once again, we have that this last expression is equal to
\begin{align*}
&\varphi_{M}(\lambda(u+\delta),u+2\delta)\circ {\mathcal I}_{M,t,t^{-1}(\lambda(u+\delta))}\circ {\mathcal J}_{\M,t,\lambda(u+\delta)} \circ {\mathcal J}_{\M,t,\lambda(u+\delta)}^{-1} \circ {\mathcal I}_{M,t,t^{-1}(\lambda(u+\delta))}^{-1}  \\
&\quad\circ \varphi_M(\lambda(u),\lambda(u+\delta)) \circ{\mathcal I}_{M,t,t^{-1}(\lambda(u))} \circ{\mathcal J}_{\M,t,\lambda(u)}\circ {\mathcal J}^{-1}_{\M,t,\lambda(u)}\circ {\mathcal I}^{-1}_{M,t,t^{-1}(\lambda(u))} \circ \varphi_M(u,\lambda(u))\\
&=\varphi_{M}(\lambda(u+\delta),u+2\delta)\circ \varphi_M(\lambda(u),\lambda(u+\delta))\circ \varphi_M(u,\lambda(u))\\
&=\varphi_{M}(u,u+2\delta).\qedhere
\end{align*}\end{proof} 

Now we can complete the proof of Theorem~\ref{InterleavingEqualsBottleneck}.  As mentioned above, by Theorem~\ref{AlgebraicStability} it suffices to show $d_I(M,N)\leq d_B(M,N)$.  Say $d_B(M,N)=\epsilon$.  Choose $\delta>0$.  By Lemma~\ref{WellBehavedApproximation}, there exist well behaved modules $M'$, $N'$ with $d_I(M,M')\leq \delta$, $d_I(N,N')\leq \delta$.  Then by Theorem~\ref{AlgebraicStability}, $d_B(M,M')\leq \delta$, $d_B(N,N')\leq \delta$, so by the triangle inequality, $d_B(M',N')\leq \epsilon+2\delta$.  By Lemma~\ref{InterleavingEqualsBottleneckForWellBehaved},  $d_I(M',N')\leq \epsilon+2\delta$.  Applying the triangle inequality again, we get that $d_I(M,N)\leq \epsilon+4\delta$.  As $\delta$ may be taken to be arbitrarily small, we have $d_I(M,N)\leq \epsilon$, which completes the proof. 
\end{proof}

\section{If $d_I(M,N)=\epsilon$ and $M,N$ are Finitely Presented then $M$ and $N$ are $\epsilon$-interleaved}\label{SectionInterleavingThm}
We now show that for finitely presented $B_n$-modules $M$ and $N$, if $d_I(M,N)=\epsilon$ then $M$ and $N$ are $\epsilon$-interleaved.  This implies that the restriction of $d_I$ to finitely presented persistence modules is a metric and, as noted in Section~\ref{InterleavingMetricSection}, yields a converse to the algebraic stability of persistence for finitely presented $B_1$-persistence modules.  Theorem~\ref{InterleavingThm} will also be of some use to us in Section~\ref{ComputationSection}.

\begin{thm}\label{InterleavingThm}If $M$ and $N$ are finitely presented $B_n$-modules and $d_I(M,N)=\epsilon$ then $M$ and $N$ are $\epsilon$-interleaved. \end{thm}    

\begin{cor}\label{MetricCorollary} $d_I$ is a metric on finitely presented $B_n$-modules. \end{cor} 

\begin{cor}[Converse to Algebraic Stability]\label{Cor:Converse} \mbox{}
\begin{itemize}
\item[(i)]If $M$ and $N$ are finitely presented $B_1$-persistence modules and $d_B(M,N)=\epsilon$ then $M$ and $N$ are $\epsilon$-interleaved.  
\item[(ii)]If $M$ and $N$ are tame $B_1$-persistence modules and $d_B(M,N)=\epsilon$ then $M$ and $N$ are $(\epsilon+\delta)$-interleaved for any $\delta>0$.
\end{itemize}
\end{cor} 

\begin{proof} (ii) follows directly from Theorem~\ref{InterleavingEqualsBottleneck}. (i) is immediate from that theorem and Theorem~\ref{InterleavingThm}. \end{proof}

For a finitely presented $B_n$-persistence module $M$, let $U_M\subset \R^n$ be the set of grades of the generators and relations in a minimal presentation for $M$.  Let $U_M^i\subset \R$ be the set of $i^{th}$ coordinates of the elements of $U_M$.  


\begin{proof}[Proof of Theorem~\ref{InterleavingThm}]  

\begin{lem}\label{FirstIsomorphismLemma} If $M$ is a finitely presented $B_n$-persistence module then for any $a\leq b\in \R^n$ such that $(a_i,b_i]\cap U_M^i=\emptyset$ for all $i$, $\varphi_M(a,b)$ is an isomorphism. \end{lem}

\begin{proof} This is straightforward; we omit the details.\end{proof}

\begin{lem}\label{FirstConsequence}If $M$ is a finitely presented $B_n$-persistence module then for any $y\in \R^n$, there exists $r\in \R_{>0}$ such that $\varphi_M(y,y+r')$ is an isomorphism for all $0\leq r'\leq r$. \end{lem}  

\begin{proof} This is an immediate consequence of Lemma~\ref{FirstIsomorphismLemma}. \end{proof}   

For a finitely presented $B_n$-persistence module $M$, let $fl_M:\R^n\to \Pi_{i=1}^n {\bar U_M^i}$ be defined by $fl_M(a_1,...,a_n)=(a'_1,...,a'_n)$, where $a'_i$ is the largest element of $U_M^i$ such that $a'_i\leq a_i$, if such an element exists, and $a'_i=-\infty$ otherwise. 

\begin{lem}\label{SecondConsequence}For any finitely presented $B_n$-module $M$ and any $y\in \R^n$ with $fl_M(y)\in \R^n$, we have that $\varphi_M(fl_M(y),y)$ is an isomorphism. \end{lem}

\begin{proof} This too is an immediate consequence of Lemma~\ref{FirstIsomorphismLemma}. \end{proof}
Having stated these preliminary results, we proceed with the proof of Theorem~\ref{InterleavingThm}.  By Lemma~\ref{FirstConsequence} and the finiteness of $U_M$ and $U_N$, there exists $\delta>0$ such that for all $z\in U_M$, $\varphi_N(z+\epsilon ,z+\epsilon+\delta)$ and $\varphi_M(z+2\epsilon ,z+2\epsilon+2\delta)$ are isomorphisms, and for all $z\in U_N$, $\varphi_M(z+\epsilon,z+\epsilon+\delta)$ and $\varphi_N(z+2\epsilon,z+2\epsilon+2\delta)$ are isomorphisms.     

By Remark~\ref{EasyInterleavingRemark}, since $d_I(M,N)=\epsilon$, $M$ and $N$ are $(\epsilon+\delta)$-interleaved.  

Theorem~\ref{InterleavingThm} then follows from the following lemma, which will also be the key ingredient in the proof of Proposition~\ref{PossibilitiesForInterleavingDistance}. 

\begin{lem}\label{SmallerInterleavingLemma}
Let $M$ and $N$ be finitely presented $B_n$-persistence modules and let $\epsilon\geq 0$ and $\delta>0$ be such that
\begin{enumerate*}
\item $M$ and $N$ are $\epsilon+\delta$-interleaved,
\item for all $z\in U_M$, $\varphi_N(z+\epsilon,z+\epsilon+\delta)$ and $\varphi_M(z+2\epsilon,z+2\epsilon+2\delta)$ are isomorphisms,
\item for all $z\in U_N$, $\varphi_M(z+\epsilon,z+\epsilon+\delta)$ and $\varphi_N(z+2\epsilon,z+2\epsilon+2\delta)$ are isomorphisms.
\end{enumerate*}
Then $M$ and $N$ are $\epsilon$-interleaved.
\end{lem}
\begin{proof} Let $f:M\to N(\epsilon+\delta)$ and $g:N\to M(\epsilon+\delta)$ be interleaving homomorphisms.  

We define $\epsilon$-interleaving homomorphisms ${\tilde f}:M\to N(\epsilon)$ and ${\tilde g}:N\to M(\epsilon)$ via their action on homogeneous summands.  First, for $z\in U_M$ define ${\tilde f}_z=\varphi^{-1}_N(z+\epsilon,z+\epsilon+\delta) \circ f_z$.  Then for arbitrary $z\in \R^n$ such that $fl_M(z)\in \R^n$ define ${\tilde f}_z=\varphi_N(fl_M(z)+\epsilon,z+\epsilon)\circ {\tilde f}_{fl_M(z)} \circ \varphi^{-1}_M(fl_M(z),z)$.  (Note that $\varphi^{-1}_M(fl_M(z),z)$ is well defined by Lemma~\ref{SecondConsequence}.)  Finally, for $z\in \R^n$ s.t. $fl_M(z)\not\in \R^n$, define ${\tilde f}_z=0$.  (If $fl_M(z)\not\in \R^n$ then $M_z=0$, so this last part of the definition is reasonable.)

Symmetrically, for $z\in U_N$ define ${\tilde g}_z=\varphi^{-1}_M(z+\epsilon,z+\epsilon+\delta) \circ g_z$.  For arbitrary $z\in \R^n$ such that $fl_N(z)\in\R^n$ define ${\tilde g}_z=\varphi_M(fl_N(z)+\epsilon,z+\epsilon)\circ {\tilde g}_{fl_N(z)} \circ \varphi^{-1}_N(fl_N(z),z)$.  For $z\in \R^n$ s.t. $fl_N(z)\not\in \R^n$, define ${\tilde g}_z=0$.   

We need to check that ${\tilde f},{\tilde g}$ as thus defined are in fact morphisms.  We perform the check for ${\tilde f}$; the check for ${\tilde g}$ is the same.

If $y\in \R^n$ is such that $fl_M(y)\not\in \R^n$, then since $M_y=0$, it's clear that ${\tilde f}_z \circ \varphi_M(y,z)=\varphi_N(y+\epsilon,z+\epsilon)\circ {\tilde f}_y$.

For $y\leq z\in \R^n$ such that $fl_M(y)\in \R^n$,
\begin{align*}
{\tilde f}_z \circ \varphi_M(y,z) &=\varphi_N(fl_M(z)+\epsilon,z+\epsilon)\circ {\tilde f}_{fl_M(z)} \circ \varphi^{-1}_M(fl_M(z),z)\circ \varphi_M(y,z)\\
&=\varphi_N(fl_M(z)+\epsilon,z+\epsilon)\circ \varphi^{-1}_N(fl_M(z)+\epsilon,fl_M(z)+\epsilon+\delta) \circ f_{fl_M(z)}\\
&\quad \circ \varphi^{-1}_M(fl_M(z),z)\circ \varphi_M(y,z)\\
&=\varphi_N(fl_M(z)+\epsilon,z+\epsilon)\circ \varphi^{-1}_N(fl_M(z)+\epsilon,fl_M(z)+\epsilon+\delta) \circ f_{fl_M(z)} \\ 
&\quad \circ \varphi_M(fl_M(y),fl_M(z))\circ \varphi^{-1}_M(fl_M(y),y)\\
&=\varphi_N(fl_M(z)+\epsilon,z+\epsilon)\circ \varphi^{-1}_N(fl_M(z)+\epsilon,fl_M(z)+\epsilon+\delta) \\
&\quad \circ \varphi_N(fl_M(y)+\epsilon+\delta,fl_M(z)+\epsilon+\delta) \circ f_{fl_M(y)}  \circ \varphi^{-1}_M(fl_M(y),y)\\
&=\varphi_N(y+\epsilon,z+\epsilon) \circ \varphi_N(fl_M(y)+\epsilon,y+\epsilon) \\ 
&\quad \circ \varphi_N^{-1}(fl_M(y)+\epsilon,fl_M(y)+\epsilon+\delta) \circ f_{fl_M(y)}  \circ \varphi^{-1}_M(fl_M(y),y)\\
&=\varphi_N(y+\epsilon,z+\epsilon)\circ {\tilde f}_y
\end{align*}
as desired.  

To finish the proof, we need to check that ${\tilde g}\circ {\tilde f}=S(M,2 \epsilon)$ and ${\tilde f}\circ {\tilde g}=S(N,2 \epsilon)$.  We perform the first check;  the second check is the same.

For $z\in \R^n$, if $fl_M(z)\not\in \R^n$ then since $M_z=0$, ${\tilde g}_{z+\epsilon}\circ {\tilde f}_z=0=\varphi_M(z,z+2\epsilon )$.

To show that the result also holds for $z$ such that $fl_M(z)\in \R^n$, we'll begin by verifying the result for $z\in U_M$.  We'll use this special case in proving the result for arbitrary $z\in \R^n$ such that $fl_M(z)\in \R^n$.  

If $z\in U_M$ then, by assumption, $\varphi_M(z+2\epsilon,z+2\epsilon+2\delta)$ is an isomorphism.  Thus, to show that ${\tilde g}_{z+\epsilon}\circ {\tilde f}_z=\varphi_M(z,z+2\epsilon )$, it suffices to show that $\varphi_M(z+2\epsilon,z+2\epsilon+2\delta) \circ {\tilde g}_{z+\epsilon}\circ{\tilde f}_z=\varphi_M(z,z+2\epsilon+2\delta)$.  

For $z\in U_M$, we have 
\begin{align*}
{\tilde g}_{z+\epsilon}\circ {\tilde f}_z&=\varphi_M(fl_N(z+\epsilon)+\epsilon,z+2\epsilon)\circ {\tilde g}_{fl_N(z+\epsilon)} \circ \varphi^{-1}_N(fl_N(z+\epsilon),z+\epsilon)\circ {\tilde f}_z \\
&=\varphi_M(fl_N(z+\epsilon)+\epsilon,z+2\epsilon)\circ \varphi^{-1}_M(fl_N(z+\epsilon)+\epsilon,fl_N(z+\epsilon)+\epsilon+\delta)\\
&\quad \circ g_{fl_N(z+\epsilon)} \circ \varphi^{-1}_N(fl_N(z+\epsilon),z+\epsilon)\circ {\tilde f}_z\\
&=\varphi_M(fl_N(z+\epsilon)+\epsilon,z+2\epsilon)\circ \varphi^{-1}_M(fl_N(z+\epsilon)+\epsilon,fl_N(z+\epsilon)+\epsilon+\delta)\\
&\quad \circ g_{fl_N(z+\epsilon)} \circ \varphi^{-1}_N(fl_N(z+\epsilon),z+\epsilon)\circ \varphi^{-1}_N(z+\epsilon,z+\epsilon+\delta) \circ f_z.
\end{align*}
Thus 
\begin{align*}
&\varphi_M(z+2\epsilon,z+2\epsilon+2\delta)\circ {\tilde g}_{z+\epsilon}\circ {\tilde f}_z \\
&=\varphi_M(z+2\epsilon,z+2\epsilon+\delta)\circ \varphi_M(fl_N(z+\epsilon)+\epsilon,z+2\epsilon)\\
&\quad \circ \varphi^{-1}_M(fl_N(z+\epsilon)+\epsilon,fl_N(z+\epsilon)+\epsilon+\delta) \circ g_{fl_N(z+\epsilon)}\\
&\quad\circ \varphi^{-1}_N(fl_N(z+\epsilon),z+\epsilon)\circ \varphi^{-1}_N(z+\epsilon,z+\epsilon+\delta) \circ f_z\\
&=\varphi_M(z+2\epsilon+\delta,z+2\epsilon+2\delta) \circ \varphi_M(fl_N(z+\epsilon)+\epsilon+\delta,z+2\epsilon+\delta) \\ 
&\quad \circ g_{fl_N(z+\epsilon)} \circ\varphi^{-1}_N(fl_N(z+\epsilon),z+\epsilon)\circ \varphi^{-1}_N(z+\epsilon,z+\epsilon+\delta) \circ f_z\\ 
&=\varphi_M(z+2\epsilon+\delta,z+2\epsilon+2\delta) \circ
g_{z+\epsilon} \circ \varphi_N(fl_N(z+\epsilon),z+\epsilon) \\ 
&\quad \circ \varphi^{-1}_N(fl_N(z+\epsilon),z+\epsilon)\circ \varphi^{-1}_N(z+\epsilon,z+\epsilon+\delta) \circ f_z\\
&=g_{z+\epsilon+\delta} \circ \varphi_N(z+\epsilon,z+\epsilon+\delta) \circ \varphi^{-1}_N(z+\epsilon,z+\epsilon+\delta) \circ f_z \\
&=g_{z+\epsilon+\delta}\circ f_z\\
&=\varphi(z,z+2\epsilon+2\delta) 
\end{align*} 
as desired.

Finally, for arbitrary $z\in \R^n$ such that $fl_M(z)\not\in \R^n$, we have, using that ${\tilde g}$ is a morphism, 
\begin{align*}
{\tilde g}_{z+\epsilon}\circ {\tilde f}_z&=\varphi_M(fl_N(z+\epsilon)+\epsilon,z+2\epsilon)\circ {\tilde g}_{fl_N(z+\epsilon)}  \\
&\quad \circ \varphi^{-1}_N(fl_N(z+\epsilon),z+\epsilon)\circ \varphi_N(fl_M(z)+\epsilon,z+\epsilon)\circ {\tilde f}_{fl_M(z)} \circ \varphi^{-1}_M(fl_M(z),z)\\
&={\tilde g}_{z+\epsilon}\circ\varphi_N(fl_N(z+\epsilon),z+\epsilon) \circ \varphi^{-1}_N(fl_N(z+\epsilon),z+\epsilon)\\
&\quad \circ \varphi_N(fl_M(z)+\epsilon,z+\epsilon)\circ {\tilde f}_{fl_M(z)} \circ \varphi^{-1}_M(fl_M(z),z)\\
&={\tilde g}_{z+\epsilon}\circ\varphi_N(fl_M(z)+\epsilon,z+\epsilon)\circ {\tilde f}_{fl_M(z)} \circ \varphi^{-1}_M(fl_M(z),z)\\
&=\varphi_M(fl_M(z)+2\epsilon,z+2\epsilon)\circ{\tilde g}_{fl_M(z)+\epsilon}\circ {\tilde f}_{fl_M(z)} \circ \varphi^{-1}_M(fl_M(z),z)\\
&=\varphi_M(fl_M(z)+2\epsilon,z+2\epsilon)\circ \varphi_M(fl_M(z),fl_M(z)+2\epsilon)\circ \varphi^{-1}_M(fl_M(z),z)\\
&=\varphi_M(z,z+2\epsilon) \end{align*} as we wanted.\end{proof}

This completes the proof of Theorem~\ref{InterleavingThm}.
\end{proof}

\begin{remark} As noted in Remark~\ref{MultidimensionExtension}, the notion of a well behaved persistence module admits a generalization to the multi-dimensional setting.  An interesting question is whether Theorem~\ref{InterleavingThm} generalizes to well behaved multidimensional persistence modules; if it does, then we obtain corresponding generalizations of Corollaries~\ref{MetricCorollary} and~\ref{Cor:Converse}.  Our proof of Theorem~\ref{InterleavingThm} does not generalize directly.
\end{remark}

%% file: Part_I/T_Extrinsic_Characterization.tex
\section{An Extrinsic Characterization of Interleaved Pairs of Multidimensonal Persistence Modules}

In this section, we introduce $(J_1,J_2)$-interleavings of pairs of $B_n$-persistence modules, a generalization of $\epsilon$-interleavings of pairs of $B_n$-persistence modules.  These generalized interleavings serve as a convenient language for expressing nuanced relationships between $B_n$-persistence modules which arise in our study of topological inference using Rips multifiltrations in Section~\ref{Sec:RipsAsymptotics}.

After defining generalized interleavings, we present Theorem~\ref{GeneralAlgebraicRealization}, our ``extrinsic" characterization of $(J_1,J_2)$-interleaved pairs of persistence modules; as noted in the introduction, this result expresses transparently the sense in which $(J_1,J_2)$-interleaved persistence modules are algebraically similar.  In particular, Theorem~\ref{GeneralAlgebraicRealization} gives an algebraically transparent characterization of $\epsilon$-interleaved pairs of modules, which we write down as Corollary~\ref{AlgebraicRealization}.  This characterization induces in an obvious way a corresponding characterization of the interleaving distance.  It is also the most important step in our proof of our main optimality result Corollary~\ref{CorOptimality}.  

\subsection{$(J_1,J_2)$-Interleavings}\label{Sec:J1J2Interleavings}
In Section~\ref{ShiftsOfModules}, we introduced shift functors and transition morphisms on $B_n$-persistence modules.  To prepare for the definition of $(J_1,J_2)$-interleavings, we define {generalized shift functors} and {generalized transition morphisms}.  

\subsubsection{Generalized Shift Functors} 

We say that a bijection $J:\R^n\to \R^n$ is {\bf order-preserving} if $\forall$ $a,b\in \R^n$, $a\leq b$ iff $J(a)\leq J(b)$.  For any order-preserving map $J:\R^n\to\R^n$, we define the {\bf generalized shift functor} $(\cdot)(J):B_n$-mod$\to B_n$-mod, as follows:
\begin{enumerate}
\item Action of $(\cdot)(J)$ on objects: For $M \in \obj(B_n$-mod$)$, we define $M(J)$ by taking $M(J)_a=M_{J(a)}$ for all $a\in \R^n$.  For all $a\leq b\in \R^n$ we take the transition map $\varphi_{M(J)}(a,b)$ to be the map $\varphi_{M}(J(a),J(b))$.
\item  Action of $(\cdot)(J)$ on morphisms: For $M,N\in \obj(B_n$-mod$)$ and $f\in \hom(M,N)$, we define $f(J):M(J)\to N(J)$ to be the homomorphism for which $f(J)_a=f_{J(a)}$ for all $a\in \R^n$.
\end{enumerate}

For $u\in \R^n$, let $J_u:\R^n\to \R^n$ be the map defined by $J_u(a)=a+u$, and for $\epsilon\in \R$, let $J_\epsilon$ denote $J_{\vec \epsilon_n}$.  $(\cdot)(J_u)$ is equal to the shift functor $(\cdot )(u)$ introduced in Section~\ref{ShiftsOfModules} and $(\cdot)(J_\epsilon)$ is equal to the shift functor $(\cdot )(\epsilon)$ introduced in the same section.  

As in the case of ordinary shift functors, for a morphism $f$ and $J$ an order-preserving map, we will sometimes abuse notation and write $f(J)$ simply as $f$.

Note the following contravariance property of generalized shift functors: If $J_1$ and $J_2$ are order-preserving, then \[(\cdot)(J_2\circ J_1)=(\cdot)(J_1)\circ (\cdot)(J_2).\]

\subsubsection{Generalized Transition Homomorphisms of $B_n$-persistence Modules}

Say an order-preserving map $J:\R^n\to\R^n$ is {\bf increasing} if $J(a)\geq a$ for all $a\in \R^n$.  For a $B_n$-persistence module $M$ and $J$ an increasing map, let $S(M,J):M\to M(J)$, the {\bf $J$-transition homomorphism}, be the homomorphism whose restriction to $M_a$ is the linear map $\varphi_M(a,J(a))$ for all $a\in \R^n$.  Note that for $\epsilon\in \R_{\geq 0}$, $S(M,\epsilon)=S(M,J_\epsilon)$, where $S(M,\epsilon)$ is as defined in Section~\ref{ShiftsOfModules}.  

The following lemma gives three easy and useful identities for generalized transition morphisms.  Of particular note is Lemma~\ref{lem:TransitionMorphismsAndShiftsForMods}(i).
\begin{lem}\label{lem:TransitionMorphismsAndShiftsForMods}\mbox{}
\begin{enumerate*}
\item[(i)] For any $f:M\to N \in \hom(B_n$-{\rm mod}) and any $J:\R^n\to \R^n$ increasing, \[S(N,J)\circ f=f(J)\circ S(M,J).\]
\item[(ii)] For any $J,J'$ increasing and $B_n$-persistence module $M$,
\[S(M,J')(J)\circ S(M,J)=S(M,J'\circ J).\]
\item[(iii)] For any $J,J'$ increasing and $B_n$-persistence module $M$,
\[S(M,J')(J)=S(M(J),J^{-1} \circ J'\circ J).\]
\end{enumerate*}
\end{lem}
\begin{proof} We'll prove (iii) and leave the proofs of (i) and (ii) to the reader.
For any $a\in \R^n$,
\begin{align*}
 S(M,J')(J)_a&=S(M,J')_{J(a)}\\
 =\varphi_M(J(a),J'\circ J(a))&=\varphi_M(J(a),J\circ J^{-1} \circ J'\circ J(a))\\
 =\varphi_{M(J)}(a,J^{-1} \circ J'\circ J(a))&=S(M(J),J^{-1} \circ J'\circ J)_a.\qedhere
\end{align*}
\end{proof}

\subsubsection{$(J_1,J_2)$-Interleavings of $B_n$-Persistence Modules} 

For $J_1,J_2:\R^n\to \R^n$ increasing, we say an ordered pair of $B_n$-persistence modules $(M,N)$ is {\bf $(J_1,J_2)$-interleaved} if there exist homomorphisms $f:M\to N(J_1)$ and $g:N\to M(J_2)$ such that 
\begin{align*}
g(J_1)\circ f&=S(M,J_2\circ J_1)\textup{ and}\\
f(J_2)\circ g&=S(N,J_1\circ J_2);
\end{align*} 
we say that $(f,g)$ is a pair of {\bf $(J_1,J_2)$-interleaving homomorphisms} for $(M,N)$.

Note the asymmetry of $(J_1,J_2)$-interleavings.  If $(M,N)$ is $(J_1,J_2)$-interleaved, it needn't be true that $(M,N)$ is $(J_2,J_1)$-interleaved.  It is however true that $(N,M)$ is $(J_2,J_1)$-interleaved.

\begin{remark}\label{Rem:General_Interleaving_Terminology_Remark}
A note on terminology: When a pair $(M,N)$ of $B_n$-persistence modules is $(J_1,J_2)$-interleaved, we will often say simply that $M,N$ are $(J_1,J_2)$-interleaved, or that $M$ and $N$ are $(J_1,J_2)$-interleaved.  Similarly, when $(f,g)$ is a pair of {\bf $(J_1,J_2)$-interleaving homomorphisms} for $(M,N)$, we will often say simply that $f,g$ are {$(J_1,J_2)$-interleaving homomorphisms} for $M,N$.
\end{remark}

Note that $M$ and $N$ are $\epsilon$-interleaved if and only if the pair $(M,N)$ is $(J_\epsilon,J_\epsilon)$-interleaved.  


\subsection{A Characterization of $(J_1,J_2)$-interleaved Pairs of Modules}\label{FirstPartOfProof}

We now present our characterization of pairs of $(J_1,J_2)$-interleaved $B_n$-persistence modules.

\subsubsection{Notation}
To state the theorem we need some notation. 

If $G=({\bar G},\iota_G)$ an $n$-graded set, and $J:\R^n\to\R^n$ is an order-preserving map, let $G(J)$ denote the $n$-graded set $({\bar G},\iota'_G)$, where $\iota'_G(y)=J^{-1}(\iota(y))$.  Note that if $J_1$ and $J_2$ are order-preserving, then $G(J_2 \circ J_1)=(G(J_2))(J_1)$.  

Similarly, if $M$ is a $B_n$-persistence module and $Q\subset M$ is a homogeneous subset, let $Q(J)\subset M(J)$ denote the image of $Q$ under the bijection between $M$ and $M(J)$ induced by the identification of each summand $M(J)_a$ with $M_{J(a)}$.    Note that if $J_1$ and $J_2$ are order-preserving, then $Q(J_2 \circ J_1)=(Q(J_2))(J_1)$.

\begin{remark}\label{ShiftInclusionRemark} For any $n$-graded set $G$ and $J:\R^n\to \R^n$ an increasing map, the homomorphism $S(\langle G(J^{-1}) \rangle,J):\langle G(J^{-1}) \rangle \to \langle G \rangle$ is injective, and so gives an identification of $\langle G(J^{-1}) \rangle$ with a submodule of $\langle G \rangle$.  More generally, if $G_1$ and $G_2$ are $n$-graded sets, we obtain in the obvious way an identification of $\langle G_1,G_2(J^{-1}) \rangle$ with a submodule of $\langle G_1,G_2 \rangle$.  In particular, for any $\epsilon\geq 0$, we obtain an identification of $\langle G(-\epsilon) \rangle$ with a submodule of $\langle G \rangle$ and, more generally, of $\langle G_1,G_2(-\epsilon) \rangle$ with a submodule of $\langle G_1,G_2 \rangle$.
\end{remark}

\begin{thm}\label{GeneralAlgebraicRealization} Let $M$ and $N$ be $B_n$-persistence modules.  For any $J_1,J_2:\R^n\to\R^n$ increasing, $(M,N)$ is $(J_1,J_2)$-interleaved if and only if there exist $n$-graded sets $\W_1,\W_2$ and homogeneous sets $\Y_1,\Y_2 \subset \langle\W_1,\W_2\rangle$ such that $\Y_1\in \langle\W_1,\W_2(J_2^{-1})\rangle$, $\Y_2\in \langle\W_1(J_1^{-1}),\W_2\rangle$, and
\begin{align*} 
M &\cong \langle\W_1,\W_2(J_2^{-1})|\Y_1,\Y_2(J_2^{-1})\rangle, \\ 
N&\cong \langle\W_1(J_1^{-1}),\W_2|\Y_1(J_1^{-1}),\Y_2\rangle. \end{align*}
If $M$ and $N$ are finitely presented, then $\W_1,\W_2,\Y_1,\Y_2$ can be taken to be finite. \end{thm}

The following corollary is immediate.

\begin{cor}\label{AlgebraicRealization} Let $M$ and $N$ be $B_n$-persistence modules.  For any $\epsilon\in \R_{\geq 0}$, $M$ and $N$ are $\epsilon$-interleaved if and only if there exist $n$-graded sets $\W_1,\W_2$ and homogeneous sets $\Y_1,\Y_2 \subset \langle\W_1,\W_2\rangle$ such that $\Y_1\in \langle\W_1,\W_2(-\epsilon)\rangle$, $\Y_2\in \langle\W_1(-\epsilon),\W_2\rangle$, 
\begin{align*} 
M &\cong \langle\W_1,\W_2(-\epsilon)|\Y_1,\Y_2(-\epsilon)\rangle \\ 
N&\cong \langle\W_1(-\epsilon),\W_2|\Y_1(-\epsilon),\Y_2\rangle. \end{align*}
If $M$ and $N$ are finitely presented, then $\W_1,\W_2,\Y_1,\Y_2$ can be taken to be finite. \end{cor}
 
\begin{proof}[Proof of Theorem~\ref{GeneralAlgebraicRealization}]
It's easy to see that if there exist $n$-graded sets $\W_1,\W_2$ and sets $\Y_1,\Y_2 \subset \langle\W_1,\W_2\rangle$ as in the statement of the theorem then $M,N$ are $(J_1,J_2)$-interleaved.    

To prove the converse, we will first express $M$ and $N$ as an isomorphic copies of two quotient persistent modules which are, in a suitable sense, algebraically similar.  We will then construct the desired presentations by lifting the structure of these quotients to free covers.

The construction by which we express $M$ and $N$ as quotients is a generalization to $(J_1,J_2)$-interleaved pairs of $B_n$-persistence modules of one introduced for $\epsilon$-interleaved pairs of $B_1$-persistence modules in the proof of \cite[Lemma 4.6]{chazal2009proximity}.

Let $f,g$ be $(J_1,J_2)$-interleaving homomorphisms for $M,N$.  

\begin{lem}\label{ChazalInterpolationLemma}  
Let $\gamma_1:M(J_1^{-1}\circ J_2^{-1})\to M\oplus N(J_2^{-1})$ be given by $\gamma_1(y)=(S(M(J_1^{-1}\circ J_2^{-1}),J_2 \circ J_1)(y),-f(y))$.  
Let $\gamma_2:N(J_2^{-1})\to  M\oplus N(J_2^{-1})$ be given by $\gamma_2(y)=(-g(y),y)$.  
Let $R\subset M\oplus N(J_2^{-1})$ be the submodule generated by $\im(\gamma_1) \cup \im(\gamma_2)$.
Then \[M\cong (M\oplus N(J_2^{-1}))/R.\]   
\end{lem}

\begin{proof}    
Let $\iota: M\to M\oplus N(J_2^{-1})$ denote the inclusion, and let $\zeta:M\oplus N(J_2^{-1}) \to  M\oplus N(J_2^{-1})/R$ denote the quotient.  We'll show that $\zeta\circ \iota$ is an isomorphism.  For any $(y_M,y_N)\in M\oplus N(J_2^{-1})$, $(-g(y_N),y_N)\in \im(\gamma_2) \subset R$, so $\zeta\circ \iota(g(y_N))=(0,y_N)+R$.  Therefore $\zeta\circ\iota(g(y_N)+y_M)=(y_M,y_N)+R.$  Hence $\zeta\circ \iota$ is surjective.  

$\zeta\circ \iota$ is injective iff $\iota(M)\cap R=0$.  It's clear that $\iota(M)\cap \im(\gamma_2)=0$.  Thus to show that $\zeta\circ \iota$ is injective it's enough to show that $\im(\gamma_1)\subset \im(\gamma_2)$.  If $y\in M(J_1^{-1}\circ J_2^{-1})$, then since $S(M(J_1^{-1}\circ J_2^{-1}),J_2\circ J_1)(y)=g\circ f(y)$, $(S(M(J_1^{-1}\circ J_2^{-1}),J_2\circ J_1)(y),-f(y))=(g\circ f(y),-f(y))=\gamma_2(-f(y))$.  Thus $\im(\gamma_1)\subset \im(\gamma_2)$ and so $\zeta\circ \iota$ is injective.  

Thus $\zeta\circ \iota$ is an isomorphism.
\end{proof}

Now let $\langle G_M|R_M\rangle$ be a presentation for $M$ and let $\langle G_N|R_N\rangle$ be a presentation for $N$.  Without loss of generality we may assume $M=\langle G_M \rangle/\langle R_M\rangle$ and $N=\langle G_N \rangle/\langle R_N\rangle$.  Let $\rho_M:\langle G_M \rangle \to M$, $\rho_N:\langle G_N \rangle \to N$ denote the quotient maps.  Then $(\langle G_M \rangle, \rho_M)$ and $(\langle G_N \rangle,\rho_N)$ are free covers for $M$ and $N$.  

Let ${\tilde f}:\langle G_M \rangle \to \langle G_N(J_1) \rangle$ be a lift of $f$ and let ${\tilde g}:\langle G_N \rangle \to \langle G_M(J_2) \rangle$ be a lift of $g$.


Let $R_{M,N}=\{y-{\tilde f}(y)\}_{y\in G_M(J_1^{-1})}$ and let $R_{N,M}=\{y-{\tilde g}(y)\}_{y\in G_N(J_2^{-1})}$.  Note that $R_{M,N}$ is a homogeneous subset of $\langle G_M(J_1^{-1}),G_N\rangle$ and $R_{N,M}$ is a homogeneous subset of $\langle G_M,G_N(J_2^{-1})\rangle$.  

Let 
\begin{align*}
P_M&=\langle G_M,G_N(J_2^{-1})|R_M,R_N(J_2^{-1}),R_{M,N}(J_2^{-1}),R_{N,M}\rangle,\\
P_N&=\langle G_M(J_1^{-1}),G_N|R_M(J_1^{-1}),R_N,R_{M,N},R_{N,M}(J_1^{-1})\rangle.
\end{align*}

$R_{M,N}(J_2^{-1})$ lies in $\langle G_M(J_1^{-1}\circ J_2^{-1}),G_N(J_2^{-1})\rangle$.  By Remark~\ref{ShiftInclusionRemark}, the map $S(G_M(J_1^{-1}\circ J_2^{-1}), J_2 \circ J_1)$ identifies $\langle G_M(J_1^{-1}\circ J_2^{-1}) \rangle$ with a subset of $\langle G_M \rangle$.  Thus $P_M$ is well defined.  By an analogous observation, $P_N$ is also well defined.  
 
We claim that $P_M$ is a presentation for $M$ and $P_N$ and is a presentation for $N$.  We'll prove that $P_M$ is a presentation for $M$; The proof that $P_N$ is a presentation for $N$ is identical.  

Let \begin{align*}
F&=\langle G_M,G_N(J_2^{-1})\rangle,\\
K&=\langle R_M,R_N(J_2^{-1}),R_{M,N}(J_2^{-1}),R_{N,M}\rangle \\
K'&=\langle R_M,R_N(J_2^{-1})\rangle.
\end{align*}

Let $p:F\to F/K'$ denote the quotient map.  Clearly, we may identify $F/K'$ with $M\oplus N(J_2^{-1})$ and $p$ with $(p_M,p_N)$. We'll check that $p$ maps $\langle R_{M,N}(J_2^{-1}) \rangle$ surjectively to $\im(\gamma_1)$ and $\langle R_{N,M} \rangle$ surjectively to $\im(\gamma_2)$, so that under the identification of $F/K'$ with $M\oplus N(J_2^{-1})$, $K/K'=R$.  Given this, it follows that $P_M$ is a presentation for $M$ by Lemma~\ref{ChazalInterpolationLemma} and the third isomorphism theorem for modules \cite{dummit1999abstract}.  

We first check that $\langle p(R_{M,N}(J_2^{-1})) \rangle = \im(\gamma_1)$.  Viewing $R_{M,N}(J_2^{-1})$ as a subset of $\langle G_M, G_N(J_2^{-1})\rangle$, $R_{M,N}(J_2^{-1})=\{S(\langle G_M(J_1^{-1}\circ J_2^{-1})\rangle,J_2\circ J_1)(y)-{\tilde f}(y)\}_{y\in G_M(J_1^{-1}\circ J_2^{-1})}$.  $S(\langle G_M(J_1^{-1}\circ J_2^{-1})\rangle,J_2\circ J_1)$ is a lift of $S(M(J_1^{-1}\circ J_2^{-1}),J_2\circ J_1)$ and ${\tilde f}$ is a lift of $f$, so for any $y\in G_M(J_1^{-1}\circ J_2^{-1})$, 
\begin{align*}
&p(S(\langle G_M(J_1^{-1}\circ J_2^{-1})\rangle,J_2\circ J_1)(y)-{\tilde f}(y))\\
&\quad=(S(M(J_1^{-1}\circ J_2^{-1}),J_2\circ J_1)(\rho_M(y)),-f(\rho_M(y)))\\
&\quad=\gamma_1(\rho_M(y)).
\end{align*}
Thus $p(R_{M,N}(J_2^{-1}))\subset \im(\gamma_1)$.  Since $G_M$ generates $\langle G_M \rangle$ and $\rho_M$ is surjective, we have that $p(\langle R_{M,N}(J_2^{-1}) \rangle) =\im(\gamma_1)$.

The check that $\langle p(R_{N,M}) \rangle = \im(\gamma_2)$ is similar to the above verification that $\langle p(R_{M,N}(J_2^{-1})) \rangle = \im(\gamma_1)$, but simpler.  $R_{N,M}=\{y-{\tilde g}(y)\}_{y\in G_N(J_2^{-1})}$.  ${\tilde g}$ is a lift of $g$ so for any $y\in G_N(J_2^{-1})$, \[p(y-{\tilde g}(y))=(-g(\rho_N(y)),\rho_N(y))=\gamma_2(\rho_N(y)).\]  Thus $p(R_{N,M})\subset \im(\gamma_2)$.  Since $G_N$ generates $\langle G_N \rangle$ and $\rho_N$ is surjective, we have that $p(\langle R_{N,M}\rangle)=\im(\gamma_2)$.   

This completes the verification that $P_M$ is a presentation for $M$.

Now, taking $\W_1=G_M$, $\W_2=G_N$, $\Y_1=R_M \cup R_{N,M}$, and $\Y_2=R_N \cup R_{M,N}$ gives the first statement of Theorem~\ref{AlgebraicRealization}.  If $M$ and $N$ are finitely presented then $G_M, G_N, R_M, R_N, R_{M,N}$, and $R_{N,M}$ can all be taken to be finite; the second statement of Theorem~\ref{AlgebraicRealization} follows. \end{proof}

%% file: Part_I/T_Geometric_Preliminaries_2.tex
\section{Geometric Preliminaries}\label{GeometricPreliminariesSection}
In this third section of preliminaries, we present preliminaries of a geometric and topological nature which we will need in the remainder of the thesis.

One organizational note before proceeding: We define multidimensional filtrations and multidimensional persistent homology here, in Section~\ref{Sec:MultidimensionalPersistentHomology1}.  In our study of topological inference in Chapter 4, we will need a slightly more general definition of multidimensional filtrations and a correspondingly more general definition of multidimensional persistent homology.  However, since these more general definitions are not needed in the remainder of this chapter, we will defer their introduction to Section~\ref{Sec:FiltrationsRevisited}.

\subsection{CW-complexes and Cellular homology}\label{Sec:CellularHomology}
Our proof of the optimality of the interleaving distance in Section~\ref{OptimalitySpecifics} will involve the construction of CW-complexes and the computation of their cellular homology.  We now briefly review finite dimensional CW-complexes and cellular homology.   

\subsubsection{Definition of a Finite-dimensional CW-complex}

A CW-complex is a topological space $X$ together with some additional data of {\it attaching maps} specifying how $X$ is assembled as the union of open disks of various dimensions.  We quote the procedural definition of a finite-dimensional CW-complex given in \cite{hatcher2002algebraic}.

Let $D^i$ denote the unit disk in $\R^i$; for $\alpha$ contained in some indexing set (which will often be implicit in our notation) let $D^i_\alpha$ be a copy of $D^i$.  When $i$ is clear from context, we will sometimes denote $D^i_\alpha$ simply as $D_\alpha$.  

A finite-dimensional CW-complex is a space $X$ constructed in the following way:
\begin{enumerate*}
\item Start with a discrete set $X^0$, the $0$-cells of X.
\item Inductively, form the i-skeleton of $X^i$ from $X^{i-1}$ by attaching $i$-cells $e^i_\alpha$ via maps $\sigma_\alpha: S^{i-1}\to X^{i-1}$.  This means that $X^i$ is the quotient space of $X^{i-1} \amalg_\alpha D^i_\alpha$ under the identifications $x\sim \sigma_\alpha(x)$ for $x\in \delta D^i_\alpha$.  The cell $e^i_\alpha$ is the homeomorphic image of $D^i_\alpha-\delta D^i_\alpha$ under the quotient map.
\item $X=X^r$ for some $r$.  We call the smallest such $r$ the {\it dimension} of $X$.
\end{enumerate*}
The {\it characteristic map} of the cell $e^i_\alpha$ is the map $\Phi_\alpha:D^i_\alpha \to X$ which is the composition $D^i_\alpha \hookrightarrow X^{i-1} \amalg_\alpha D^i_\alpha \to X^i \hookrightarrow X$, where the middle map is the quotient map defining $X^i$.

A {\it subcomplex} of a CW-complex $X$ is a closed subspace $A$ of $X$ which is a union of the cells of $X$; those cells contained in $A$ are taken to have the same attaching maps as they do in $X$.

\subsubsection{Cellular Homology}\label{CWHomologyDef}  
We mention only what we need about cellular homology to prove our optimality result Theorem~\ref{MainOptimality}.  For a more complete discussion and proofs of the results stated here, see e.g. \cite{hatcher2002algebraic} or \cite{bredon1993topology}.

For $i\in \Z_{\geq 0}$, we'll let $H_i$ denote the $i^{th}$ singular homology functor with coefficients in the field $k$.

For $X$ a CW-complex and $i\in \NN$, let $d^{X}_i:H_i(X^i,X^{i-1})\to H_{i-1}(X^{i-1},X^{i-2})$ denote the map induced by the boundary map in the long exact sequence of the pair $(X^i,X^{i-1})$.  It can be checked that the $d^{X}_i$ give 
\[\cdots \stackrel{\delta^X_{i+1}}{\rightarrow} H_i(X^i,X^{i-1})\stackrel{\delta^X_i}{\rightarrow} H_{i-1}(X^{i-1},X^{i-2})\stackrel{\delta^X_{i-1}}{\rightarrow}\cdots \stackrel{\delta^X_1}{\rightarrow}H_0(X^0)\to 0\] 
the structure of a chain complex, and that the $i^{th}$ homology vector space of this chain complex, denoted $H^{CW}_i(X)$, is isomorphic to $H_i(X)$.

It can be shown that a choice of generator for $H_i(D^i,S^{i-1})\cong \Z$ induces a choice of basis for $H_i(X^i,X^{i-1})$ whose elements correspond bijectively to the $i$-cells of $X$.  We now fix a choice of generator $H_i(D^i,S^{i-1})$ for each $i\in \NN$.\footnote{Such a choice is induced e.g. by the standard orientation on $D^i$.}  We can then think of $H_i(X^i,X^{i-1})$ as the $k$-vector space generated by the $i$-cells of $X$. 

It follows from the equality $H^{CW}_0(X)=H_0(X)$ that in the case that $X$ has a single $0$-cell, $d^{X}_1=0$.  

For $i>1$, the {\it cellular boundary formula} gives an explicit expression for $d^{X}_i$.  To prepare for the formula, we note first that for $i\in \NN$, the choice of generator for $H_i(D^i,S^{i-1})$ induces a choice of generator $a_i$ for $H_i(D^i/S^{i-1})$ via the quotient map $D^i\to D^i/S^{i-1}$.  Also, the choice of generator for $H_{i+1}(D^{i+1},S^i)$ induces a choice of generator $b_i$ for $H_i(S^i)$ via the boundary map in the long exact sequence of the pair $(D^{i+1},S^i)$.  For each $i\in \NN$, choose $\rho^i:D^i/S^{i-1} \to S^i$ to be any homeomorphism such that $\rho^i_*:H_i(D^i/S^{i-1}) \to H_i(S^i)$ sends $a_i$ to $b_i$.

For $i\in \NN$ and an $i$-cell $e^i_\beta$ of $X$, let $(e^i_\beta)^c$ denote the compliment of $e^i_\beta$ in $X^i$, and let $q_\beta:X^i\to X^i/(e^i_\beta)^c$ denote the quotient map.  $\rho^i$ and $\Phi_\beta$ induce an identification of $q_\beta(X^i)$ with $S^i$.  

By a compactness argument \cite[Section A.1]{hatcher2002algebraic}, for any $i$-cell $e^i_\alpha$ the image of the attaching map $\sigma_\alpha$ of $e^i_\alpha$ meets only finitely many cells.  

For $i>1$, the cellular boundary formula states that \[\delta^X_i(e^i_\alpha)=\sum_{\im(\sigma_\alpha)\cup e^{i-1}_\beta \ne \emptyset} \deg(q_\beta \circ \sigma_\alpha) e^{i-1}_\beta .\]  Here, for any map $f:S^{i-1}\to S^{i-1}$, $\deg(f)$ denotes the field element $a\in k$ such that $f_*:H_{i-1}(S^{i-1})\to H_{i-1}(S^{i-1})$ is multiplication by $a$.         

We can endow the set of CW-complexes with the structure of a category by taking $\hom(X,Y)$ for CW-complexes $X,Y$ to be the set of continuous maps $f:X \to Y$ such that $f(X^i)\subset Y^i$ for all $i$.  We call maps $f\in \hom(X,Y)$ {\it cellular maps}.  It can be shown that a cellular map $f$ induces a map $H^{CW}_i(f):H^{CW}_i(X)\to H^{CW}_i(Y)$ in such a way that $H^{CW}_i$ becomes a functor.

Further, there exists a natural isomorphism \cite{mac1998categories} $\kappa:H^{CW}_i\to {\bar H}_i$, where ${\bar H}_i$ is the restriction of $H_i$ to the category of CW-complexes.  


\subsection{Multidimensional Filtrations}\label{Sec:FiltrationsDef1}
Fix $n\in \NN$.

Define an {\bf $n$-filtration} $X$ to be a collection of topological spaces $\{X_a\}_{a\in \R^n}$, together with a collection of continuous maps $\{\phi_X(a,b):X_a\to X_b\}_{a\leq b\in \R^n}$ such that if $a\leq b\leq c\in\R^n$ then $\phi_X(b,c)\circ \phi_X(a,b)=\phi_X(a,c)$.  Given two $n$-filtrations $X$ and $Y$, we define a morphism $f$ from $X$ to $Y$ to be a collection of continuous functions $\{f_a\}_{a\in \R^n}:X_a \to Y_a$ such that for all $a\leq b\in \R^n$, $f_b\circ \phi_X(a,b)=\phi_Y(a,b)\circ f_a$.  This definition of morphism gives the $n$-filtrations the structure of a category.  Let $n$-filt denote this category.  

Define a {\bf cellular $n$-filtration} to be a collection of CW-complexes $\{X_a\}_{a\in \R^n}$, together with inclusions of subcomplexes $\{\phi_X(a,b):X_a\to X_b\}_{a\leq b\in \R^n}$.  Given two cellular $n$-filtrations $X$ and $Y$, we define a morphism $f$ from $X$ to $Y$ to be a collection of cellular maps $\{f_a\}_{a\in \R^n}:X_a \to Y_a$ such that for all $a\leq b\in \R^n$, $f_b\circ \phi_X(a,b)=\phi_Y(a,b)\circ f_a$.  This definition of morphism gives the cellular $n$-filtrations the structure of a category.  

Simplicial $n$-filtrations can be defined analogously.


 
\subsection{Multidimensional Persistent Homology}\label{Sec:MultidimensionalPersistentHomology1}

The multidimensional persistent homology functor $H_i$ is a generalization of the ordinary homology functor with field coefficients to the setting where the source is an $n$-filtration and the target is a $B_n$-persistence module.
We first present a definition of the singular multidimensional persistent homology functor; we introduce cellular multidimensional persistent homology below.

\subsubsection{Singular Multidimensional Persistent Homology} 
For a topological space $X$ and $j\in \Z_{\geq 0}$, let $C_j(X)$ denote the $j^{th}$ singular chain module of $X$, with coefficients in $k$.  For $X,Y$ topological spaces and $f:X\to Y$ a continuous map, let $f^\#:C_j(X)\to C_j(Y)$ denote the map induced by $f$.  

For $X$ an $n$-filtration, define $C_j(X)$, the $j^{th}$ singular chain module of $X$, as the $B_n$-persistence module for which $C_j(X)_u=C_j(X_u)$ for all $u\in \R^n$ and for which $\varphi_{C_j(X)}(u,v)=\phi_X^\#(u,v)$.  Note that for any $j\in \Z_{\geq 0}$, the collection of  boundary maps $\{\delta_j:C_j(X_u)\to C_{j-1}(X_u)\}_{u\in \R^n}$ induces a boundary map $\delta_j:C_j(X)\to C_{j-1}(X)$.  These boundary maps give $\{C_j(X)\}_{j\geq 0}$ the structure of a chain complex.  We define the $H_j(X)$, the $j^{th}$ persistent homology module of $X$, to be the $j^{th}$ homology module of this complex.  For $X$ and $Y$ two $n$-filtrations, a morphism $f\in \hom(X,Y)$ induces in the obvious way a morphism $H_j(f):H_j(X)\to H_j(Y)$, making $H_j:n$-filt$\to B_n$-mod a functor.  

\subsubsection{Cellular Multidimensional Persistent Homology} 
A construction analogous to the one above, with cellular chain complexes used in place of singular chain complexes, yields a definition of the cellular multidimensional persistent homology of cellular $n$-filtrations.  For a cellular filtration $X$, let $C_j^{CW}(X)$ denote the $j^{th}$ cellular chain module of $X$, let $\delta_j^{CW}:C_j^{CW}(X)\to C_{j-1}^{CW}(X)$ denote the $j^{th}$ cellular chain map of $X$, and let $H_j^{CW}(X)$ denote the $j^{th}$ cellular multidimensional persistence module of $X$.  

\begin{remark}\label{CellularAndSingularPersistentHomology}It follows from the naturality of the isomorphisms between singular and cellular homology that the singular and cellular multidimensional persistent homology modules of a cellular $n$-filtration are isomorphic.\end{remark}

%
%
%
%
\subsection{Functors from Geometric Categories to Categories of $n$-filtrations}\label{GeoFunctorsSection}

In applications of persistent homology one typically has some geometric\footnote{We are using the word ``geometric" here in an informal (and rather broad) sense.} category of interest and a functor $F$ from that category to $n$-filt for some $n\in \NN$; one then studies and works with the functor $H_i \circ F$.  These composite functors are also generally referred to as $i^{th}$ persistent homology functors.  There thus are a number of different such $i^{th}$ persistent homology functors with different sources, each determined by a different choice of the functor $F$. 

We next define several functors $F:C \to n$-filt, where $C$ is some geometric category.  We introduced special cases of some of the definitions presented here in Section~\ref{Sec:FirstDefinitionsOfFiltrations}.  We'll use these functors to formulate and prove our inference results in Chapter 4.

In Section~\ref{MultidimensionalStabilitySection} we will prove stability results for $H_i \circ F$ for each of functors $F$ introduced here.    

In the following examples, we omit the specification of the action of these functors on morphisms.  This should be clear from context.  



Recall from Section~\ref{Sec:FirstDefs} that for $a=(a_1,...,a_n)\in \hat \R^n$ and $b\in \hat \R$, we let $(a,b)=(a_1,...,a_n,b)\in \hat \R^{n+1}$.  

\begin{example} Sublevelset Filtrations \end{example}
   
Let $C^{S}$ be the category defined as follows:
\begin{enumerate*}
\item Objects of $C^{S}$ are pairs $(X,f)$, where $X$ is a topological space and $f:X\to {\R}^n$ is a function.  
\item If $(X,f),(X',f')\in \obj(C^{S})$, then we define $\hom_{C^{S}}((X,f),(X',f'))$ to be the set of functions $\gamma:X\to X'$ such that $f(x)\geq f'(\gamma(x))$ for all $x \in X$.  
\end{enumerate*}

For $X$ a topological space, let $C_{X}^{S}$ denote the subcategory of $C^{S}$ consisting of pairs of the form $(X,f)$.

Define the functor $F^{S}:C^S\to n$-filt on objects by taking \[F^S(X,f)_a=f_a\] for all $a\in \R^n$ (and taking the map $\phi_{F^S(X,f)}(a,b)$ to be the inclusion for all $a\leq b\in \R^n$.)  We refer to $F^S$ as the {\bf sublevelset filtration functor}, to $F^S(X,f)$ as the {\bf sublevelset filtration of $f$}, and to $F^S(X,-f)$ as the {\bf superlevelset filtration of $f$}.

\begin{example}\label{SOExample} Sublevelset-Offset Filtrations \end{example}
We define two variants of the sublevelset-offset filtration functor, the {\it closed} and {\it open} variants.  It is the closed variant that is of greater interest to us and that we use in the statement of our theorems; we define the open variant purely as a technical tool for use in the proofs of the results of Chapter 4.  

Let $C^{SO}$ be the category defined as follows:
\begin{enumerate*}
\item Objects of $C^{SO}$ are quadruples $(X,Y,d,f)$, where $(Y,d)$ is a metric space, $X\subset Y$, and $f:X\to {\R}^n$ is a function.  
\item If $(X,Y,d,f),(X',Y',d',f')\in \obj(C^{SO})$, then we define $\hom_{C^{SO}}((X,Y,d,f),(X',Y,d',f'))$ to be the set of functions $\gamma:Y\to Y'$ such that $\gamma(X)\subset X'$, $f(x)\geq f'(\gamma(x))$ for all $x \in X$, and $d(y_1,y_2)\geq d'(\gamma(y_1),\gamma(y_2))$ for all $y_1,y_2\in Y$.  
\end{enumerate*}

For $(X,X,d,f)\in C^{SO}$, we will sometimes write $(X,d,f)$ as shorthand for $(X,X,d,f)$.  For $Y$ a fixed topological space, let $C_{(Y,d)}^{SO}$ denote the subcategory of $C^{SO}$ consisting of quadruples of the form $(\cdot,Y,d,\cdot)$.

Define the functor $F^{SO}:C^{SO}\to (n+1)$-filt on objects by taking \[F^{SO}(X,Y,d,f)_{(a,b)}=\{y\in Y|d(y,f_a)\leq b \}\] for all $(a,b)\in \R^n\times \R$.  We refer to $F^{SO}$ as the {\bf (closed) sublevelset-offset filtration functor}, to $F^{SO}(X,Y,d,f)$ as the {\bf sublevelset-offset filtration of $f$}, and to $F^{SO}(X,Y,d,-f)$ as the {\bf superlevelset-offset filtration of $f$}.

Similarly, define the functor $F^{SO-Op}:C^{SO}\to (n+1)$-filt on objects by taking \[F^{SO-Op}(X,Y,d,f)_{(a,b)}=\{y\in Y|d(y,f_a)< b\}\] for all $(a,b)\in \R^n\times \R$.  We refer to $F^{SO-Op}$ as the {\bf open sublevelset-offset filtration functor}.


\begin{example} Sublevelset-\Cech Filtrations \end{example}
As we did for sublevelset-offset filtrations, we define closed and open variants of the sublevelset-\Cech filtration functor.  

Let $C^{SCe}$ denote the full subcategory of $F^{SO}$ whose objects are the quadruples $(X,Y,d,f)$ with $X$ finite; let $C_{(Y,d)}^{SCe}$ denote the full subcategory of $C^{SCe}$ consisting of quadruples of the form $(\cdot,Y,d,\cdot)$.

Recall from Section~\ref{Sec:FirstDefinitionsOfFiltrations} that given a metric space $(Y,d)$, a subset $X\subset Y$, and $b\in \R$, we let $\Cech(X,Y,d,b)$, the {\it (closed) \Cech complex} of $(X,d)$ with parameter $b$, be the abstract simplicial complex with vertex set $X$ such that for $l\geq 2$ and $x_1,x_2,..,x_l\in X$, $\Cech(X,Y,d,b)$ contains the $(l-1)$-simplex $[x_1,...,x_l]$ iff there is a point $y\in Y$ such that $d(y,x_i)\leq b$ for $1\leq i \leq l$.

Let $\Cech^{\circ}(X,Y,d,b)$, the {\it open \Cech complex} of $(X,d)$ with parameter $b$, be defined the same way as the closed \Cech complex, except that we specify that $\Cech^{\circ}(X,Y,d,b)$ contains the $(l-1)$-simplex $[x_1,...,x_l]$ iff there is a point $y\in Y$ such that $d(y,x_i)<b$ for $1\leq i \leq l$.

Define the functor $F^{\SCe}:C^{SCe}\to (n+1)$-filt on objects by taking \[F^{\SCe}(X,Y,d,f)_{(a,b)}=\Cech(f_a,Y,d,b)\] for all $(a,b)\in \R^n\times \R$.  We refer to $F^{\SCe}$ as the {\bf (closed) sublevelset-offset filtration functor}, to $F^{\SCe}(X,Y,d,f)$ as the {\bf sublevelset-offset filtration of $f$}, and to $F^{\SCe}(X,Y,d,-f)$ as the {\bf superlevelset-offset filtration of $f$}.

Similarly, define the functor $F^{SCe-Op}:C^{SCe}\to (n+1)$-filt on objects by taking \[F^{SCe-Op}(X,Y,d,f)_{(a,b)}=\Cech^\circ(f_a,Y,d,b)\] for all $(a,b)\in \R^n\times \R$.  We refer to $F^{SCe-Op}$ as the {\bf open sublevelset-\Cech filtration functor}.

\begin{example} Sublevelset-Rips Filtrations \end{example}
Let $C^{SR}$ be the full subcategory of $C^{SO}$ whose objects are quadruples $(X,X,d,f)$, where $X$ is a finite metric space.

Recall from Section~\ref{Sec:FirstDefinitionsOfFiltrations} that given a finite metric space $(X,d)$ and $b \in \R$, we let $\Rips(X,d,b)$, the {\it Rips complex} of $(X,d)$ with parameter $b$, be the maximal abstract simplicial complex with vertex set $X$ such that for $x_1,x_2\in X$, the 1-skeleton of $R(X,d,b)$ contains the edge $[x_1,x_2]$ iff $d(x_1,x_2)\leq 2b$.

Define the functor $F^{SR}:C^{SR}\to (n+1)$-filt on objects by taking \[F^{SR}(X,d,f)_{(a,b)}=\Rips(f_a,d,b)\] for all $(a,b)\in \R^n\times \R$.

We refer to $F^{SR}$ as the {\bf sublevelset-Rips filtration functor}, to $F^{SR}(X,d,f)$ as the {\bf sublevelset-Rips filtration of $f$}, and to $F^S(X,d,-f)$ as the {\bf superlevelset-Rips filtration of $f$}.

The functor $F^{SR}$ has previously been considered (in the case $n=1$) in \cite{carlsson2009theory}.

\subsection{Metrics on Geometric Categories}
We now define metrics on the isomorphism classes of objects of the geometric categories we defined in the last section.  We'll use these to formulate stability results for $d_I$ and to formulate various notion of optimality for $d_I$.  

\subsubsection{A Metric on $\obj(C_X^S)$}
Let $X$ be a topological space.  We define a metric $d^S_X$ on $\obj(C^S_X)$ by taking \[d^S_X((X,f_1),(X,f_2))=\sup_{x\in X} \|f_1(x)-f_2(x)\|_{\infty}.\]  

\subsubsection{A Metric on $\obj(C_{(Y,d)}^{SO})$}
Let $(Y,d')$ be a metric space.  We now define a metric $d^{SO}_{(Y,d')}$ on $\obj(C_{(Y,d')}^{SO})$ by taking 
\begin{align*}
d^{SO}_{(Y,d')}&((X_1,Y,d',f_1),(X_2,Y,d',f_2))=\\
\max\{\,&\sup_{x_1 \in X_1} \inf_{x_2 \in X_2} \max(d'(x_1,x_2),\|f_1(x_1)-f_2(x_2)\|_{\infty}),\,\\ 
            &\sup_{x_2 \in X_2} \inf_{x_1 \in X_1} \max(d'(x_1,x_2),\|f_1(x_1)-f_2(x_2)\|_{\infty})\,\}
\end{align*}
This distance is a sort of ``function aware" variant of the Hausdorff distance.  Since $C_{(Y,d)}^{SCe}$ is a subcategory of $C_{(Y,d)}^{SO}$, our metric on $\obj(C_{(Y,d)}^{SO})$ restricts to a metric on $\obj(C_{(Y,d)}^{SCe})$.

\subsubsection{A Metric on $\obj^*(C^{SR})$}
Generalizing in a mild way a definition of \cite{chazal2009gromov}, we define a metric $d^{SR}$ on $\obj^*(C^{SR})$.  For $f_X,f_Y\equiv 0$, $d^{SR}((X,d_X,f_X),(X,d_X,f_Y))$ will be equal to the Gromov-Hausdorff metric \cite{chazal2009gromov}.  (In fact, the definition extends to the subcategory of $C^{SO}$ whose objects are the triplets $(X,d_X,f_X)$ with $X$ compact, but we won't need the extra generality here.)  

To define $d^{SR}$, we need some preliminary definitions and notation.  Define a {\it correspondence} between two sets $X$ and $Y$ to be a subset $C\in X\times Y$ such that $\forall x\in X$, $\exists y\in Y$ such $(x,y)\in C$, and $\forall y\in Y$, $\exists x\in X$ s.t. $(x,y)\in C$.  Let $C(X,Y)$ denote the set of correspondences between $X$ and $Y$.

For $(X,d_X,f_X),(Y,d_Y,f_Y)\in C^{SR}$, define $\Gamma_{X,Y}:X\times Y \times X \times Y\to \R_{\geq 0}$ by \[\Gamma_{X,Y}(x,y,x',y')=|d_X(x,x')-d_Y(y,y')|.\]
For $C\in C(X,Y)$, define $\Gamma_C$ as $\sup_{(x,y),(x',y')\in C} \Gamma_{X,Y}(x,y,x',y')$, and define $|f_X-f_Y|_C$ to be $\sup_{(x,y)\in C} \|f_X(x)-f_Y(y)\|_{\infty}$.  Informally, $\Gamma_C$ is the maximum distortion of the metrics under the correspondence $C$, and $|f_X-f_Y|_C$ is the maximum distortion of the functions under $C$.

Now define we define $d^{SR}$ by taking \[d^{SR}((X,d_X,f_X),(Y,d_Y,f_Y))=\inf_{C\in C(X,Y)} \max(\frac{1}{2}\Gamma_C,|f_X-f_Y|_C).\]

\subsection{Stability Results for Ordinary Persistence}


There are two main geometric stability results for ordinary persistence in the literature.  (Though see also the generalization \cite{carlsson2009zigzag}) 
Each is a consequence of the algebraic stability of persistence \cite{chazal2009proximity}.

\begin{thm}[1-D Stability Result for $C_X^S$ \cite{chazal2009proximity}]\label{OrdinarySublevelsetStability}  For any $i\in \Z_{\geq 0}$, topological space $X$, and functions $f_1,f_2:X\to \R$ such that $H_i \circ F^S(X,f_1)$ and $H_i \circ F^S(X,f_2)$ are tame, \[d_B(H_i \circ F^S(X,f_1),H_i \circ F^S(X,f_2))\leq d^S_X((X,f_1),(X,f_2)).\] \end{thm}

For a 2-D filtration $F$, let ${\rm diag}(F)$ denote the 1-D filtration for which ${\rm diag}(F)_a=F_{(a,a)}$. 

\begin{thm}[1-D Stability Result for $C^{SR}$ \cite{chazal2009gromov}]\label{OrdinaryRipsPersistenceStability} For finite metric spaces $(X,d_X),(Y,d_Y)$ and functions $f_X:X\to \R$, $f_Y:Y\to \R$, 
\[d_B(H_i\circ {\rm diag}(F^{SR}(X,d_X,f_X)),H_i\circ {\rm diag}(F^{SR}(Y,d_Y,f_Y)))\leq d^{SR}((X,d_X,f_X),(Y,d_Y,f_Y)).\] \end{thm}

We'll see in Section~\ref{MultidimensionalStabilitySection} that both of these results admit generalizations to the setting of multidimensional persistence in terms of the interleaving metric.

%% file: Part_I/T_Stability.tex
\section{Stability Properties of the Interleaving Distance}\label{MultidimensionalStabilitySection}

In this section, we observe that multidimensional persistent homology is stable with respect to the interleaving distance in four senses senses analogous to those in which ordinary persistent homology is known to be stable.  As noted in the introduction, there is not much mathematical work to do here.  Nevertheless, these observations are significant (if not particularly surprising) not only because they show that the interleaving distance is in several respects a well behaved distance, but also because stability is closely related to the optimality of distances on persistence modules as we define it in Section~\ref{OptimalityGeneralities}; insofar as we wish to understand the optimality properties of the interleaving distance, the stability properties of the interleaving distance are of interest.     




\subsection{Stability of Sublevelset Multidimensional Persistent Homology}

\begin{thm}\label{MultidimensionalSublevelsetStability}For any topological space $X$ and pairs $(X,f_1),(X,f_2)\in \obj(C^S_X)$ we have, for any $i\in \Z_{\geq 0}$, \[d_I(H_i \circ F^S(X,f_1),H_i \circ F^S(X,f_2))\leq d^S_X((X,f_1),(X,f_2)).\] \end{thm}

The case $n=1$ is Theorem~\ref{OrdinarySublevelsetStability}.

\begin{proof}  
Let $d^S_X((X,f_1),(X,f_2))=\epsilon$.  Then for any $u\in \R^n$, $F^S(X,f_1)_u\subset F^S(X,f_2)_{u+\epsilon}$ and $F^S(X,f_2)_u\subset F^S(X,f_1)_{u+\epsilon}$.  The images of these inclusions under the $i^{th}$ singular homology functor define $\epsilon$-interleaving morphisms between $H_i \circ F^S(X,f_1)$ and $H_i \circ F^S(X,f_2)$.  Thus $d_I(H_i \circ F^S(X,f_1),H_i \circ F^S(X,f_2))\leq \epsilon$ as needed. 
\end{proof}

\subsection{Stability of Sublevelset-Offset Multidimensional Persistent Homology}

\begin{thm} \label{MultidimensionalSublevelsetOffsetStability} For any metric space $(Y,d')$ and $(X_1,Y,d',f_1),$ $(X_2,Y,d',f_2)$ $\in \obj(C_{(Y,d')}^{SO})$ we have, for any $i\in \Z_{\geq 0}$, \[d_I(H_i \circ F^{SO}(X_1,Y,d',f_1),H_i \circ F^{SO}(X_2,Y,d',f_2))\leq d^{SO}_{(Y,d')}((X_1,Y,d',f_1),(X_2,Y,d',f_2)).\]\end{thm}
\begin{proof} 
This is similar to proof of the previous result.  Let \[d^{SO}_{(Y,d')}((X_1,Y,d',f_1),(X_2,Y,d',f_2))=\epsilon.\]  Then for any $u\in \R^n$, $r\in \R$, and $\epsilon'>\epsilon$, 
\begin{align*}
F^{SO}(X_1,Y,d',f_1)_{(u,r)}&\subset F^{SO}(X_2,Y,d',f_2)_{(u+\epsilon',r+\epsilon')}\textup{ and}\\
F^{SO}(X_2,Y,d',f_2)_{(u,r)}&\subset F^{SO}(X_1,Y,d',f_1)_{(u+\epsilon',r+\epsilon')}.
\end{align*}
The images of these inclusions under the $i^{th}$ singular homology functor define $\epsilon'$-interleaving morphisms between $H_i \circ F^{SO}(X_1,Y,d',f_1)$ and $H_i \circ F^{SO}(X_2,Y,d',f_2)$.  The result follows. 
\end{proof}

\subsection{Stability of Sublevelset-\Cech Multidimensional Persistent Homology}

\begin{thm}\label{MultidimensionalCechPersistenceStability} For any $(X_1,Y,d',f_1)$,$(X_2,Y,d,f_2)\in \obj(C_{(Y,d')}^{SCe})$ and $i\in \Z_{\geq 0}$, \[d_I(H_i\circ F^{\SCe}(X_1,Y,d',f_1),H_i\circ F^{\SCe}(X_2,Y,d',f_2))\leq d^{SO}_{(Y,d')}((X_1,Y,d',f_1),(X_2,Y,d',f_2)).\] \end{thm}        

\begin{proof}
First note that the definition of $F^{\SCe}$ extends to quadruples $(\tilde X,Y,d',f)$ where $\tilde X$ is a multiset\footnote{Informally, a multiset is a set whose elements are allowed to have multiplicity greater than or equal to one.} whose points each have finite multiplicity and whose underlying set $X$ is a finite subset of $Y$.  Further note that for any $(a,b)\in \R^n\times \R$, $F^{\SCe}(X,Y,d',f)_{(a,b)}$ is a deformation retract of $F^{\SCe}(\tilde X,Y,d',f)_{(a,b)}$, and the deformation retracts can be taken to commute with the inclusion maps in these filtrations.  Thus \[H_i\circ F^{\SCe}(\tilde X,Y,d',f)_{(a,b)}\cong H_i\circ F^{\SCe}(X,Y,d',f)_{(a,b)}.\]

Let $d^{SO}_{(Y,d')}((X_1,Y,d',f_1),(X_2,Y,d',f_2)=\epsilon$.  Then we can choose multisets $\tilde X_1$ and $\tilde X_2$ whose underlying sets are $X_1$ and $X_2$ respectively, such that for all $\epsilon'>\epsilon$ there is a bijection $\lambda:\tilde X_1\to \tilde X_2$ such that 
\begin{enumerate*}
\item for all $x\in \tilde X_1$, $d'(x,\lambda(x))\leq \epsilon;$ and $\|f_1(x)-f_2(\lambda(x))\|_{\infty}\leq\epsilon'$ and
\item for all $x\in \tilde X_2$, $d'(x,\lambda^{-1}(x))\leq \epsilon'$ and $\|f_1(x)-f_2(\lambda^{-1}(x))\|_{\infty}\leq\epsilon'$.
\end{enumerate*}
This implies that $H_i\circ F^{\SCe}(\tilde X_1,Y,d',f_1)$ and $H_i\circ F^{\SCe}(\tilde X_2,Y,d',f_2))$ are $\epsilon'$-interleaved.  Thus $H_i\circ F^{\SCe}(X_1,Y,d',f_1)$ and $H_i\circ F^{\SCe}(X_2,Y,d',f_2))$ are also $\epsilon'$-interleaved, and the result follows.  
\end{proof}

\subsection{Stability of Sublevelset-Rips Multidimensional Persistent Homology}

\begin{thm}\label{MultidimensionalRipsPersistenceStability} For $(X_1,d_1,f_1),(X_2,d_2,f_2)\in \obj(C^{SR})$ we have, for any $i\in \Z_{\geq 0}$, \[d_I(H_i\circ F^{SR}(X_1,d_1,f_1),H_i\circ F^{SR}(X_2,d_2,f_2))\leq d^{SR}((X_1,d_1,f_1),(X_2,d_2,f_2)).\] \end{thm}          

The proof of this is a very minor modification of the argument given in \cite{chazal2009gromov} to prove Theorem~\ref{OrdinaryRipsPersistenceStability}.  

Note that Theorem~\ref{MultidimensionalRipsPersistenceStability} implies Theorem~\ref{OrdinaryRipsPersistenceStability}: When $n=1$, if $H_i\circ F^{SR}(X_1,d_1,f_1)$ and $H_i\circ F^{SR}(X_2,d_2,f_2)$ are $\epsilon$-interleaved, for $\epsilon\in \R_{\geq 0}$, then $H_i\circ {\rm diag}(F^{SR}(X_1,d_1,f_1))$ and $H_i\circ {\rm diag}(F^{SR}(X_2,d_2,f_2))$ are $\epsilon$-interleaved.

%% file: Part_I/T_Optimality_Generalities_2.tex
\section{Optimal Pseudometrics}\label{OptimalityGeneralities}

In this section we introduce a relative notion of optimality of pseudometrics on persistence modules and their discrete invariants.  This relative notion of optimality is quite general and specializes to a number of different notions of optimality of psuedometrics of interest in the context of multidimensional persistence.

We also present some first theoretical observations about optimal pseudometrics.  We'll exploit these in Section~\ref{OptimalitySpecifics} to prove our optimality result for the interleaving distance.    

\subsection{A General Definition of Optimal Pseudometrics} 
Let $Y$ be a set.  We define a {\bf relative structure} on $Y$ to be a triple ${\mathcal R}=(\T,X_\T,f_\T)$, where $\T$ is a set, $X_\T=\{(X_s,d_s)\}_{s\in \T}$ is a collection of pseudometric spaces indexed by $\T$, and $f_\T=\{f_s:X_s\to Y\}_{s\in \T}$ is a collection of functions.  Let $\im(f_\T)=\cup_{s\in \T} \im(f_s)$.

If $Y$ is a set and ${\mathcal R}=(\T,X_\T,f_\T)$ is a relative structure on $Y$, we say a semi-pseudometric $d$ on $Y$ is {\bf ${\mathcal R}$-stable} if for every $s\in \T$ and $x_1,x_2\in X_s$ we have $d(f_s(x_1),f_s(x_2)) \leq d_s(x_1,x_2)$.

We say a pseudometric $d$ on $Y$ is {\bf ${\mathcal R}$-optimal} if $d$ is ${\mathcal R}$-stable, and for every other ${\mathcal R}$-stable pseudometric $d'$ on $Y$, we have $d'(y_1,y_2)\leq d(y_1,y_2)$ for all $y_1,y_2\in \im(f_\T)$.   

The following lemma is immediate, but important to understanding our definition of ${\mathcal R}$-optimality.

\begin{lem}\label{OptimalityLemma} An ${\mathcal R}$-stable pseudometric $d$ is ${\mathcal R}$-optimal iff for any other ${\mathcal R}$-stable pseudometric $d'$, $s\in \T$, and $x_1,x_2\in X_s$, 
\[|d_s(x_1,x_2)-d(f_s(x_1),f_s(x_2))|\leq |d_s(x_1,x_2)-d'(f_s(x_1),f_s(x_2))|.\] \end{lem}

Note that if an ${\mathcal R}$-optimal pseudometric $d$ on $Y$ exists, its restriction to $\im(f_\T)\times \im(f_\T)$ is unique.  

\subsection{Examples} 
Here we give several examples of sets $Y$ and relative structures $(\T,X_\T,f_\T)$ on $Y$ for which it would be interesting or useful from the standpoint of the theory and application of multidimensional persistent homology to identify an ${\mathcal R}$-optimal pseudometric.  In Section~\ref{OptimalitySpecifics} we will focus exclusively on the relative structures of Examples~\ref{AllDimsSublevelsetEx} and~\ref{SingleDimSublevelsetEx}; we leave it to future work to investigate in detail the optimality of pseudometrics with respect to the relative structures of Examples~\ref{SublevelsetOffsetOptimality}-\ref{DiscreteInvariantExample}.  

For examples~\ref{AllDimsSublevelsetEx}-\ref{VariantExample}, let $Y=\obj^*(B_n$-mod$)$.

\begin{example}\label{AllDimsSublevelsetEx}Let ${\mathcal R}_{1}=(\T,X_\T,f_\T)$, where $\T$ is the set of pairs $\{(T,i)|T$ is a topological space, $i\in \Z_{\geq 0}\}$, $X_{(T,i)}=\obj(C_T^S)$ for $(T,z)\in \T$, $d_{(T,i)}=d^S_T$, and $f_{(T,i)}$ is given by $f_{(T,i)}(T,g)=H_i \circ F^S(T,g)$. \end{example}

\begin{remark}\label{InterpretationRemark} By Lemma~\ref{OptimalityLemma}, an ${\mathcal R}_{1}$-stable pseudometric $d$ is ${\mathcal R}_{1}$-optimal iff for any other ${\mathcal R}_{1}$-stable pseudometric $d'$, any $(T,i)\in \T$, and any $(T,g_1),(T,g_2)\in X_{(T,i)}$, 
\begin{align*}
&|d^S_T((T,g_1),(T,g_2))-d(H_i \circ F^S(T,g_1),H_i \circ F^S(T,g_2))|\\
\leq &|d^S_T((T,g_1),(T,g_2))-d'(H_i \circ F^S(T,g_1),H_i \circ F^S(T,g_2))|.
\end{align*}  
This says that an ${\mathcal R}_{1}$-optimal psuedometric is one for which the $L_\infty$ distance between any two functions defined on the same topological space is preserved under the multidimensional persistent homology functor as faithfully as is possible for any choice of ${\mathcal R}_{1}$-stable psuedometric on $\obj^*(B_n$-mod$)$. \end{remark}  

For the relative structures ${\mathcal R}$ in the rest of the examples below, ${\mathcal R}$-optimality has an interpretation by way of Lemma~\ref{OptimalityLemma} analogous to that of Remark~\ref{InterpretationRemark}.

\begin{example}\label{SingleDimSublevelsetEx} For $i\in \Z_{\geq 0}$, let ${\mathcal R}_{1,i}=(\T,X_\T,f_\T)$, where $\T$ is the set of topological spaces, $X_T=\obj(C_T^S)$ for $T\in \T$, $d_T=d^S_T$, and $f_{T}$ is given by $f_{T}(T,g)=H_i \circ F^S(T,g)$. \end{example}

Since the definitions of ${\mathcal R}_{1}$ and ${\mathcal R}_{1,i}$ are similar, one might expect that there's a relationship between ${\mathcal R}_{1}$-optimality and ${\mathcal R}_{1,i}$-optimality.  Corollary~\ref{CorOptimality} establishes such a relationship in the cases $k=\Q$ and $k=\Z/p\Z$ for some prime $p$.

Our reason for considering both ${\mathcal R}_{1}$ and ${\mathcal R}_{1,i}$, rather than only ${\mathcal R}_{1,i}$, is primarily technical: In short, the reason is that Section~\ref{OptimalitySpecifics} give a proof of the optimality of $d_I$ relative to ${\mathcal R}_{1,i}$ only for $i\in \NN$ but not for $i=0$; by considering optimality of $d_I$ relative to ${\mathcal R}_{1}$ we are able to obtain an optimality result which doesn't depend on $i$, and thus has a certain aesthetic appeal.  I suspect that $d_I$ is also ${\mathcal R}_{1,0}$ optimal; once this is shown to be the case, there will, in my view, be little reason to consider ${\mathcal R}_{1}$-optimality in formulating the results of this thesis.  

\begin{example}\label{SublevelsetOffsetOptimality} Let $\T$ be the set of triples $\{(T,d',i)|T$ is a topological space, $d'$ is a metric on $T$, $i\in \Z_{\geq 0}\}$, $X_{(T,d',i)}=\obj(C^{SO}_{(T,d')})$ for $(T,d',z)\in \T$, $d_{(T,d',i)}=d^{SO}_{(T,d')}$, and $f_{(T,d',i)}$ be given by $f_{(T,d',i)}(U,T,d',g)=H_i \circ F^{SO}(U,T,d',g)$.  
\end{example}

\begin{example}\label{SublevelsetRipsOptimality} Let $\T$ be the singleton set $\{s\}$.  Let $X_s=\obj^*(C^{SR})$, $d_s=d^{SR}$, and $f_s$ be given by $f_s(T,d',g)=H_i \circ F^{SR}(T,d',g)$. \end{example}

\begin{example}\label{SublevelsetCechOptimality} Let $\T$ be as in example~\ref{SublevelsetOffsetOptimality}.  Define $X_\T$ by taking $X_{(T,d',i)}=\obj^*(C^{SCe})$ and $d_{(T,d',i)}=d^{SO}_{(T,d')}$.  Let $f_{(T,d',i)}(U,T,d',g)=H_i \circ F^{SCe}(U,T,d',g)$. \end{example}

\begin{example}\label{VariantExample} We can present variants of examples~\ref{SublevelsetOffsetOptimality},~\ref{SublevelsetRipsOptimality},~and~\ref{SublevelsetCechOptimality} where we only consider homology in a single dimension, in the same way we did for Example~\ref{AllDimsSublevelsetEx} in Example~\ref{SingleDimSublevelsetEx}. \end{example}


\begin{example}\label{DiscreteInvariantExample} Let $W:\obj^*(B_n$-mod$)\to Y$ be a discrete invariant \cite{carlsson2009theory} with values in a set $Y$, let $(\T,X_\T,f_\T')$ be any relative structure on $\obj^*(B_n$-mod$)$, and let $f_\T=W\circ f_\T'$.  Then $(\T,X_\T,f_\T)$ is a relative structure on $W$.

For example, we can take $W$ to be the rank invariant \cite{carlsson2009theory} and $(\T,X_\T,f_\T')$ to be the relative structure of Example~\ref{AllDimsSublevelsetEx}.  \end{example}

\subsection{Induced Semi-Pseudometrics and a Condition for the Existence of Optimal Pseudometrics}\label{InducedSemiPseudometrics}
We'll see here that a relative structure ${\mathcal R}=(\T,X_\T,f_\T)$ on a set $Y$ induces a semi-pseudometric $d_{\mathcal R}$ on $\im(f_\T)$ with a nice property.  

For $y_1,y_2 \in \im(f_\T)$, let  \[A(y_1,y_2)=\{(s,x_1,x_2)|s\in \T,x_1\in X_s,x_2\in X_s,f_s(x_1)=y_1,f_s(x_2)=y_2\}.\]

Now define $d_{\mathcal R}(y_1,y_2)=\inf_{(s,x_1,x_2)\in A(y_1,y_2)} d_s(x_1,x_2)$.  $d_{\mathcal R}$ is an ${\mathcal R}$-stable semi-pseudometric.  In general it needn't satisfy the triangle inequality.  However, if $\T$ is a singleton set, as for instance in Example~\ref{SublevelsetRipsOptimality}, then $d_{\mathcal R}$ does satisfy the triangle inequality and is a pseudometric.

\begin{lem}\label{InducedMetricOptimality} for any ${\mathcal R}$-stable pseudometric $d$ on $\im(f_\T)$, $d\leq d_{\mathcal R}$.  In particular, if $d_{\mathcal R}$ is a pseudometric, it is ${\mathcal R}$-optimal.  \end{lem}

\begin{proof} Let $d$ be an ${\mathcal R}$-stable pseudometric on $\im(f_\T)$.  Since $d$ is ${\mathcal R}$-stable, $d(y_1,y_2)\leq d_s(x_1,x_2)$ for all $(s,x_1,x_2)\in A(y_1,y_2)$.  Thus $d(y_1,y_2)\leq \inf_{(s,x_1,x_2)\in A(y_1,y_2)} d_s(x_1,x_2)=d_{\mathcal R}(y_1,y_2)$.  \end{proof}

It's easy to see that a pseudometric $d$ on $\im(f_\T)$ can be extended (non-uniquely) to a metric on $Y$; if $d$ is ${\mathcal R}$-optimal then, by definition, any extension to a pseudometric on $Y$ is as well.  Thus by the Lemma~\ref{InducedMetricOptimality}, if $d_{\mathcal R}$ is a pseudometric, then an ${\mathcal R}$-optimal pseudometric exists on $Y$; its restriction to $\im(f_\T)$ is $d_{\mathcal R}$.  It follows, for example, that for ${\mathcal R}$ as in Example~\ref{SublevelsetRipsOptimality} an ${\mathcal R}$-optimal metric exists. 

In Section~\ref{OptimalitySpecifics}, we'll show that for $Y=\obj^*(B_n$-mod$)$, $i\in \NN$, and $k=\Q$ or $k=\Z/p\Z$ for some prime $p$, the restriction of $d_I$ to the domain of $d_{{\mathcal R}_{1,i}}$ is equal to $d_{{\mathcal R}_{1,i}}$, so that $d_I$ is ${\mathcal R}_{1,i}$-optimal.  It will follow easily that $d_I$ is also ${\mathcal R}_1$-optimal.

%% file: Part_I/T_Optimality_Specifics_9.tex
\section{Optimality of the Interleaving Distance (Relative to Sublevelset Persistence)}\label{OptimalitySpecifics}

\subsubsection{Summary of Results} 

The central result of this section is the following theorem:

\begin{thm}\label{MainOptimality} For $k=\Q$ or $\Z/p\Z$ for some prime $p$, and $i\in {\mathbb N}$, $d_I$ is ${\mathcal R}_{1,i}$-optimal. \end{thm}

This theorem also yields the following weaker optimality result, which has aesthetic advantage of not depending in its formulation on a choice of homology dimension.
    
\begin{cor}\label{CorOptimality} For $k=\Q$ or $\Z/p\Z$ for some prime $p$, $d_I$ is ${\mathcal R}_1$-optimal. \end{cor} 

We also present analogues of Theorems~\ref{MainOptimality} and Corollary~\ref{CorOptimality} for well behaved $B_1$-persistence modules which hold for arbitrary fields $k$; these are given as Theorem~\ref{WellBehavedOptimality}.  Our analogue of Theorem~\ref{MainOptimality} for well behaved $B_1$-persistence modules holds for $i\in \Z_{\geq 0}$.

\begin{remark}
I suspect that theorem~\ref{MainOptimality} holds, more generally, for arbitrary fields and for $i\in \Z_{\geq 0}$, though we do not prove that here.  To orient the reader, it may be helpful to point out that such a generalization would subsume both Corollay~\ref{CorOptimality} and Theorem~\ref{WellBehavedOptimality}.
\end{remark}


\subsection{Optimality of Interleaving Distance and Geometric Lifts of Interleavings}

\begin{proof}[Proof of Theorem~\ref{MainOptimality}] Fix $k=\Q$ or $k=\Z/p\Z$ for some prime $p$.  For $i\in \Z_{\geq 0}$, Lemma~\ref{InducedMetricOptimality} implies that to prove that $d_I$ is ${\mathcal R}_{1,i}$-optimal, it's enough to show that the restriction of $d_I$ to the domain of $d_{{\mathcal R}_{1,i}}$ is equal to $d_{{\mathcal R}_{1,i}}$. We show that for $i\in {\mathbb N}$, this follows from the following proposition, which we can think of as guaranteeing the existence of ``geometric lifts" of interleavings to the category $C^S$.

\begin{prop}\label{RealizationProp}[Existence of Geometric Lifts of Interleavings] Let $k=\Q$ or $k=\Z/p\Z$ for some prime $p$.  If $i\in {\mathbb N}$, and $M$ and $N$ are $\epsilon$-interleaved $B_n$-modules, then there exists a CW-complex $X$ and continuous functions $\gamma_M,\gamma_N:X\to \R^n$ such that $H_i\circ F^S(X,\gamma_M)\cong M$, $H_i\circ F^S(X,\gamma_N)\cong N$, and $\|\gamma_M-\gamma_N\|_{\infty}=\epsilon$.\end{prop}

Sections~\ref{FirstPartOfProof}-\ref{AdaptingProof} are devoted to the proof of Proposition~\ref{RealizationProp}.

\begin{cor}\label{RealizationCor} For $k=\Q$ or $\Z/p\Z$ for some prime $p$, for every $B_n$-module $M$ and $i\in {\mathbb N}$, there exists a $CW$-complex $X$ and a continuous function $\gamma:X\to \R^n$ such that $H_i\circ F^S(X,\gamma)\cong M$. \end{cor}  

To see that Proposition~\ref{RealizationProp} implies that $d_I$ is equal to $d_{{\mathcal R}_{1,i}}$ on their common domain, let $M$ and $N$ be two $B_n$-persistence modules such that $d_I(M,N)=\epsilon$.  Choose $\delta>0$.  Then $M$ and $N$ are are $(\epsilon+\delta)$-interleaved.  By the proposition, there exists a topological space $X$ and maps $\gamma_M,\gamma_N:X\to \R^n$ such that $H_i\circ F^S(X,\gamma_M)\cong M$, $H_i\circ F^S(X,\gamma_N)\cong N$, and $\|\gamma_M-\gamma_N\|_{\infty}=\epsilon+\delta$.  Thus by stability, $d_{{\mathcal R}_{1,i}}(M,N)\leq \epsilon+\delta$.  Since this holds for all $\delta>0$, 
$d_{{\mathcal R}_{1,i}}(M,N)\leq \epsilon$.  By Lemma~\ref{InducedMetricOptimality}, $d_{{\mathcal R}_{1,i}}(M,N)=\epsilon$.
\end{proof}   

Note that the extension of Theorem~\ref{MainOptimality} to the case $i=0$ would follow by the same argument from the following conjectural extension of Proposition~\ref{RealizationProp} to the case $i=0$.  

\begin{conjecture} Let $k=\Q$ or $k=\Z/p\Z$ for some prime $p$.  Let $M$ and $N$ be $\epsilon$-interleaved $B_n$-modules such that for some topological spaces $X_M$, $X_N$ and functions $f_M:X_M\to \R^n$, $f_N:X_N\to \R^n$, $H_0\circ F^S(X_M,f_M)\cong M$ and $H_0\circ F^S(X_N,f_N)\cong N$.  Then there exists a CW-complex $X$ and continuous functions $\gamma_M,\gamma_N:X\to \R^n$ such that $H_0\circ F^S(X,\gamma_M)\cong M$, $H_0\circ F^S(X,\gamma_N)\cong N$, and $\|\gamma_M-\gamma_N\|_{\infty}=\epsilon$.
\end{conjecture}

We do not prove this conjecture here.

\begin{proof}[Proof of Corollary~\ref{CorOptimality}] Fix any $i\in {\mathbb N}$.  Write ${\mathcal R}_1=(S,X_S,f_S)$, ${\mathcal R}_{1,i}=(S',X_{S'},f_{S'})$.  By Theorem~\ref{MultidimensionalSublevelsetStability}, $d_I$ is ${\mathcal R}_1$-stable.  Further, any ${\mathcal R}_1$-stable pseudometric $d'$ on $\obj^*(B_n$-mod$)$ is ${\mathcal R}_{1,i}$-stable.  By Corollary~\ref{RealizationCor}, when $k=\Q$ or $\Z/p\Z$ for some prime $p$, $\im(f_{S'})=\im(f_{S})=\obj^*(B_n$-mod$)$.  Thus, if $d'$ is any ${\mathcal R}_1$ stable metric, $d'(M,N)\leq d_I(M,N)$ for any $M,N\in \im(f_S)$ by the ${\mathcal R}_{1,i}$-optimality of $d_I$.  Thus $d_I$ is ${\mathcal R}_1$-optimal. \end{proof}

\subsection{Optimality of the Bottleneck Distance}
Before proceeding with the proof of Proposition~\ref{RealizationProp} we present optimality results for the bottleneck distance.  

Let $Y$ denote the set of isomorphism classes of tame $B_1$-persistence modules.  Since the bottleneck distance $d_B$ is only defined between elements of $Y$, formulating statements about the optimality of $d_B$ requires that we consider a relative structure on $Y$ rather than on $\obj^*(B_1$-mod$)$.  We can (in the obvious way) define restrictions ${\mathcal R}_2$ and ${\mathcal R}_{2,i}$ of the relative structures ${\mathcal R}_1$ and ${\mathcal R}_{1,i}$ to relative structures on $Y$. 

Then, given Proposition~\ref{RealizationProp}, the proofs of Theorem~\ref{MainOptimality} and Corollary~\ref{CorOptimality} adapt to give that for $k=\Q$ or $k=\Z/p\Z$ for $p$ a prime, $d_B$ is ${\mathcal R}_{2}$-optimal, and for any $i\in \NN$, $d_B$ is also ${\mathcal R}_{2,i}$-optimal.

Moreover, if we further restrict our attention to well behaved $B_1$-persistence modules then we can prove an optimality result for arbitrary fields $k$ and persistent homology of any dimension.  

The key is the following variant of Proposition~\ref{RealizationProp}:

\begin{prop}\label{WellBehavedRealizationProp} Let $k$ be any field and let $i\in \Z_{\geq 0}$.  Let $M$ and $N$ be well behaved $B_1$-persistence modules with $d_B(M,N)=\epsilon$.  If $i=0$, assume further that for some topological spaces $X_M$, $X_N$ and functions $f_M:X_M\to \R$, $f_N:X_N\to \R$, $H_0\circ F^S(X_M,f_M)\cong M$ and $H_0\circ F^S(X_N,f_N)\cong N$.  Then for any $\delta>0$, there exists a CW-complex $X$ and continuous functions $\gamma_M,\gamma_N:X\to \R$ such that $H_i\circ F^S(X,\gamma_M)\cong M$, $H_i\circ F^S(X,\gamma_N)\cong N$, and $\|\gamma_M-\gamma_N\|_{\infty}\leq\epsilon+\delta$.\end{prop}

\begin{proof}  
An easy constructive proof, similar on a high level to our proof of Proposition~\ref{RealizationProp} below but much simpler, follows from the definition of $d_B$ and the structure theorem for well behaved persistence modules (Theorem~\ref{WellBehavedStructureThm}).  We leave the details to the reader.   
\end{proof}

Now let $Y'$ denote the set of isomorphism classes of well behaved $B_1$-persistence modules.  Let ${\mathcal R}_3$ and ${\mathcal R}_{3,i}$ denote the restrictions of the relative structures ${\mathcal R}_1$ and ${\mathcal R}_{1,i}$ to relative structures on $Y'$.

Our optimality result for the bottleneck distance on well behaved persistence modules is the following:

\begin{thm}\label{WellBehavedOptimality}
For any field $k$ and $i\in \Z_{\geq 0}$, $d_B$ is ${\mathcal R}_{3,i}$-optimal.  Further, $d_B$ is ${\mathcal R}_{3}$-optimal.  
\end{thm}

\begin{proof}
Using Proposition~\ref{WellBehavedRealizationProp} in place of Proposition~\ref{RealizationProp}, the proofs of Theorem~\ref{MainOptimality} and Corollary~\ref{CorOptimality} carry over with only minor modifications to give the result.
\end{proof}

\begin{remark}\label{FrosiniRemark} In the case of 0-D ordinary persistence, our Theorem~\ref{WellBehavedOptimality} implies a slight weakening of the optimality result of \cite{d2010natural}.  It is easy to check that in the geometric context considered in \cite{d2010natural}, the persistent homology modules obtained must satisfy property 2 in the definition of a well behaved $B_1$-persistence module, but they needn't satisfy property 1.  To strengthen our optimality result to a full generalization of the result of \cite{d2010natural}, we'd want a generalization of Theorem~\ref{WellBehavedStructureThm} to all $B_1$-persistence modules satisfying property 2.  We presume that this can be obtained via a slight strengthening of Theorem~\ref{TamePersistenceStructureThm}, but we do not pursue the details of this here.   
\end{remark}


\subsection{Proof of Existence of Geometric Lifts of Interleavings, Part 1: Constructing the CW-complex}\label{ConstructingComplexSection}
The rest of Section~\ref{OptimalitySpecifics} is devoted to the proof of Proposition~\ref{RealizationProp}.

Corollary~\ref{AlgebraicRealization} gives us $n$-graded sets $\W_1,\W_2$ and homogeneous sets $\Y_1,\Y_2 \subset \langle \W_1,\W_2 \rangle$ such that  $\Y_1\in \langle\W_1,\W_2(-\epsilon)\rangle$, $\Y_2\in \langle\W_1(-\epsilon),\W_2\rangle$, and
\begin{align*} 
M &\cong \langle\W_1,\W_2(-\epsilon)|\Y_1,\Y_2(-\epsilon)\rangle,\\ 
N&\cong \langle\W_1(-\epsilon),\W_2|\Y_1(-\epsilon),\Y_2\rangle. \end{align*}

Given such $\W_1,\W_2,\Y_1,\Y_2$, we proceed with the proof of Proposition~\ref{RealizationProp}, beginning with the construction of the CW-complex $X$ whose existence is posited by the proposition.  

Write $\W=\W_1 \cup \W_2$ and $\Y=\Y_1 \cup \Y_2$.  If $M=N=0$, Proposition~\ref{RealizationProp} clearly holds, so we may assume without loss of generality that $M$ and $N$ are not both trivial.  Under this assumption, $\W\ne \emptyset$.

For an $n$-graded set $S$, let $GCD(S)=(v_1,...,v_n)$, where $v_i=\inf_{s\in S} gr(s)_i$.  In general, some of the components of $GCD(\W)$ may be equal to $-\infty$.  To simplify our exposition, we first present the remainder of the proof of Proposition~\ref{RealizationProp} in the special case that $GCD(\W)\in \R^n$.  The adaptation of our proof to the general case is reasonably straightforward; we outline this adaptation in Section~\ref{AdaptingProof}.

We'll define $X$ so that

\begin{enumerate*}
\item $X$ has a single 0-cell $B$.
\item $X$ has an $i$-cell $e^i_w$ for each $w\in \W$. 
\item $X$ has an $(i+1)$-cell $e^{i+1}_y$ for each $y\in \Y$.
\end{enumerate*}

For such $X$, the attaching map for each $i$-cell $e^i_w$ must be the constant map to $B$.  To define $X$, then, we need only to specify the attaching map $\sigma_y:S^i \to X^i$ for each $y\in \Y$.       

We do this for $k=\Q$, and leave to the reader the easy adaptation of the construction (and its use in the remainder of the proof of Proposition~\ref{RealizationProp}) to the case $k=\Z /p\Z$.  

For any $y\in \Y$, we may choose a finite set $\W_y\subset \W$ such that $gr(w)\leq gr(y)$ for each $w\in \W_y$, and 
\begin{align}\label{EquationForY}
y=\sum_{w\in \W_y} a'_{wy} \varphi_{\langle\W\rangle}(gr(w),gr(y))(w)
\end{align}
 for some $a'_{wy}\in \Q$.  There's an integer $z$ such that for each $a'_{wy}$ in the sum, $z a'_{wy}\in \Z$.  Let $a_{wy}=z a'_{wy}$.  For $w\not \in W_y$, define $a_{wy}=0$.



\begin{lem}\label{AttachingMapsLem}
There exists a choice of attaching map $\sigma_y:S^i \to X^i$ for each $y\in \Y$ such that the CW-complex $X$ constructed via these attaching maps has $\delta^X_{i+1}(e^{i+1}_y)=\sum_{w\in \W} a_{wy} e^i_w$ for all $y\in \Y$.
\end{lem}

\begin{proof}  
Let $\rho^i$ be as defined in Section~\ref{CWHomologyDef}.  For each $w\in \W$, $\rho^i$ and the characteristic map $\Phi_w$ induce an identification of $\im(\Phi_w)$ with an $i$-sphere $S_w^i$.  We have that $(X^i,B)=\wedge_{w\in \W} (S^i_w,B)$.  Choose a basepoint $o\in S^i$ and for each $w\in \W_y$, let $\sigma_{wy}:(S^i,o)\to (S^i_w,B)$ be a based map of degree $a_{wy}$.  $[\sigma_{wy}]\in \pi_i(X^i,B)$, where $\pi_i(X^i,B)$ denotes the $i^{th}$ homotopy group of $X^i$ with basepoint $B$.  

Order the elements of $\W_y$ arbitrarily and call them $w_1,...w_l$.  Let $\sigma_y:(S^i,o)\to(X^i,B)$ be a map in $[\sigma_{w_1y}]\cdot [\sigma_{w_2y}] \cdot ... \cdot [\sigma_{w_ly}]\in \pi_i(X^i,B)$.  Then for any $w\in \W$, $q_w\circ \sigma_y$ is a map of degree $a_{wy}$.  (See Section~\ref{CWHomologyDef} for the definition of $q_w$). By the definition of $\delta^X_{i+1}$ given in Section~\ref{CWHomologyDef}, the lemma now follows. 
\end{proof}

We choose the attaching maps $\sigma_y$ so that $\delta^X_{i+1}$ has the property specified in Lemma~\ref{AttachingMapsLem}. 

\subsection{Proof of Existence of Geometric Lifts of Interleavings, Part 2: Defining $\gamma_M$ and $\gamma_N$}

Having defined the CW-complex $X$, we next define $\gamma_M,\gamma_N:X\to \R^n$.

Let ${\tilde X}=\{B\} \amalg_{w\in \W} D^i_w \amalg_{y\in \Y} D^{i+1}_y$. 

$X$ is the quotient of $\tilde X$ under the equivalence relation generated by the attaching maps of the cells of $X$.  Let $\pi:\tilde X  \to X$ denote the quotient map.  For a topological space $A$, let ${\mathfrak C}(A,\R^n)$, denote the space of continuous functions from $A$ to $\R^n$.  The map $\tilde \cdot:{\mathfrak C}(X,\R^n)\to {\mathfrak C}(\tilde X,\R^n)$ defined by ${\tilde f}(x)=f(\pi(x))$ is a bijective correspondence between elements of ${\mathfrak C}(X,\R^n)$ and elements of ${\mathfrak C}(\tilde X,\R^n)$ which are constant on equivalence classes.   

In what follows, we'll define $\gamma_M$ and $\gamma_N$ by specifying their lifts ${\tilde \gamma_M},{\tilde \gamma_N}$.

We'll take each of our functions ${\tilde \gamma_M},{\tilde \gamma_N}$ to have the property that for each disk of $\tilde X$, the restriction of the function to any {\it radial line segment} (i.e. a line segment from the origin of the disk to the boundary of the disk) is linear.  Given this assumption, to specify each function it is enough to specify its values on the origins of each disk of $\tilde X$.  We now do this.       

For any $i\in {\mathbb N}$ and any unit disk $D$ in $\R^i$, let $O(D)$ denote the origin of $D$.  

\begin{itemize*}
\item ${\tilde \gamma_M}(B)={\tilde \gamma_N}(B)=GCD(\W)$; 
\end{itemize*}
\begin{itemize*}
\item For $x\in \W_1 \cup \Y_1$, ${\tilde \gamma_M}(O(D_x))=gr(x)$;
\item For $x\in \W_2 \cup \Y_2$, ${\tilde \gamma_M}(O(D_x))=gr(x(-\epsilon))$.
\end{itemize*}
\begin{itemize*}
\item For $x\in \W_1 \cup \Y_1$, ${\tilde \gamma_N}(O(D_x))=gr(x(-\epsilon))$;
\item For $x\in \W_2 \cup \Y_2$, ${\tilde \gamma_N}(O(D_x))=gr(x)$.
\end{itemize*}

\begin{lem}\label{FunctionDistanceLemma}
$\|\gamma_M-\gamma_N\|_\infty=\epsilon$.
\end{lem}

\begin{proof} Assume that for a disk $D$ of $\tilde X$, $|{\tilde \gamma_M}(a)-{\tilde \gamma_N}(a)|\leq \epsilon$ for all $a\in \delta D$, and that $|{\tilde \gamma_M}(O(D))-{\tilde \gamma_N}(O(D))|=\epsilon$.  We'll show that then $|{\tilde \gamma_M}(a)-{\tilde \gamma_N}(a)|\leq \epsilon$ for all $x\in D$.  Applying this result once gives that the lemma holds on the restriction of $\gamma_M,\gamma_N$ to $X^i$.  Applying the result a second time gives that the lemma holds on all of $X$.     

To show that $|{\tilde \gamma_M}(a)-{\tilde \gamma_N}(a)|\leq \epsilon$, let $x$ be a point in $D$ and write $a=tO(D)+(1-t)b$ for some $b\in \delta D$, and $0\leq t\leq 1$.   Since the restrictions of ${\tilde \gamma_M}$ and ${\tilde \gamma_N}$ to any radial line segment from $O(D)$ to $\delta D$ are linear, we have that ${\tilde \gamma_M}(a)=t{\tilde \gamma_M}(O(D))+(1-t){\tilde \gamma_M}(b)$, and ${\tilde \gamma_N}(a)=t{\tilde \gamma_N}(O(D))+(1-t){\tilde\gamma_N}(b)$.  Thus $|{\tilde \gamma_M}(a)-{\tilde \gamma_N}(a)|\leq t|{\tilde \gamma_M}(O(D))-
{\tilde \gamma_N}(O(D))|+(1-t)|{\tilde \gamma_M}(b)-{\tilde \gamma_N}(b)|\leq t\epsilon+(1-t)\epsilon=\epsilon$ as needed.\end{proof}

\subsection{Proof of Existence of Geometric Lifts of Interleavings, Part 3: Showing that $H_i\circ F^S(X,\gamma_M)\cong M$, $H_i\circ F^S(X,\gamma_N)\cong N$}\label{LastPartOfProof}
Now it remains to show that $H_i\circ F^S(X,\gamma_M)\cong M$, $H_i\circ F^S(X,\gamma_N)\cong N$.  We'll show that $H_i\circ F^S(X,\gamma_M)\cong M$; the argument that $H_i\circ F^S(X,\gamma_N)\cong N$ is essentially same.  

For $a\in \R^n$, let ${\F}_a$ denote the subcomplex of $X$ consisting of only those cells $e$ such that $\gamma_M(O(D(e)))\leq a$, where in this expression $D(e)$ is the disk of $\tilde X$ whose interior maps to $e$ under $\pi$.  $\{{\F}_a\}_{a\in \R^n}$ defines a cellular filtration, which we'll denote ${\F}$.  Let $X_a=F^S(X,\gamma_N)_a$.  It's easy to see that ${\F}_a$ is a deformation retract of $X_a$.  Further, the inclusions of each ${\F}_a \hookrightarrow X_a$ define a morphism $\chi$ of filtrations; this morphism of filtrations maps under $H_i$ to a morphism $H_i(\chi): H_i({\F})\to H_i(F^S(X,\gamma_M))$ of $B_n$-persistence modules whose maps 
$H_i(\chi)_a:H_i({\F}_a)\to H_i(X_a)$ are isomorphisms.  Any homomorphism of $B_n$-persistence modules whose action on each homogeneous summand is a vector space isomorphism must be an isomorphism of $B_n$-persistence modules, so $H_i(\chi)$ is an isomorphism.  Thus, to prove that $H_i\circ F^S(X,\gamma_M)\cong M$, it's enough to show that $H_i({\F})\cong M$.

By Remark~\ref{CellularAndSingularPersistentHomology}, $H_i({\F})\cong H^{CW}_i({\F})$.

Note that ${\F}$ has the property that each cell $e$ of $X$ has a unique minimal grade of appearance $gr_{\F}(e)$ in ${\F}$.  
Since each cell has a unique minimal grade of appearance, for any $j\in \Z_{\geq 0}$, $C^{CW}_j({\F})$ is free: \[C^{CW}_j({\F})=\oplus_{e^j\subset X \text{a {\it j}-cell }}B_n(-gr_{\F}(e^j)).\]  The usual identification of $j$-cells of $X$ with a basis for $C^{CW}_j(X)$ extends in the obvious way to an identification of the $j$-cells of $X$ with a basis for $C^{CW}_j({\F})$.  

Moreover, the boundary homomorphism $\delta^X_{i+1}:C^{CW}_{i+1}(X)\to C^{CW}_i(X)$ and the boundary homomorphism 
$\delta^{\F}_{i+1}:C^{CW}_{i+1}({\F}) \to C^{CW}_i({\F})$ are related in a simple way:

\begin{lem}\label{PersistentAttachingMapsLem} $\delta^{\F}_{i+1}(e^{i+1}_y)=\sum_{w\in \W} a_{wy} \varphi_{C^{CW}_i({\F})}(gr_{\F}(e^i_w),gr_{\F}(e^{i+1}_y))(e^i_w)$.
\end{lem}

\begin{proof} Recall that we constructed $X$ in such a way that for $y\in \Y$, $\delta^X_{i+1}(e^{i+1}_y)=\sum_{w\in \W} a_{wy} e^i_w$.  The result follows in a routine way from this expression for $\delta^X_{i+1}(e^{i+1}_y)$ and the definition of the boundary map $\delta^{\F}_{i+1}$. \end{proof}

Now note that we have $\delta^{\F}_i=0$.  If $i\ne 1$ this is follows from the fact that ${\F}$ has no $i-1$ cells.  If $i=1$, it is still true because of the isomorphism between cellular and singular persistent homology: we must have $C^{CW}_0({\F})\cong B_n(-gr_{\F}(B))\cong H_0(\F)\cong H^{CW}_0(\F)$, so $\delta_1=0$.

Therefore $H^{CW}_i({\F})=C^{CW}_i({\F})/\im(\delta^{\F}_{i+1})$.

The bijection which sends $w\in \W$ to the cell $e^i_w$ induces an isomorphism $\Lambda:\langle\W_1 \cup \W_2(-\epsilon)\rangle\to C^{CW}_i({\F})$.  

By the expression (\ref{EquationForY}) for $y$ in terms of $a'_{wy}$ given in Section~\ref{ConstructingComplexSection}, for $y\in \Y_1 \cup \Y_2(-\epsilon)$, \[\Lambda(y)=\sum_{w\in \W} a'_{wy} \varphi_{C^{CW}_i({\F})}(gr_{\F}(e^i_w),gr_{\F}(e^{i+1}_y))(e^i_w). \]
Thus 
\begin{align*}
&\Lambda(\langle \Y_1\cup \Y_2(-\epsilon)\rangle)\\
&=\langle\{\sum_{w\in \W} a'_{wy} \varphi_{C^{CW}_i({\F})}(gr_{\F}(e^i_w),gr_{\F}(e^{i+1}_y))(e^i_w)\}_{y\in \Y}\rangle\\
&= \langle\{\sum_{w\in \W} a_{wy} \varphi_{C^{CW}_i({\F})}(gr_{\F}(e^i_w),gr_{\F}(e^{i+1}_y))(e^i_w)\}_{y\in \Y}\rangle\\ 
&=\im(\delta^{\F}_{i+1}) \end{align*}

by Lemma~\ref{PersistentAttachingMapsLem}.  
$\Lambda$ therefore descends to an isomorphism between $C^{CW}_i({\F})/\im(\delta^{\F}_{i+1})$ and $\langle \W_1,\W_2(-\epsilon)\rangle /\langle \Y_1,\Y_2(-\epsilon)\rangle$.  This shows that $H^{CW}_i(\F)=M$ and thus completes the proof of Proposition~\ref{RealizationProp} in the special case that $GCD(\W)\in \R^n$.

\subsection{Proof of Existence of Geometric Lifts of Interleavings, Part 4: Adapting the Proof to the Case That $GCD(\W)\not \in \R^n$}\label{AdaptingProof}
As noted above, the proof of Proposition~\ref{RealizationProp} we have given for the special case that $GCD(\W)\in \R^n$ adapts readily to a proof for the general case.  We now outline this adaptation, leaving to the reader the straightforward details.

Let $X'$ be the standard CW-complex structure on $\R$.  That is, we take each $z\in \Z$ to be a $0$-cell in $X'$, and for each $z\in \Z$, we take the interval $(z,z+1)$ to be a $1$-cell in $X'$.  

To carry out the proof for the general case, we first present a modified version of the construction of the CW-complex $X$ given in Section~\ref{ConstructingComplexSection}.  In this modified version, we construct the CW-complex $X$ so that

\begin{enumerate*}
\item $X'$ is a subcomplex of $X$.
\item $X$ has an $i$-cell $e^i_w$ for each $w\in \W$. 
\item $X$ has an $(i+1)$-cell $e^{i+1}_y$ for each $y\in \Y$.
\item As a set, $X=X'\amalg \{e^i_w\}_{w\in\W} \amalg \{e^{i+1}_y\}_{y\in \Y}.$
\end{enumerate*}

For $r=(r_1,...,r_n)\in\R^n$, let $\lfloor r \rfloor=\max\{z\in \Z|z\leq r_j$ for $1\leq j\leq n\}$.  For all $w\in \W$, let the attaching map of $e^i_w$ be the constant map to the $0$-cell $\lfloor gr(w) \rfloor \subset X'$. 

This defines $X^i$.  To complete the construction of $X$, it remains only to specify the attaching map $\sigma_y:S^i \to X^i$ for each $y\in \Y$.       

Let $a_{wy}$ be defined as in Section~\ref{ConstructingComplexSection}.  An analogue of Lemma~\ref{AttachingMapsLem} holds in our present setting and admits a similar proof.  Invoking this result, for each $y\in \Y$ we choose $\sigma_y:S^i \to X^i$ such that the CW-complex $X$ constructed via these attaching maps satisfies $\delta^X_{i+1}(e^{i+1}_y)=\sum_{w\in \W} a_{wy} e^i_w$ for each $y\in \Y$.

For $r\in X'=\R$, we define $\gamma_M(r)=\gamma_N(r)=r\vec 1$.  As in the special case, we define the values of $\tilde \gamma_M$ and $\tilde \gamma_N$ at the origin of the disks $D^i_w$ and $D^{i+1}_y$ as follows: 

\begin{itemize*}
\item For $x\in \W_1 \cup \Y_1$, ${\tilde \gamma_M}(O(D_x))=gr(x)$;
\item For $x\in \W_2 \cup \Y_2$, ${\tilde \gamma_M}(O(D_x))=gr(x(-\epsilon))$.
\end{itemize*}
\begin{itemize*}  
\item For $x\in \W_1 \cup \Y_1$, ${\tilde \gamma_N}(O(D_x))=gr(x(-\epsilon))$;
\item For $x\in \W_2 \cup \Y_2$, ${\tilde \gamma_N}(O(D_x))=gr(x)$.
\end{itemize*}

As in the special case, we require the restriction of $\tilde \gamma_M$ and $\tilde \gamma_N$ to radial line segments of the disks $D^{i}_w$ and $D^{i+1}_y$ to be linear.  This completes the specification of $\gamma_M$ and $\gamma_N$.  

The argument of Lemma~\ref{FunctionDistanceLemma} shows that $\|\gamma_M-\gamma_N\|_\infty=\epsilon$, and the argument of Section~\ref{LastPartOfProof} adapts in a straightforward way to show that $H_i\circ F^S(X,\gamma_M)\cong M$ and $H_i\circ F^S(X,\gamma_N)\cong N$. \qed

%% file: Part_I/T_Computation.tex
\section{Reducing the Evaluation of $d_I$ to Deciding Solvability of Quadratics}\label{ComputationSection}

Let ${\mathcal MQ}(k)$ denote the set of multivariate systems of quadratic equations over the field $k$.

Fix $n\in {\mathbb N}$ and Let $M$ and $N$ be finitely presented $B_n$-persistence modules.  Let $q$ be the total number of generators and relations in a minimal presentation for $M$ and in a minimal presentation for $N$.  We show in this section that given minimal presentations for $M$ and $N$, for any $\epsilon>0$ deciding whether $M$ and $N$ are $\epsilon$-interleaved is equivalent to deciding the solvability of an instance of ${\mathcal MQ}(k)$ with $O(q^2)$ unknowns and $O(q^2)$ equations.

We also show that $d_I$ must be equal to one of the elements of an order $O(q^2)$ subset of $\R_{\geq 0}$ defined in terms of the grades of generators and relations of $M$ and $N$.  Thus, by searching through these values, we can compute $d_I$ by deciding whether $M$ and $N$ are $\epsilon$-interleaved for $O(\log q)$ values of $\epsilon$.  That is, we can compute $d_I(M,N)$ by deciding the solvability of $O(\log q)$ instances of ${\mathcal MQ}(k)$.
   
If e.g. $k$ is a field of prime order, a standard algorithm based on Gr\"obner bases determines the solvability of systems in ${\mathcal MQ}(k)$.  ${\mathcal MQ}(k)$ is NP-complete, however, and this algorithm is for general instances of ${\mathcal MQ}(k)$ prohibitively inefficient.  We leave it to future work to investigate the complexity and tractability in practice of deciding the solvability of systems in ${\mathcal MQ}(k)$ arising from our reduction.  

In practice, we are interested in computing the interleaving distance between the simplicial persistent homology modules of two simplicial $n$-filtrations.  To apply the reduction presented here to this problem, we need a way of computing a presentation of the multidimensional persistent homology module of a simplicial $n$-filtration; strictly speaking, our reduction does not require that the presentations of our modules be minimal.  However, to minimize the number and size of the quadratic systems we need to consider in computing the interleaving distance via this reduction, we do want the presentations we compute to be minimal.

We hope to address the problem of computing a minimal presentation of the simplicial persistent homology module of a simplicial $n$-filtration in future work. 

\subsection{Linear Algebraic Representations of Homogeneous Elements and Morphisms of Free $B_n$-persistence Modules}

\subsubsection{Representing Homogeneous Elements of Free $B_n$-persistence Modules as Vectors}\label{HomogeneousEltsAsVectors}
Given a finitely generated free $B_n$-persistence module $F$ and an (ordered) basis $B=b_1,...,b_l$ for $F$, we can represent a homogeneous element $v\in F$ as a pair $([v,B],gr(v))$ where $[v,B]\in k^l$ is a vector: if $v=\sum_{i:gr(v)\geq gr(b_i)} a_i \varphi_F(gr(b_i),gr(v))(b_i)$, with each $a_i\in k$, then for $1\leq i\leq l$ we define 
\[[v,B]_i=
\begin{cases}
a_i &\text{if }gr(v)\geq gr(b_i)\\
0 &\text{otherwise.} 
\end{cases}
\]

\begin{remark}\label{StandardBasisRemark} Note that for $1\leq i\leq l$, $[b_i,B]=\e_i$, where $\e_i$ denotes the $i^{th}$ standard basis vector in $k^l$. \end{remark}

If $V\subset F$ is a homogeneous set, we define $[V,B]=\{[v,B]|v\in V\}$. 

\subsubsection{Representing Morphisms of $B_n$-persistence Modules as Matrices}
Given finitely generated $B_n$-persistence modules $F$ and $F'$ and (ordered) bases $B=b_1,...,b_l$ and $B'=b'_1,...,b'_m$ for $F$ and $F'$ respectively, let $Mat_k(B,B')$ denote the set of $m\times l$ matrices $A$ with entries in $k$ such that $A_{ij}=0$ whenever $gr(b_j)<gr(b'_i)$.

We can represent a morphism $f\in\hom(F,F')$ as a matrix $[f,B,B']\in Mat_k(B,B')$, where if $f(b_j)=\sum_{i:gr(b_j)\geq gr(b'_i)} a_{ij} \varphi_{F'}(gr(b'_i),gr(b_j))(b'_i)$, with each $a_{ij}\in k$, then 
\[[f,B,B']_{ij}=
\begin{cases}
a_{ij} &\text{if }gr(b_j)\geq gr(b'_i)\\
0 &\text{otherwise.} 
\end{cases}
\]
\begin{lem}\label{MatrixRepBijectionLemma} The map $[\cdot,B,B']:\hom(F,F')\to Mat_k(B,B')$ is a bijection.\end{lem}
\begin{proof} The proof is straightforward.\end{proof}
Note also the following additional properties of these matrix representations of morphisms between free $B_n$-modules:
\begin{lem}\label{MatrixRepProperties} Let $F,F',F''$ be free $B_n$-persistence modules with ordered bases $B,B',B''$.  
\begin{enumerate*}
\item[(i)]If $f_1,f_2\in \hom(F,F')$ then $[f_1+f_2,B,B']=[f_1,B,B']+[f_2,B,B']$,
\item[(ii)]If $f_1\in \hom(F,F')$, $f_2\in \hom(F',F'')$ then $[f_2\circ f_1,B,B'']=[f_2,B',B''][f_1,B,B']$,
\item[(iii)]For any $\epsilon\geq 0$, $[S(F,\epsilon),B,B(\epsilon)]=I_{|B|}$, where for $m\in {\mathbb N}$, $I_m$ denotes the $m\times m$ identity matrix.
\end{enumerate*}
\end{lem}
\begin{proof} The proof of each of these results is straightforward. \end{proof}

For a graded set $W$ and $u\in \R^n$, let $W^u=\{y\in W|gr(y)\leq u\}$.

\begin{lem}\label{MappingIntoSubModLemma} If $F_1$,$F_2$ are free $B_n$-persistence modules with bases $B_1$,$B_2$ and $W_1\subset F_1$,$W_2\subset F_2$ are sets of homogeneous elements then a morphism $f:F_1\to F_2$ maps $\langle W_1 \rangle$ into $\langle W_2 \rangle$ iff $[f,B_1,B_2][w,B_1]\in \newspan[W_2^{gr(w)},B_2]$ for every $w\in W_1$. 
\end{lem}      

\begin{proof} 
This is straightforward.  
\end{proof}

\subsection{Deciding Whether Two $B_n$-persistence Modules are $\epsilon$-interleaved is Equivalent to Deciding the Solvability of a System in ${\mathcal MQ}(k)$}

Let $\langle G_M|R_M \rangle$, $\langle G_N|R_N \rangle$ be presentations for finitely presented $B_n$-modules $M$ and $N$, and assume the elements of each of the sets $G_M,G_N,R_M,R_N$ are endowed with a total order, which may be chosen arbitrarily.  For a finite ordered set $T$ and $1\leq i\leq |T|$, let $T_{,i}$ denote the $i^{th}$ element of $T$.

We now define six matrices of variables, each with some of the variables constrained to be 0.
\begin{itemize*}
\item Let ${\mathbf A}$ be an $|G_N|\times |G_M|$ matrix of variables, with ${\mathbf A}_{ij}=0$ iff $gr(G_{M,j})<gr(G_{N,i})+\epsilon$. 
\item Let ${\mathbf B}$ be an $|G_M|\times |G_N|$ matrix of variables, with ${\mathbf B}_{ij}=0$ iff $gr(G_{N,j})<gr(G_{M,i})+\epsilon$.
\item Let ${\mathbf C}$ be an $|R_N|\times |R_M|$ matrix of variables, with ${\mathbf C}_{ij}=0$ iff $gr(R_{M,j})<gr(R_{N,i})+\epsilon$.
\item Let ${\mathbf D}$ be an $|R_M|\times |R_N|$ matrix of variables, with ${\mathbf D}_{ij}=0$ iff $gr(R_{N,j})<gr(R_{M,i})+\epsilon$.
\item Let ${\mathbf E}$ be an $|R_M|\times |G_M|$ matrix of variables, with ${\mathbf E}_{ij}=0$ iff $gr(G_{M,j})<gr(R_{M,i})+2\epsilon$.
\item Let ${\mathbf F}$ be an $|R_N|\times |G_N|$ matrix of variables, with ${\mathbf F}_{ij}=0$ iff $gr(G_{N,j})<gr(R_{N,i})+2\epsilon$.
\end{itemize*}

Let $T_M$ denote the $|G_M|\times|R_M|$ matrix whose $i^{th}$ column is $[R_{M,i},G_M]$ and let $T_N$ denote the $|G_N|\times|R_N|$ matrix whose $i^{th}$ column is $[R_{N,i},G_N]$.

\begin{thm}\label{InterleavingAndQuadratics} $M$ and $N$ are $\epsilon$-interleaved iff the multivariate system of quadratic equations
\begin{align*}
{\mathbf A}T_M&=T_N{\mathbf C}\\  
{\mathbf B}T_N&=T_M{\mathbf D}\\ 
\mathbf{BA}-I_{|G_M|}&=T_M{\mathbf E}\\
\mathbf{AB}-I_{|G_N|}&=T_N{\mathbf F}
\end{align*}        
has a solution.
\end{thm}

\begin{proof}

To prove the result, we proceed in three steps.  First, we observe that for any free covers $(F_M,\rho_M)$ and $(F_N,\rho_N)$ of $M$ and $N$, the existence of $\epsilon$-interleaving morphisms between $M$ and $N$ is equivalent to the existence of a pair of morphisms between $F_M$ and $F_N$ having certain properties.  We then note that the existence of such maps is equivalent to the existence of two matrices, one in $Mat_k(G_M,G_N)$ and the other in $Mat_k(G_N,G_M)$, having certain properties.  Finally, we observe that the existence of such matrices is equivalent to the existence of a solution to the above multivariate system of quadratics.

Let $(F_M,\rho_M)$ and $(F_N,\rho_N)$ be free covers of $M$ and $N$.  

\begin{lem}\label{LiftsOfInterleavingsLemma} $M$ and $N$ are $\epsilon$-interleaved iff there exist morphisms ${\tilde f}:F_M\to F_N(\epsilon)$ and ${\tilde g}:F_N\to F_M(\epsilon)$ such that
\begin{enumerate*}
\item ${\tilde f}(\ker(\rho_M))\subset (\ker(\rho_N))(\epsilon)$,
\item ${\tilde g}(\ker(\rho_N))\subset (\ker(\rho_M))(\epsilon)$,
\item ${\tilde g}\circ {\tilde f}-S(F_M,2\epsilon)\subset (\ker(\rho_M))(2\epsilon)$,
\item ${\tilde f}\circ {\tilde g}-S(F_N,2\epsilon)\subset (\ker(\rho_N))(2\epsilon)$.
\end{enumerate*}
\end{lem}

We'll call morphisms ${\tilde f},{\tilde g}$ satisfying the above properties {\bf $\epsilon$-interleaved lifts} of the free covers $(F_M,\rho_M)$ and $(F_N,\rho_N)$.

\begin{proof} Let $f:M\to N(\epsilon)$ and $g:N\to M(\epsilon)$ be interleaving morphisms.  Then by Lemma~\ref{ExistenceAndHomotopyUniquenessOfLifts} there exist lifts ${\tilde f}:F_M \to F_N(\epsilon)$ and ${\tilde g}:F_N\to F_M(\epsilon)$ of $f$ and $g$.  By the definition of a lift, ${\tilde f}$ and ${\tilde g}$ satisfy properties 1 and 2 in the statement of the lemma.  ${\tilde g}\circ {\tilde f}$ is a lift of $g\circ f=S(M,2\epsilon)$.  $S(F_M,2\epsilon)$ is also a lift of $S(M,2\epsilon)$, so by the uniqueness up to homotopy of lifts (Lemma~\ref{ExistenceAndHomotopyUniquenessOfLifts}), ${\tilde f}$ and ${\tilde g}$ satisfy property 3.  The same argument shows that ${\tilde f}$ and ${\tilde g}$ satisfy property 4.

The converse direction is straightforward; we omit the details. 
\end{proof}

Now let $F_M=\langle G_M \rangle$, $F_N=\langle G_N \rangle$, and let $\rho_M:F_M \to F_M/\langle R_M \rangle$, $\rho_N:F_N \to F_N/\langle R_N \rangle$ be the quotient maps.  Since the interleaving distance between two modules is an isomorphism invariant of the modules, we may assume without loss of generality that $F_M/\langle R_M \rangle=M$ and $F_N/\langle R_N \rangle=N$.  Then $(F_M,\rho_M)$ and $(F_N,\rho_N)$ are free covers of $M$ and $N$.

\begin{lem}\label{MatrixLemma} $M$ and $N$ are $\epsilon$-interleaved iff there exist matrices $A\in Mat_k(G_M,G_N)$ and $B\in Mat_k(G_N,G_M)$ such that
\begin{enumerate*}
\item $A[w,G_M]\in \newspan [R_N^{gr(w)+\epsilon},G_N]$ for all $w\in R_M$,
\item $B[w,G_N]\in \newspan [R_M^{gr(w)+\epsilon},G_M]$ for all $w\in R_N$,
\item $(BA-I_{|G_M|})(\e_i)\in \newspan [R_M^{gr(G_{M,i})+2\epsilon},G_M]$ for $1\leq i \leq |G_M|$,
\item $(AB-I_{|G_N|})(\e_i)\in \newspan [R_N^{gr(G_{N,i})+2\epsilon},G_N]$ for $1\leq i \leq |G_N|$.
\end{enumerate*}
\end{lem}

\begin{proof}
By Lemma~\ref{LiftsOfInterleavingsLemma}, $M$ and $N$ are $\epsilon$-interleaved iff there exists $\epsilon$-interleaved lifts ${\tilde f}:F_M\to F_N$ and ${\tilde g}:F_N\to F_M$ of the free covers $(F_M,\rho_M)$ and $(F_N,\rho_N)$.  

By Lemma~\ref{MappingIntoSubModLemma}, morphisms ${\tilde f}:F_M\to F_N$ and ${\tilde g}:F_N\to F_M$, are $\epsilon$-interleaved lifts iff
\begin{enumerate*}
\item $[{\tilde f},G_M,G_N(\epsilon)][w,G_M]\in \newspan [R_N(\epsilon)^{gr(w)},G_N(\epsilon)]$ for all $w\in R_M$,
\item $[{\tilde g},G_N,G_M(\epsilon)][w,G_N]\in \newspan [R_M(\epsilon)^{gw(w)},G_M(\epsilon)]$ for all $w\in R_N$,
\item $[{\tilde g}\circ {\tilde f}-S(F_M,2\epsilon),G_M,G_M(2\epsilon)][w,G_M]\in \newspan [R_M(2\epsilon)^{gr(w)},G_M(2\epsilon)]$ for all $w\in G_M$,
\item $[{\tilde f}\circ {\tilde g}-S(F_N,2\epsilon),G_N,G_N(2\epsilon)][w,G_N]\in \newspan [R_N(2\epsilon)^{gr(w)},G_N(2\epsilon)]$ for all $w\in G_N$.
\end{enumerate*}
By Lemma~\ref{MatrixRepProperties}, 

\[[{\tilde g}\circ {\tilde f}-S(F_M,2\epsilon),G_M,G_M(2\epsilon)]=[{\tilde g},G_N,G_M(\epsilon)][{\tilde f},G_M,G_N(\epsilon)]-I_{|G_M|}\] 
and
\[[{\tilde f}\circ {\tilde g}-S(F_N,2\epsilon),G_N,G_N(2\epsilon)]=[{\tilde f},G_M,G_N(\epsilon)][{\tilde g},G_N,G_M(\epsilon)]-I_{|G_N|}.\]
Also, by Remark~\ref{StandardBasisRemark}, for $1\leq i\leq |G_M|$, $[G_{M,i},G_M]=\e_i$, where $\e_i$ is the $i^{th}$ standard basis vector in $k^{|G_M|}$.  Similarly, for $1\leq i\leq |G_N|$, $[G_{N,i},G_N]=\e_i$, where $\e_i$ is the $i^{th}$ standard basis vector in $k^{|G_N|}$.

Finally, note that we have that 
\begin{align*}
[R_N(\epsilon)^{gr(w)},G_N(\epsilon)]&=[R_N^{gr(w)+\epsilon},G_N] \text{\rm{ for all }} w\in R_M,\\
[R_M(\epsilon)^{gw(w)},G_M(\epsilon)]&=[R_M^{gr(w)+\epsilon},G_M] \text{\rm{ for all }} w\in R_N,\\
[R_M(2\epsilon)^{gr(w)},G_M(2\epsilon)]&=[R_M^{gr(w)+2\epsilon},G_M] \text{\rm{ for all }} w\in G_M,\\ 
[R_N(2\epsilon)^{gr(w)},G_N(2\epsilon)]&=[R_N^{gr(w)+2\epsilon},G_N] \text{\rm{ for all }} w\in G_N. 
\end{align*}

Using all of these observations, Lemma~\ref{MatrixLemma} now follows from Lemma~\ref{MatrixRepBijectionLemma}.
\end{proof}

Finally, Theorem~\ref{InterleavingAndQuadratics} follows from Lemma~\ref{MatrixLemma} by way of elementary matrix algebra and, in particular, the basic fact that for $l,m\in {\mathbb N}$ and vectors $v,v_1,...,v_l$ in $k^m$, $v\in \newspan(v_1,...,v_l)$ iff there exists a vector $w\in k^{l}$ such that $v=V w$, where $V$ is the $m\times l$ matrix whose $i^{th}$ column is $v_i$.
\end{proof}

\begin{remark}\label{ComputationRemark} Note that the size of the system of quadratic equations in the statement of Theorem~\ref{InterleavingAndQuadratics} is $O(q^2)$, where $q$ is the total number of generators and relations in the presentations for $M$ and $N$.  For any $\epsilon\geq 0$, the system of quadratics has as few variables and equations as possible when the presentations for $M$ and $N$ are minimal.
\end{remark}

\subsection{Determining Possible Values for $d_I(M,N)$}

Let $M$ and $N$ be finitely presented $B_n$-modules, and let $U_M^i$, $U_N^i$, $U_M$, and $U_N$ be as defined at the beginning of Section~\ref{SectionInterleavingThm}.  Let \[U_{M,N}=\bigcup_i\left( \{|x-y|\}_{x\in U_M^i,y\in U_N^i} \cup \{\frac{1}{2}|x-y|\}_{x,y\in U_M^i} \cup \{\frac{1}{2}|x-y|\}_{x,y\in U_N^i}\right) \cup \{0,\infty\}.\]

Note that $|U_{M,N}|=O(q^2)$, where as above $q$ is the total number of generators and relations in a minimal presentation for $M$ and a minimal presentation for $N$.

\begin{prop}\label{PossibilitiesForInterleavingDistance} $d_I(M,N)\in U_{M,N}$. \end{prop}

\begin{proof} Assume that for some $\epsilon'>0$, $\epsilon'\not\in U_{M,N}$, $M$ and $N$ are $\epsilon'$-interleaved.  Let $\epsilon$ be the largest element of $U_{M,N}$ such that $\epsilon'>\epsilon$, and let $\delta=\epsilon'-\epsilon$.

We'll check that $M,N,\epsilon$ and $\delta$ satisfy the hypotheses of Lemma~\ref{SmallerInterleavingLemma}.  The lemma then implies that $M$ and $N$ are $\epsilon$-interleaved.  The result follows.

By assumption, $M$ and $N$ are $(\epsilon+\delta)$-interleaved, so the first hypothesis of Lemma~\ref{SmallerInterleavingLemma} is satisfied.  We'll show that the second hypothesis is satisfied; the proof that the third hypothesis is satisfied is the same as that for the second hypothesis.

If $z\in U_{M}$ then for no $i$, $1\leq i\leq n$, can an element of $U_N^i$ lie in $(z_i+\epsilon,z_i+\epsilon+\delta]$; if, to the contrary, for some $i$ there were an element $u\in U_N^i\cap(z+\epsilon,z+\epsilon+\delta]$, then we would have $|u-z_i|\in U_{M,N}$, and $\epsilon <|u-z_i|\leq \epsilon+\delta$, which contradicts the way we chose $\epsilon$ and $\delta$.  
Thus by Lemma~\ref{FirstIsomorphismLemma}, $\varphi_N(z+\epsilon,z+\epsilon+\delta)$ is an isomorphism.

Similarly, for no $i$, $1\leq i\leq n$, can an element of $U_M^i$ lie in $(z_i+2\epsilon,z_i+2\epsilon+2\delta]$; if, to the contrary, for some $i$ there were an element $u\in U_M^i\cap(z+2\epsilon,z+2\epsilon+2\delta]$, then we would have $\frac{1}{2}|u-z_i|\in U_{M,N}$, and $\epsilon <\frac{1}{2}|u-z_i|\leq \epsilon+\delta$, which again contradicts the way we chose $\epsilon$ and $\delta$.  By Lemma~\ref{FirstIsomorphismLemma}, $\varphi_M(z+2\epsilon,z+2\epsilon+2\delta)$ is an isomorphism.

Thus the second hypothesis of Lemma~\ref{SmallerInterleavingLemma} is satisfied by our $M,N,\epsilon$,$\delta$, as we wanted to show. \end{proof}

%% file: Part_I/T_Discussion_I.tex
\section{Discussion of Future Work on the Interleaving Distance}\label{DiscussionSection}

Theorem~\ref{InterleavingEqualsBottleneck}, Corollary~\ref{MetricCorollary}, and Corollary~\ref{CorOptimality} show that the interleaving distance is a natural generalization of the bottleneck distance to the setting of multidimensional persistence.

Insofar as the interleaving distance is in fact a good choice of distance on multidimensional persistence modules, the question of how to compute it is interesting and, it seems to us, potentially important from the standpoint of applications.  The results of Section~\ref{ComputationSection} suggest a path towards the development of a theory of computation of the interleaving distance.  As noted in Section~\ref{ComputationSection}, to exploit the connection with multivariate quadratics in the development of such a theory in practice, we need in particular a way of computing minimal presentations of simplicial homology modules of simplicial $n$-filtrations.  We hope to address this problem in future work. 

Corollary~\ref{CorOptimality} demonstrates that the interleaving distance is optimal in the sense of Example~\ref{AllDimsSublevelsetEx} when $k=\Q$ or $\Z/p\Z$.  However, our discussion of optimality of pseudometrics in Section~\ref{OptimalityGeneralities} raises many more questions than it answers.  Some of the more interesting questions are:

\begin{enumerate*}
\item Can we extend the result of Theorem~\ref{MainOptimality} to arbitrary ground fields?  
\item Can we extend the result of Theorem~\ref{MainOptimality} to the case $i=0$?  
\item Can we prove that the interleaving distance is ${\mathcal R}$-optimal for ${\mathcal R}$ any of the relative structures on $\obj^*(B_n$-mod$)$ defined in Examples~\ref{SublevelsetOffsetOptimality}-\ref{VariantExample}?  The case of Example~\ref{SublevelsetRipsOptimality} is of particular interest to us.  We have observed in Section~\ref{InducedSemiPseudometrics} that in this case an ${\mathcal R}$-optimal pseudometric does exist.    
\item\label{Generalizations} Can we obtain analogous results about the optimality of pseudometrics on more general types of persistent homology modules?  For instance, can we prove a result analogous to Theorem~\ref{MainOptimality} for levelset zigzag persistence \cite{carlsson2009zigzag}?
\end{enumerate*}

An interesting question related to question~\ref{Generalizations} above is whether there is a way of algebraically reformulating the bottleneck distance for zigzag persistence modules as an analogue of the interleaving distance in such way that the definition generalizes to a larger classes of quiver representations \cite{derksen2005quiver}.

It seems quite likely that Theorem~\ref{AlgebraicStability}, the algebraic stability result of \cite{chazal2009proximity}, generalizes to a theorem which quantifies the similarity between the persistence diagrams of a pair of tame $(J_1,J_2)$-interleaved $B_1$-persistence modules.  It would be nice to have a generalized algebraic stability theorem of this kind.

Finally, we mention again that it would be nice to have an extension of Theorem~\ref{WellBehavedStructureThm} to a structure theorem for arbitrary tame $B_1$-persistence modules, and an extension of Corollary~\ref{MetricCorollary} to well behaved tame $B_n$-persistence modules. 

%% file: Part_II/T_Filtrations_Revisited.tex
In this chapter we introduce and study weak and strong interleavings and interleaving distances on multidimensional filtrations.  See Section~\ref{Sec:Chapter3Overview} for an overview of the chapter.

\section{$u$-filtrations and their Persistent Homology}\label{Sec:FiltrationsRevisited}
In this section we define $u$-filtrations, a mild generalization of the $n$-filtrations introduced in Section~\ref{Sec:FiltrationsDef1} which allow for the coordinates of points in the index set of the spaces in a multidimensional filtration to have finite upper bounds.  We also generalize the multidimensional persistent homology functor, introduced in Section~\ref{Sec:MultidimensionalPersistentHomology1}, to a category whose objects are $u$-filtrations.  These generalizations will be useful in formulating our inference results in Chapter 4.

\subsection{$u$-Filtrations}
Fix $n\in \NN$.  For $u\in {\R}^n$, we now define the category $u$-filt of $u$-filtrations.  

Define a {\bf $u$-filtration} $X$ to be a collection of topological spaces $\{X_a\}_{a<u}$, together with a collection of continuous functions $\{\phi_X(a,b):X_a\to X_b\}_{a\leq b<u}$ such that if $a\leq b\leq c<u$ then $\phi_X(b,c)\circ \phi_X(a,b)=\phi_X(a,c)$.  We call the functions $\{\phi_X(a,b):X_a\to X_b\}_{a\leq b<u}$ {\it transition maps}.

Note that an $\vec\infty_n$-filtration is the same as an $n$-filtration, as defined in Section~\ref{Sec:FiltrationsDef1}.  To avoid ambiguity, we stipulate that in this thesis, whenever $n\in \NN$, by $n$-filtration we will always mean an $\vec\infty_n$-filtration, and not a $u$-filtration with $u=n$.  

Given two $u$-filtrations $X$ and $Y$, we define a morphism $f$ from $X$ to $Y$ to be a collection of continuous functions $\{f_a\}_{a<u}:X_a \to Y_a$ such that for all $a\leq b<u$, $f_b\circ \phi_X(a,b)=\phi_Y(a,b)\circ f_a$.  This definition of morphism gives the $u$-filtrations the structure of a category.  Let $u$-filt denote this category. 
      
For all $u' \leq u$, we can define the {\bf restriction functor} $R_{u'}:u$-filt $\to u'$-filt in the obvious way.  (The dependence of $R_{u'}$ on $u$ is implicit in our notation.)  We'll use these functors heavily in what follows.  

\subsection[The Multi-D Persistent Homology Functor on $u$-filtrations]{The Multidimensional Persistent Homology Functor on $u$-filtrations}\label{Sec:GeneralizedPersistentHomologyDefinition}


Here we define the multidimensional persistent homology functor $H_i:u$-filt $\to B_n$-mod for any $u\in \hat \R^n$, extending the definition introduced in Section~\ref{Sec:MultidimensionalPersistentHomology1} for $u=\infty_n$.


As in Section~\ref{Sec:MultidimensionalPersistentHomology1}, for a topological space $X$ and $i\in \Z_{\geq 0}$, let $C_i(X)$ denote the $i^{th}$ singular chain module of $X$, with coefficients in $k$; for $X,Y$ topological spaces and $f:X\to Y$ a continuous map, let $f^\#:C_i(X)\to C_i(Y)$ denote the map induced by $f$.  

For $X$ a $u$-filtration, define $C_i(X)$, the $i^{th}$ singular chain module of $X$, to be the $B_n$-persistence module for which 
\begin{itemize*}
\item $C_i(X)_a=C_i(X_a)$ for all $a<u$,
\item $C_i(X)_a=0$ for all $a\not <u$,
\item $\varphi_{C_i(X)}(a,b)=\phi_X^\#(a,b)$ for $b<u$, 
\item $\varphi_{C_i(X)}(a,b)=0$ for $b \not <u$.
\end{itemize*}

Note that for any $i\in \Z_{\geq 0}$, the collection of boundary maps $\{\delta_i:C_i(X_a)\to C_{i-1}(X_a)\}_{a<u}$ induces a boundary map $\delta_i:C_i(X)\to C_{i-1}(X)$.  These boundary maps give $\{C_i(X)\}_{i\in \Z_{\geq 0}}$ the structure of a chain complex.  We define the $H_i(X)$, the $i^{th}$ persistent homology module of $X$, to be the $i^{th}$ homology module of this complex.  For $X$ and $Y$ two $u$-filtrations, a morphism $f\in \hom(X,Y)$ induces in the obvious way a morphism $H_i(f):H_i(X)\to H_i(Y)$, making $H_i:u$-filt$\to B_n$-mod a functor.  

\subsubsection{Compatibility of Multidimensional Persistent Homology With Restriction Functors}
For all $u\in {\hat \R}^n$, we define the {\bf restriction functor} $R_{u}:B_n$-mod $\to B_n$-mod by taking $R_u(M)$ of a $B_n$-persistence module $M$ to be such that 

\begin{itemize*}
\item $R_u(M)_a=M_a$ for all $a<u$,
\item $R_u(M)_a=0$ for all $a\not <u$,
\item $\varphi_{R_u(M)}(a,b)=\varphi_M(a,b)$ for $b<u$, 
\item $\varphi_{R_u(M)}(a,b)=0$ for $b \not <u$.
\end{itemize*}

We define the action of $R_{u}$ on morphisms in the obvious way.

\begin{remark} Note that if $M$ is a $B_n$-persistence module having a finite presentation $\langle G|R \rangle$ \cite{lesnick2011optimality}, then for any $u\in \hat \R^n$, $R_u(M)$ has a presentation consisting of at most $|G|$ generators and $|R|+n|G|$ relations.
\end{remark}

\begin{lem}[Compatibility of persistent homology with restriction functors]\label{lem:PersistentHomologyIsCompatibleWRestriction}
For any $u'\leq u\in \R^n$, \[H_i \circ R_{u'}=R_{u'} \circ H_i.\] \end{lem}

\begin{proof} The proof is easy; we omit it. \end{proof}

%% file: Part_II/T_Interleavings_Of_Filtrations.tex
\section[Strong Interleavings and the Strong Interleaving Distance]{Strong Interleavings and the Strong Interleaving Distance on $u$-Filtrations}\label{Sec:StrongInterleavings}

\subsubsection{Overview}
In this section, we introduce strong interleavings and the strong interleaving distance on $u$-filtrations, and develop their basic theory; in Section~\ref{Sec:WeakInterleavings} we will introduce and study our homotopy theoretic variants of these, weak interleavings and the weak interleaving distance.   

As we noted in the introduction, our inference results in Chapter 4 are formulated using weak interleavings and the weak interleaving distance.  In this sense, weak interleavings and the weak interleaving distance are of more interest to us than strong interleavings and the strong interleaving distance.  Nevertheless, because the program of understanding weak interleavings is, at least in a loose sense, related to the (simpler) program of understanding strong interleavings, we choose to devote some effort here to writing down some basic results about strong interleavings and the strong interleaving distance.

To begin this section, we define shift functors and transition morphisms for categories of $u$-filtrations.  Using these, we define strong $(J_1,J_2)$-interleavings of $u$-filtrations and the strong interleaving distance; the definitions are similar to those we introduced for $B_n$-persistence modules in Chapter 2.  Whereas we first introduced $\epsilon$-interleavings of $B_n$-persistence modules in Section~\ref{InterleavingsOfModules} and then presented the more general definition of strong $(J_1,J_2)$-interleavings of $B_n$-persistence modules in Section~\ref{Sec:J1J2Interleavings}, we proceed here directly with the general definition of strong interleavings of $u$-filtrations.

We present a characterization result for strongly $(J_1,J_2)$-interleaved pairs of $u$-filtrations, for an important class of $u$-filtrations.  We also present two optimality results for the strong interleaving distance.  These are analogues of our optimality result Theorem~\ref{MainOptimality} for the interleaving distance on $B_n$-persistence modules. 

\subsection{Definitions of Strong Interleavings and the Strong Interleaving Distance}\label{Sec:StrongInterleavingsDef}

\subsubsection{Shift Functors}
Fix $u\in {\hat \R}^n$ and let $J:\R^n\to \R^n$ be an order-preserving map.  We define the $J$-shift functor $(\cdot)(J):u$-filt$\to J^{-1}(u)$-filt in a way analogous to the way we did for $B_n$-persistence modules:
\begin{enumerate}
\item Action of $(\cdot)(J)$ on objects: For $X\in \obj(u$-filt$)$, we define $X(J)$ by taking $X(J)_a=X_{J(a)}$ for all $a<J^{-1}(u)$.  We take the maps $\{\phi_{X(J)}(a,b)\}_{a\leq b<J^{-1}(u)}$ to be induced by those of $X$ in the obvious way.  
\item  Action of $(\cdot)(J)$ on morphisms: For $X,Y\in \obj(u$-filt$)$ and $f\in \hom(X,Y)$, we define $f(J):X(J)\to Y(J)$ to be the morphism for which $f(J)_a=f_{J(a)}$ for all $a<J^{-1}(u)$.
\end{enumerate}

Note that the dependence of the definition of the functor $(\cdot)(J)$ on the choice of $u$ is implicit in our notation.  

As in the case of $B_n$-persistence modules, for $\epsilon\in\R$ we let $(\cdot)(\epsilon)$ denote the functor $(\cdot)(J_\epsilon)$.  




\subsubsection{Transition Morphisms}
For a $u$-filtration $X$ and $J:\R^n\to\R^n$ increasing, let $S(X,J):R_{J^{-1}(u)}(X)\to X(J)$, the {\bf $J$-transition morphism}, be the morphism whose restriction to $X_a$ is the map $\phi_X(a,J(a))$ for all $a<J^{-1}(u)$.  

For $\epsilon\in \R_{\geq 0}$, let $S(X,\epsilon)=S(X,J_\epsilon)$.

The expected analogue of Lemma~\ref{lem:TransitionMorphismsAndShiftsForMods} holds, with the same proof.

\begin{lem}\label{lem:TransitionMorphismsAndShiftsForFilts}\mbox{}
\begin{enumerate*}
\item[(i)] For any $f:X\to Y \in \hom(u$-filt) and any $J:\R^n\to \R^n$ increasing, \[S(Y,J)\circ f=f(J)\circ S(X,J).\]
\item[(ii)] For any $J,J'$ increasing and $u$-filtration $X$,
\[S(X,J')(J)\circ S(X,J)=S(X,J'\circ J).\]
\item[(iii)] For any $J,J'$ increasing and $u$-filtration $X$,
\[S(X,J')(J)=S(X(J),J^{-1} \circ J'\circ J).\]
\end{enumerate*}
\end{lem}

\subsubsection{Strong $(J_1,J_2)$-Interleavings of $u$-filtrations}

For $J_1,J_2:\R^n\to\R^n$ increasing, we say that a pair $(X,Y)$ of $u$-filtrations is {\bf strongly $(J_1,J_2)$-interleaved} (or, more colloquially, that $X$ and $Y$ are strongly $(J_1,J_2)$-interleaved) if there exist morphisms $f:R_{J_1^{-1}(u)}(X)\to Y(J_1)$ and $g:R_{J_2^{-1}(u)}(Y)\to X(J_2)$ such that 
\begin{align*}
g(J_1)\circ R_{J_1^{-1}\circ J_2^{-1}(u)}(f)&=S(X,J_2\circ J_1)\textup{ and}\\ 
f(J_2)\circ R_{J_2^{-1}\circ J_1^{-1}(u)}(g)&=S(Y,J_1\circ J_2);
\end{align*}
we say $(f,g)$ is a pair of {\bf strong $(J_1,J_2)$-interleaving morphisms} for $(X,Y)$.  

\begin{remark}
When discussing $(J_1,J_2)$-interleavings and $(J_1,J_2)$-interleaving morphisms of $u$-filtrations, we will follow terminological conventions analogous to those introduced for $(J_1,J_2)$-interleaved $B_n$-persistence modules in Remark~\ref{Rem:General_Interleaving_Terminology_Remark}.
\end{remark}

If for $\epsilon\geq 0$, a pair of $u$-filtrations is strongly $(J_\epsilon,J_\epsilon)$-interleaved, we say that they are strongly $\epsilon$-interleaved. 

Since the definition of strong $(J_1,J_2)$-interleavings is a bit complex, the reader may find it illuminating to see the definition in the special case of $\epsilon$-interleavings: For $\epsilon\in [0,\infty)$, two $u$-filtrations $X$ and $Y$ are strongly $\epsilon$-interleaved if and only if there exist morphisms $f:R_{u-\epsilon}(X)\to Y(\epsilon)$ and $g:R_{u-\epsilon}(Y)\to X(\epsilon)$ such that 
\begin{align*}
g(\epsilon)\circ R_{u-2\epsilon}(f)&=S(X,2\epsilon)\textup{ and}\\
f(\epsilon)\circ R_{u-2\epsilon}(g)&=S(Y,2\epsilon).
\end{align*}
  
When $u={\vec\infty}_n$ the definition of $\epsilon$-interleavings simplifies yet further to give a definition closely analogous to that of $\epsilon$-interleavings of $B_n$-persistence modules: For $\epsilon\in [0,\infty)$, two $n$-filtrations $X$ and $Y$ are strongly $\epsilon$-interleaved if and only if there exist morphisms $f:X\to Y(\epsilon)$ and $g:Y\to X(\epsilon)$ such that 
\begin{align*}
g(\epsilon)\circ f&=S(X,2\epsilon)\textup{ and}\\
f(\epsilon)\circ g&=S(Y,2\epsilon).
\end{align*}

\subsubsection{The Strong Interleaving Distance}
We define $d_{SI}:\obj^*(u$-filt$)\times \obj^*(u$-filt$)\to [0,\infty]$, the {\bf strong interleaving distance}, by taking \[d_{SI}(X,Y)=\inf \{\epsilon\in \R_{\geq 0}|X\textup{ and }Y\textup{ are strongly }\epsilon\textup{-interleaved}\}.\]

\subsection{Stability Results for the Strong Interleaving distance}

$d_{SI}$ satisfies stability properties with respect to the functors $F^S$, $F^{SO}$, $F^{\SCe}$ analogous to the stability properties Theorems~\ref{MultidimensionalSublevelsetStability}-\ref{MultidimensionalCechPersistenceStability} satisfied by the interleaving distance on $B_n$-persistence modules with respect to the functors $H_i\circ F^S$, $H_i\circ F^{SO}$, and $H_i\circ F^{\SCe}$.  The formulation and proofs of these results are easy modifications of those of Section~\ref{MultidimensionalStabilitySection}.

\subsection{First Examples of $(J_1,J_2)$-interleaved Filtrations}

\begin{example}\label{Ex:OpenAndClosedDeltaInterleaved} \mbox{}
\begin{enumerate}
\item[(i)] For any $\delta>0$ and $(X,Y,d,f)\in \obj(C^{SO})$, $F^{SO}(X,Y,d,f),F^{SO-Op}(X,Y,d,f)$ are strongly $\delta$-interleaved.
\item[(ii)] For any $\delta>0$ and $(X,Y,d,f)\in \obj(C^{SCe})$, $F^{\SCe}(X,Y,d,f),F^{SCe-Op}(X,Y,d,f)$ are strongly $\delta$-interleaved.
\end{enumerate}
\end{example}

\begin{example}\label{Ex:Rips_and_Cech_Generalized_Interleaving}
For $n\in \NN$, let $\id_n:\R^n\to \R^n$ denote the identity map.  

Let $\J_1:\R\to \R$ be given by $\J_1(x)=2x$. If $(Y,d)$ is any metric space, $X\subset Y$, $F^{R}(X,d)$ is the usual Rips filtration on $(X,d)$, and $F^{\Ce}(X,Y,d)$ is the usual \Cech filtration on $(X,Y,d)$, then the inclusions (\ref{Eq:Rips_And_Cech}) in Section~\ref{Sec:RipsInferenceIntro} immediately imply that $F^{R}(X,d),F^{\Ce}(X,Y,d)$ are $(\J_1,\id_1)$-interleaved.

To generalize this example, for $n\in \NN$, let $\J_{n+1}:\R^{n+1}\to \R^{n+1}$ be given by $\J_{n+1}((a,b))=(a,2b)$ for all $a\in \R^n,$ $b\in \R$.
Then it again follows immediately from the inclusions (\ref{Eq:Rips_And_Cech}) in Section~\ref{Sec:RipsInferenceIntro} that for any $(X,Y,d,\gamma)\in \obj(C^{SCe})$, $F^{SR}(X,d,\gamma),F^{\SCe}(X,Y,d,\gamma)$ are $(\J_{n+1},\id_{n+1})$-interleaved.
\end{example}

\subsection{A Characterization of Strongly Interleaved Pairs of Modules of Nesting Type}

We would like to have a transparent characterization of strongly $(J_1,J_2)$-interleaved pairs of $u$-filtrations, as Theorem~\ref{GeneralAlgebraicRealization} gives us for $(J_1,J_2)$-interleaved pairs of $B_n$-persistence modules.

It is not clear how obtain such a characterization in full generality; the obstacle to adapting the approach of Theorem~\ref{GeneralAlgebraicRealization} is that, at least naively, in the category of $u$-filt there is no reasonably behaved notion of free covers, as we have in the category $B_n$-mod.

However, we present a characterization result, Theorem~\ref{Thm:CharacterizationOfStrongInterleavings}, for strongly $(J_1,J_2)$-interleaved pairs $u$-filtrations, provided we restrict attention to the $u$-filtrations in a full subcategory {\bf $u$-nest} of $u$-filt which is large enough to contain all of the examples of $u$-filtrations that we have occasion to consider in this thesis.  In the special case that $u={\vec \infty}_n$, this characterization is very simple and transparent (Corollary~\ref{Cor:CharacterizationForNFiltrations}).  For general $u$, the characterization is not as transparent, but it does imply a transparent necessary condition for the existence of a strong $(J_1,J_2)$-interleaving between two $u$-filtrations of nested type, Corollary~\ref{Cor:CharacterizationCorollary}, similar to our necessary and sufficient condition for $n$-filtrations.

\subsubsection{$u$-filtrations of Nested Type}

We define $u$-nest to be the full subcategory of $u$-filt whose objects are filtrations $X$ such that for all $a\leq b<u$, the transition map $\phi_X(a,b)$ is injective.  We refer to the objects of $u$-filt as {\bf $u$-filtrations of nested type}.\footnote{Note that in the literature on persistent homology, filtrations of nested type are typically the only kinds of filtrations defined and studied; there is not a huge loss of generality in restricting attention to these.}

As we will now see, we can regard $u$-filtrations of nested type as an $n$-parameter family of nested subspaces of an ambient topological space; our characterization of strongly $(J_1,J_2)$-interleaved $u$-filtrations follows from the adoption of this viewpoint. 

\subsubsection{Colimits of $u$-filtrations}
Let $\Top$ denote the category of topological spaces.

Define the {\bf diagonal functor} $\Delta:\Top\to u$-filt to act on objects $W$ in $\Top$ by taking $\Delta(W)_a=W$ for all $a<u$ and $\phi_{\Delta(W)}(a,b)=\id_W$ for all $a\leq b<u$; we define the action of $\Delta$ on $\hom(\Top)$ in the obvious way.
  
It is well known that the usual category of topological spaces has all small colimits \cite{mac1998categories,dwyer1995homotopy}.  In particular, we can define the colimit of a $u$-filtration X.  This is a topological space $\colim(X)$ together with a morphism $s^X:X\to \Delta(\colim(X))$ satisfying the following universal property: If $W$ is a topological space and $s':X\to \Delta(W)$ is a morphism, then there is a unique map $t:\colim(X)\to W$ such that $s'=\Delta(t) \circ s^X$.

For a $u$-filtration $X$ $\colim(X)$ is, as a set, the quotient of $\amalg_{a<u} X_a$ by the relation $\sim$ generated by taking $x\sim\phi_X(a,b)(x)$, for all $a\leq b<u$ and $x\in X_a$.  We define $s^X$ so that for each $a<u$, $s^X_a$ takes each element of $X_a$ to its equivalence class in $\colim(X)$.  We take $\colim(X)$ to be endowed with the {\it final topology}.  This means that a set $U\in \colim(X)$ is open if and only if $(s^X_a)^{-1}(U)$ is open for all $a<u$.   

We can, in the obvious way, also define define an action of $\colim(\cdot)$ on $\hom(u$-filt) to obtain a functor $\colim: u$-filt$\to \Top$. 

\begin{lem}
If $X$ is a $u$-filtration of nested type then $s^X_a$ is injective for all $a<u$.  
\end{lem}

\begin{proof}
The proof is straightforward; we leave it to the reader.  
\end{proof}

\subsubsection{Colimit Representations of $u$-Filtrations of Nested Type}
Let $\Sub(u)=\{y|y<u\}$.  Let $A_u$ denote the set of subsets $S$ of $\Sub(u)$ such that if $y\in S$ and $y\leq y'<u$ then $y'\in S$.

Let $CL_u$ denote the category whose objects are pairs $(W,f)$, where $W$ is a topological space and $f:W\to A_u$ is a function.  For $(W_1,f_1)$, $(W_2,f_2)\in \obj(CL_u)$, let $\hom((W_1,f_1),(W_2,f_2))$ denote continuous maps $g:W_1\to W_2$ such that for all $y\in W_1$, $f_1(y)\subset f_2(g(y))$.

We'll define a functor $\A:u$-filt$\to CL_u$ such that $\A$ restricts to an equivalence between $u$-nest and its image under $\A$.

First, for $X$ a $u$-filtration, we define a function $ROA_X:\colim(X)\to A_u$ by \[ROA_X(x)=\{y|y<u\textup{ and } x\textup{ lies in the image of } s^X_y\}.\]  We say that $ROA_X(x)$ is the {\bf region of appearance of $x$ in $X$}.  We define $\A$ by taking $\A(X)=(\colim(X),ROA_X)$ for any $u$-filtration $X$ and $\A(f)=\colim(f)$ for all $f\in \hom(u$-filt).

We can define a functor $\filt:CL_u\to u$-nest by taking $\filt((W,f))_a=\{y\in W|a\in f(y)\}$ and endowing $\filt((W,f))_a$ with the subspace topology.  The action of $\filt$ on $\hom(CL_u)$ is defined in the obvious way.

\begin{lem}\label{lem:ColimEquivalence}
The restriction of $\A$ to $u$-$\textup{nest}$ is equivalence of categories between $u$-nest and the full subcategory ${CL'_u}$ of $CL_u$ whose set of objects is the image of $\obj(u$-$\textup{nest})$ under $\A$; the restriction of $\filt$ to ${CL'_u}$ is an inverse of the restriction of $\A$ to $u$-\textup{nest}.  
\end{lem}

\begin{proof}
This is straightforward; we leave the details to the reader. 
\end{proof}

\subsubsection{Characterization of Interleaved pairs of Filtrations of Nested Type}

Now we are ready to offer our characterization of $(J_1,J_2)$-interleaved pairs of $u$-filtrations of nested type.  Roughly, it says that $u$-filtrations $X,Y$ are $(J_1,J_2)$-interleaved if and only if there is a pair of inverse homeomorphisms between ``large" subspaces of $\colim(X)$ and $\colim(Y)$ which is compatible with the data of $ROA_X$ and $ROA_Y$ in an appropriate sense, and additionally these homeomorphisms admit certain extensions.

Note that if $u\in \hat \R^n$, $u'\leq u$, and $X$ is a $u$-filtration, then $\colim(R_{u'}(X))$ can be identified with a subspace of $\colim(X)$.

\begin{thm}\label{Thm:CharacterizationOfStrongInterleavings}
For $J_1,J_2$ increasing, a pair $(X,Y)$ of $u$-filtrations of nested type is $(J_1,J_2)$-interleaved if and only if 
there are maps 
\begin{align*}
&f:\colim(R_{J_1^{-1}(u)}(X))\to \colim(Y) \textup{ and}\\
&g:\colim(R_{J_2^{-1}(u)}(Y))\to \colim(X)
\end{align*}
such that 
\begin{enumerate*}
\item $f$ and $g$ restrict to inverse homeomorphisms between 
\begin{align*}
\colim(R_{J_1^{-1}\circ J_2^{-1}(u)}(X))&\cup g(\colim(R_{J_2^{-1}\circ J_1^{-1}(u)}(Y)))\textup{ and }\\
\colim(R_{J_2^{-1}\circ J_1^{-1}(u)}(Y))&\cup f(\colim(R_{J_1^{-1}\circ J_2^{-1}(u)}(X))),
\end{align*}
\item $J_1(ROA_X(y))\subset ROA_Y(f(y))$ for all $y\in \colim(R_{J_1^{-1}(u)}(X))$, and 
\item $J_2(ROA_Y(y))\subset ROA_X(g(y))$ for all $y\in \colim(R_{J_2^{-1}(u)}(Y))$.
\end{enumerate*}
\end{thm}

\begin{proof}
The proof is a straightforward application of Lemma~\ref{lem:ColimEquivalence}; we leave the unwinding of definitions to the reader.
\end{proof}

The following immediate consequence of the last theorem gives our transparent necessary condition for the existence of a $(J_1,J_2)$-interleaving between two $u$-filtrations.

\begin{cor} \label{Cor:CharacterizationCorollary} If a pair $(X,Y)$ of $u$-filtrations of nested type is $(J_1,J_2)$-interleaved then there 
are sets $D_X, D_Y$ with 
\begin{align*}
\colim(R_{J_1^{-1}\circ J_2^{-1}(u)}(X))&\subset D_X \subset \colim(X),\\
\colim(R_{J_2^{-1}\circ J_1^{-1}(u)}(Y))&\subset D_Y \subset \colim(Y),
\end{align*} and a homeomorphism $f:D_X\to D_Y$ such that 
\begin{enumerate*}
\item $J_1(ROA_X(y))\subset ROA_Y(f(y))$ for all $y\in D_X$, and 
\item $J_2(ROA_Y(y))\subset ROA_X(f^{-1}(y))$ for all $y\in D_Y$.
\end{enumerate*}
\end{cor}

When $u={\vec\infty}_n$, Theorem~\ref{Thm:CharacterizationOfStrongInterleavings} reduces to the following:

\begin{cor}\label{Cor:CharacterizationForNFiltrations}
A pair $(X,Y)$ of $n$-filtrations of nested type is $(J_1,J_2)$-interleaved if and only if 
there is a homeomorphism $f:\colim(X)\to \colim(Y)$, such that
\begin{enumerate*}
\item $J_1(ROA_X(y))\subset ROA_Y(f(y))$ for all $y\in \colim(X)$, and 
\item $J_2(ROA_Y(y))\subset ROA_X(f^{-1}(y))$ for all $y\in \colim(Y)$.
\end{enumerate*}
\end{cor}

\subsection{Optimality Properties of the Strong Interleaving Distance}

Using Corollary~\ref{Cor:CharacterizationForNFiltrations}, we now observe that $d_{SI}$ is optimal is a sense analogous to which we have proven $d_I$ to be optimal in Section~\ref{OptimalitySpecifics}.  We also show that $d_I$ is $\RR$-optimal, for $\RR$ a relative structure defined in terms of $d_{SI}$.

The results here, while perhaps not especially surprising in light of the optimality results of Section~\ref{OptimalitySpecifics}, are of significance in that they show that optimality theory for $d_{SI}$ behaves as one would hope, given those results.

\begin{thm}\label{Thm:FiltrationOptimality1}
Let ${\mathcal R}_{4}=(\T,X_\T,f_\T)$ be the relative structure on $\obj^*(n$-$\filt)$ (as defined in Section~\ref{OptimalityGeneralities}) such that $\T$ is  the set of topological spaces, $X_{T}=\obj(C_T^S)$ for all $T\in\T$, $d_{T}=d^S_T$, and $f_T$ is given by $f_{T}(T,g)=F^S(T,g)$.  Then $d_{SI}$ is $\RR_{4}$-optimal.
\end{thm}

\begin{proof}
We have observed above that $d_{SI}$ is $\RR_{4}$-stable.  By essentially the same argument by which Theorem~\ref{MainOptimality} follows from Proposition~\ref{RealizationProp}, to prove the result it's enough to show that if two $n$-filtrations $X,Y\in \im(f_\T)$ are strongly $\epsilon$-interleaved then there exists a topological space $T$ and functions $\gamma_X,\gamma_Y:T\to \R^n$ such that $F^S(T,\gamma_X)\cong X$,  $F^S(T,\gamma_Y)\cong Y$, and $\sup_{y\in T}\|f_1(y)-f_2(y)\|_\infty\leq \epsilon$.  This follows from Corollary~\ref{Cor:CharacterizationForNFiltrations}.  
\end{proof}

\begin{thm}\label{Thm:FiltrationOptimality2}
For $i\in \NN$, let $\RR_{5,i}$ be the relative structure on $\obj^*(B_n$-\mod$)$ such that $\T$ is the singleton set $s$, $X_s=\obj(n$-\filt), $d_s=d_{SI}$, and $f_{s}$ is given by $f_{s}(X)=H_i(X)$.  If $k=\Q$ or $k=\Z/p\Z$ for some prime $p$ then $d_I$ is $\RR_{5,i}$-optimal.
\end{thm}

\begin{proof}
Cleary $d_I$ is $\R_{5,i}$-stable.  Optimality follows from Proposition~\ref{RealizationProp} in essentially the same way that Theorem~\ref{MainOptimality} follows from the same proposition.  
\end{proof}

\section[Weak Interleavings and the Weak Interelaving Distance]{Weak Interleavings and the Weak Interelaving Distance on $u$-Filtrations}\label{Sec:WeakInterleavings}
\subsubsection{Overview}
We now turn to the definition and basic theory of weak interelavings and the weak interleaving distance.

To begin, we prepare for the definition of weak interleavings by reviewing localization of categories and defining the localization of the category of $u$-filtrations with respect to {\it levelwise homotopy equivalences}.  We observe that the shift and restriction functors descend to the localized category.  Using the descents of these functors, we define weak $(J_1,J_2)$-interleavings and the weak interleaving distance.  

After proving some basic facts about weak interleavings,  
we show that persistent homology is stable with respect to the weak $(J_1,J_2)$-interleavings on $u$-filtrations and $(J_1,J_2)$-interleavings on $B_n$-persistence modules.  This result will be very useful to us for passing from our results about $u$-filtrations to corresponding results about persistent homology modules.

We finish the section by posing some questions about the weak interleaving distance which we hope to answer in future work.

\subsection{Localization of Categories}
Localization of categories is a standard construction in homotopy theory.  It is analogous to the localization of rings and modules studied in commutative algebra.

\subsubsection{Definition}
Let $C$ be a category and let $W\subset  \hom(C)$ be a class of morphisms.  A {\it localization} of $C$ with respect to $W$ is a category $C[W^{-1}]$ together with a functor $\Gamma:C\to C[W^{-1}]$ such that

\begin{enumerate}
\item[(i)] $\Gamma(f)$ is an isomorphism for each $f\in W$.
\item[(ii)] [universality] Whenever $G:C\to D$ is a functor carrying elements of $W$ to isomorphisms, there exists a unique functor $G':C[W^{-1}]\to D$ such that $G'\circ \Gamma=G$.   
\end{enumerate}
Condition $(ii)$ guarantees that if $\Gamma:C\to D$ and $\Gamma':C\to D'$ are two localizations of $C$ with respect to $W$, then there is a unique isomorphism of categories $\Theta:D\to D'$ such that $\theta\circ\Gamma=\Gamma'$.  We will often denote $C[W^{-1}]$ as $Ho(C)$, suppressing the dependence of this category on $W$.

If C is a {\it small category}, meaning that $\obj(C)$ and $\hom(C)$ are each sets (as opposed to proper classes \cite{jech2003set}), then a localization of $C$ with respect to any set of morphisms $W\subset \hom(C)$ exists---see \cite[Section 42.4]{dwyer1997model} for a simple proof of this, using the existence of pushouts in the category of small categories \cite{mac1998categories}.

The localization of a category, when it exists, admits an explicit construction---see \cite[Section 42.7]{dwyer1997model}.  From the form of this construction we obtain the following lemma.

\begin{lem}\label{lem:GeneratorsOfLocalization}
For any category $C$, $W\subset \hom(C)$ there is a localization $\Gamma:C\to C[W^{-1}]$ such that $\obj(C[W^{-1}])=\obj(C)$ and $\hom(C[W^{-1}])$ is generated under composition by the images under $\Gamma$ of elements of $\hom(C)$ together with the inverses of the images under $\Gamma$ of elements of $W$.
\end{lem}

\subsubsection{The Connection Between Localization and Closed Model Categories}\label{Sec:LocalizationAndClosedModelCategories}
Let $W\subset \hom(\Top)$ denote the homotopy equivalences.  Then it can be shown that $\Top[W^{-1}]$ is isomorphic to the usual homotopy category of topological spaces (i.e. the category having the same objects as $\Top$ and homotopy classes of continuous maps as morphisms) \cite{dwyer1995homotopy}.  

In fact, there is a very general setting in which a localization can be constructed in a homotopy theoretic fashion---this is the setting of {\it closed model categories}, as introduced by the 1967 monograph of Daniel Quillen \cite{quillen1967homotopical}.   A closed model category is a category together with certain extra structure that allows for the axiomatic development of homotopy theory in that category.  In more detail, a closed model category $C$ is a category together with three distinguished classes of morphisms called {\it weak equivalences}, {\it fibrations}, and {\it cofibrations}.  The category is required to have the property that all finite limits and colimits exist \cite{dwyer1995homotopy}, and the weak equivalences, fibrations, and cofibrations are required to satisfy certain conditions.  Let $W$ denote the weak equivalences of $C$.  Quillen showed that the closed model structure on $C$ allows for a homotopy theoretic construction of $C[W^{-1}]$ generalizing the construction of the usual homotopy category of topological spaces.  For more details, see the expository articles \cite{dwyer1995homotopy,dwyer1997model}.

We will not need to appeal to the theory of closed model categories in this thesis, and indeed it is not clear that the particular localization we consider in our definition of weak interleavings admits an interpretation in terms of closed model categories.  However, a very closely related notion of localization does admit such an interpretation; see Remark~\ref{ClosedModelStructureRemark}.  In any case, it may be useful for the reader to keep in mind the connection between localization and homotopy theory provided by Quillen's theory, even if we do not take advantage of it here.  As the reader will see in Section~\ref{Sec:WeakInterleavingsQuestions}, there are important questions about weak interleavings which this thesis raises but does not answer, and it is quite plausible that closed model categories and axiomatic homotopy theory will be useful in addressing these questions.

\subsubsection{Localization of $u$-filt with Respect to Levelwise Homotopy Equivalences}
 
For any $u\in{\hat \R}^n$, let $W\subset \hom(u$-filt) denote the {\it levelwise homotopy equivalences}, the morphisms $f$ for which $f_a$ is a homotopy equivalence for all $a<u$.  In what follows, we'll work only with localizations of $u$-filt with respect to $W$.  Thus from now on, $Ho(u$-filt) will always denote $u$-filt$[W^{-1}]$, and $\Gamma: u$-filt $\to Ho(u$-filt) will denote the localization functor; the dependence of $\Gamma$ on $u$ will be implicit.

\begin{remark}\label{ClosedModelStructureRemark}  I am not aware of any way of endowing $u$-filt with the structure of a closed model category having $W$ the class of weak equivalences.  However, let $W'\subset \hom(u$-filt) denote the {\it levelwise weak homotopy equivalences}, the morphisms $f$ for which $f_a$ is a {\it weak homotopy equivalence} \cite{hatcher2002algebraic} for all $a<u$.  It is known that $u$-filt can be given the structure of a closed model category having $W'$ as the class of weak equivalences; see the discussion at the end of \cite[Section 10]{dwyer1995homotopy} and the references given there.
\end{remark}

\subsection{Definitions of Weak Interleavings and the Weak Interleaving Distance}\label{Sec:WeakInterleavingsDef}

\subsubsection{Shift Functors and Restriction Functors in the Homotopy Category}
By the universal property of the localization of a category, for any $u\in {\hat \R}^n$ and $J:\R^n\to \R^n$ order-preserving, the functor $(\cdot) (J):u$-filt $\to J^{-1}(u)$-filt descends to a functor from $Ho(u$-filt) to $Ho(J^{-1}(u)$-filt).   By abuse of notation, we also write this functor as $(\cdot)(J)$. 
Note that \[\Gamma \circ (\cdot)(J)=(\cdot)(J)\circ \Gamma.\] 

Similarly, by the universal property of the localization of a category, for any $u'\leq u\in {\hat \R}^n$ the functor $R_{u'}:u$-filt $\to u'$-filt descends to a functor from $Ho(u$-filt) to $Ho(u'$-filt).   By abuse of notation, we also write this functor as $R_{u'}$. 
Note that \[\Gamma \circ R_{u'}=R_{u'}\circ \Gamma.\] 

\subsubsection{Commutativity in the Localized Category of Transition Morphisms with Arbitrary Morphisms}
The following is an analogue of Lemma~\ref{lem:TransitionMorphismsAndShiftsForFilts}(i) in the category $Ho(u$-filt).  It is the key step in our proofs of Lemmas~\ref{lem:InterleavingsUnderExpansion} and~\ref{Lem:InterleavingTriangleInequality} below.

\begin{lem}\label{lem:CommutativityInHomotopyCategory}
For any $u\in \hat \R^n$, $u$-filtrations $X$ and $Y$, morphism $f\in\hom_{Ho(u-\filt)}(X,Y)$, and $J$ increasing, we have that \[f(J)\circ \Gamma(S(X,J))=\Gamma(S(Y,J))\circ f.\]
\end{lem}
\begin{proof}
By Lemma~\ref{lem:GeneratorsOfLocalization} we can write $f=f_1\circ ...\circ f_l$ where for each $i$, either $f_i=\Gamma(\tilde f)$ for some $\tilde f\in \hom(C)$ or $f_i=\Gamma(\tilde f)^{-1}$ for some $\tilde f \in W$.

Note that $f(J)=f_1(J)\circ ... \circ f_l(J)$.  Thus to prove the result it's enough to check that for $\tilde f \in W$, 
\[\Gamma(\tilde f)^{-1}(J)\circ S(X,J)=S(Y,J) \circ \Gamma(\tilde f)^{-1}.\]  
From the equation \[\tilde f(J)\circ S(X,J)=S(Y,J) \circ \tilde f,\]
we have that \[\Gamma(\tilde f(J))\circ \Gamma(S(X,J))=\Gamma(S(Y,J)) \circ \Gamma(\tilde f).\]
Composing both sides of the equation on the left by $\Gamma(\tilde f)^{-1}(J)$ and on the right by $\Gamma(\tilde f)^{-1}$, we obtain that
\[\Gamma(\tilde f)^{-1}(J)\circ \Gamma(S(X,J))=\Gamma(S(Y,J)) \circ \Gamma(\tilde f)^{-1}\] as desired.
\end{proof}

\subsubsection{Weak $(J_1,J_2)$-interleavings of $u$-filtrations}
For $J_1,J_2:\R^n\to\R^n$ increasing, we say that an ordered pair of $u$-filtrations $(X,Y)$ is {\bf weakly $(J_1,J_2)$-interleaved} if there exist morphisms $f:\Gamma (R_{J_1^{-1}(u)}(X))\to \Gamma (Y(J_1))$ and $g:\Gamma(R_{J_2^{-1}(u)}(Y))\to \Gamma (X(J_2))$ such that 
\begin{align*}
g(J_1)\circ R_{J_1^{-1}\circ J_2^{-1}(u)}(f)=\Gamma (S(X,J_2\circ J_1))\\
f(J_2)\circ R_{J_2^{-1}\circ J_1^{-1}(u)}(g)=\Gamma (S(Y,J_1\circ J_2)); 
\end{align*}
we say $(f,g)$ is a pair of weak $(J_1,J_2)$-interleaving morphisms for $(X,Y)$.  

For $\epsilon\geq 0$, we define weak $\epsilon$-interleavings in the expected way.

\subsubsection{The Weak Interleaving Distance on $u$-filtrations}\label{Sec:WeakInterleavingDef}
    
We define $d_{WI}:\obj^*(u$-filt$)\times \obj^*(u$-filt$)\to [0,\infty]$, the {\bf weak interleaving distance}, by taking \[d_{WI}(X,Y)=\inf \{\epsilon\in \R_{\geq 0}|X\textup{ and }Y\textup{ are weakly }\epsilon\textup{-interleaved}\}.\]

It follows from Lemma~\ref{Lem:InterleavingTriangleInequality}(ii) below that $d_{WI}$ is a pseudometric on $\obj^*(u$-filt).  It is easily checked that this pseudometric is not a metric.

\subsection{Stability Results for the Weak Interleaving Distance}
$d_{WI}$ satisfies stability properties with respect to the functors $F^S$, $F^{SO}$, $F^{SR}$, and $F^{\SCe}$ analogous to the stability properties  Theorems~\ref{MultidimensionalSublevelsetStability}-\ref{MultidimensionalRipsPersistenceStability} satisfied by the interleaving distance on $B_n$-persistence modules with respect to the functors $H_i\circ F^S$, $H_i\circ F^{SO}$, $H_i\circ F^{SR}$, and $H_i\circ F^{\SCe}$.  The formulation and proofs of these results are easy modifications of those of Theorems~\ref{MultidimensionalSublevelsetStability}-\ref{MultidimensionalRipsPersistenceStability}; analogues of Theorems~\ref{MultidimensionalSublevelsetStability}-\ref{MultidimensionalCechPersistenceStability} follow via Lemma~\ref{lem:StrongImpliesWeak} below from the analogues of these theorems for $d_{SI}$ mentioned in Section~\ref{Sec:StrongInterleavingsDef}.  The analogue of Theorem~\ref{MultidimensionalRipsPersistenceStability} for $d_{WI}$ is proved in essentially the same way as Theorem~\ref{MultidimensionalRipsPersistenceStability}, but requires Proposition~\ref{prop:ZeroInterleavingOfSOandCech}(i), a lift of the persistent nerve theorem of \cite{chazal2008towards} to the level of filtrations.     

\subsection{Basic Results about Weak Interleavings}

\begin{lem}\label{lem:StrongImpliesWeak}
If $u$-filtrations $X,Y$ are strongly $(J_1,J_2)$-interleaved then $X,Y$ are weakly $(J_1,J_2)$-interleaved.
\end{lem}

\begin{proof} It's easy to check that if $f,g$ are strong $(J_1,J_2)$-interleaving morphisms for $X,Y$ then $\Gamma(f),\Gamma(g)$ are weak $(J_1,J_2)$-interleaving morphisms for $X,Y$.  We leave the details to the reader. 
\end{proof}

\begin{lem}\label{lem:InterleavingsUnderRestriction}
For any $u'\leq u \in \hat \R^n$ and $(J_1,J_2)\in \R^n$, if $u$-filtrations $X,Y$ are strongly (weakly) $(J_1,J_2)$-interleaved, then $R_{u'}(X),R_{u'}(Y)$ are also strongly (weakly) $(J_1,J_2)$-interleaved. 
\end{lem}

\begin{proof}  It is easy to check that if $f,g$ are strong (weak) $(J_1,J_2)$-interleaving morphisms for $X,Y$, then $R_{u'}(f),R_{u'}(g)$ are strong (weak) $(J_1,J_2)$-interleavings for $R_{u'}(X),R_{u'}(Y)$.
\end{proof}

To prove our consistency results, we will require the following lemma. 

\begin{lem}\label{lem:InterleavingsUnderExpansion}
For any $u'\leq u \in \hat \R^n$ and $\epsilon\geq 0$, if $X,Y$ are $u$-filtrations such that $R_{u'}(X),R_{u'}(Y)$ are strongly (weakly) $\epsilon$-interleaved then $X,Y$ are strongly (weakly) $(J_{\epsilon+u-u'},J_{\epsilon+u-u'})$-interleaved. 
\end{lem}

\begin{proof}  Let $f:R_{u'-\epsilon}(X)\to R_{u'}(Y)(\epsilon)$ and $g:R_{u'-\epsilon}(Y)\to R_{u'}(X)(\epsilon)$ be strong $\epsilon$-interleaving homomorphisms for $R_{u'}(X),R_{u'}(Y)$.  Note that $R_{u'}(X)(\epsilon)=R_{u'-\epsilon}(X(\epsilon))$ and $R_{u'}(Y)(\epsilon)=R_{u'-\epsilon}(Y(\epsilon))$.  Thus the homomorphisms $S(Y(\epsilon),J_{u-u'})\circ f$ and $S(X(\epsilon),J_{u-u'})\circ g$ are well defined, and it's easy to check that these are strong $(J_{\epsilon+u-u'},J_{\epsilon+u-u'})$-interleaving homomorphisms for $X,Y$.  

Invoking Lemma \ref{lem:CommutativityInHomotopyCategory}, the same argument also gives the result for weak interleavings.\end{proof}

The next lemma will prove very useful to us.  For $u=(u_1,...,u_n), u'=(u'_1,...,u'_n)\in {\hat \R}^n$, let $\gcd(u,u')=(\min(u_1,u'_1),...,\min(u_n,u'_n))$.

\begin{lem}[Triangle Inequality for $(J_1,J_2)$-interleavings]\label{Lem:InterleavingTriangleInequality}
\mbox{}
\begin{enumerate*}
\item[(i)] Suppose we are given $u\in \hat\R^n$, $J_1,J_2,J_3,J_4:\R^n\to \R^n$ increasing, $u_1,u_2\leq u$, and $u$-filtrations $X_1, X_2, X_3$ such that $R_{u_1}(X_1),R_{u_1}(X_2)$ are strongly (weakly) $(J_1,J_2)$-interleaved and $R_{u_2}(X_2),R_{u_2}(X_3)$ are strongly (weakly) $(J_3,J_4)$-interleaved. Then $R_{\gcd(u_1,u_2)}(X_1),R_{\gcd(u_1,u_2)}(X_3)$ are strongly (weakly) $(J_3\circ J_1,J_2\circ J_4)$-interleaved.
\item[(ii)] In particular, if we have $\epsilon_1,\epsilon_2\geq 0$ such that $X_1,X_2$ are strongly (weakly) $\epsilon_1$-interleaved and $X_2,X_3$ are strongly (weakly) $\epsilon_2$-interleaved, then $X_1,X_3$ are strongly (weakly) $(\epsilon_1+\epsilon_2)$-interleaved. 
\end{enumerate*}
\end{lem}

\begin{proof} By Lemma~\ref{lem:InterleavingsUnderRestriction}, it suffices to assume that $u_1=u_2=u$.  It's easy to check that if $f_1,g_1$ are strong $(J_1,J_2)$-interleaving homomorphisms for $X_1,X_2$ and $f_2,g_2$ are strong $(J_3,J_4)$-interleaving homomorphisms for $X_2,X_3$ then $f_2(J_1) \circ f_1,g_1(J_4)\circ g_2$ are strong $(J_3\circ J_1,J_2\circ J_4)$-interleaving homomorphisms for $X_1,X_3$.  This gives $(i)$ for strong interleavings.  The case of weak interleavings follows via the same argument, using Lemma~\ref{lem:CommutativityInHomotopyCategory}.\end{proof}

We noted in Remark~\ref{EasyInterleavingRemark} that if $\epsilon<\epsilon'\in \R$ and two $B_n$-persistence modules are $\epsilon$-interleaved, then they are $\epsilon'$-interleaved.  Here is the analogous statement for $(J_1,J_2)$-interleavings, on the level of filtrations.  (The same proof gives corresponding result for $(J_1,J_2)$-interleaved $B_n$-persistence modules.)

For $J,J':\R^n\to \R^n$ bijections, say that $J\leq J'$ if $J(a)\leq J'(a)$ for all $a\in \R^n$.

\begin{lem}\label{lem:LessThanInterleavings}
Let $J_1,J'_1,J_2,J'_2$ be increasing and $J_1\leq J'_1$, $J_2\leq J'_2$.  If $u$-filtrations $X,Y$ are strongly (weakly) $(J_1,J_2)$-interleaved then $X,Y$ are strongly (weakly) $(J'_1,J'_2)$-interleaved.
\end{lem}

\begin{proof}
Since $J_1$, and $J'_1$ are increasing, there is an increasing map $J''_1:\R^n\to \R^n$ such that $J'_1=J''_1\circ J_1$.  Similarly, there is an increasing map $J''_2:\R^n\to \R^n$ such that $J'_2=J''_2\circ J_2$.  Let $f,g$ be strong $(J_1,J_2)$-interleaving morphisms for $X,Y$.  We show that $S(Y,J''_1)(J_1)\circ f,S(X,J''_2)(J_2)\circ g$ are strong $(J'_1,J'_2)$-interleaving momorphisms for $X,Y$.  By Lemma~\ref{lem:TransitionMorphismsAndShiftsForFilts},
\begin{align*}
& S(X,J''_2)(J_2)(J'_1)\circ g(J'_1) \circ S(Y,J''_1)(J_1)\circ f\\
\quad &=S(X,J''_2)(J_2 \circ J'_1)\circ g(J'_1) \circ S(Y(J_1),J_1^{-1} \circ J'_1)\circ f\\
\quad &=S(X,J''_2)(J_2 \circ J'_1)\circ g(J'_1) \circ f(J_1^{-1} \circ J'_1)\circ S(X,J_1^{-1} \circ J'_1)\\
\quad &=S(X,J''_2)(J_2 \circ J'_1)\circ S(X,J_2 \circ J_1)(J_1^{-1} \circ J'_1)\circ S(X,J_1^{-1} \circ J'_1)\\
\quad &=S(X,J''_2)(J_2 \circ J'_1)\circ S(X,J_2 \circ J_1 \circ J_1^{-1} \circ J'_1)\\
\quad &=S(X,J''_2 \circ J_2 \circ J'_1)\\
\quad &=S(X,J'_2 \circ J'_1).
\end{align*}
The symmetric argument shows that \[S(Y,J''_1)(J_1)(J'_2)\circ f(J'_2) \circ S(X,J''_2)(J_2)\circ g=S(Y,J'_1 \circ J'_2).\]  This gives the result for strong interleavings; the result follows for weak interleavings by the same argument, using Lemma \ref{lem:CommutativityInHomotopyCategory}.
\end{proof}

\subsection{Stability of Persistent Homology with Respect to Interleavings}

In this section, we prove that persistent homology is stable with respect to weak interleavings on $u$-filtrations and interleavings on $B_n$-persistence modules.  

\begin{thm}\label{Thm:StabilityWRTInterleavings}
For any $i\in \Z_{\geq 0}$, $u\in {\hat \R}^n$, and $J_1,J_2:\R^n\to\R^n$ increasing, if $u$-filtrations $X,Y$ are weakly $(J_1,J_2)$-interleaved then $H_i(X),H_i(Y)$ are $(J_1,J_2)$-interleaved.
\end{thm}

\begin{cor}
For any $i\in \Z_{\geq 0}$, $u\in {\hat \R}^n$, and $u$-filtrations $X$ and $Y$, \[d_{I}(H_i(X),H_i(Y))\leq d_{WI}(X,Y).\]
\end{cor}

\begin{proof}[Proof of Theorem~\ref{Thm:StabilityWRTInterleavings}]
We begin with a couple of lemmas.

\begin{lem} For any $i\in \Z_{\geq 0}$ and $J:\R^n\to \R^n$ order-preserving, \[H_i \circ (\cdot)(J)=(\cdot)(J)\circ H_i.\]  \end{lem}

\begin{proof} We leave the straightforward proof to the reader. \end{proof}

Note that if $f\in \hom(B_n$-mod) is such that $f_a$ is an isomorphism for all $a\in \R^n$, then $f$ is an isomorphism.  Thus, since the singular homology functor maps homotopy equivalences to isomorphisms, $H_i$ takes elements of $W$ to isomorphisms in $B_n$-mod.  Then by the universal property of localizations, for any $u\in {\hat \R}^n$, there's a functor $\tilde H_i:Ho(u$-filt$)\to B_n$-mod such that \[H_i=\tilde H_i\circ \Gamma.\]

\begin{lem} For any $i\in \Z_{\geq 0}$ and $J$ increasing, \[\tilde H_i \circ (\cdot)(J)=(\cdot)(J)\circ \tilde H_i.\]  
\end{lem}

\begin{proof}
$H_i\circ (\cdot)(J)$ takes elements of $W$ to isomorphisms in $B_n$-mod.  Thus by the universal property of localization there exists a 
unique functor $G: Ho(u$-filt$)\to B_n$-mod such that $G\circ \Gamma=H_i\circ (\cdot)(J)$.  We have that  
\begin{align*}
\tilde H_i \circ (\cdot)(J)\circ \Gamma=\tilde H_i \circ \Gamma\circ (\cdot)(J)=H_i\circ (\cdot)(J)=(\cdot)(J) \circ H_i=(\cdot)(J) \circ \tilde H_i \circ \Gamma.
\end{align*}
Thus by the uniqueness property of $G$, $\tilde H_i \circ (\cdot)(J)=G=(\cdot)(J) \circ \tilde H_i$.
\end{proof} 

Now let $f:\Gamma(R_{J_1^{-1}(u)}(X))\to \Gamma(Y(J_1))$ and $g:\Gamma(R_{J_2^{-1}(u)}(Y))\to \Gamma(X(J_2))$ be weak $(J_1,J_2)$-interleaving morphisms for $X,Y$.

The homomorphisms 
\begin{align*}
&\tilde H_i(f): H_i(R_{J_1^{-1}(u)}(X))\to  H_i(Y)(J_1),\\
&\tilde H_i(g): H_i(R_{J_2^{-1}(u)}(Y))\to  H_i(X)(J_2)
\end{align*}
extend to homomorphisms $f_E: H_i(X)\to  H_i(Y)(J_1)$ and $g_E: H_i(Y)\to  H_i(X)(J_2)$ by taking $f_E$ to be identically zero on all homogenous summands $H_i(X)_a$ of $H_i(X)$ such that $a\not <J_1^{-1}(u))$, and taking $g_E$ to be identically zero on all homogenous summands $H_i(Y)_a$ of $H_i(Y)$ such that $a\not <J_2^{-1}(u)$.

We'll now show that  $f_E,g_E$ are $(J_1,J_2)$-interleaving homomorphisms for $H_i(X),H_i(Y)$.  We have that 
\begin{align*}
\tilde H_i(g)(J_1)\circ \tilde H_i(R_{J_1^{-1}\circ J_2^{-1}(u)}(f))&=\tilde H_i(g(J_1)) \circ \tilde H_i(R_{J_1^{-1}\circ J_2^{-1}(u)}(f))\\
\qquad =\tilde H_i(g(J_1)\circ R_{J_1^{-1}\circ J_2^{-1}(u)}(f))&=\tilde H_i \circ \Gamma(S(X,J_2\circ J_1))\\
\qquad =H_i(S(X,J_2\circ J_1))&=S(H_i(X),J_2\circ J_1).
\end{align*}
This implies that for $a< J_1^{-1}\circ J_2^{-1}(u)$, \[(g_E(J_1) \circ f_E)_a=(\tilde H_i(g)(J_1)\circ \tilde H_i(R_{J_1^{-1}\circ J_2^{-1}(u)}(f)))_a=S(H_i(X),J_2\circ J_1)_a.\]  On the other hand, for $a\not < J_1^{-1}\circ J_2^{-1}(u)$, \[(g_E(J_1) \circ f_E)_a=0=S(H_i(X),J_2\circ J_1)_a.\]  Thus, $g_E(J_1) \circ f_E=S(H_i(X),J_2\circ J_1)$.  The symmetric argument shows that $f_E(J_2) \circ g_E=S(H_i(Y),J_1\circ J_2)$.  This shows that $f_E,g_E$ are $(J_1,J_2)$-interleaving homomorphisms for $H_i(X),H_i(Y)$, as desired. \end{proof}

\begin{example}\label{Ex:Homology_Rips_and_Cech_Generalized_Interleaving}
By Example~\ref{Ex:Rips_and_Cech_Generalized_Interleaving} and Theorem~\ref{Thm:StabilityWRTInterleavings}, we have that for any $i\geq 0$ and
$(X,Y,d,\gamma)\in \obj(C^{SCe})$, the $B_1$-persistence modules $H_i\circ F^{R}(X,d),H_i \circ F^{\Ce}(X,Y,d)$ are $(\J_1,\id_1)$-interleaved, and the $B_{n+1}$-persistence modules $H_i\circ F^{SR}(X,d,\gamma),H_i\circ F^{\SCe}(X,Y,d,\gamma)$ are $(\J_{n+1},\id_{n+1})$-interleaved.
\end{example}

\subsection{On Our Choice of Definition of Weak Interleavings}\label{Sec:OurChoiceOfDefinitionWeakInterleavings}
To define weak interleavings, we have considered the localization of $u$-filt with respect to levelwise weak equivalences.  An alternative approach, and one worth understanding because it is very simple, is to define weak interleavings by passing to a category of filtrations over the homotopy category of topological spaces.  We take a moment here to explain this alternative approach to defining weak interleavings and the weak interleaving distance and to discuss the relationship between the distance thus obtained and $d_{WI}$.

For $u\in \hat \R^n$, let $u$-filt$^*$ be the category whose objects are collections of topological spaces $\{X_a\}_{a<u}$ together with morphisms $\{\phi_X(a,b)\in \hom_{Ho(\Top)}(X_a,X_b)\}_{a\leq b<u}$ such that if $a\leq b\leq c<u$ then $\phi_X(b,c)\circ \phi_X(a,b)=\phi_X(a,c)$.  

Given two objects $X$ and $Y$ of $u$-filt$^*$, we define a morphism $f\in \hom_{u\textup{-filt}^*}(X,Y)$ to be a collection of morphisms $\{f_a\in \hom_{Ho(\Top)}(X_a,Y_a)\}_{a<u}$ such that for all $a\leq b<u$, $f_b\circ \phi_X(a,b)=\phi_Y(a,b)\circ f_a$.  

The usual functor $\Gamma': \Top\to Ho(\Top)$ induces a functor $\Gamma':u$-filt$\to u$-filt$^*$.  This functor takes levelwise weak equivalences to isomorphisms.  Further, we can define shift and restriction functors on the categories $u$-filt$^*$ just as we have for the categories $u$-filt.  Thus for any $u\in \hat \R^n$ we have all the structure we need to formulate an alternative definition of weak $(J_1,J_2)$-interleavings of $u$-filtrations via the functor $\Gamma'$ analogous to the one we have formulated via the localization functor $\Gamma$.  We'll call the weak interleavings thus defined {\it $A$-weak interleavings}.  The definition of $A$-weak $\epsilon$-interleavings induces a definition of an {\it $A$-weak interleaving distance} on $u$-filtrations which we denote $d_{AWI}$.

By the universal property of localization, we have a functor $\Theta: Ho(u$-filt)$\to u$-filt$^*$ such that $\Gamma'=\Theta\circ \Gamma$.  It follows that if two filtrations are weakly $(J_1,J_2)$-interleaved, then they are $A$-weakly $(J_1,J_2)$-interleaved, and in particular, $d_{AWI}\leq d_{WI}$.  This implies that all of our inference results in this thesis which are formulated in terms of weak interleavings are still true if instead formulated in terms of $A$-weak interleavings.  

We would like the distance with which we develop our inferential theory to be as sensitive as possible, subject to the conditions that it has reasonable stability properties (such as those mentioned in Section~\ref{Sec:WeakInterleavingsDef}) and that the distance between filtrations $X$ and $Y$ is $0$ whenever there exists a levelwise homotopy equivalence $f:X\to Y$.  Thus the fact that $d_{AWI}\leq d_{WI}$ offers some justification for our working with $d_{WI}$ rather than $d_{AWI}$.

I do not yet know if it is in fact true that $d_{AWI}=d_{WI}$, though I suspect that this equality does not hold in general.  More generally, we can ask the following: Under what circumstances does the existence of an $A$-weak $(J_1,J_2)$-interleaving between two $u$-filtrations imply the existence of a weak $(J_1,J_2)$-interleaving between the two filtrations?  This question is closely related to the problem of understanding the relationship between a homotopy category of diagrams of topological spaces of a given shape and the category of diagrams of the same shape over the homotopy category of topological spaces.  It is well known that these two categories are usually not equivalent \cite[Remark 10.3]{dwyer1995homotopy}, and in fact there are results in the homotopy theory literature which quantify the difference between the two categories \cite{cordier1986vogt}.  We leave it to the future to study how those results bear on the question of when $d_{AWI}=d_{WI}$.

\subsection{The Theory of Weak Interleavings and the Optimality of $d_{WI}$: Open Questions}\label{Sec:WeakInterleavingsQuestions}

To close this section, we present some questions regarding theoretical properties of weak interleavings and the interleaving distance.  Answers to these questions would, it seems, offer us a satisfactory understanding of weak interleavings and the weak interleaving distance.  
\begin{enumerate*}

\item Can we use homotopy colimits to give an analogue for weak interleavings of our characterization Theorem~\ref{Thm:CharacterizationOfStrongInterleavings}?

\item In particular, can we give a transparent topological characterization of weakly $0$-interleaved pairs of filtrations of nested type analogous to the characterization \cite[Corollary 0.21]{hatcher2002algebraic} of pairs of homotopy equivalent topological spaces as pairs of deformation retracts of a common embedding space?   

\item If $u$-filtrations $X,Y$ are weakly $(J_1,J_2)$-interleaved, do there exist strongly $(J_1,J_2)$-interleaved $u$-filtrations $X',Y'$ with $d_{WI}(X,X')=d_{WI}(Y,Y')=0$?  

\item Is $d_{WI}$ optimal in the sense that $d\leq d_{WI}$ for any pseudometric $d$ on $\obj(u$-filt) such that $d\leq d_{SI}$ and $d(X,Y)=0$ whenever $X,Y$ are weakly $0$-interleaved?  A positive answer to this question would provide strong justification for the use of weak interleavings in the development of the theory of topological inference.  Note that a positive answer to the last question would imply a positive answer to this question.

\item A question already raised in Section~\ref{Sec:OurChoiceOfDefinitionWeakInterleavings}: Under what circumstances does the existence of an $A$-weak $(J_1,J_2)$-interleaving for a pair of $u$-filtrations imply the existence of a weak $(J_1,J_2)$-interleaving for the pair?

\item Can an arbitrary $u$-filtration be well approximated (with respect to $d_{WI}$) by a $u$-filtration of nesting type (e.g. by mapping telescope type constructions)?  
\end{enumerate*}

%% file: Part_II/T_Preliminaries_3.tex
In this chapter we apply the interleaving machinery introduced and studied in the previous two chapters to formulate and prove topological inference results for multidimensional filtrations.  See Section~\ref{Sec:Chapter4Overview} for an overview of the chapter.

\section{Inference Preliminaries}\label{Sec:InferencePreliminaries}

\subsection{Basic Notation}  

Thoughtout this chapter, fix $p\in [1,\infty]$ and $m\in \NN$.  Let $d^p$ denote the $L^p$ metric on $\R^m$.  

For $(X,d)$ a metric space and $X'\subset X$, we'll often abuse notation slightly and let $d$ also denote the restriction of $d$ to $X'$.

If $(X,d)$ is a metric space, $x\in X$, and $r\in \R_{\geq 0}$, we let $B_d(x,r)$ denote the closed metric ball of radius $r$ centered at $x$.  That is, $B_d(x,r)=\{x'\in X|d(x,x')\leq r\}$.

Let $(X,d)$ be a metric space, and $S\subset X$.  We say a set $L\subset X$ is an {\it $\epsilon$-sample} of $S$ (w.r.t $d$) if for any $s\in S$, there exists some $l\in L$ such that $d(s,l)\leq \epsilon$.

If $(X,d_X)$ and $(Y,d_Y)$ are metric spaces, a function $\gamma:X\to Y$ is said to be {\it $c$-Lipschitz} if for all $x_1,x_2\in X$, $d_Y(\gamma(x_1),\gamma(x_2))\leq c\cdot d_X(x_1,x_2)$.  In this thesis, the codomain $(Y,d_Y)$ of a Lipschitz function will always be
$(\R^n,d^\infty)$ for some $n$.   When we want to make explicit the metric $d$ on the domain of a Lipchitz function $f$, we will refer to $f$ as a $c$-Lipchitz function (w.r.t. $d$).

\subsection{Riemannian Manifolds and Probability Density Functions}\label{Sec:ManifoldsAndDensities}
In this thesis Riemannian manifolds will always be understood to be manifolds with boundary.  

Let $M$ be a Riemannian manifold of dimension $l$.  The Riemannian structure on $M$ induces a metric $d^M$ on $M$, the {\it geodesic metric}.  In turn, the geodesic metric induces a measure $\H_M^l$ on $(M,\B_M)$, the {\it $l$-dimensional Hausdorff measure} \cite{burago2001course}.  (Here $\B_M$ is the Borel $\sigma$-algebra of $M$).  From now on we'll write $\H_M^l$ as $\H_M$; $l$ will be implicit in this notation.

When $M=\R^m$, endowed with the standard Euclidean metric, $\H_M$ is the usual Lebesgue measure on $\R^n$.

For $y\in M$ and $r>0$, we say a ball $B\equiv B_{d^M}(y,r)$ is {\it strongly convex} if for every pair of points $y',y''$ in the closure of B, there exists a unique shortest path in $M$ between $y'$ and $y''$, and the interior of this path is included in $B$.  Define $\rho(M)$, the strong convexity radius of $M$, by \[\rho(M)=\inf_{y\in M} \sup_{r>0} \{r | B_{d^M}(y,r) \textup{ is strongly convex}\} \]  As noted in \cite{chazal2009persistence}, $\rho(M)$ is positive when $M$ is compact.  When $M$ is a Euclidean space, $\rho(M)=\infty$.   

A density function on $M$ is a $\B_M$-measurable function $\gamma: M\to [0,\infty)$ such that $\int_M \gamma \, {\mathrm d} \H_M=1$.  A density function $\gamma$ defines a probability measure ${\mathcal P}_\gamma$ on $M$ with the property that for any $A\in \B_M$, ${\mathcal P}_\gamma(A)=\int_A \gamma \,{\mathrm d} \H_M.$

\subsection{Density Estimators}\label{Sec:DensityEstimators}
As we explained in the introduction, the superlevelset-Rips filtrations we consider in the formation of our consistency results are filtered by density estimators.  We review here some basic concepts and results related to density estimation that we will need in what follows.  Some of the material here also appeared in Section~\ref{Sec:ResultOfChazal}; we include that material again here for readability's sake.

Let $M$ be a Riemannian manifold and let ${\mathcal D}(M)$ be the set of density functions on $M$.  We'll define a density estimator $E$ on $M$ to be a sequence of functions $\{E_z:M^z\to {\mathcal D}(M)\}_{z\in \NN}$ such that for any $y\in M$, the function $E_{z,y}: M^z\to \R$ defined by $E_{z,y}(T)=(E_{z}(T))(y)$ is measurable.  By slight abuse of terminology, we'll also refer to the individual functions $E_z$ as density estimators.

We formulate our results in terms of pairs $(E,\gamma)$ of density estimators and density functions satisfying one of two properties.  Let $T_z$ be a random sample of ${\mathcal P}_\gamma$ of size $z$.  The first property is: 

\begin{enumerate}
\item[\bf{A1}:] $E_z(T_z)$ converges uniformly in probability to $\gamma$.
\end{enumerate}

In stating the second property we'll assume that  $(E,\gamma)$ is defined on $\R^m$ for some $m\in \NN$.  The property is:
\begin{enumerate}
\item[\bf{A2}:] $E_z(T_z)$ converges uniformly in probability to the convolution of $\gamma$ with some kernel function $K$.
\end{enumerate}

A1 is known to hold for kernel density estimators on Euclidean spaces, for a wide class of kernels and density functions $\gamma$, provided the kernel width tends to 0 at an appropriate rate as $z\to \infty$ \cite{gine2002rates}.  Further, Pelletier has shown that the notion of kernel density estimators extends to Riemannian manifolds \cite{pelletier2005kernel}, and a recent article by Henry and Rodriguez \cite{henry2009kernel} shows that under mild assumptions on $(E,\gamma)$ and a similar condition on the rate at which the bandwidth of the kernel tends to 0 as $z\to \infty$, assumption A1 holds for the estimators defined by Pelletier.  In fact, each of the cited results gives a.s. uniform convergence of $E_z(T_z)$ to $\gamma$.

A2 also is known to hold for kernel density estimators $E$ with kernel $K$ on Euclidean spaces, for a wide class of kernels $K$ and density functions $\gamma$, when the bandwidth of the estimator $E_z$ is held fixed as $z$ varies \cite[Proposition 9]{rinaldo2010generalized}.

Let $\gamma:M\to \R$ be a density function, $\tilde \gamma:M \to \R$ be another density function, $q\in (0,1]$, $C>0$, $z\in \NN$, and $E_z:M^z\to  {\mathcal D}(M)$ be a density estimator.  We'll say that $E_z$ is a {\bf $(q,C)$-density estimator of $\tilde \gamma$ w.r.t. ${\mathcal P}_\gamma$} if for $T_z$ a finite i.i.d. random sample of ${\mathcal P}_\gamma$ of size $z$, \[P(\sup_{l\in T_z} |E_z(T_z)(l)-\tilde\gamma(l)|> C)\leq q.\]  When $E_z$ is a {\bf $(q,C)$-density estimator of $\gamma$ w.r.t. ${\mathcal P}_\gamma$}, we'll say simply that $E_z$ is a $(q,C)$-density estimator of $\gamma$.

The two cases of interest to us will be where, in the above definition, $\tilde \gamma=\gamma$ and, in the special case that $M=\R^m$, $\tilde\gamma=\gamma*K$ for a kernel $K$.

\begin{lem}\label{BoundsForEstimatorsSatisfyingA1orA2} \mbox{}
\begin{enumerate*}
\item[(i)]If $M$ is a Riemannian manifold, $\gamma:M\to\R$ is a density function, and $E$ is density estimator on $M$ such that $(E,\gamma)$ satisfies $A1$, then for any $C>0$ and $q\in (0,1]$ there exists $z_0$ such that for all $z\geq z_0$, $E_z$ is a $(q,C)$-density estimator of $\gamma$.  

\item[(ii)]If $\gamma:\R^m\to\R$ is a density function and $E$ is density estimator on $\R^m$ such that $(E,\gamma)$ satisfies $A2$ for some kernel $K$, then for any $C\in (0,\infty)$ and $q\in (0,1]$ there exists $z_0$ such that for all $z \geq z_0$, $E_z$ is a $(q,C)$-density estimator of $\gamma*K$w.r.t. ${\mathcal P}_\gamma$.  
\end{enumerate*}
\end{lem}

\begin{proof} This is immediate from the definitions. \end{proof}

%% file: Part_II/T_Deterministic_Approximation.tex
\section{Deterministic Approximation of Multidimensional Filtrations via Discrete Filtrations}\label{Sec:LInfApproximation}

\subsubsection{Overview}

In this section, we prove two theorems concerning the deterministic topological approximation of sublevelset-offset filtrations via discrete filtrations.  The first theorem concerns approximation by sublevelset-\Cech filtrations; the second theorem concerns approximation by sublevelset-Rips filtrations.

These results are multidimensional analogues of \cite[Theorem 3.1]{chazal2009analysis} and its extension \cite[Theorem 4.5]{chazal2009persistence} to good samplings of sublevelsets of $\gamma$.  They hold on the level of filtrations rather than merely on the level of persistent homology modules.  By Theorem~\ref{Thm:StabilityWRTInterleavings}, we also obtain analogues of these results on the level of persistent homology modules.

In addition to our results on the topological approximation of sublevelset-offset filtrations, we present results on the deterministic approximation of multidimensional sublevelset filtrations and their persistent homology via \Cech and Rips complexes with fixed scale parameter.

\subsection{Weak $0$-interleavings of Open Sublevelset-offset and Open Sublevelset-\Cech Filtrations}
To prove our approximation results, we need the following proposition, which offers some motivation for our definition of weak interleavings of $u$-filtrations; see Remark~\ref{Rem:InterleavingDefinition} below.    

\begin{prop}\label{prop:ZeroInterleavingOfSOandCech}\mbox{}
\begin{enumerate*}
\item[(i)] For any $X\subset \R^m$ and $\gamma:X\to\R^n$, $R_{\vec\infty_{n+1}} \circ F^{SO-Op}(X,\R^m,d^p,\gamma)$ and $R_{\vec\infty_{n+1}} \circ F^{SCe-Op}(X,\R^m,d^p,\gamma)$ are weakly 0-interleaved.
\item[(ii)] Let $M$ be a Riemannian manifold.  For any $X\subset M$ and $\gamma:X\to\R^n$, $R_{(\vec\infty_n,\rho(M))} \circ F^{SO-Op}(X,M,d^M,\gamma)$ and $R_{(\vec\infty_n,\rho(M))} \circ F^{SCe-Op}(X,M,d^M,\gamma)$ are weakly 0-interleaved.
\end{enumerate*}
\end{prop}

Note that in (ii), if $\rho(M)=0$ then all spaces of each of the two filtrations are the empty set, so in this case (ii) holds vacuously.

\begin{proof}
We give the proof for (ii); the proof of (i) is the same. 

The same argument that Chazal and Oudot use to prove the persistent nerve lemma \cite{chazal2008towards} gives us that there exists an $(\vec\infty_{n},\rho(M))$-filtration $Z$ (a ``Mayer-Vietoris blowup filtration," to borrow the terminology of \cite{zomorodian2008localized}) such that there are levelwise homotopy equivalences $f:Z\to R_{(\vec\infty_n,\rho(M))}\circ F^{SO-Op}(X,M,d^M,\gamma)$, and $g:R_{(\vec\infty_n,\rho(M))} \circ  Z\to F^{SCe-Op}(X,M,d^M,\gamma)$ in $\hom((\vec\infty_{n},\rho(M))$-filt).  Thus $\Gamma \circ R_{(\vec\infty_n,\rho(M))} \circ F^{SO-Op}(X,M,d^M,\gamma)$ and $\Gamma \circ R_{(\vec\infty_n,\rho(M))} \circ F^{SCe-Op}(X,M,d^M,\gamma)$ are isomorphic in $Ho((\vec\infty_{n},\rho(M))$-filt), and hence are weakly $0$-interleaved.
\end{proof}

\begin{remark}\label{Rem:InterleavingDefinition}
Note that for $M,X$, and $\gamma$ defined as in the statement of Proposition~\ref{prop:ZeroInterleavingOfSOandCech}(ii), it needn't be true that $R_{(\vec\infty_n,\rho(M))} \circ F^{SO-Op}(X,M,d^M,\gamma)$ and $R_{(\vec\infty_n,\rho(M))} \circ F^{SCe-Op}(X,M,d^M,\gamma)$ are strongly 0-interleaved.  For example, Let $M=[0,1]$ (endowed with the Euclidean metric), $X=\{0,1\}$, and define $\gamma:\{0,1\}\to \R$ by $\gamma(0)=\gamma(1)=0$.  Then $\rho(M)=\infty$.  It's straightforward to check that there can be no pair of strong $0$-interleaving homomorphisms between $F^{SO-Op}(X,M,d^M,\gamma)$ and $F^{SCe-Op}(X,M,d^M,\gamma)$. 

Nevertheless, it is natural to think of $R_{(\vec\infty_n,\rho(M))} \circ F^{SO-Op}(X,M,d^M,\gamma)$ and $R_{(\vec\infty_n,\rho(M))} \circ F^{SCe-Op}(X,M,d^M,\gamma)$ as being ``topologically equivalent."  This motivates a choice of pseudometric on $u$-filtrations with respect to which $u$-filtrations which are isomorphic in $Ho(u$-filt) are distance $0$ from one another.

 
\end{remark}

\subsection{Topological Approximation of Sublevelset-Offset Filtrations via Sublevelset-\Cech Filtrations}\label{Sec:CechApproximation}
We state our first approximation result, Theorem~\ref{Thm:CechApproximation}, in two parts.  It is easily seen that each is a special case of a more general result, but the fully general form of the result is not particularly interesting and it seems more expedient to just state the two special cases of interest separately.   

Theorem~\ref{Thm:CechApproximation}(i) says that for any $W\subset \R^m$, Lipchitz function $\gamma:W\to \R^n$, $T\subset W$ and $\tilde \gamma: T\to \R^n$, $F^{\SCe}(T,\R^m,d^p,\tilde \gamma)$ gives a good topological approximation to $F^{SO}(W,\R^m,d^p,\gamma)$ when the Hausdorff distance between $T$ and $W$ is small and $\sup_{l\in T} \|\tilde\gamma(l)-\gamma(l)\|$ is also small.  

In fact, Theorem~\ref{Thm:CechApproximation}(i) says more generally that for any $u\in {\hat \R}^n$, $R_{u,\infty} \circ F^{\SCe}(T,\R^m,d^p,\tilde \gamma)$ gives a good topological approximation to $R_{u,\infty} \circ  F^{SO}(W,\R^m,d^p,\gamma)$ when the Hausdorff distance between $T\cap \gamma_u$ and $\gamma_u$ is small and $\sup_{l\in T} \|\tilde\gamma(l)-\gamma(l)\|$ is also small.  

Theorem~\ref{Thm:CechApproximation}(ii) gives a variant of this result for Riemannian manifolds.

\begin{thm}\label{Thm:CechApproximation}\mbox{}
\begin{enumerate}
\item[(i)] Let $W$ be a subset of $\R^m$, let $\gamma:W\to \R^n$ be a $c$-Lipschitz function (w.r.t $d^p$) for some $c>0$, let $c'=\max(1,c)$, and let $u\in {\hat \R}^n$.  Let $T\subset W$ be a $\frac{\epsilon}{c'}$-sample of $\gamma_u$ (w.r.t $d^p)$.  Let $C\in \R_{\geq 0}$ and $\tilde \gamma:T\to \R^n$ be a function such that $\|\tilde \gamma(l)-\gamma(l)\|_\infty\leq C$ for all $l\in T$.
Then \[d_{WI}(R_{(u,\infty)} \circ F^{SO}(W,\R^m,d^p,\gamma),R_{(u,\infty)} \circ F^{\SCe}(T,\R^m,d^p,\tilde\gamma))\leq \epsilon+C.\]

\item[(ii)] Let $M$ be a Riemannian manifold and let $\gamma:M\to \R^n$ be a $c$-Lipschitz function (w.r.t $d^M$) for some $c>0$, let $c'=\max(1,c)$, and let $u\in {\hat \R}^n$.  Let $T\subset M$ be a $\frac{\epsilon}{c'}$-sample of $\gamma_u$ (w.r.t $d^M)$.  Let $C\in \R_{\geq 0}$ and $\tilde \gamma:T\to \R^n$ be a function such that $\|\tilde \gamma(l)-\gamma(l)\|_\infty\leq C$ for all $l\in T$.
Then, \[d_{WI}(R_{(u,\rho(M))} \circ F^{SO}(M,d^M,\gamma),R_{(u,\rho(M))} \circ F^{\SCe}(T,M,d^M,\tilde\gamma))\leq \epsilon+C.\]
\end{enumerate}
\end{thm}

\begin{proof}
We present the proof of (i); the proof of (ii) is essentially the same as the proof of (i), using Proposition~\ref{prop:ZeroInterleavingOfSOandCech}(ii) in place of Proposition~\ref{prop:ZeroInterleavingOfSOandCech}(i).

We begin with a lemma.  Let $\gamma_T:T\to \R^n$ denote the restriction of $\gamma$ to $T$.
\begin{lem}\label{Lem:FirstStep} $R_{(u,\infty)} \circ F^{SO}(W,\R^m,d^p,\gamma)$ and $ R_{(u,\infty)} \circ F^{SO}(T,\R^m,d^p,\gamma_T)$ are strongly $\epsilon$-interleaved.
\end{lem}

\begin{proof}
For any $(a,b)\in \R^n\times \R$, let $\gamma_{a,b}$ denote $F^{SO}(W,\R^m,d^p,\gamma)_{(a,b)}$ and let $\lambda_{a,b}$ denote $F^{SO}(T,\R^m,d^p,\gamma_T)_{(a,b)}$.  Consider some $(a,b)\leq (u,\infty)$ and $p\in \gamma_{a,b}$.  There is a point $p'\in \gamma_a$ with $d^p(p,p')\leq b$.  Since $T$ is an $\frac{\epsilon}{c'}$-sample of $\gamma_a$, there is a point $p''\in T$ with $d^p(p'',p')\leq \frac{\epsilon}{c'}$.  Thus $d^p(p'',p)\leq \frac{\epsilon}{c'}+b$.  Since $\gamma$ is $c$-Lipschitz, $\gamma(p'')\leq a+\frac{c\epsilon}{c'}\leq a+\epsilon$.  Hence $p\in \lambda_{a+\epsilon,b+\frac{\epsilon}{c'}}$.  Therefore $\gamma_{a,b}\subset \lambda_{a+\epsilon,b+\frac{\epsilon}{c'}}\subset \lambda_{a+\epsilon,b+\epsilon}$.  The inclusions thus define a morphism $f:R_{(u,\infty)} \circ F^{SO}(W,\R^m,d^p,\gamma)\to R_{(u,\infty)} \circ F^{SO}(T,\R^m,d^p,\gamma_T)(\epsilon)$.

Now let $p\in \lambda_{a,b}$.  Then there is a point $p'\in T$ with $d^p(p,p')\leq b$ and $\gamma(p')\leq a$.  Thus $p\in \gamma_{a,b}$.  Therefore $\lambda_{a,b}\subset \gamma_{a,b}\subset \gamma_{a+\epsilon,b+\epsilon}$.  The inclusions thus define a morphism $g:R_{(u,\infty)} \circ F^{SO}(T,\R^m,d^p,\gamma_T)\to R_{(u,\infty)} \circ F^{SO}(W,\R^m,d^p,\gamma)(\epsilon)$.  

Clearly, $f$ and $g$ are a pair of strong interleaving $\epsilon$-interleaving homomorphisms.  Thus $R_{(u,\infty)} \circ F^{SO}(W,\R^m,d^p,\gamma)$ and $R_{(u,\infty)} \circ F^{SO}(T,\R^m,d^p,\gamma_T)$ are strongly $\epsilon$-interleaved, as desired.
\end{proof}

We now observe that we have a chain of interleaving relationships between filtrations:
\begin{itemize*}
\item Lemma~\ref{Lem:FirstStep} tells us that $R_{(u,\infty)} \circ F^{SO}(W,\R^m,d^p,\gamma)$ and $ R_{(u,\infty)} \circ F^{SO}(T,\R^m,d^p,\gamma_T)$ are strongly $\epsilon$-interleaved.
\item By Example~\ref{Ex:OpenAndClosedDeltaInterleaved}(i), for any $\delta>0$, $F^{SO}(T,\R^m,d^p,\gamma_T)$ and $F^{SO-Op}(T,\R^m,d^p,\gamma_T)$ are strongly $\delta$-interleaved. 
\item By Proposition~\ref{prop:ZeroInterleavingOfSOandCech}(i), $F^{SO-Op}(T,\R^m,d^p,\gamma_T)$ and $F^{SCe-Op}(T,\R^m,d^p,\gamma_T)$ are weakly 0-interleaved.  
\item By Example~\ref{Ex:OpenAndClosedDeltaInterleaved}(ii), $F^{SCe-Op}(T,\R^m,d^p,\gamma_T)$ and $F^{\SCe}(T,\R^m,d^p,\gamma_T)$ are strongly $\delta$-interleaved.  
\item $F^{\SCe}(T,\R^m,d^p,\gamma_T)$ and $F^{\SCe}(T,\R^m,d^p,\tilde\gamma)$ are strongly $C$-interleaved.
\end{itemize*}

Applying Lemmas~\ref{lem:StrongImpliesWeak} and~\ref{Lem:InterleavingTriangleInequality} several times gives us that $R_{(u,\infty)}\circ F^{SO}(W,\R^m,d^p,\gamma)$ and $R_{(u,\infty)} \circ F^{\SCe}(T,\R^m,d^p,\tilde\gamma)$ are $(\epsilon+C+2\delta)$-interleaved.  Since this holds for all $\delta>0$, \[d_{WI}(R_{(u,\infty)}\circ F^{SO}(W,\R^m,d^p,\gamma),R_{(u,\infty)} \circ F^{SR}(T,\R^m,d^p,\tilde\gamma))\leq \epsilon+C,\] as we wanted to show. 
\end{proof}

\begin{remark}\label{Rem:Strengthening} Theorem~\ref{Thm:CechApproximation} can be strengthened somewhat using the language of $(J_1,J_2)$-interleavings, but the strengthening is not especially interesting, so we choose to frame the result using the simpler language of $\epsilon$-interleavings.
\end{remark}

\begin{cor}\label{Cor:CechHomologyApproximation}
For any $i\geq 0$,
\begin{enumerate}
\item[(i)] under the same assumptions as in Theorem~\ref{Thm:CechApproximation}(i),  \[d_I(H_i \circ R_{(u,\infty)}\circ F^{SO}(W,\R^m,d^p,\gamma),H_i \circ R_{(u,\infty)} \circ F^{\SCe}(T,\R^m,d^p,\tilde \gamma))\leq \epsilon+C.\]
\item[(ii)] under the same assumptions as in Theorem~\ref{Thm:CechApproximation}(ii), \[d_I(H_i \circ R_{(u,\rho(M))}\circ F^{SO}(M,d^M,\gamma),H_i \circ R_{(u,\rho(M))} \circ F^{\SCe}(T,M,d^M,\tilde \gamma))\leq \epsilon+C.\]
\end{enumerate}
\end{cor}

\begin{proof}
This follows immediately from Theorem~\ref{Thm:StabilityWRTInterleavings}. 
\end{proof}

\subsection{Topological Approximation of Sublevelset-Offset Filtrations via Sublevelset-Rips Filtrations}\label{Sec:RipsApproximation}

Our second main result of this section, Theorem~\ref{Thm:RipsApproximation}, is an analogue of Theorem~\ref{Thm:CechApproximation} for sublevelset-Rips filtrations rather than sublevelset-\Cech filtrations.  We formulate the result using weak $(J_1,J_2)$-interleavings.

For $(J_1,J_2)$ increasing, we'll say that two $u$-filtrations $X$ and $Y$ are strongly (weakly) {\bf almost $(J_1,J_2)$-interleaved} if for all $\epsilon>0$, $X$ and $Y$ are strongly (weakly) $(J_\epsilon \circ J_1,J_\epsilon \circ J_2)$-interleaved.  Similarly, we'll say that two $B_n$-persistence modules $M$ and $N$ are {\bf almost $(J_1,J_2)$-interleaved} if for all $\epsilon>0$, $M$ and $N$ are $(J_\epsilon \circ J_1,J_\epsilon \circ J_2)$-interleaved. 

For the statement theorem, recall that in Example~\ref{Ex:Rips_and_Cech_Generalized_Interleaving} we defined the map $\J_{n+1}:\R^{n+1}\to \R^{n+1}$ by $\J_{n+1}((a,b))=(a,2b)$ for $a\in \R^n,$ $b\in \R$. 
\begin{thm}\label{Thm:RipsApproximation}\mbox{}
\begin{enumerate}
\item[(i)] Under the same assumptions as in Theorem~\ref{Thm:CechApproximation}(i), $R_{(u,\infty)}\circ F^{SO}(W,\R^m,d^p,\gamma)$ and $R_{(u,\infty)} \circ F^{SR}(T,d^p,\tilde \gamma)$ are almost weakly $(J_{\epsilon+C},\J_{n+1}\circ J_{\epsilon+C})$-interleaved.

\item[(ii)] Under the same assumptions as in Theorem~\ref{Thm:CechApproximation}(ii), $R_{(u,\rho(M))}\circ F^{SO}(M,d^M,\gamma)$ and $R_{(u,\rho(M))} \circ F^{SR}(T,d^M,\tilde \gamma)$ are almost weakly $(J_{\epsilon+C},\J_{n+1}\circ J_{\epsilon+C})$-interleaved.
\end{enumerate}
\end{thm}

\begin{proof}
We present the proof of (i); the proof of (ii) is essentially the same as the proof of (i), using Theorem~\ref{Thm:CechApproximation}(ii) in place of Theorem~\ref{Thm:CechApproximation}(i).

By Theorem~\ref{Thm:CechApproximation}(i), for any $\delta>0$ $R_{(u,\infty)}\circ F^{SO}(W,\R^m,d^p,\gamma)$ and $R_{(u,\infty)} \circ F^{\SCe}(T,\R^m,d^p,\tilde\gamma)$ are $(\epsilon+C+\delta)$-interleaved.  By Example~\ref{Ex:Rips_and_Cech_Generalized_Interleaving}, $F^{\SCe}(T,\R^m,d^p,\tilde\gamma)$ and $F^{SR}(T,d^p,\tilde\gamma)$ are $(\id_{n+1},\J_{n+1})$-interleaved.  Thus by Lemmas~\ref{lem:InterleavingsUnderRestriction} and~\ref{Lem:InterleavingTriangleInequality}, $R_{(u,\infty)}\circ F^{SO}(W,\R^m,d^p,\gamma)$ and $R_{(u,\infty)} \circ F^{SR}(T,d^p,\tilde \gamma)$ are $(J_{\epsilon+C+\delta},\J_{n+1}\circ J_{\epsilon+C+\delta})$-interleaved.  Thus by Lemma~\ref{lem:LessThanInterleavings} they are also $(J_{2\delta} \circ J_{\epsilon+C},J_{2\delta} \circ \J_{n+1}\circ J_{\epsilon+C})$-interleaved.  The result follows.
\end{proof}

\begin{remark} Theorem~\ref{Thm:RipsApproximation} can be tightened by tightening the result of Theorem~\ref{Thm:CechApproximation} on which it depends; see Remark~\ref{Rem:Strengthening}.
\end{remark}

\begin{cor}
For any $i\geq 0$,
\begin{enumerate}
\item[(i)] under the same assumptions as in Theorem~\ref{Thm:CechApproximation}(i),  $H_i \circ R_{(u,\infty)}\circ F^{SO}(W,\R^m,d^p,\gamma)$ and $H_i \circ R_{(u,\infty)} \circ F^{SR}(T,d^p,\tilde \gamma)$ are almost $(J_{\epsilon+C},\J_{n+1}\circ J_{\epsilon+C})$-interleaved.

\item[(ii)] under the same assumptions as in Theorem~\ref{Thm:CechApproximation}(ii), $H_i \circ R_{(u,\rho(M))}\circ F^{SO}(M,d^M,\gamma)$ and $H_i \circ R_{(u,\rho(M))} \circ F^{\SCe}(T,d^M,\tilde \gamma)$ are almost $(J_{\epsilon+C},\J_{n+1}\circ J_{\epsilon+C})$-interleaved.
\end{enumerate}
\end{cor}

\subsection{Approximating Multidimensional Sublevelset Persistence via Discrete Filtrations with Fixed Scale Parameter}
We now observe that the deterministic result \cite[Theorem 3.7]{chazal2009analysis}, on the approximation of sublevelset persistence of $\R$-valued functions via filtered Rips complexes with fixed scale parameter, admits a straightforward generalization $\R^n$-valued functions, $n\geq 1$.  This generalization, Theorem~\ref{Thm:RipsFixedScaleApproximation}, is proven in the same way as  \cite[Theorem 3.7]{chazal2009analysis}, using the interleaving distance in place of the bottleneck distance.  

Unlike the results of the previous sections, this result gives an approximation only on the level of persistent homology and not one on the level of filtrations.  However, we prove a variant of the result, Theorem~\ref{Thm:CechFixedScaleApproximation}, formulated using filtered \Cech complexes with fixed scale parameter instead of filtered Rips complexes with fixed scale parameter, which does hold on the level of filtrations.  This latter result is new even for 1-D filtrations; we will use it in Section~\ref{Sec:CechFixedScaleInference} to present a result, Theorem~\ref{Thm:CechFixedScaleConsistency}, on the topological inference of the superlevelset filtration of a density function $\gamma$ via a filtered \Cech complex with fixed scale parameter built on i.i.d. samples of $\P_\gamma$.

Given a finite metric space $(X,d,f)\in C^{SR}$, and $\delta>0$, let $F^{SR}(X,d,f,\delta)$ denote the $n$-filtration obtained by fixing the last parameter in the $(n+1)$-filtration $F^{SR}(X,d,f)$ to be $\delta$.  For $u\in \R^n$, $i\in \Z_{\geq 0}$, and $\delta_2\geq\delta_1>0$, let \[H_i^u(X,d,f,\delta_1,\delta_2)\subset H_i \circ R_{u} \circ F^{SR}(X,d,f,\delta_2)\] be the image 
of the map $H_i(j)$, where \[j: R_u \circ F^{SR}(X,d,f,\delta_1)\hookrightarrow R_u \circ F^{SR}(X,d,f,\delta_2)\] is the inclusion.

\cite[Theorem 5.1]{chazal2009persistence} was stated for geodesic metrics on Riemannian manifolds, but the analogous result holds for $L^p$ metrics on $\R^m$, $1\leq p \leq \infty$.  Similarly, we state our generalization for geodesic metrics on Riemannian manifolds, but an analogue holds for $L^p$ metrics as well.

\begin{thm}\label{Thm:RipsFixedScaleApproximation}
Let $M$ be a Riemannian manifold and let $\gamma:M\to \R^n$ be a $c$-Lipschitz function (w.r.t $d^M$).  For $u\in \hat \R^n$, let $T$ be an $\epsilon$-sample of $\gamma_u$ (w.r.t. $d^M$).  Let $\tilde \gamma:T\to \R^n$ be such that $|\tilde \gamma(l)-\gamma(l)|\leq C$ for all $l\in T$.  If $\epsilon<\frac{\rho(M)}{4}$ then for any $i\in \Z_{\geq 0}$ and $\delta\in [2\epsilon, \frac{\rho(M)}{2}]$,
\[d_I(H_i^u(T,d^M,\tilde \gamma,\delta,2\delta),H_i \circ R_u \circ F^S(M,\gamma))\leq 2c\delta+C.\]
\end{thm}

\begin{proof}
As noted above, the proof of \cite[Theorem 3.7]{chazal2009analysis} adapts directly.  The one difference is that \cite[Theorem 3.7]{chazal2009analysis} is formulated in terms of ``open" variants of the sublevelset filtration and the sublevelset-Rips filtration, analogous to the open variant of the sublevelset-offset filtration presented in Section~\ref{GeoFunctorsSection}.  Thus to adapt the proof of \cite[Theorem 3.7]{chazal2009analysis} to our setting we need to appeal to example~\ref{Ex:OpenAndClosedDeltaInterleaved}.

We note also that using our module-theoretic definition of interleavings, the proof of \cite[Theorem 3.7]{chazal2009analysis} may be written down in terms of module homomorphisms rather than in terms of maps of vector spaces, as in \cite{chazal2009analysis}.  This makes the proof somewhat less cumbersome.  
\end{proof}


Now we now formulate our variant of Theorem~\ref{Thm:RipsFixedScaleApproximation} for \Cech filtrations.  As for Theorem~\ref{Thm:RipsFixedScaleApproximation} we state the result for geodesic metrics on Riemannian manifolds, but an analogue holds for the $L^p$ metric as well.

For $(X,Y,d,f)\in C^{SCe}$ and $\delta>0$, let $F^{\SCe}(X,Y,d,f,\delta)$ denote the $n$-filtration obtained by fixing the last parameter in the $(n+1)$-filtration $F^{\SCe}(X,Y,d,f)$ to be $\delta$.

\begin{thm}\label{Thm:CechFixedScaleApproximation}
Let $M$ be a compact Riemannian manifold and $\gamma:M\to \R^n$ be a $c$-Lipschitz function (w.r.t $d^M$).  For $u\in \hat \R^n$, let $T$ be an $\epsilon$-sample of $\gamma_u$ (w.r.t. $d^M$).  Let $\tilde \gamma:T\to \R^n$ be such that $|\tilde \gamma(l)-\gamma(l)|\leq C$ for all $l\in T$.  Then for any $\delta\geq\epsilon$, 
\[d_{WI}(R_u \circ F^{\SCe}(T,M,d^M,\tilde \gamma,\delta),R_u \circ F^S(M,\gamma))\leq c\delta+C.\]
\end{thm}

\begin{proof}
The proof is similar to the proof of Theorem~\ref{Thm:CechApproximation}.  We leave the details to the reader.  
\end{proof}

%% file: Part_II/T_Main_Inference_Results.tex
\section{Bounds on the Probability that an I.I.D. Sample of a Manifold is an $\epsilon$-sample of a Superlevelset}\label{Sec:CoverageBounds}

\subsubsection{Overview}

\cite[Section 4.2]{chazal2009persistence} presents a bound on the probability that an i.i.d. sample of a Riemannian manifold $M$ with density $\gamma$ is an $\epsilon$-sample (w.r.t $d^M$) of a superlevelset of $\gamma$.

Here we recall that bound and apply it to obtain a bound on the probability that an i.i.d. sample of a submanifold $M$ of $\R^m$ with density $\gamma$ is an $\epsilon$-sample (w.r.t $d^p$)  of a superlevelset of $\gamma$.

In the next section, we will use these bounds to prove our main inference results.

\subsubsection{Definitions and Results}

Let $M$ be a Riemannian manifold.  Let $A$ be a subset of $M$.  
For $r\in \R$, let \[\V(A,r)=\inf_{x\in A} \H(B_{d^M}(x,r)).\]  Let $\NNN(A,r)\in \NN\cup\{\infty\}$ be the {\it $r$-covering number} of $A$---that is, $\NNN(A,r)$ is the minimum number of closed $d$-balls of the same radius $r$ needed to cover $A$ (the balls do not need to be centered in $A$).

The following lemma is \cite[Lemma 4.3]{chazal2009persistence}.

\begin{lem}\label{Lem:CoveringProbability}
Let $M$ be a Riemannian manifold, $\gamma:M\to \R$ be a $c$-Lipschitz probability density function (w.r.t $d^M$), and $c'=\max(1,c)$.  Let $T_z$ be an i.i.d. sample of ${\mathcal P}_\gamma$ of size $z$.  Then for any $\epsilon>0$ and $\alpha>\frac{c\epsilon}{c'}$, $T_z$ is an $\frac{\epsilon}{c'}$-sample of $-\gamma_{-\alpha}$ (w.r.t. to $d^M$) with probability at least $1-\NNN(-\gamma_{-\alpha},\frac{\epsilon}{2c'}) e^{-z(\alpha-\frac{c\epsilon}{c'})\V(-\gamma_{-\alpha},\frac{\epsilon}{2c'})}.$
\end{lem}

\begin{remark}\label{Rem:GeodesicCoverageRemark}

As noted in \cite{chazal2009persistence}, Lemma~\ref{Lem:CoveringProbability} tells us that if $\NNN(-\gamma_{-\alpha},\frac{\epsilon}{2c'})$ is positive and $\V(-\gamma_{-\alpha},\frac{\epsilon}{2c'})$ is finite for every $\epsilon>0$ then the Hausdorff distance between $T_z\cap -\gamma_{-\alpha}$ and $-\gamma_{-\alpha}$ converges in probability to $0$ as $z\to \infty$.
\end{remark}

\begin{remark}\label{Rem:VandNareWellBehaved}
It is observed in \cite{chazal2009persistence} that $\NNN(-\gamma_{-\alpha},\epsilon)$ is positive and $\V(-\gamma_{-\alpha},\epsilon)$ is finite for any $\alpha>0$ and $\epsilon>0$, provided $M$ is compact or, more generally, the sectional curvature of $M$ is bounded above and below.
\end{remark}

Now we extend Lemma~\ref{Lem:CoveringProbability} to the case of $L^p$ metrics on submanifolds of $\R^m$.  For $1\leq p\leq \infty$, let 

\begin{equation*}
K(p)=
\begin{cases} 
\sqrt m &\text{if }p\in [1,2),
\\
1 &\text{if }p\in [2,\infty].
\end{cases}
\end{equation*}

\begin{lem}\label{Lem:LPCoveringProbability}
Let $M$ be a submanifold of $\R^m$, $\gamma:M\to \R$ be a $c$-Lipschitz probability density function (w.r.t $d^p$), let $c'=\max(1,c)$, and let $T_z$ be an i.i.d. sample of ${\mathcal P}_\gamma$ of size $z$.  Then for any $\epsilon>0$ and $\alpha>\frac{c\epsilon}{c'}$, $T_z$ is an $\frac{\epsilon}{c'}$-sample of $-\gamma_{-\alpha}$ (w.r.t. $d^p$) with probability at least $1-\NNN(-\gamma_{-\alpha},\frac{\epsilon}{2c'K(p)}) e^{-z(\alpha-\frac{c\epsilon}{c'})\V(-\gamma_{-\alpha},\frac{\epsilon}{2c'K(p)})}$.
\end{lem}

\begin{proof}
For any $y_1,y_2\in M$, 
\begin{equation}\label{eq:DistanceInequalities} d^p(y_1,y_2)\leq K(p)\ d^2(y_1,y_2)\leq K(p)\ d^M(y_1,y_2).
\end{equation}
Thus a $c$-Lipschitz function w.r.t. $d^p$ is a $cK(p)$-Lipschitz density function (w.r.t $d^M$).

Let $c''=\max(1,cK(p))$.  Lemma~\ref{Lem:CoveringProbability} tells us that for $\gamma:M\to \R$ a $cK(p)$-Lipschitz probability density function (w.r.t $d^M$), $T_z$ an i.i.d. sample of ${\mathcal P}_\gamma$ of size $z$, $\epsilon'>0$, and $\alpha>\frac{cK(p)\epsilon'}{c''}$, $T_z$ is an $\frac{\epsilon'}{c''}$-sample of $-\gamma_{-\alpha}$ (w.r.t. $d^M$) with probability at least $1-\NNN(-\gamma_{-\alpha},\frac{\epsilon'}{2c''}) e^{-z(\alpha-\frac{cK(p)\epsilon'}{c''})\V(-\gamma_{-\alpha},\frac{\epsilon'}{2c''})}$.

Now substituting $\epsilon'=\frac{c''\epsilon}{c'K(p)}$ in the above statement, we obtain for any $\epsilon>0$, $\alpha>\frac{c\epsilon}{c'}$, $T_z$ is an $\frac{\epsilon}{c'K(p)}$-sample of $M$ (w.r.t. to $d^M$) with probability at least $1-\NNN(-\gamma_{-\alpha},\frac{\epsilon}{2c'K(p)}) e^{-z(\alpha-\frac{c\epsilon}{c'})\V(-\gamma_{-\alpha},\frac{\epsilon}{2c'K(p)})}$.  Thus by (\ref{eq:DistanceInequalities}), with at least the same probability $T_z$ is an $\frac{\epsilon}{c'}$-sample of $M$ (w.r.t. $d^p$), as we wanted to show. 
\end{proof}

Note that remark~\ref{Rem:GeodesicCoverageRemark} adapts immediately to the setting of Lemma~\ref{Lem:LPCoveringProbability}.

\section{Inference of the Superlevelset-offset Bifiltration of a Density Function via Discrete Bifiltrations Built on I.I.D. Samples}\label{Sec:InferenceResults}

We're now ready to present our results on the inference of the superlevelset-offset filtration of a density function $\gamma$ on a Riemannian manifold from an i.i.d. sample of $\P_\gamma$.  These are the culmination of the theory we have developed so far in this thesis.

We first present results for inference using superlevelset-\Cech filtrations.  Theorem~\ref{Thm:CechInference} shows that with high probability, superlevelset-\Cech filtrations built on sufficiently large i.i.d. samples of $\P_\gamma$ (and filtered by the superlevelsets of a well behaved density estimator) give good approximations to the superlevelset-offset filtration of $\gamma$, with respect to $d_{WI}$.  Theorem~\ref{Thm:CechConsistency}, an asymptotic form of this result, then tells us that such superlevelset-\Cech filtrations are {\it consistent} estimators (w.r.t. $d_{WI}$) of the superlevelset-offset filtration of $\gamma$.  

After presenting our results for superlevelset-\Cech filtraitons, we present analogues of each these results for superlevelset-Rips filtrations.  The main result is Theorem~\ref{Thm:RipsAsymptotics}, which shows that in the large sample limit, the superlevelset-Rips filtration built on an i.i.d. sample of 
$\P_\gamma$ (and filtered by the superlevelsets of a well behaved density estimator) and the superlevelset-offset filtration of $\gamma$ are weakly $(\J_{n+1},\id_{n+1})$-interleaved.  Since as observed in Example~\ref{Ex:Rips_and_Cech_Generalized_Interleaving}, superlevelset-Rips and superlevelset-\Cech filtration functors built on the same input data are always strongly  $(\J_{n+1},\id_{n+1})$-interleaved, Theorem~\ref{Thm:RipsAsymptotics} says that in the large sample limit, the superlevelset-Rips and superlevelset-Offset filtrations satisfy in the weak sense the same interleaving relationship that superlevelset-Rips and superlevelset-\Cech filtrations always satisfy in the strong sense.


\subsection{Inference Results for Superlevelset-\Cech Bifiltrations}\label{Sec:CechConsistency}

\begin{thm}[Inference of superlevelset-offset filtrations via superlevelset-\Cech filtrations (finite sample case)]\label{Thm:CechInference}
Let $M$ be a Riemannian manifold, $\gamma:M\to \R$ be a density function, $T_z$ be a finite i.i.d. random sample of ${\mathcal P}_\gamma$ of size $z$, and $E_z$ be a {$(q,C)$-density estimator} of another density function $\tilde\gamma: M\to \R$ for some $C>0$ and $q\in [0,1]$.  

\begin{enumerate}
\item[(i)] If $M$ is an embedded submanifold of $\R^m$, $\gamma$ is $c$-Lipschitz w.r.t. $d^p$ for some $c>0$, and $c'=\max(1,c)$ then for any $\epsilon>0$ and $\alpha>\frac{c\epsilon}{c'}$, \[d_{WI}(R_{(-\alpha,\infty)}\circ F^{\SCe}(T_z,\R^m,d^p,-E_z(T_z)),R_{(-\alpha,\infty)} \circ F^{SO}(M,\R^m,d^p,-\tilde \gamma))\leq \epsilon+C\] with probability at least $1-\NNN(-\gamma_{-\alpha},\frac{\epsilon}{2c'K(p)}) e^{-z(\alpha-c\epsilon/c')\V(-\gamma_{-\alpha},\frac{\epsilon}{2c'K(p)})}-q$.
\item[(ii)] If $\gamma$ is $c$-Lipschitz w.r.t. $d^M$ for some $c>0$, and $c'=\max(1,c)$, then for any $\epsilon>0$ and $\alpha>\frac{c\epsilon}{c'}$, \[d_{WI}(R_{(-\alpha,\rho(M))}\circ F^{\SCe}(T_z,M,d^M,-E_z(T_z)),R_{(-\alpha,\rho(M))} \circ F^{SO}(M,d^M,-\tilde \gamma))\leq \epsilon+C\] with probability at least $1-\NNN(-\gamma_{-\alpha},\frac{\epsilon}{2c'}) e^{-z(\alpha-c\epsilon/c')\V(-\gamma_{-\alpha},\frac{\epsilon}{2c'})}-q$.
\end{enumerate}
\end{thm}

\begin{proof}
To prove (i), first note that in the event that 
\[\{\sup_{l\in T_z} |\tilde\gamma(l)-E_z(T_z)(l)|\leq C\} \cap \{T_z\textup{ is an }\frac{\epsilon}{c'}\textup{-sample of }M\textup{ (w.r.t. }d^p)\},\] Theorem~\ref{Thm:CechApproximation}(i) applies to give that \[d_{WI}(R_{(-\alpha,\infty)} \circ F^{\SCe}(T_z,\R,d^p,-E_z(T_z)),R_{(-\alpha,\infty)} \circ F^{SO}(M,\R^m,d^p,-\tilde\gamma))\leq \epsilon+C.\] 

By Lemma~\ref{Lem:LPCoveringProbability}, the definition of a $(q,C)$-density estimator, and the union bound, the probability of this event is at least $1-\NNN(-\gamma_{-\alpha},\frac{\epsilon}{2c'})e^{-z(\alpha-c\epsilon/c')\V(-\gamma_{-\alpha},\frac{\epsilon}{2c'})}-q$.  This gives (i). 

The proof of (ii) is the same as that of (i), using Lemma~\ref{Lem:CoveringProbability} in place of Lemma~\ref{Lem:LPCoveringProbability} and Theorem~\ref{Thm:CechApproximation}(ii) in place of Theorem~\ref{Thm:CechApproximation}(i). 
\end{proof}

\begin{thm}[Consistent estimation of superlevelset-offset filtration via superlevelset-\Cech filtrations]\label{Thm:CechConsistency} 
Let $M$ be Riemannian manifold with sectional curvature bounded above and below, let $\gamma:M\to \R$ be a density function, and let $T_z$ be an i.i.d. sample of ${\mathcal P}_\gamma$ of size $z$.
\begin{enumerate*}
\item[(i)]If $M$ is a submanifold of $\R^m$, $\gamma$ is c-Lipchitz (w.r.t. $d^p$) for some $c>0$ and $E$ is a density estimator on $M$ such that 
$(E,\gamma)$ satisfies $A1$ then
\[d_{WI}(R_{(0,\infty)} \circ F^{\SCe}(T_z,\R^m,d^p,-E_z(T_z)),R_{(0,\infty)} \circ F^{SO}(M,\R^m,d^p,-\gamma))\xrightarrow{P}0.\]  
\item[(ii)]If $M=\R^m$, $\gamma$ is c-Lipchitz (w.r.t. $d^p$) for some $c>0$ and $E$ is a density estimator on $\R^m$ such that $(E,\gamma)$ satisfies $A2$ for some kernel $K$ then 
\[d_{WI}(R_{(0,\infty)} \circ F^{\SCe}(T_z,\R^m,d^p,-E_z(T_z)),R_{(0,\infty)} \circ F^{SO}(\R^m,d^p,-\gamma*K))\xrightarrow{P}0.\]  
\item[(iii)]If $\gamma:M\to \R$ is $c$-Lipschitz (w.r.t. $d^M$) for some $c>0$ and $E$ is a density estimator such that $(E,\gamma)$ satisfies $A1$ then 
\[d_{WI}(R_{(0,\rho(M))} \circ F^{\SCe}(T_z,M,d^M,-E_z(T_z)),R_{(0,\rho(M))} \circ F^{SO}(M,d^M,-\gamma))\xrightarrow{P}0.\]
\end{enumerate*}
\end{thm}

\begin{proof} 
To prove (i), we need to show that for any $\epsilon>0$ and $q\in (0,1]$, there exists some $z_1\in \NN$ such that for all $z\geq z_1$, 
\[d_{WI}(R_{(0,\infty)} \circ F^{\SCe}(T_z,\R^m,d^p,-E_z(T_z)),R_{(0,\infty)} \circ F^{SO}(M,\R^m,d^p,-\gamma))\leq\epsilon\] with probability at least $1-q$.
Since $(E,\gamma)$ satisfies $A1$, Lemma~\ref{BoundsForEstimatorsSatisfyingA1orA2}(i) tells us that for any $q\in (0,1]$ and $\epsilon>0$, there exists $z_0\in\NN$ such that for all $z\geq z_0$, $E_z$ is a $(\frac{q}{2},\frac{\epsilon}{4})$-density estimator of $\gamma$.  

Choose $\epsilon'\in (0, \frac{\epsilon}{4})$.  Then $\frac{\epsilon}{2}>\frac{c\epsilon'}{c'}$.  Since $M$ has bounded absolute sectional curvature, it follows from Remark~\ref{Rem:VandNareWellBehaved} that there exists some $z_1\in \NN$ with $z_1>z_0$ such that for all $z\geq z_1$, $\NNN(-\gamma_{-\frac{\epsilon}{2}},\frac{\epsilon'}{2c'})e^{-z(\frac{\epsilon}{2}-c\epsilon'/c')\V(-\gamma_{-\frac{\epsilon}{2}},\frac{\epsilon'}{2c'})}<\frac{q}{2}$.

Invoking Theorem~\ref{Thm:CechInference}(i) with $\alpha=\frac{\epsilon}{2}$, for all $z\geq z_1$ \[d_{WI}(R_{(-\frac{\epsilon}{2},\infty)} \circ F^{\SCe}(T_z,,\R^m,d^p,-E_z(T_z)),R_{(-\frac{\epsilon}{2},\infty)} \circ F^{SO}(M,\R^m,d^p,-\gamma))\leq \epsilon'+\frac{\epsilon}{4}\leq \frac{\epsilon}{2}\] with probability at least $1-q$.

Then by Lemma~\ref{lem:InterleavingsUnderExpansion}, for all $z>z_1$, \[d_{WI}(R_{(0,\infty)} \circ F^{\SCe}(T_z,\R^m,d^p,-E_z(T_z)),R_{(0,\infty)} \circ F^{SO}(M,\R^m,d^p,-\gamma))\leq\frac{\epsilon}{2}+\frac{\epsilon}{2}=\epsilon\] with probability at least $1-q$.  This completes the proof of (i).

The proof of (ii) is essentially the same as that of statement (i), using Lemma~\ref{BoundsForEstimatorsSatisfyingA1orA2}(ii) in place of Lemma~\ref{BoundsForEstimatorsSatisfyingA1orA2}(i).  The proof of (iii) is also essentially the same as that of statement (i), using Theorem~\ref{Thm:CechInference}(ii) in place of Theorem~\ref{Thm:CechInference}(i).
\end{proof}

\begin{cor}\label{Cor:CechHomologyConsistency}
For any $i\in \Z_{\geq 0}$,
\begin{enumerate}
\item[(i)]Under the same assumptions as in the statement of Theorem~\ref{Thm:CechConsistency}(i), 
\[d_I(H_i \circ R_{(0,\infty)} \circ F^{\SCe}(T_z,\R^m,d^p,-E_z(T_z)),H_i \circ R_{(0,\infty)} \circ F^{SO}(M,\R^m,d^p,-\gamma))\xrightarrow{P}0.\]

\item[(ii)]Under the same assumptions as in the statement of Theorem~\ref{Thm:CechConsistency}(ii), 
\[d_I(H_i \circ R_{(0,\infty)} \circ F^{\SCe}(T_z,\R^m,d^p,-E_z(T_z)),H_i \circ R_{(0,\infty)} \circ F^{SO}(\R^m,d^p,-\gamma*K))\xrightarrow{P}0.\]

\item[(iii)]Under the same assumptions as in the statement of Theorem~\ref{Thm:CechConsistency}(iii), 
\[d_I(H_i \circ R_{(0,\rho(M))} \circ F^{\SCe}(T_z,M,d^M,-E_z(T_z)),H_i \circ R_{(0,\rho(M))} \circ F^{SO}(M,\R^m,d^M,-\gamma))\xrightarrow{P}0.\]
\end{enumerate}
\end{cor}

\begin{proof}
This follows immediately from Theorems~\ref{Thm:CechConsistency} and~\ref{Thm:StabilityWRTInterleavings}.
\end{proof}

\subsection{Inference Results for Superlevelset-Rips Bifiltrations}\label{Sec:RipsAsymptotics}
%
%
%

\subsubsection{Convergence in Probability up to Interleavings}
To describe our asymptotic results for superlevelset-Rips filtrations it will be convenient for us to introduce some terminology.
Let $J_1,J_2:\R^n\to \R^n$ be increasing maps and let $u\in \hat\R^n$.  We say that a sequence $\{X_z\}_{z\in \NN}$ of random $u$-filtrations converges in probability to a random $u$-filtration $X$ up to weak $(J_1,J_2)$-interleaving, and we write \[X_z\xrightarrow{P,J_1,J_2} X\] if for all $\epsilon>0$ 
\[\lim_{z\to \infty} P(X_z\textup{ and }X\textup{ are weakly }(J_\epsilon\circ J_1,J_\epsilon\circ J_2)\textup{-interleaved})=1.\]
Note that $X_n\xrightarrow{P,\id_n,\id_n} X$ if and only if $d_{WI}(X_n,X)\xrightarrow{P} 0$.

Similarly, we say that a sequence $\{M_z\}_{z\in \NN}$ of random $B_n$-persistence modules converges in probability to a random $B_n$-persistence module $M$ up to weak $(J_1,J_2)$-interleaving, and we write \[M_z\xrightarrow{P,J_1,J_2} M\] if for all $\epsilon>0$ 
\[\lim_{z\to \infty} P(M_z\textup{ and }M\textup{ are weakly }(J_\epsilon\circ J_1,J_\epsilon\circ J_2)\textup{-interleaved})=1.\]
$M_n\xrightarrow{P,\id_n,\id_n} M$ if and only if $d_{I}(M_z,M)\xrightarrow{P} 0$.

\begin{thm}[Asymptotics of estimation of superlevelset-offset filtration via superlevelset-Rips filtrations]\label{Thm:RipsAsymptotics}
\mbox{}
\begin{enumerate*}
\item[(i)]Under the same assumptions as in Theorem~\ref{Thm:CechConsistency}(i), 
\[R_{(0,\infty)} \circ F^{SR}(T_z,d^p,-E_z(T_z))\xrightarrow{P,\J_{n+1},\id_{n+1}} R_{(0,\infty)} \circ F^{SO}(M,\R^m,d^p,-\gamma).\]  
\item[(ii)]Under the same assumptions as in Theorem~\ref{Thm:CechConsistency}(ii), 
\[R_{(0,\infty)} \circ F^{SR}(T_z,d^p,-E_z(T_z))\xrightarrow{P,\J_{n+1},\id_{n+1}} R_{(0,\infty)} \circ F^{SO}(\R^m,d^p,-\gamma*K).\]  
\item[(iii)]Under the same assumptions as in Theorem~\ref{Thm:CechConsistency}(iii), 
\[R_{(0,\rho(M))} \circ F^{SR}(T_z,d^M,-E_z(T_z))\xrightarrow{P,\J_{n+1},\id_{n+1}} R_{(0,\rho(M))} \circ F^{SO}(M,d^M,-\gamma).\]
\end{enumerate*}
\end{thm}

\begin{proof} 
To prove (i), we need to show that under the same assumptions as in Theorem~\ref{Thm:CechConsistency}(i), for any $\epsilon>0$ and $q\in (0,1]$, there exists some $z'\in \NN$ such that for all $z\geq z'$, 
$R_{(0,\infty)} \circ F^{SR}(T_z,d^p,-E_z(T_z))$ and $R_{(0,\infty)} \circ F^{SO}(M,\R^m,d^p,-\gamma)$ are $(J_\epsilon \circ \J_{n+1},J_\epsilon \circ \id_{n+1})$-interleaved with probability at least $1-q$.

By Example~\ref{Ex:Rips_and_Cech_Generalized_Interleaving},  $F^{SR}(T,d^p,\tilde\gamma)$ and $F^{\SCe}(T,\R^m,d^p,\tilde\gamma)$ are $(\J_{n+1},\id_{n+1})$-interleaved.

Theorem~\ref{Thm:CechConsistency} tells us that there exists some $z''>0$ such that for all $z>z''$, 
$R_{(0,\infty)} \circ F^{\SCe}(T_z,\R^m,d^p,-E_z(T_z))$ and $R_{(0,\infty)} \circ F^{SO}(M,\R^m,d^p,-\gamma)$ are $\epsilon$-interleaved with probability at least $1-q$.  

By Lemma~\ref{Lem:InterleavingTriangleInequality} then, for all $z>z''$, $R_{(0,\infty)} \circ F^{SR}(T_z,d^p,-E_z(T_z))$ and $R_{(0,\infty)} \circ F^{SO}(M,\R^m,d^p,-\gamma)$ are
 $(J_\epsilon \circ \J_{n+1},J_\epsilon)$-interleaved with probability at least $1-q$.  Thus taking $z'=z''$ gives the result.  

The proofs of (ii) and (iii) are essentially the same as that of (i), using Theorems~\ref{Thm:CechConsistency}(ii) and~\ref{Thm:CechConsistency}(iii) in place of Theorem~\ref{Thm:CechConsistency}(i).
\end{proof}

\begin{cor}\label{Cor:RipsHomologyAsymptotics}\mbox{}
For any $i\in \Z_{\geq 0}$,
\begin{enumerate}
\item[(i)]Under the same assumptions as in the statement of Theorem~\ref{Thm:CechConsistency}(i),  
\[H_i \circ R_{(0,\infty)} \circ F^{SR}(T_z,d^p,-E_z(T_z))\xrightarrow{P,\J_{n+1},\id_{n+1}}  H_i \circ R_{(0,\infty)} \circ F^{SO}(M,\R^m,d^p,-\gamma).\]

\item[(ii)]Under the same assumptions as in the statement of Theorem~\ref{Thm:CechConsistency}(ii), 
\[H_i \circ R_{(0,\infty)} \circ F^{SR}(T_z,d^p,-E_z(T_z))\xrightarrow{P,\J_{n+1},\id_{n+1}} H_i \circ R_{(0,\infty)} \circ F^{SO}(\R^m,d^p,-\gamma*K).\]

\item[(iii)]Under the same assumptions as in the statement of Theorem~\ref{Thm:CechConsistency}(iii), 
\[H_i \circ R_{(0,\rho(M))} \circ F^{SR}(T_z,d^M,-E_z(T_z))\xrightarrow{P,\J_{n+1},\id_{n+1}} R_{(0,\rho(M))}\circ F^{SO}(M,d^M,-\gamma).\]
\end{enumerate}
\end{cor}

\begin{proof}
This follows immediately from Theorem~\ref{Thm:RipsAsymptotics} and Theorem~\ref{Thm:StabilityWRTInterleavings}.
\end{proof}

\section{Inference of the Superlevelset Filtration of a Probability Density via Filtered \Cech Complexes}\label{Sec:CechFixedScaleInference}
In this section we apply Theorem~\ref{Thm:CechFixedScaleApproximation} to obtain a version of the inference theorem of \cite[Theorem 5.1]{chazal2009persistence} which holds on the level of filtrations rather than only on the level of persistent homology modules.  Unlike \cite[Theorem 5.1]{chazal2009persistence}, this result holds for only estimators defined using \Cech filtrations, not for estimators defined using Rips filtrations; the result of \cite[Theorem 5.1]{chazal2009persistence} 
for estimators defined using Rips filtrations does not in general lift to the level of filtrations.

We state the result for geodesic metrics on Riemannian manifolds; an analogous result holds for $L^p$-metrics on submanifolds of $\R^m$.

\begin{thm}[Inference of superlevelset filtrations]\label{Thm:CechFixedScaleInference}
Let $M$ be a Riemannian manifold, let $\gamma:M\to \R$ be a $c$-Lipschitz density function w.r.t. $d^M$ for some $c>0$.  Let $T_z$ be a finite i.i.d. random sample of ${\mathcal P}_\gamma$ of size $z$. Let $E_z$ be a {$(q,C)$-density estimator} of another density function $\tilde\gamma: \R^m\to \R$ for some $C>0$ and $q\in [0,1]$.  Then for any $\epsilon>0$, $\delta\geq\epsilon$, and $\alpha>c\epsilon$.  
\[d_{WI}(R_{(-\alpha,\rho(M))}\circ F^{\SCe}(T_z,M,d^M,-E_z(T_z),\delta),R_{(-\alpha,\rho(M))} \circ F^{S}(M,-\tilde \gamma))\leq c\delta+C\] with probability at least $1-\NNN(-\gamma_{-\alpha},\frac{\epsilon}{2}) e^{-z(\alpha-c\epsilon)\V(-\gamma_{-\alpha},\frac{\epsilon}{2})}-q$.
\end{thm}

\begin{proof}
The proof is the essentially same as the proof of Theorem~\ref{Thm:CechInference}, using Theorem~\ref{Thm:CechFixedScaleApproximation} in place of Theorem~\ref{Thm:CechApproximation}.
\end{proof}

\begin{thm}[convergence in probability of superlevelset-\Cech filtrations with fixed scale parameter to superlevelset filtration of density function]\label{Thm:CechFixedScaleConsistency} 
Let $M$ be a Riemannian manifold with bounded absolute sectional curvature, let $\gamma:M\to \R$ be a $c$-Lipschitz density function (w.r.t. $d^M$) for some $c>0$, and let $T_z$ be an i.i.d. sample of ${\mathcal P}_\gamma$ of size $z$.
\begin{enumerate*}
\item[(i)]If $M=\R^m$ and $E$ is a density estimator on $\R^m$ such that $(E,\gamma)$ satisfies $A1$ then there exists a sequence of positive real numbers $\{\delta_z\}_{z\in \NN}$ such that 
\[d_{WI}(R_{(0,\infty)} \circ F^{\SCe}(T_z,\R^m,d^p,-E_z(T_z),\delta_z),R_{(0,\infty)} \circ F^{S}(\R^m,-\gamma))\xrightarrow{P}0.\]  

\item[(ii)]If $M=\R^m$ and $E$ is a density estimator on $\R^m$ such that $(E,\gamma)$ satisfies $A2$ for some kernel $K$ then there exists a sequence  of positive real numbers $\{\delta_z\}_{z\in \NN}$ such that 
\[d_{WI}(R_{(0,\infty)} \circ F^{\SCe}(T_z,\R^m,d^p,-E_z(T_z),\delta_z),R_{(0,\infty)} \circ F^{S}(\R^m,-\gamma*K))\xrightarrow{P}0.\]  

\item[(iii)]If $E$ is a density estimator such that $(E,\gamma)$ satisfies $A1$ then there exists a sequence  of positive real numbers $\{\delta_z\}_{z\in \NN}$ such that 
\[d_{WI}(R_{(0,\rho(M))} \circ F^{\SCe}(T_z,M,d^M,-E_z(T_z),\delta_z),R_{(0,\rho(M))} \circ F^{S}(M,-\gamma))\xrightarrow{P}0.\]  

\end{enumerate*}
\end{thm}

\begin{proof}
We'll prove (iii).  The proofs of (i) and (ii) are essentially the same, using Lemma~\ref{BoundsForEstimatorsSatisfyingA1orA2}(ii) in place of Lemma~\ref{BoundsForEstimatorsSatisfyingA1orA2}(i) and an analogue of Theorem~\ref{Thm:CechFixedScaleInference} for $L^p$ metrics on $\R^m$.

To prove (i), we need to find a sequence of positive real numbers $\{\delta_z\}_{z\in \NN}$ such that for any $\epsilon>0$ and $p\in (0,1]$, there exists some $z_1\in \NN$ such that for all $z\geq z_1$, 
\[d_{WI}(R_{(0,\rho(M))} \circ F^{\SCe}(T_z,M,d^M,-E_z(T_z),\delta_z),R_{(0,\rho(M))} \circ F^{S}(M,-\gamma))\leq\epsilon\] with probability at least $1-p$.
To construct our sequence $\{\delta_z\}_{z\in \NN}$, we first choose a monotonically decreasing sequence $\xi:\NN\to (0,1]$ such that $\lim_{z\to \infty} \xi(z)=0$.
Since $(E,\gamma)$ satisfies $A1$, Lemma~\ref{BoundsForEstimatorsSatisfyingA1orA2}(i) tells us that there exists a strictly increasing sequence of natural numbers $j:\NN\to \NN$ such that for any $y\in \NN$ and any $z\geq j(y)$, $E_z$ is a $(\frac{\xi(y)}{2},\frac{\xi(y)}{4})$-density estimator of $\gamma$.

Choose a second monotonically decreasing sequence $\xi':\NN\to (0,1]$ such that for each $z\in \NN$, $\xi'(z)<\frac{\xi(z)}{4})$.  Let $c'=\max(1,c)$.
For each $z\in \NN$, $\frac{\xi(z)}{2}>\frac{c\xi'(z)}{c'}$.  Since $M$ has bounded absolute sectional curvature, it follows from Remark~\ref{Rem:VandNareWellBehaved} that there exists a strictly increasing sequence of natural numbers $l:\NN\to \NN$ with $l(z)\geq j(z)$ for all $z\in \NN$ such that for $y\in \NN$ and all $z\geq l(y)$, \[\NNN(-\gamma_{-\frac{\xi(y)}{2}},\frac{\xi'(y)}{2c'})e^{-z(\frac{\xi(y)}{2}-c\xi'(y)/c')\V(-\gamma_{-\frac{\xi(y)}{2}},\frac{\xi'(y)}{2c'})}<\frac{\xi(y)}{2}.\]

Invoking Theorem~\ref{Thm:CechFixedScaleInference}, with the variables $(\alpha,\epsilon,\delta)$ in the statement of that Theorem set equal to $(\frac{\xi(y)}{2},\frac{\xi'(y)}{c'},\frac{\xi'(y)}{c'})$, we obtain the following result, which we state as a lemma.
\begin{lem}\label{lem:LittleSubLemmaForConsistency}
For all $y\in \NN$, $\xi(y)$, $\xi'(y)$, $l(y)$ defined as above, and all $z\geq l(y)$ 
\begin{align*}
&d_{WI}(R_{(-\frac{\xi(y)}{2},\rho(M))} \circ F^{\SCe}(T_z,M,d^M,-E_z(T_z),\xi'(y)),R_{(-\frac{\xi(y)}{2},\rho(M))} \circ F^{S}(M,-\gamma))\\
&\quad \leq \xi'(y)+\frac{\xi(y)}{4}\leq \frac{\xi(y)}{2}
\end{align*}
with probability at least $1-\xi(y)$.
\end{lem}

Now define $l^{-1}:\NN\to \NN$ by 
\begin{equation*}
l^{-1}(z)=
\begin{cases}
\max \{z'\in \NN |l(z')\leq z\} &\text{if }l(1)\leq z,
\\
1&\text{otherwise.}  
\end{cases}
\end{equation*}
Define $\{\delta_z\}_{z\in \NN}$ by taking $\delta_z=\xi'(l^{-1}(z))$ for all $z\in \NN$.

Now choose $\epsilon>0$ and $p\in (0,1]$.  There's some $y\in \NN$ such that $\xi(y)<\min(\epsilon,p)$.  We claim that for all $z\geq l(y)$, 
\[d_{WI}(R_{(-\frac{\xi(y)}{2},\rho(M))} \circ F^{\SCe}(T_z,M,d^M,-E_z(T_z),\delta_z),R_{(-\frac{\xi(y)}{2},\rho(M))} \circ F^{SO}(M,-\gamma))\leq \frac{\xi(y)}{2}\] with probability at least $1-\xi(y)$.

Then, given the claim, we set $z_1=l(y)$.  By Lemma~\ref{lem:InterleavingsUnderExpansion}, for all $z\geq z_1$, 
\begin{align*}
&d_{WI}(R_{(0,\rho(M))} \circ F^{\SCe}(T_z,M,d^M,-E_z(T_z),\delta_y),R_{(0,\rho(M))} \circ F^{S}(M,-\gamma))\\
&\quad \leq\frac{\xi(y)}{2}+\frac{\xi(y)}{2}=\xi(y)\leq \epsilon
\end{align*}
with probability at least $1-\xi(y)\geq 1-p$, which completes the proof of (i).

To prove the claim, we plug in $y=l^{-1}(z)$ into Lemma~\ref{lem:LittleSubLemmaForConsistency} to obtain that
\begin{align*}
&d_{WI}(R_{(-\frac{\xi(l^{-1}(z))}{2},\rho(M))} \circ F^{\SCe}(T_z,M,d^M,-E_z(T_z),\xi'(l^{-1}(z))),R_{(-\frac{\xi(l^{-1}(z))}{2},\rho(M))} \circ F^{S}(M,-\gamma))
\\&\leq \xi'(l^{-1}(z))+\frac{\xi(l^{-1}(z))}{4}\leq \frac{\xi(l^{-1}(z))}{2}
\end{align*}
 with probability at least $1-\xi(l^{-1}(z))$.

Since $l$ is strictly increasing, when $z\geq l(y)$, $l^{-1}(z)\geq l^{-1}(l(y))=y$.  Since the sequences $\xi$ and $\xi'$ are monotonically decreasing, we then have that $\xi(l^{-1}(z))\leq \xi(y)$ and $\xi'(l^{-1}(z))\leq \xi'(y)$.  Thus,
\begin{align*}
&d_{WI}(R_{(-\frac{\xi(l^{-1}(z))}{2},\rho(M))} \circ F^{\SCe}(T_z,M,d^M,-E_z(T_z),\xi'(l^{-1}(z))),R_{(-\frac{\xi(l^{-1}(z))}{2},\rho(M))} \circ F^{S}(M,-\gamma))\\
&\quad \leq \frac{\xi(y)}{2}
\end{align*}
with probability at least $1-\xi(y)$.  Applying Lemma~\ref{lem:InterleavingsUnderRestriction} and using the definition of $\delta_z$, we obtain that 
\begin{align*}
&d_{WI}(R_{(-\frac{\xi(y)}{2},\rho(M))} \circ F^{\SCe}(T_z,M,d^M,-E_z(T_z),\delta_z),R_{(-\frac{\xi(y)}{2},\rho(M))} \circ F^{SO}(M,-\gamma))\\
&\quad \leq \xi'(y)+\frac{\xi(y)}{4}\leq \frac{\xi(y)}{2}
\end{align*}
 with probability at least $1-\xi(y)$.  This proves the claim and completes the proof of (i).

\end{proof}

%% file: Part_II/T_Inference_Discussion.tex
\section{Future Work on Persistence-Based Topological Inference}\label{Sec:InferenceDiscussion}

We close this chapter with a discussion of directions for future work on the statistical aspects of persistence-based topological inference.

\subsubsection{Almost Sure Convergence Results}

Our asymptotic results are formulated in terms of convergence in probability and a variant of it which we have called {\it convergence in probability up to interleavings.}  It would be interesting to know whether our results can be strengthened to almost sure convergence results.

\subsubsection{Bootstrap Confidence Balls for Persistence}

The natural next step in the development of the theory of persistence-based topological inference is to develop a theory of confidence regions for the consistent estimators studied in this chapter.   This problem has hitherto not been addressed in any published work, for even for 1-D persistent homology. 

In the setting of Theorem~\ref{Thm:CechInference}, it seems that it should be possible to employ the bootstrap methodology \cite{efron1993introduction} to compute approximate confidence balls (in the pseudometric space of multidimensional filtrations given by the weak interleaving interleaving distance) for a superlevelset-\Cech estimator $X$ of the superlevelset-offset filtration $Y$ of a probability density function. 

For $a\in (0,1)$, to compute an approximate $a$-confidence ball for $X$ the simplest and most natural approach would be to
\begin{enumerate*}
\item Construct bootstrap replicates $X_1,...,X_l$ of the filtration $X$ by taking superlevelset-\Cech filtrations of bootstrap replicates of the point cloud data on which $X$ was built.  
\item Compute the smallest value of $y$ such that a fraction $a$ of the bootstrap replicates $X_1,...X_l$ lie within weak interleaving distance $y$ of $X$.  
\end{enumerate*}
We expect that, under mild hypotheses, the ball of radius $y$ centered at $X$ would then be an approximate $a$-confidence ball for $X$, and that an analogous result would hold on the level of persistent homology modules.  

A well developed theory along these lines of bootstrap confidence balls for persistent homology estimates, together with an efficient computational pipeline for computing the radii of confidence balls, could add in a significant way to power of persistent homology as a tool for understanding the qualitative structure of random data.

\subsubsection{Testing Persistence-Based Hypotheses}
Another important problem in topological inference is to develop a theory of hypothesis testing for hypotheses formulated using the language of persistence.  Given the close relationship between hypothesis testing and computation of confidence regions, this problem is much in the same spirit as the problem of computing approximate confidence balls for the estimators of this thesis.
 
The problem of testing hypotheses formulated using persistence has hitherto not yet been treated by the topological data analysis community, yet seems to be of basic importance, particularly in the case of $0^{th}$ persistent homology.  To explain, on the one hand $0^{th}$ persistent homology offers an elegant and user-friendly language for describing the multi-modality of functions in a way that is sensitive to the ``size" of the modes.  On the other hand, mode detection of densities and regression functions is an old and well studied problem in statistics, and one of fundamental interest---see, for example \cite{silverman1981using} and \cite[Chapter 20]{wasserman2004all}.  Classical statistics has developed approaches to testing for multi-modality based on smoothing ideas \cite{silverman1981using}, but as the classical language for describing multi-modality is impoverished relative to that offered by persistent homology, there seems little doubt that persistent homology would have something significant to offer to the problem of testing multi-modality hypotheses. 

\subsubsection{Inference of Persistent Homology at the Cycle Level}
It is possible to geometrically represent the persistent homology module of a filtration, at the chain level, by choosing a presentation for the persistent homology module and then choosing representative cycles and boundaries for the generators and relations in the presentation.  There are of course many different ways to make such choices, most of which are not geometrically interesting.  However, for filtrations whose topological spaces are equipped with a sufficient amount of geometric structure we can in principle define notions of optimality of chain-level representations of persistence modules in such a way that the optimal choice will be geometrically meaningful. 
Then, using such notions, we can pose persistence-based inference problems directly on the chain level.  The results on clustering of \cite{chazal2009persistence} are loosely in this spirit, for $0^{th}$ persistent homology.  It would be very interesting to understand if and how those results adapt to higher persistent homology and to the multidimensional setting.  

The problem of pursuing inference at the chain level is an extremely important one, at least to the extent that we are interested in the use of persistent homology for topological inference: A descriptor of a probability distribution defined only using the isomorphism class of a persistent homology module gives us only indirect information about the qualitative structure of the probability distribution; ultimately we want to understand how the algebraic features of the persistent homology descriptor correspond to specific geometric features of the probability distribution.  For that we need to consider representatives of those algebraic features on the chain level.

%% file: Part_II/T_Discussion_Final.tex
In this thesis, we have introduced and studied $(J_1,J_2)$-interleavings on multidimensional persistence modules and strong and weak $(J_1,J_2)$-interleavings on multidimensional filtrations.  We have undertaken a careful study of such interleavings and of the interleaving distances $d_I$, $d_{SI}$, and $d_{WI}$ defined in terms of them.  We have applied interleavings and interleaving distances to adapt the persistence-based topological inference result \cite[Theorem 5.1]{chazal2009persistence} to the multidimensional setting and directly to the level of filtrations.  The culmination of our efforts is a pair of results, Theorems~\ref{Thm:CechConsistency} and~\ref{Thm:RipsAsymptotics}, describing the topological asymptotics of random sublevelset-\Cech and sublevelset-Rips bifiltrations (filtered by the superlevelsets of a density estimator) in the large sample limit.  These theorems put on firm mathematical footing the idea, first put forth in~\cite{carlsson2009theory}, that such filtrations should encode topological information about the probability density function of a probability distribution generating the point cloud data on which the bifiltrations are built.

One of the central themes of this thesis has been that
\begin{enumerate*}
\item To formulate results in the theory of persistence-based topological inference, one first needs to select notions of similarly between filtrations and persistence modules.
\item Much of the substantive theoretical work to be done in the study of topological persistence lies in understanding what the right choices of those notions of similarity are.
\end{enumerate*}
Indeed, we have spent the bulk of our effort in this thesis developing the theory of interleavings and interleaving distances.  The theory we have developed makes a strong, if still incomplete, case that interleavings are natural tools with which to frame the theory of topological inference of multifiltrations.

Our discussions of directions for future work in Sections~\ref{DiscussionSection}, \ref{Sec:WeakInterleavingsQuestions}, and \ref{Sec:InferenceDiscussion} make it clear that there is much more work to be done on the theory of interleavings of multidimensional persistence modules, on the theory of weak interleavings of multidimensional filtrations, and on the statistical aspects of persistence-based topological inference.  I believe that progress in these directions, seen as a part of the larger program of fleshing out the statistical foundations of topological data analysis, stands a good chance of contributing something fundamental to the way scientists, engineers, and statisticians think about and perform the qualitative analysis data.

%% file: Part_I/T_Coherence_2.tex
\section{The Coherence of $B_n$}\label{CoherenceSection}

\subsection{Coherence: Basic Definitions and Results}

$k[x_1,...,x_n]$ is well known to be a Noetherian ring.  Finitely generated modules over Noetherian rings have some very nice algebraic properties.  Here we define a standard weakening of the Noetherian property called coherence.  Analogues of many of the same nice algebraic properties that hold for finitely generated modules over Noetherian rings hold for finitely presented modules over coherent rings.  In particular, we have Corollary~\ref{NiceCoherenceProperty}, which we will use in Appendix~\ref{MinimalPresentationAppendix} to prove Theorem~\ref{MinimalPresentationTheorem}.  

\begin{definition} For $R$ a ring, we say an $R$-module $M$ is {\bf coherent} if $M$ is finitely generated and every finitely generated submodule of $M$ is finitely presented.  We say a ring $R$ is coherent if it is a coherent module over itself. \end{definition}

Coherent commutative rings and coherent modules are well studied; the following results are standard.  The reader may refer to \cite{glaz1989commutative} for the proofs.

\begin{prop}\label{NoetherianRingIsCoherent} If $R$ is a Noetherian ring then $R$ is coherent. \end{prop}

\begin{thm} If $R$ is a coherent ring then every finitely presented $R$-module is coherent. \end{thm}

\begin{thm} If $f:M\to N$ is a morphism between coherent $R$-modules $M$ and $N$ then $\ker(f)$, $\im(f)$, and $\coker(f)$ are coherent $R$-modules. \end{thm}

Combining these last two theorems immediately gives

\begin{cor}\label{NiceCoherenceProperty}If $R$ is a coherent ring and $f:M\to N$ is a morphism between fintely presented $R$-modules $M$ and $N$, then $\ker(f)$, $\im(f)$, and $\coker(f)$ are finitely presented. \end{cor}

\subsection{The Ring $B_n$ is Coherent}
\begin{thm}\label{CoherenceThm} For any $n\in {\mathbb N}$, $B_n$ is coherent. \end{thm}

\begin{proof} The key to the proof is the following theorem:

\begin{thm}[\glazcitationhack]\label{GlazThm} Let $\{R_{\alpha}\}_{\alpha \in S}$ be a directed system of rings and let $R=\lim_{\to} R_{\alpha}$.  Suppose that for ${\alpha}\leq {\beta}$, $R_{\beta}$ is a flat $R_{\alpha}$ module and that $R_{\alpha}$ is coherent for every $\alpha$.  Then $R$ is a coherent ring. \end{thm}

First, recall that $\R$ is a vector space over $\Q$.  We'll say that $a_1,...,a_l\in \R_{\geq 0}$ are {\it rationally independent} if they are linearly independent as vectors in $\R$ over the field $\Q$.

We next extend this definition to vectors in $\R^n_{\geq 0}$: We say a finite set $V\subset \R^n_{\geq 0}$ is {\it rationally independent} if
\begin{enumerate*}
\item $V$ is the union of sets $V_1,...,V_n$, where each element of $V_i$ has a non-zero $i^{th}$ coordinate and all other coordinates are equal to zero.  
\item For any $i$, if $a_1,...,a_l$ are the non-zero coordinates of the elements of $V_i$ (listed with multiplicity), then $a_1,...,a_l$ are rationally independent in the sense defined above.
\end{enumerate*}

We define an {\it n-grid} to be a monoid generated by some rationally independent set $V\subset \R^n_{\geq 0}$.  Denote the $n$-grid generated by the rationally independent set $V$ as $\Gamma(V)$.  $\Gamma(V)$ is a submonoid of $\R^n_{\geq 0}$. 

\begin{lem}\label{FirstCoherenceLem} If $V$ is a rationally independent set, then the $n$-grid generated by $V$ is isomorphic to $\Z_{\geq 0}^{|V|}$.  \end{lem}

\begin{proof}The proof is straightforward; we omit it. \end{proof}

As noted in Section~\ref{MonoidRings}, for any $m \in {\mathbb N}$, $k[\Z_{\geq 0}^{m}]\cong k[x_1,...,x_m]$.  As the latter ring is Noetherian, it is coherent by proposition~\ref{NoetherianRingIsCoherent}.  Thus if $G$ is an $n$-grid, $k[G]$ is coherent.


\begin{lem}\label{RatIndSetExistenceLem} For any finite set $A\subset \R_{\geq 0}$, there's a rationally independent set $B\subset \R_{\geq 0}$ such that $A$ lies in the monoid generated by $B$. \end{lem} 

\begin{proof} We proceed by induction on the number of elements $l$ in the set $A$.  The base case is trivial.  Now assume the result holds for sets of order $l-1$.  Write $A=\{a_1,...,a_l\}$.  By the induction hypothesis there exists a finite rationally independent set $A'=\{a'_1,...,a'_m\}$ such that $\{a_1,...,a_{l-1} \}$ lies in $\Gamma(A')$.  If $A' \cup {a_l}$ is rationally independent, take $B=A' \cup {a_l}$.  Otherwise $a_l=q_1a'_1+...+q_{m-1}a'_{m-1}$ for some $q_1,...,q_{l-1}\in \Q$; we may take $B=\{q'_1a'_1,...,q'_{l-1}a'_{m-1}\}$, where $q'_i=1/b_i$ for some $b_i\in {\mathbb N}$ such that $q_i=a/b_i$ for some $a\in \Z_{\geq 0}$. \end{proof}
                        
\begin{lem} The set of $n$-grids forms a directed system under inclusion with direct limit $\R^n_{\geq 0}$.  \end{lem}

\begin{proof} To show that the set of $n$-grids forms a directed system, we need that given two n-grids $\G_1$ and $G_2$, there's an n-grid $G_3$ such that $G_1\subset G_3$ and $G_2\subset G_3$.  This follows readily from Lemma~\ref{RatIndSetExistenceLem}; we leave the details to the reader. Any element of $\R^n_{\geq 0}$ lies in an $n$-grid, so $\R^n_{\geq 0}$ must be the colimit of the directed system.  \end{proof}

For a monoid $A$ and a submonoid $A'\subset A$, we have $k[A']\subset k[A]$.  This implies the following:

\begin{lem}\label{DirectedSystemLemma} The set of rings $\{k[G]|G$ is an $n$-grid$\}$ has the structure of a directed system induced by the directed system structure on the set of $n$-grids, and $B_n$ is the direct limit of this directed system. \end{lem}

\begin{prop}\label{FreenessProp} Given two positive n-grids $G',G$ with $G' \subset G$, $k[G]$ is a free $k[G']$ module. \end{prop}

\begin{proof} We begin by establishing a couple of lemmas.

\begin{lem}\label{ExtensionLem} For any rationally independent set $V'$ and $n$-grid $A$ containing $\Gamma(V')$, there is a rationally independent set $V$ such that $A=\Gamma(V)$ and such that for each $a\in V'$, $V$ contains an element of the form $a/b$ for some $b\in {\mathbb N}$. \end{lem}

We call $V$ an {\it extension} of $V'$.   

\begin{proof} The proof of Lemma~\ref{ExtensionLem} is similar to the proof of Lemma~\ref{RatIndSetExistenceLem}; we omit it. \end{proof}
  
Let ${\mathcal S}$ denote the set of maximal sets of the form $g+G'\equiv\{g+g'|g'\in G'\}$ for some $g\in G$.   

\begin{lem}\label{PartitionLemma} The sets ${\mathcal S}$ form a partition of $G$. \end{lem}

\begin{proof}  It's enough to show that if $g_1+G',g_2+G'\in {\mathcal S}$ and $g_1+G'\cap g_2+G'\ne \emptyset$, then $g_1+G'=g_2+G'$.

Let $V'$ be a rationally independent set with $\Gamma(V')=G'$, and let $V$ be an extension of $V'$ with $\Gamma(V)=G$.    
Write $V'=\{v_1,...,v_l\}$ and $V=\{v_1/b_1,...,v_l/b_l,v_{l+1},...,v_m\}$ for some $b_1,...,b_l\in {\mathbb N}$. 

Assume there exist $g'_1,g'_2\in G'$ such that $g_1+g'_1=g_2+g'_2$.  We'll show that there then exists an element $g_3\in G$ such that $g_1,g_2\in g_3+G$.  By the maximality of $g_1+G'$ and $g_2+G'$, this implies $g_1+G'=g_2+G'$, as needed.  We write 
\begin{align*}
g_1&=y_1v_1/b_1 +\cdots +y_lv_l/b_l + y_{l+1} v_{l+1}+\cdots +y_m v_m,\\ 
g_2&=z_1v_1/b_1 +\cdots +z_lv_l/b_l+ z_{l+1} v_{l+1}+\cdots +z_m v_m,\\
g'_1&=y'_1 v_1+\cdots +y'_l v_l,\\
g'_2&=z'_1 v_1+\cdots +z'_l v_l.
\end{align*}
for some $y_1,...,y_m,z_1,...,z_m,y'_1,...,y'_l,z'_1,...,z'_l\in \Z$.
By the rational independence of $V$ and the fact that $g_1+g'_1=g_2+g'_2$, we have that $y_i=z_i$ for $l+1\leq i\leq m$.  

Define
\begin{align*}
g_3&=\min(y_1,z_1)v_1/b_1 +\cdots +\min(y_l,z_l)v_l/b_1 +y_{l+1} v_{l+1}+\cdots +y_m v_m,\\
g''_1&=(y_1-\min(y_1,z_1))v_1/b_1 +\cdots +(y_l-\min(y_l,z_l))v_l/b_l, \\
g''_2&=(z_1-\min(y_1,z_1))v_1/b_1 +\cdots +(z_l-\min(y_l,z_l)) v_l/b_l.
\end{align*}
$g_3+g''_1=g_1$ and $g_3+g''_2=g_2$, so if we can show that $g''_1,g''_2\in G'$ we are done.

By the rational independence of $V$ and the fact that $g_1+g'_1=g_2+g'_2$, for $1\leq i\leq l$ we have that $\min(y_i,z_i)/b_i+\max(y'_i,z'_i)=y_i/b_i+y'_i$.  This implies that $\max(y'_i,z'_i)-y'_i=(y_i-\min(y_i,z_i))/b_i$.  In particular, the term on the right hand side lies in $\Z_{\geq 0}$.  Thus $g''_1\in G'$.  The same argument shows $g''_2\in G'$. \end{proof}

Now we are ready to complete the proof of Proposition~\ref{FreenessProp}.  It's easy to see that for any $s\in {\mathcal S}$, the natural action of $G'$ on $s$ extends to give $k[s]$ the structure of a free $k[G']$ module of rank 1.  It follows from Lemma~\ref{PartitionLemma} that the sets $\{k[s]\}_{s\in {\mathcal S}}$ have trivial intersection as $k[G']$-submodules of $k[G]$.  We then have that as a $k[G']$ module, $k[G]=\oplus_{s\in {\mathcal S}} k[s]$, and so in particular $k[G]$ is a free $k[G']$-module, as we wanted to show. \end{proof}

Given Lemma~\ref{DirectedSystemLemma} and Proposition~\ref{FreenessProp}, Theorem~\ref{GlazThm} applies to give that $B_n$ is coherent, since free modules are flat \cite{eisenbud1995commutative}.  \end{proof}

%% file: Part_I/T_Minimal_Presentations_Algebra_3.tex
 \section{Minimal Presentations of $B_n$-persistence Modules}\label{MinimalPresentationAppendix}
This section is devoted to the proof of Theorem~\ref{MinimalPresentationTheorem}.

\subsection{Free Hulls}
We first observe that some standard results about resolutions and minimal resolutions of modules over local rings adapt to $B_n$-persistence modules.  We'll only be interested in the specialization of such results to the $0^{th}$ modules in a free resolution, and for the sake of simplicity we phrase the results only for this special case.  However, the results discussed here do extend to statements about free resolutions of finitely presented $B_n$-persistence modules.

Let $\mathfrak m$ denote the ideal of $B_n$ generated by the set \[\{v\in B_n|\text{v is homogeneous and }gr(v)>0\}.\]  Define a {\it free hull} of $M$ to be a free cover $(F_M,\rho_M)$ such that $\ker(\rho_M)\subset {\mathfrak m}F_M$.

Nakayama's lemma \cite{eisenbud1995commutative} is a key ingredient in the proofs of the results about free resolutions over local rings that we would like to adapt to our setting.  To adapt these proofs, we need an $n$-graded version of Nakayama's lemma.

\begin{lem}[Nakayama's Lemma for Persistence Modules]\label{PersistenceNakayama} Let $M$ be a finitely generated $B_n$-persistence module.  If $y_1,...,y_m\in M
$ have images in $M/{\mathfrak m}M$ that generate the quotient, then $y_1,...,y_m$ generate $M$. \end{lem}
\begin{proof} The usual Proof of Nakayama's lemma \cite{eisenbud1995commutative} carries over with only minor changes.\end{proof}

\begin{lem}\label{FreeHullIsMinimal} A free cover $(F_M,\rho_M)$ of a finitely generated $B_n$-persistence module $M$ is a free hull iff a basis for $F_M$ maps under $\rho_M$ to a minimal set of generators for $M$. \end{lem}

\begin{proof} Given the adaptation Lemma~\ref{PersistenceNakayama} of Nakayama's lemma to our setting, the proof of \cite[Lemma 19.4]{eisenbud1995commutative} gives the result.\end{proof}

It follows easily from Lemma~\ref{FreeHullIsMinimal} that a free hull exists for any finitely generated $B_n$-persistence module $M$.
Corollary~\ref{FreeHullIsUnique} below gives a uniqueness result for free hulls.

\begin{thm}\label{FreeHullsAsSummands} If $(F_M,\rho_M)$ is a free hull of a finitely presented $B_n$-persistence module $M$ and $(F'_M,\rho'_M)$ is any free cover of $M$, then $F_M$ includes as a direct summand of $F'_M$ in such a way that $F'_M\cong F_M \oplus F''_M$ for some free module $F''_M$, and $\ker(\rho'_M)=\ker(\rho_M)\oplus F''_M \subset F_M \oplus F''_M$. \end{thm}

\begin{proof}[Sketch of Proof.] 
The statement of the theorem is the specialization to $0^{th}$ modules in the free resolutions of $M$ of an adaptation of \cite[Theorem 20.2]{eisenbud1995commutative} to our $B_n$-persistence setting.  To modify Eisenbud's proof of \cite[Theorem 20.2]{eisenbud1995commutative} to obtain a proof of Theorem~\ref{FreeHullsAsSummands}, one needs to invoke the coherence of $B_n$ and use Corollary~\ref{NiceCoherenceProperty} to show that $\ker(\rho_M)$ is finitely generated.  Given this, the strategy of proof adapts in a straightforward way.
\end{proof}
  
\begin{cor}[Uniqueness of free hulls]\label{FreeHullIsUnique} If $M$ is a finitely presented $B_n$-persistence module, and $(F_M,\rho_M)$, $(F'_M,\rho'_M)$ are two free hulls of $M$, then there is an isomorphism from $F_M$ to $F'_M$ which is a lift of the identity map of $M$. \end{cor} 

\begin{proof} By Theorem~\ref{FreeHullsAsSummands}, we can identify $F_M$ with a submodule of $F'_M$ in such a way that $F'_M=F_M \oplus F''_M$ for some free module $F''_M$ and $\ker(\rho'_M)=\ker(\rho)\oplus F''_M \subset F_M \oplus F''_M$.  Since $F'_M$ is a free hull, we must have $\ker(\rho'_M)\in {\mathfrak m}F'_M$, which implies $F''_M=0$.  The result follows.  \end{proof}

\begin{cor}\label{GradesOfMinimalSetOfGeneratorsUnique} If $M$ is a finitely presented $B_n$-persistence module and $B,B'$ are two minimal sets of generators for $M$, then $gr(B)=gr(B')$.\end{cor}

\begin{proof} This follows from Corollary~\ref{FreeHullIsUnique} and Lemma~\ref{FreeHullIsMinimal}.\end{proof}

\subsection{Proof of Theorem~\ref{MinimalPresentationTheorem}}
Recall that a minimal presentation $\langle G|R \rangle$ of a $B_n$-persistence module $M$ is one such that
\begin{enumerate*}
\item the quotient $\langle G\rangle \to \langle G\rangle/\langle R\rangle$ maps $G$ to a minimal set of generators for $\langle G\rangle/\langle R\rangle$.
\item $R$ is a minimal set of generators for $\langle R \rangle$.
\end{enumerate*}

Let $M$ be a finitely presented $B_n$-persistence module.  Let $\langle G|R \rangle$ be a minimal presentation of $M$.  We need to show that for any other presentation $\langle G'|R' \rangle$ of $M$, $gr(G)\leq gr(G')$ and $gr(R)\leq gr(R')$.  

Let $\psi:\langle G\rangle/\langle R \rangle \to M$ and $\psi':\langle G'\rangle/\langle R' \rangle \to M$ be isomorphisms, let $\pi:\langle G \rangle \to \langle G\rangle/\langle R \rangle$ and $\pi':\langle G' \rangle \to \langle G'\rangle/\langle R' \rangle$ be the quotient homomorphisms, let $\rho=\psi\circ \pi$, and let $\rho'=\psi'\circ \pi'$.  Then by Lemma~\ref{FreeHullIsMinimal}, $(\langle G \rangle,\rho)$ is a free hull of $M$, and $(\langle G' \rangle,\rho')$ is a free cover of $M$.  

By Theorem~\ref{FreeHullsAsSummands}, $\langle G \rangle$ includes as a direct summand of $\langle G' \rangle$.  The image of $G$ under this inclusion can be extended to a basis for $\langle G' \rangle$.  Recall that if $B$ and $B'$ are two bases for a free $B_n$-persistence module $F$, then $gr(B)=gr(B')$.  We thus have that $gr(G)\leq gr(G')$.

Theorem~\ref{FreeHullsAsSummands} also implies that $\langle R'\rangle\cong \langle R\rangle \oplus F$ for some free $B_n$-persistence module $F$.  Let $B$ be a basis for $F$.  Then $R \cup B$ is a minimal set of generators for $\langle R\rangle \oplus F$.   Let $R''$ denote the image of $R'$ under an isomorphism from $\langle R'\rangle$ to $\langle R\rangle \oplus F$ and let $p:\langle R\rangle \oplus F\to \langle R\rangle$ denote projection onto the first summand.  Since $p$ is surjective, $p(R'')$ is a set of homogeneous generators for $\langle R\rangle$.  

Since $\langle G \rangle$ and $M$ are finitely presented, by Corollary~\ref{NiceCoherenceProperty} $\ker(\rho)=\langle R\rangle$ is also finitely presented.  Then by Corollary~\ref{GradesOfMinimalSetOfGeneratorsUnique}, $gr(R)\leq gr(p(R''))$.  Since $gr(p(R''))\leq gr(R'')=gr(R')$ we have that $gr(R)\leq gr(R')$.\qed

%% file: Lesnick-Thesis-For-Arxiv.bbl
\begin{thebibliography}{10}

\bibitem{adams2011morse}
H.~Adams, A.~Atanasov, and G.~Carlsson.
\newblock Morse theory in topological data analysis.
\newblock {\em Arxiv preprint arXiv:1112.1993}, 2011.

\bibitem{bredon1993topology}
G.E. Bredon.
\newblock {\em {Topology and geometry}}.
\newblock Springer Verlag, 1993.

\bibitem{burago2001course}
D.~Burago, Y.~Burago, S.~Ivanov, and American~Mathematical Society.
\newblock {\em A course in metric geometry}.
\newblock American Mathematical Society Providence, 2001.

\bibitem{carlsson2009topology}
G.~Carlsson.
\newblock {Topology and data}.
\newblock {\em American Mathematical Society}, 46(2):255--308, 2009.

\bibitem{carlsson2009zigzag}
G.~Carlsson, V.~De~Silva, and D.~Morozov.
\newblock {Zigzag persistent homology and real-valued functions}.
\newblock In {\em Proceedings of the 25th annual symposium on Computational
  geometry}, pages 247--256. ACM, 2009.

\bibitem{carlsson2008local}
G.~Carlsson, T.~Ishkhanov, V.~De~Silva, and A.~Zomorodian.
\newblock {On the local behavior of spaces of natural images}.
\newblock {\em International Journal of Computer Vision}, 76(1):1--12, 2008.

\bibitem{carlsson2009theory}
G.~Carlsson and A.~Zomorodian.
\newblock {The theory of multidimensional persistence}.
\newblock {\em Discrete and Computational Geometry}, 42(1):71--93, 2009.

\bibitem{chazal2009proximity}
F.~Chazal, D.~Cohen-Steiner, M.~Glisse, L.J. Guibas, and S.Y. Oudot.
\newblock {Proximity of persistence modules and their diagrams}.
\newblock In {\em Proceedings of the 25th annual symposium on Computational
  geometry}, pages 237--246. ACM, 2009.

\bibitem{chazal2009gromov}
F.~Chazal, D.~Cohen-Steiner, L.J. Guibas, F.~M{\'e}moli, and S.Y. Oudot.
\newblock {Gromov-Hausdorff stable signatures for shapes using persistence}.
\newblock In {\em Proceedings of the Symposium on Geometry Processing}, pages
  1393--1403. Eurographics Association, 2009.

\bibitem{chazal2011geometric}
F.~Chazal, D.~Cohen-Steiner, and Q.~M{\'e}rigot.
\newblock Geometric inference for probability measures.
\newblock {\em Foundations of Computational Mathematics}, pages 1--19, 2011.

\bibitem{chazal2009persistence}
F.~Chazal, L.J. Guibas, S.~Oudot, and P.~Skraba.
\newblock {Persistence-based clustering in Riemannian manifolds}.
\newblock {\em INRIA Technical Report}, 2009.

\bibitem{chazal2009analysis}
F.~Chazal, L.J. Guibas, S.Y. Oudot, and P.~Skraba.
\newblock {Analysis of scalar fields over point cloud data}.
\newblock In {\em Proceedings of the twentieth Annual ACM-SIAM Symposium on
  Discrete Algorithms}, pages 1021--1030. Society for Industrial and Applied
  Mathematics, 2009.

\bibitem{chazal2008towards}
F.~Chazal and S.Y. Oudot.
\newblock {Towards persistence-based reconstruction in Euclidean spaces}.
\newblock In {\em Proceedings of the twenty-fourth annual symposium on
  Computational geometry}, pages 232--241. ACM, 2008.

\bibitem{cohen2007stability}
D.~Cohen-Steiner, H.~Edelsbrunner, and J.~Harer.
\newblock {Stability of persistence diagrams}.
\newblock {\em Discrete and Computational Geometry}, 37(1):103--120, 2007.

\bibitem{cordier1986vogt}
J.M. Cordier and T.~Porter.
\newblock VogtÕs theorem on categories of homotopy coherent diagrams.
\newblock In {\em Math. Proc. Camb. Phil. Soc}, volume 100, pages 65--90.
  Cambridge Univ Press, 1986.

\bibitem{de2009persistent}
V.~De~Silva and M.~Vejdemo-Johansson.
\newblock Persistent cohomology and circular coordinates.
\newblock In {\em Proceedings of the 25th annual symposium on Computational
  geometry}, pages 227--236. ACM, 2009.

\bibitem{derksen2005quiver}
H.~Derksen and J.~Weyman.
\newblock {Quiver representations}.
\newblock {\em Notices of the AMS}, 52(2):200--206, 2005.

\bibitem{dummit1999abstract}
D.S. Dummit and R.M. Foote.
\newblock {\em Abstract algebra}.
\newblock Wiley, 1999.

\bibitem{dwyer1997model}
W.G. Dwyer, P.~Hirschhorn, and D.~Kan.
\newblock Model categories and more general abstract homotopy theory: a work in
  what we like to think of as progress.
\newblock {\em preprint}, 1997.

\bibitem{dwyer1995homotopy}
W.G. Dwyer and J.~Spalinski.
\newblock Homotopy theories and model categories.
\newblock {\em Handbook of algebraic topology}, pages 73--126, 1995.

\bibitem{d2010natural}
M.~d'Amico, P.~Frosini, and C.~Landi.
\newblock Natural pseudo-distance and optimal matching between reduced size
  functions.
\newblock {\em Acta applicandae mathematicae}, 109(2):527--554, 2010.

\bibitem{edelsbrunner2008persistent}
H.~Edelsbrunner and J.~Harer.
\newblock {Persistent homology—a survey}.
\newblock In {\em Surveys on discrete and computational geometry: twenty years
  later: AMS-IMS-SIAM Joint Summer Research Conference, June 18-22, 2006,
  Snowbird, Utah}, volume 453, page 257. Amer Mathematical Society, 2008.

\bibitem{edelsbrunner2010computational}
H.~Edelsbrunner and J.~Harer.
\newblock {\em Computational topology: an introduction}.
\newblock American Mathematical Society, 2010.

\bibitem{efron1993introduction}
B.~Efron and R.~Tibshirani.
\newblock {\em An introduction to the bootstrap}, volume~57.
\newblock Chapman \& Hall/CRC, 1993.

\bibitem{eisenbud1995commutative}
D.~Eisenbud.
\newblock {\em {Commutative algebra with a view toward algebraic geometry}}.
\newblock Springer, 1995.

\bibitem{ghrist2008barcodes}
R.~Ghrist.
\newblock Barcodes: The persistent topology of data.
\newblock {\em BULLETIN-AMERICAN MATHEMATICAL SOCIETY}, 45(1):61, 2008.

\bibitem{gine2002rates}
E.~Gin{\'e} and A.~Guillou.
\newblock Rates of strong uniform consistency for multivariate kernel density
  estimators.
\newblock {\em Annales de l'Institut Henri Poincare (B) Probability and
  Statistics}, 38(6):907--921, 2002.

\bibitem{glaz1989commutative}
S.~Glaz.
\newblock {\em {Commutative coherent rings}}.
\newblock Springer, 1989.

\bibitem{hatcher2002algebraic}
A.~Hatcher.
\newblock {\em {Algebraic topology}}.
\newblock Cambridge Univ Press, 2002.

\bibitem{henry2009kernel}
G.~Henry and D.~Rodriguez.
\newblock Kernel density estimation on riemannian manifolds: Asymptotic
  results.
\newblock {\em Journal of Mathematical Imaging and Vision}, 34(3):235--239,
  2009.

\bibitem{jech2003set}
T.J. Jech.
\newblock {\em Set theory}.
\newblock Springer Verlag, 2003.

\bibitem{lesnick2011optimality}
M.~Lesnick.
\newblock The optimality of the interleaving distance on multidimensional
  persistence modules.
\newblock {\em Arxiv preprint arXiv:1106.5305}, 2011.

\bibitem{mac1998categories}
S.~Mac~Lane.
\newblock {\em {Categories for the working mathematician}}, volume~5.
\newblock Springer verlag, 1998.

\bibitem{nicolau2011topology}
M.~Nicolau, A.J. Levine, and G.~Carlsson.
\newblock Topology based data analysis identifies a subgroup of breast cancers
  with a unique mutational profile and excellent survival.
\newblock {\em Proceedings of the National Academy of Sciences}, 108(17):7265,
  2011.

\bibitem{niyogi2009finding}
P.~Niyogi, S.~Smale, and S.~Weinberger.
\newblock Finding the homology of submanifolds with high confidence from random
  samples.
\newblock {\em Twentieth Anniversary Volume:}, pages 1--23, 2009.

\bibitem{pelletier2005kernel}
B.~Pelletier.
\newblock Kernel density estimation on riemannian manifolds.
\newblock {\em Statistics \& probability letters}, 73(3):297--304, 2005.

\bibitem{quillen1967homotopical}
D.G. Quillen.
\newblock {\em Homotopical algebra}, volume~43.
\newblock Springer Verlag, 1967.

\bibitem{rinaldo2010generalized}
A.~Rinaldo and L.~Wasserman.
\newblock Generalized density clustering.
\newblock {\em The Annals of Statistics}, 38(5):2678--2722, 2010.

\bibitem{scott1992multivariate}
D.W. Scott.
\newblock {\em Multivariate density estimation}, volume 139.
\newblock Wiley Online Library, 1992.

\bibitem{silverman1981using}
B.W. Silverman.
\newblock Using kernel density estimates to investigate multimodality.
\newblock {\em Journal of the Royal Statistical Society. Series B
  (Methodological)}, pages 97--99, 1981.

\bibitem{singh2007topological}
G.~Singh, F.~M{\'e}moli, and G.~Carlsson.
\newblock Topological methods for the analysis of high dimensional data sets
  and 3d object recognition.
\newblock In {\em Eurographics Symposium on Point-Based Graphics}, volume~22,
  2007.

\bibitem{singh2008topological}
G.~Singh, F.~Memoli, T.~Ishkhanov, G.~Sapiro, G.~Carlsson, and D.L. Ringach.
\newblock Topological analysis of population activity in visual cortex.
\newblock {\em Journal of vision}, 8(8), 2008.

\bibitem{skraba2010persistence}
P.~Skraba, M.~Ovsjanikov, F.~Chazal, and L.~Guibas.
\newblock {Persistence-based segmentation of deformable shapes}.
\newblock In {\em Computer Vision and Pattern Recognition Workshops (CVPRW),
  2010 IEEE Computer Society Conference on}, pages 45--52. IEEE, 2010.

\bibitem{wasserman2004all}
L.~Wasserman.
\newblock {\em All of statistics: a concise course in statistical inference}.
\newblock Springer Verlag, 2004.

\bibitem{webb1985decomposition}
C.~Webb.
\newblock {Decomposition of graded modules}.
\newblock {\em American Mathematical Society}, 94(4), 1985.

\bibitem{zomorodian2005computing}
A.~Zomorodian and G.~Carlsson.
\newblock {Computing persistent homology}.
\newblock {\em Discrete and Computational Geometry}, 33(2):249--274, 2005.

\bibitem{zomorodian2008localized}
A.~Zomorodian and G.~Carlsson.
\newblock Localized homology.
\newblock {\em Computational Geometry}, 41(3):126--148, 2008.

\end{thebibliography}
